\documentclass[twoside, a4paper, leqno]{article}

\usepackage{amsmath, amssymb, amsthm}
\usepackage{mathtools}
\usepackage{mathrsfs}
\usepackage[sort&compress, square, comma, numbers]{natbib} 
\usepackage{hyperref}
\usepackage{bm}
\usepackage{graphicx}        
\usepackage{multicol}        
\usepackage[bottom]{footmisc}
\usepackage{comment}
\usepackage{geometry}

\newtheorem{theorem}{Theorem}
\newtheorem{corollary}{Corollary}
\newtheorem{problem}{Problem}
\newtheorem{recommendation}{Mesh Refinement Recommendation}
\newtheorem{algorithm}{Algorithm}

\newcommand{\NN}{\mathbb{N}}

\newcommand{\RR}{\mathbb{R}}

\newcommand{\rmd}{\,\mathrm{d}}
\newcommand{\rme}{\mathrm{e}}

\newcommand{\ce}{\coloneqq}
\newcommand{\ec}{\eqqcolon}

\newcommand{\pd}[2]{\frac{\partial #1}{\partial #2}}
\newcommand{\lap}{\Delta}
\newcommand{\rot}{\nabla\times}
\renewcommand{\div}{\nabla\cdot}
\newcommand{\grad}{\nabla}

\newcommand{\x}{\mathbf{x}}
\newcommand{\y}{\mathbf{y}}
\newcommand{\bnu}{\mathbf{n}}
\newcommand{\eps}{\varepsilon}
\newcommand{\bphi}{\boldsymbol{\phi}}
\newcommand{\E}{\mathbf{E}}
\renewcommand{\H}{\mathbf{H}}

\renewcommand{\P}{\mathbf{P}}
\newcommand{\G}{\mathbf{G}}
\renewcommand{\l}{ {\boldsymbol{\lambda} }}

\newcommand{\oeps}{{\bar{\eps}}}
\newcommand{\oE}{{\bar{\E}}}
\newcommand{\ol}{{\bar{\l}}}
\newcommand{\ih}{{\Pi_h}}
\newcommand{\rh}{{r_h}}
\newcommand{\maxwell}{\mathscr{D}}
\newcommand{\adjoint}{\mathscr{A}}
\renewcommand{\th}{\mathcal{T}_h}
\newcommand{\It}{\mathcal{I}_\tau}
\newcommand{\ot}{ {\Omega_T} }
\newcommand{\gt}{ {\Gamma_T} }
\newcommand{\diam}{\operatorname{diam}}
\newcommand{\vht}{V_h}
\newcommand{\uht}{U_h}

\newcommand{\abs}[1]{\left\lvert #1 \right\rvert}
\newcommand{\norm}[1]{\left\lVert{\textstyle #1 }\right\rVert}
\newcommand{\sca}[2]{\left\langle{\textstyle #1},\,{\textstyle #2}\right\rangle}
\newcommand{\jm}[2]{\left\{\textstyle{ #1 }\right\}_{\mathrm{#2}}}
\newcommand{\aj}[2]{\left[\textstyle{ #1 }\right]_{\mathrm{#2}}}

\graphicspath{   
{FIGURES/}
}

\begin{document}

\author{John Bondestam Malmberg 
\thanks{
Department of Mathematical Sciences, Chalmers University of Technology and Gothenburg University, SE-42196 Gothenburg, Sweden, e-mail: \texttt{john.bondestam.malmberg@chalmers.se}
}
\and
Larisa Beilina
\thanks{
Department of Mathematical Sciences, Chalmers University of Technology and Gothenburg University, SE-42196 Gothenburg, Sweden, e-mail: \texttt{larisa@chalmers.se}
}
}

\title{An Adaptive Finite Element Method in Quantitative Reconstruction of Small Inclusions from Limited Observations}

\maketitle

\begin{abstract}
We consider a coefficient inverse problem for the dielectric permittivity in Maxwell's equations, with data consisting of boundary measurements of one or two backscattered or transmitted waves. The problem is treated using a Lagrangian approach to the minimization of a Tikhonov functional, where an adaptive finite element method forms the basis of the computations. A new a posteriori error estimate for the coefficient is derived. The method is tested successfully in numerical experiments for the reconstruction of two, three, and four small inclusions with low contrast, as well as the reconstruction of a superposition of two Gaussian functions.
\end{abstract}

\section{Introduction} \label{intro}
In this note we study an adaptive finite element method for the reconstruction of a dielectric permittivity function $\eps=\eps(\x)$, $\x=(x_1,\,x_2,\,x_3)\in\Omega$, where $\Omega\subset\RR^3$ is a bounded domain with (piecewise) smooth boundary $\Gamma$. We consider data consisting of boundary measurements of a small number (one or two) backscattered or transmitted waves. This is a coefficient inverse problem (CIP) for Maxwell's equations, where the dielectric permittivity function $\eps$, acting as the coefficient in the equations, characterizes an inhomogeneous, isotropic, non-magnetic, non-conductive medium in $\Omega$. Although we are in this note concerned with the reconstruction of a real-valued coefficient $\eps$, our future applications are, eventually, in medical imaging, such as microwave imaging of breast cancer, and early diagnosis of stroke (see \cite{fhp06, mfrdp07, mpr95, gmkdp12} for details).

The method studied is based on a Lagrangian approach to the minimization of a Tikhonov functional, where the functions involved are approximated by piecewise polynomials in a finite element method. Such an approach to coefficient inverse problems of the type we consider has previously been studied extensively in the context of a two-stage procedure in \cite{b11, bkk10, btkm14, btkm14b}. The main application considered in those publications was detection of explosives. With such applications, the low amount of data one can expect -- usually only backscattered data for a single incident wave of one frequency -- makes the problem of reconstruction challenging. In the context of medical imaging, it is possible to gather more data, and the challenges are instead low contrast between healthy and non-healthy tissue (seldom larger than 2 in microwave imaging of malign tumors, see for example \cite{jzlj94}), and small size of inclusions.

In the theoretical part of this paper, we derive a new, direct a posteriori error estimate for the dielectric permittivity function to be reconstructed. Qualitative computable, or a posteriori, error estimates are an essential tool for finite element based adaptive algorithms in the optimization approach to our inverse problem. Previously, in \cite{b11}, such an estimate was given for our coefficient inverse problem. However, this estimate was an indirect one, estimating the size of an error in the computed Lagrangian, as apposed to a direct estimate of the error in the computed permittivity as presented here. A similar estimate, also in the computed Lagrangian, was shown for a modified Maxwell system in \cite{m14}.

We illustrate our theoretical results with several numerical examples. With the above-mentioned aspects of contrast and size of inclusions in mind, we present reconstructions of two, three, and four different small inclusions, respectively, of low contrast. We also evaluate how variations in the type of data collected affect the reconstruction. More precisely, we consider the effect of working with only backscattered data, with only transmitted data, with both backscattered and transmitted data, as well as backscattered data from two anti-parallel incident waves. In addition to these reconstructions of inclusions, we also present a reconstruction of a more complicated function, to wit, a superposition of two gaussians.

Here is an outline of the remaining part of this note: In the Section~\ref{dir_inv} we present the mathematical formulations of the direct and inverse problems and present the basic results prior to discretization of the problems. In Section~\ref{error_analysis} we state the finite element formulations, perform the error analysis, and summarize the results in terms of a mesh refinement strategy.  The adaptive algorithm for solving the inverse problem is described in Section~\ref{algorithm}, and numerical examples are given in Section~\ref{numex}. Section~\ref{conclusion} concludes the paper.

\section{The direct and inverse problems} \label{dir_inv}
Before proceeding with the mathematical statement of the problem, we introduce some notation. For the bounded polygonal domain $\Omega\subset \RR^3$ with boundary $\Gamma$, we write $\ot\ce \Omega\times(0,\,T)$ and $\gt\ce\Gamma\times(0,\,T)$, where $T>0$ is a (sufficiently large) fixed time. If $X\subset\RR^n$, $n\in\NN$, is a domain, we denote by $\sca{\cdot}{\cdot}_X$ and $\norm{\cdot}_X$ the $L_2$-inner product and norm, respectively, over the domain $X$.

Let 
\begin{equation*}
V^\eps_0\ce \{v\in C(\bar{\Omega}): \grad{v}\in[L_\infty(\Omega)]^3\}.
\end{equation*}
In the spirit of \cite{bkk10}, we define a finite dimensional subspace of piecewise linears 
\begin{equation*}
V^\eps\ce \{v\in V^\eps_0: v\rvert_K\in P^1(K)\,\forall K\in \mathcal{T}_\delta\},
\end{equation*}
where $\mathcal{T}_\delta= \{K\}$ is a very fine triangulation of $\bar\Omega$ and $P^1(K)$ denotes the set of polynomials of degree no greater than 1 over $K$. We equip this space with the $L_2$-norm over $\Omega$, and define the set of admissible dielectric permittivity functions
\begin{equation} \label{admissible}
U^\eps\ce\{v\in V^\eps: 1 \leq v(\x) \leq \eps_{\mathrm{max}}~\forall \x\in\Omega,~v\rvert_\Gamma\equiv1,~\grad v\rvert_\Gamma\equiv 0\}
\end{equation}
for some known but not necessarily small upper bound $\eps_{\mathrm{\max}}$. The space $V^\eps$ can be thought of as a very fine finite element space (see Section \ref{error_analysis}), hence it is of (large but) finite dimension, which justifies the use of the $L_2$-norm. Although finite dimensional, this space is considered too large to handle numerically, and our goal is to approximate $\eps\in V^\eps$ by an element of a smaller subspace.

The set $U^\eps$ is defined to describe a heterogeneous medium in $\Omega$, immersed in a constant background with permittivity 1 in $\RR^3\setminus\Omega$.

Under the assumption that $\eps\in U^\eps$ we consider Maxwell's equations for an isotropic, non-magnetic, non-conductive medium in $\Omega$:
\begin{align}
\pd{(\mu\H)}{t} + \rot\E &= 0 && \text{in } \ot, \label{mw1}\\
\pd{(\eps\E)}{t} - \rot\H &= 0 && \text{in } \ot, \label{mw2}\\
 \div(\mu\H) = \div(\eps\E) &= 0 && \text{in } \ot, \label{mw3}
\end{align}
where $\H=\H(\x,\,t)$ and $\E=\E(\x,\,t)$, $(\x,\,t)\in\ot$, denote the magnetic and electric fields, respectively, and $\mu>0$ is the constant magnetic permeability. By scaling, we may assume that $\mu=1$.

To obtain an equation involving only $\eps$ and $\E$, we combine the curl of \eqref{mw1} and derivative of \eqref{mw2} with respect to $t$ to obtain the second order equation
\begin{align*}
\eps\pd{^2\E}{t^2} + \rot(\rot\E) &= 0 && \text{in } \ot.
\end{align*}
To incorporate \eqref{mw3} we introduce a penalty term $-s\grad(\div(\eps\E))$, for a fixed gauge parameter $s\geq1$ (see details in \cite{m03, pl91}), thus, after completing with boundary and initial conditions, we obtain the system
\begin{equation} \label{maxwells}
\begin{aligned}
&\eps\pd{^2\E}{t^2}-\lap\E + \grad(\div\E) - s\grad(\div(\eps\E)) = 0 &&\text{in } \ot,\\
&\pd{\E}{\bnu} = \P &&\text{on } \gt, \\
&\E(\cdot,\,0) = \pd{\E}{t}(\cdot, \,0) = 0 &&\text{in } \Omega,
\end{aligned}
\end{equation}
where we have used the expansion $\rot(\rot\E)=-\lap\E+\grad(\div\E)$ to simplify further manipulations of the equation. We use the notation $\pd{}{\bnu}=\bnu\cdot\grad$, where $\bnu$ denotes the outward unit normal on $\Gamma$. $\P\in[L_2(\gt)]^3$ is given Neumann data (see Section~4 of \cite{btkm14} for details). For well-posedness of problems of this class, we refer to \cite{l85, e10}.

The mathematical statement of the coefficient inverse problem is:

\begin{problem} \label{cip}
Given time-resolved boundary observations $\G\in[L_2(\gt)]^3$ of the electric field, determine $\eps\in U^\eps$ such that $\E = \G$ on $\gt$.
\end{problem}

The observations $\G$ represents either experimental or (partially) simulated data. For details on how to obtain such data, see \cite{btkm14}.

Inverse problems, such as Problem~\ref{cip}, are typically ill-posed in nature. Thus we cannot expect to be able to find an exact unique solution for any given data $\G$ (which in practice \emph{will} contain noise). Instead we will follow the concept of regularization (see for instance \cite{ehn96, tgsy95}) and assume that the given data $\G$ is a perturbation of ideal data $\G^*$, for which there exists a unique solution $\eps^*$ to Problem~\ref{cip}. The goal is then to systematically compute an approximation of $\eps^*$, a so-called regularized solution which hereafter will be denoted simply by $\eps$, which is as close to $\eps^*$ as can be achieved given the level of noise $\norm{\G - \G^*}_\gt$. Such a regularized solution is obtained by minimizing a Tikhonov functional, to be defined below (see equation~\eqref{tikhonovfunctional}).

Uniqueness of the solution of coefficient inverse problems of this type is typically obtained via the method of Carleman estimates \cite{bk81}. Examples where this method is applied to inverse problems for Maxwell's equations can be found in, for example, \cite{k86}, \cite{bcs12} for simultaneous reconstruction of two coefficients, and \cite{li05, ly07} for bi-isotropic and anisotropic media.
However, this technique requires non-vanishing initial conditions for the underlying partial differential equation, which is not the case here. Thus, currently, uniqueness of the solution for the problem we study is not known. For the purpose of this work, we will assume that uniqueness holds. This assumption is justified by the numerical reconstruction results presented in the experimental works \cite{btkm14, btkm14b}.

We introduce the space 
\begin{equation*}
V^\mathrm{dir}\ce\{\mathbf{v}\in[H^1(\ot)]^3:\mathbf{v}(\cdot,\,0)=0\}
\end{equation*}
for solutions to the direct problem, and
\begin{equation*}
V^\mathrm{adj}\ce\{\mathbf{v}\in[H^1(\ot)]^3:\mathbf{v}(\cdot,\,T)=0\}
\end{equation*}
for adjoint solutions. Both spaces are equipped with the usual norm and inner product on $[H^1(\ot)]^3$.
Then, by multiplying the first equation in \eqref{maxwells} by a test function $\bphi\in V^\mathrm{adj}$ and integration over $\ot$, we obtain, after integration by parts,
\begin{equation} \label{weakmaxwell}
\begin{aligned}
0 &=
 -\sca{\eps\pd{\E}{t}}{\pd{\bphi}{t}}_\ot
 + \sca{\eps\pd{\E}{t}(\cdot,\,T)}{\bphi(\cdot,\,T)}_\Omega
 - \sca{\eps\pd{\E}{t}(\cdot,\,0)}{\bphi(\cdot,\,0)}_\Omega \\&\quad
 + \sca{\grad\E}{\grad\bphi}_\ot 
 - \sca{\pd{\E}{\bnu}}{\bphi}_\gt
 - \sca{\div\E}{\div\bphi}_\ot
 + \sca{\div\E}{\bnu\cdot\bphi}_\gt \\&\quad
 + s\sca{\div(\eps\E)}{\div\bphi}_\ot
 - s\sca{\div(\eps\E)}{\bnu\cdot\bphi}_\gt\\
&= 
 -\sca{\eps\pd{\E}{t}}{\pd{\bphi}{t}}_\ot
 + \sca{\grad\E}{\grad\bphi}_\ot 
 - \sca{\div\E}{\div\bphi}_\ot \\&\quad
 + s\sca{\div(\eps\E)}{\div\bphi}_\ot
 - \sca{\P}{\bphi}_\gt \\
& \ec \maxwell(\eps,\,\E,\,\bphi),
\end{aligned}
\end{equation}
where the second equality holds because $\bphi(\cdot,\,T)=0$, $\pd{\E}{t}(\cdot,\,0)=0$, $\pd{\E}{\bnu}=\P$ on $\gt$, and $\div\E=\div(\eps\E)=0$ on $\gt$. This leads to the following weak description of the electric field:
\begin{problem} \label{weakdirect}
Given $\eps\in U^\eps$, determine $\E\in V^\mathrm{dir}$ such that $\maxwell(\eps,\,\E,\,\bphi)=0$ for every $\bphi\in V^\mathrm{adj}$.
\end{problem}

Let $\E_\eps\in V^\mathrm{dir}$ denote the solution to Problem~\ref{weakdirect} for a given $\eps\in U^\eps$. We can then define the Tikhonov functional $F\colon U^\eps\to \RR_+$,
\begin{equation} \label{tikhonovfunctional}
F(\eps)=F(\eps,\,\E_\eps)\ce \frac{1}{2}\norm{(\E_\eps - \G)z_\delta}_\gt^2 + \frac{\alpha}{2}\norm{\eps - \eps_0}_\Omega^2,
\end{equation}

\begin{center}
  \begin{figure}
    \includegraphics[width=0.9\textwidth]{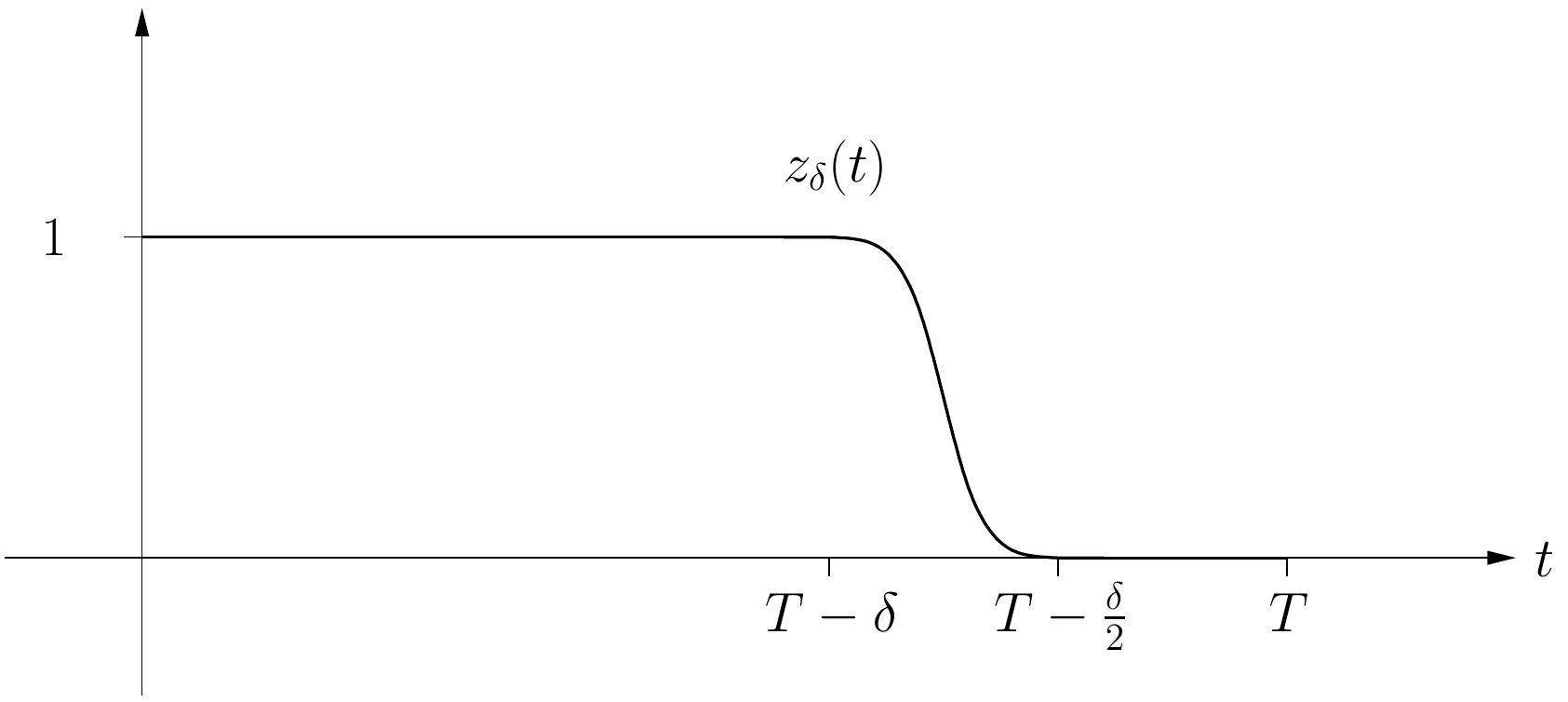}
    \caption{Schematic illustration of the cut-off function $z_\delta$ appearing in the Tikhonov functional \eqref{tikhonovfunctional}.}\label{zdelta}
  \end{figure}
\end{center}
where $\alpha>0$ is a regularization parameter and $z_\delta=z_\delta(t)\in C^\infty([0,\,T])$ is a cut-off function for the data, dropping from a constant level of 1 to a constant level of 0 within the small interval $(T-\delta,\,T-\delta/2)$, $0 < \delta \ll T$, as schematically shown in Figure \ref{zdelta}. The function $z_\delta$ is introduced to ensure data compatibility in the adjoint problem arising in the minimization of \eqref{tikhonovfunctional}.

How to choose the regularization parameter $\alpha$ with respect to the level of noise in the data is a widely studied topic. Several methods exist, examples are the (generalized) discrepancy principle \cite{tgsy95} and iterative methods \cite{bks11}. In future studies, we are planning to investigate iterative methods, similarly with \cite{h15}. However, for the results presented here, we regard $\alpha$ as a fixed parameter.

We assume that the initial approximation $\eps_0$ is sufficiently close to an ideal solution $\eps^*$, corresponding to noiseless data $\G^*$ in Problem~\ref{cip}.  If so, then by Theorem~3.1 of \cite{bkk10}, the Tikhonov functional $F$ is strongly convex in a neighborhood $\mathcal{N}\subset V^\eps$ of $\eps_0$. More precisely, we have the estimate
\begin{equation} \label{convexity}
\frac{\alpha}{2}\norm{\eps_1 - \eps_2}^2_\Omega \leq F'(\eps_1;\,\eps_1 - \eps_2) - F'(\eps_2;\,\eps_1 - \eps_2)
\end{equation}
for every $\eps_1$, $\eps_2\in\mathcal{N}\cap U^\eps$, where $F'(\eps;\,\oeps)$ denotes the Fr\'echet derivative of $F$ at $\eps$, acting on $\oeps$.

Throughout the remaining part of this text we will assume that the hypothesis of Theorem~3.1 of \cite{bkk10}, and hence strong convexity, holds. Then we may seek a minimizer $\eps\in U^\eps$ of $F$ by applying any gradient based method (such as steepest descent, quasi-Newton, or conjugate gradient), starting from $\eps_0$.

Such an approach requires that we compute the Fr\'echet derivative of $F$, which is complicated since it involves the implicit dependence of $\E_\eps$ upon $\eps$. To simplify the analysis, in the spirit of optimal control (see for example \cite{bkr00, kl10} for the general theory and some specific examples), we introduce the Lagrangian associated to the problem of minimizing $F(\eps,\,\E)$, $\eps\in U^\eps$, $\E\in V^\mathrm{dir}$, with $\maxwell(\eps,\,\E,\,\bphi)=0$ for all $\bphi\in V^\mathrm{adj}$ acting as a constraint. This Lagrangian is
\begin{equation*} \label{lagrangian}
L(u)\ce F(\eps,\,\E) + \maxwell(\eps,\,\E,\,\l),
\end{equation*}
where $u=(\eps,\,\E,\,\l)\in U \ce U^\eps\times V^\mathrm{dir}\times V^\mathrm{adj} \subset V\ce V^\eps\times V^\mathrm{dir}\times V^\mathrm{adj}$, $F(\eps,\,\E)$ was defined in \eqref{tikhonovfunctional}, and $\maxwell(\eps,\,\E,\,\l)$ was defined in \eqref{weakmaxwell}.

We can now minimize $F$ over $U^\eps$ by finding a stationary point of $L$ over $U$. With the strong convexity as above, this would imply that we solve
\begin{problem} \label{minlag}
Find $u\in U$ such that $L'(u;\,v)=0$ for every $v\in V$.
\end{problem}
Again we use the notation $L'(u;\,v)$ for the Fr\'echet derivative of $L$ at $u$, acting on $v$. It can be shown 
that
\begin{equation*}
L'(u;\,v)=\pd{L}{\eps}(u;\,\oeps) + \pd{L}{\E}(u;\,\oE) + \pd{L}{\l}(u;\,\ol),
\end{equation*}
where $u=(\eps,\,\E,\,\l)\in U$, $v=(\oeps,\,\oE,\,\ol)\in V$, and
\begin{equation}
\begin{aligned} \label{dl}
\pd{L}{\eps}(u;\,\oeps) &\ce \alpha\sca{\eps-\eps_0}{\oeps}_\Omega
 - \sca{\pd{\E}{t}\cdot\pd{\l}{t}}{\oeps}_\ot
 + s\sca{\div\l}{\div(\oeps\E)},\\
 \pd{L}{\E}(u;\,\oE) &\ce \sca{(\E-\G)z_\delta^2}{\oE}_\gt
 - \sca{\eps\pd{\l}{t}}{\pd{\oE}{t}}_\ot
 + \sca{\grad\l}{\grad\oE}_\ot\\ &\qquad
 - \sca{\div\l}{\div\oE}_\ot
 + s\sca{\div\l}{\div(\eps\oE)}  \ec \adjoint(\eps,\,\l,\,\oE),\\
\pd{L}{\l}(u;\,\ol) &= \maxwell(\eps,\,\E,\,\ol).
\end{aligned}
\end{equation}

In particular, we note that the solution $u=(\eps,\,\E,\,\l)$ to Problem~\ref{minlag} must satisfy $\maxwell(\eps,\,\E,\,\ol)=0$ for every $\ol\in V^\mathrm{adj}$ and $\adjoint(\eps,\,\l,\,\oE) = 0$ for every $\oE\in V^\mathrm{dir}$. The former means that $\E$ solves Problem~\ref{weakdirect} and the latter that $\l$ solves the following adjoint problem:
\begin{problem} \label{weakadjoint}
Given $\eps\in U^\eps$, determine $\l\in V^\mathrm{adj}$ such that $\adjoint(\eps,\,\l,\,\bphi)=0$ for every $\bphi\in V^\mathrm{dir}$.
\end{problem}

The functional $\adjoint$ in Problem~\ref{weakadjoint} was defined in \eqref{dl}. The problem can be seen as a weak analogue of the following system, adjoint to \eqref{maxwells}:
\begin{equation*}
\begin{aligned}
&\eps\pd{^2\l}{t^2} - \lap\l + \grad(\div\l) - s\eps\grad(\div\l) = 0 && \text{in }  \ot,\\
&\pd{\l}{\bnu}=-(\E - \G)z_\delta^2 && \text{on } \gt,\\
&\l(\cdot,\,T) = \pd{\l}{t}(\cdot,\,T) = 0 && \text{in } \Omega. 
\end{aligned}
\end{equation*}

These observations will be used in the error analysis to be described below. But first we shall make some remarks concerning the relation between the Fr\'echet derivative of Tikhonov functional and that of the Lagrangian.

Let $u_\eps = (\eps,\,\E_\eps,\,\l_\eps)$ be the element of $U$ obtained by taking $\E_\eps$ as the solution to Problem~\ref{weakdirect} and $\l_\eps$ as the solution to Problem~\ref{weakadjoint} for the given $\eps\in U^\eps$. Then, under assumption of sufficient stability of the weak solutions $\E_\eps$ and $\l_\eps$ with respect to $\eps$, the observation that
\begin{equation*}
F(\eps) = F(\eps,\,\E_\eps) = F(\eps,\,\E_\eps) + \maxwell(\eps,\,\E_\eps,\,\l_\eps) = L(u_\eps),
\end{equation*}
(as $\maxwell(\eps,\,\E_\eps,\,\l_\eps)=0$) leads to 
\begin{equation} \label{dj2dl}
F'(\eps;\,\cdot) = \pd{L}{\eps}(u_\eps;\,\cdot).
\end{equation}
Estimate \eqref{convexity} and identity \eqref{dj2dl} will play an important role in the error analysis for the Tikhonov functional and for the coefficient.

\section{Finite element formulations and error analysis} \label{error_analysis}
In this section we will give finite element formulations for discretizing Problems~\ref{weakdirect}, \ref{minlag} and \ref{weakadjoint}. After that we will turn to the error analysis. We begin by defining finite-dimensional analogues of the spaces $V^\eps$, $V^\mathrm{dir}$, $V^\mathrm{adj}$, and $V$, as well as subsets corresponding to $U^\eps$ and $U$.

Let $\th\ce\{K\}$ be a triangulation of $\Omega$ and let $\It$ be a uniform partition of $(0,\,T)$ into subintervals $(t_k,\,t_{k+1}]$, $t_k=k\tau$, $k=0,\,\ldots,\,N_\tau$, of length $\tau=T/N_\tau$. With $\th$ we associate a mesh-function $h=h(\x)$ such that 
\begin{equation} \label{meshfunction}
h(\x)=\diam(K)
\end{equation}
for $\x\in K\in \th$. On these meshes we define
\begin{equation*}
\begin{aligned}
V_h^\eps & \ce \{v\in V^\eps: v\rvert_K\in P^q(K)~\forall K\in \th\},\\
U_h^\eps & \ce V_h^\eps\cap U^\eps,\\
\vht^\mathrm{dir}& \ce \{v\in V^\mathrm{dir}:v\rvert_{K\times I}\in [P^1(K)]^3\times P^1(I)~\forall K\in\th~\forall I\in\It\},\\
\vht^\mathrm{adj}& \ce \{v\in V^\mathrm{adj}:v\rvert_{K\times I}\in [P^1(K)]^3\times P^1(I)~\forall K\in\th~\forall I\in\It\},\\
\vht &\ce V_h^\eps \times\vht^\mathrm{dir}\times\vht^\mathrm{adj},\\
\uht &\ce U_h^\eps \times\vht^\mathrm{dir}\times\vht^\mathrm{adj},
\end{aligned}
\end{equation*}
where $P^n(X)$ denotes the space of polynomials of degree at most $n\in\NN$ over $X$, and the degree $q$ used in the finite-dimensional analogue $V_h^\eps$ of $V^\eps$ is at least 1. Observe that the dependence on the step size $\tau$ in time is not explicitly included in the notation for the finite-dimensional spaces. This is justified by the fact that $\tau$ should be selected with regard to $h$ in accordance with the Courant-Friedrichs-Lewy condition.

Using these spaces we can state finite element versions of Problems~\ref{weakdirect} and \ref{weakadjoint} as Problem~\ref{femdirect} and Problem~\ref{femadjoint}, respectively, as follows:

\begin{problem}\label{femdirect}
Given $\eps\in U^\eps$, determine $\E_h\in \vht^\mathrm{dir}$ such that $\maxwell(\eps,\,\E_h,\,\bphi_h)=0$ for every $\bphi_h\in \vht^\mathrm{adj}$.
\end{problem}

\begin{problem}\label{femadjoint}
Given $\eps\in U^\eps$, determine $\l_h\in \vht^\mathrm{adj}$ such that $\adjoint(\eps,\,\l_h,\,\bphi_h)=0$ for every $\bphi\in \vht^\mathrm{dir}$.
\end{problem}

The finite-dimensional analogue for Problem~\ref{minlag} is:
\begin{problem} \label{femminlag}
Find $u_h=(\eps_h,\,\E_h,\,\l_h)\in \uht$ such that $L'(u_h,\,v)=0$ for every $v\in \vht$.
\end{problem}
The same remark that was made in conjunction with Problem~\ref{minlag} is also valid here: it holds that $\E_h$ solves Problem~\ref{femdirect} and $\l_h$ solves Problem~\ref{femadjoint} for $\eps=\eps_h$.

Having stated problems \ref{cip}--\ref{femminlag}, let us briefly summarize how these problems relate to each other, before turning to the error analysis.

Recall that our main goal is to find the dielectric permittivity which gave rise to the observed data $\G$, in other words to solve Problem~\ref{cip}. Due to ill-posedness, this goal is practically unattainable, and we focus instead on finding a regularized solution, which can be done by finding a stationary point to the Lagrangian, that is, by solving Problem~\ref{minlag}. This, in turn, relies upon solving the direct and adjoint problems for the electric field, problems \ref{weakdirect} and \ref{weakadjoint}, respectively. However, problems \ref{weakdirect}--\ref{weakadjoint} cannot, in general, be solved exactly, but approximately through their finite-dimensional analogues, problems \ref{femdirect}--\ref{femminlag}.

The purpose of the error estimation below is to quantify the discrepancy between the solutions to problems \ref{weakdirect}--\ref{weakadjoint}, and the solutions to problems \ref{femdirect}--\ref{femminlag}, in order to be able to adaptively refine the latter three problems so that the solution to Problem~\ref{femminlag}, the approximation of the regularized solution, fits the solution to Problem~\ref{minlag}, the true regularized solution, as closely as desired. We will now focus on this estimation.

We begin by introducing some additional notation. For $v=(\eps,\,\E,\,\l)\in V$ we denote (with some slight abuse of notation) its interpolant in $\vht$ by 
\begin{align*}
\ih v= (\ih\eps,\,\ih\E,\,\ih\l),
\end{align*}
and the interpolation error by 
\begin{align*}
\rh v = v-\ih v = (\rh\eps,\, \rh\E,\,\rh\l).
\end{align*}
 We will also need to consider jumps of discontinuous functions over $\th$ and $\It$. Let $K_1$, $K_2\in\th$ share a common face $f$. For $\x\in f$ we define
\begin{equation} \label{spatial_jump}
\jm{v}{s}(\x)\ce\lim_{\y\to\x,\,\mathbf{y}\in K_1}v(\y) + \lim_{\y\to\x,\,\mathbf{y}\in K_2}v(\y),
\end{equation}
so that in particular if $v=w\bnu$, where $w$ is piecewise constant on $\th$ and $\bnu$ is the outward unit normal, then $\jm{v}{s}=\jm{w\bnu}{s}=(w\bnu)\rvert_{K_1} + (w\bnu)\rvert_{K_2}$ is the normal jump across $f$. We extend $\jm{\cdot}{s}$ to every face in $\th$ by defining $\jm{v}{s}(\x)=0$ for $\x\in K\cap\Gamma$, $K\in \th$. The corresponding maximal jump is defined by
\begin{equation} \label{abs_spatial_jump}
\aj{v}{s}(\x) \ce \max_{\y\in\partial K}\abs{\jm{v}{s}(\y)}, \quad \x\in K\in \th,
\end{equation}
where $\partial K$ denotes the boundary of $K$.

For jumps in time, we define
\begin{equation} \label{time_jump}
\jm{v}{t}(t_k)\ce\begin{cases} {\displaystyle \lim_{s\to0+}}\big(v(t_k+s)-v(t_k-s)\big), & k=1,\,\ldots,\,N_\tau-1,\\
0 & k=0,\,N_\tau,\end{cases}
\end{equation}
and
\begin{equation} \label{abs_time_jump}
\aj{v}{t}(t)\ce 
\max\{\abs{\jm{v}{t}(t_k)},\,\abs{\jm{v}{t}(t_{k+1})}\}
\quad t\in(t_k,\,t_{k+1}).
\end{equation}


In the theorems and proofs to be presented, we will frequently use the symbols $\approx$ and $\lesssim$ to denote approximate equality and inequality, respectively, where higher order terms (with respect to mesh-size or errors) are neglected. We let $C$ denote variuos constants of moderate size which are independent of mesh-sizes and the the unknown functions.

We now proceed with an error estimate for the coefficient. An error estimate for the Tikhonov functional will follow as a corollary.

\begin{theorem} (A posteriori error estimate for the coefficient.) \label{ape_norm}
Suppose that the initial approximation $\eps_0$ and the regularization parameter $\alpha$ are such that the strong convexity estimate \eqref{convexity} holds. Let $u=(\eps,\,\E,\,\l)\in U$ be the solution to Problem~\ref{minlag}, and let $u_h=(\eps_h,\,\E_h,\,\l_h)\in \uht$ be the solution to Problem~\ref{femminlag}, computed on meshes $\th$ and $\It$. Then there exists a constant $C$, which does not depend on $u$, $u_h$, $h$, or $\tau$, such that 
\begin{equation} \label{the_estimate}
\norm{\eps - \eps_h}_\Omega \lesssim \frac{2C}{\alpha}( \eta + \norm{R_\eps}_\Omega),
\end{equation}
where $\eta =  \eta(u_h)$ is defined by
\begin{equation} \label{eta}
\begin{aligned}
\eta &\ce
\sca{ \frac{1}{\tau} \abs{\aj{\pd{\l_h}{t}}{t}} + s\abs{\div\l_h} }{h\abs{\aj{\pd{\E_h}{\bnu}}{s}} + \tau\abs{\aj{\pd{\E_h}{t}}{t}}  }_\ot \\ &\qquad
 + \sca{\frac{1}{\tau}\abs{\aj{\pd{\E_h}{t}}{t}}}{h\abs{\aj{\pd{\l_h}{\bnu}}{s}} + \tau\abs{\aj{\pd{\l_h}{t}}{t}} }_\ot\\&\qquad
 + s\sca{\abs{\div\l_h}}{\abs{\aj{\pd{\E_h}{\bnu}}{s}} + \tau\abs{\aj{\pd{\div\E_h}{t}}{t}} }_\ot\\&\qquad
 + s\sca{\abs{\div\E_h} + \abs{\E_h}}{ \abs{\aj{\pd{\l_h}{\bnu}}{s}} + \tau\abs{\aj{\pd{\div\l_h}{t}}{t}} }_\ot,
\end{aligned}
\end{equation}
and
\begin{equation} \label{residual}
R_\eps \ce \alpha(\eps_h - \eps_0)
 - \int_0^T\pd{\E_h}{t}\cdot\pd{\l_h}{t}\rmd t
 + \frac{s}{2h}\int_0^T\aj{(\div\l_h)(\bnu\cdot\E_h)}{s}\rmd t.
\end{equation}
\end{theorem}

\begin{proof}
Using strong convexity \eqref{convexity}, we obtain
\begin{equation*}
\norm{\eps - \eps_h}_\Omega^2\leq  \frac{2}{\alpha}\left(F'(\eps;\,\eps - \eps_h) - F'(\eps_h;\,\eps-\eps_h)\right).
\end{equation*}
Since $\eps$ minimizes $F(\eps)$ we have $F'(\eps;\eps-\eps_h)=0$ and thus
\begin{equation} \label{pfthm2_00}
\norm{\eps-\eps_h}_\Omega^2\leq \frac{2}{\alpha}\abs{F'(\eps_h;\,\eps - \eps_h)} = \frac{2}{\alpha}\abs{\pd{L}{\eps}(\tilde u;\,\eps - \eps_h)}, 
\end{equation}
where we have denoted by $\tilde\E$ and $\tilde\l$ the solutions to Problem~\ref{weakdirect} and Problem~\ref{weakadjoint}, respectively, with permittivity $\eps_h$, and set $\tilde{u} =(\eps_h,\,\tilde\E,\,\tilde\l)\in U$. The last equality follows from \eqref{dj2dl}. We remark that \eqref{pfthm2_00} implies $\norm{\eps - \eps_h}_\Omega \leq \frac{2}{\alpha}\norm{F'(\eps_h)}$, which would be an effective estimate. Unfortunately, we cannot compute $F'(\eps_h)$ exactly, since it depends on the exact solutions $\tilde\E$ and $\tilde\l$ as indicated in the text, hence further estimation is needed.

We expand
\begin{equation} \label{pfthm2_0}
\begin{aligned}
\abs{\pd{L}{\eps}(\tilde u;\,\eps - \eps_h)} &= \abs{\pd{L}{\eps}(\tilde u;\,\eps - \eps_h) -\pd{L}{\eps}(u_h;\,\eps - \eps_h) + \pd{L}{\eps}(u_h;\,\eps - \eps_h)}\\
&\leq \abs{\pd{L}{\eps}(\tilde u;\,\eps - \eps_h) -\pd{L}{\eps}(u_h;\,\eps - \eps_h)} + \abs{\pd{L}{\eps}(u_h;\,\eps - \eps_h)}\\
&\ec \abs{\Theta_1} + \abs{\Theta_2},
\end{aligned}
\end{equation}
and estimate the two terms $\abs{\Theta_1}$ and $\abs{\Theta_2}$ separately.

For $\Theta_1$ we assume that the second partial derivatives of $L$ exist, and use the linearization
\begin{align*}
\Theta_1& = \pd{L}{\eps}(\tilde u;\,\eps - \eps_h) - \pd{L}{\eps}(u_h;\,\eps - \eps_h)\\ & =
 \pd{^2L}{\eps^2}(u_h;\,\eps_h - \eps_h;\,\eps -\eps_h) + o(\norm{\eps_h - \eps_h}_\Omega) \\ &\quad +
\pd{^2L}{\E\partial\eps}(u_h;\,\tilde\E - \E_h;\eps - \eps_h) + o(\lVert\tilde\E - \E_h\rVert_{H^1(\ot)}) \\ &\quad
 + \pd{^2L}{\l\partial\eps}(u_h;\,\tilde\l - \l_h;\eps - \eps_h) + o(\lVert\tilde\l - \l_h\rVert_{H^1(\ot)}),
\end{align*}
where $\pd{^2L}{\E\partial\eps}$ and $\pd{^2L}{\l\partial\eps}$ denote mixed second partial Fr\'echet derivatives of $L$. The first two terms vanish, since the first components of $\tilde u$ and $u_h$ are both $\eps_h$, and again the remainder terms are neglected as they are of higher order with respect to the error. Thus, after exchanging the order of differentiation, we are left with
\begin{equation}
\begin{aligned}\label{pfthm2_1}
\Theta_1 &\approx \pd{^2L}{\E\partial\eps}(u_h;\,\tilde\E - \E_h;\eps - \eps_h) + \pd{^2L}{\l\partial\eps}(u_h;\,\tilde\l - \l_h;\eps - \eps_h) \\ & =
 D_1\rvert_{\eps - \eps_h}\left(\pd{L}{\E}(u_h;\,\tilde\E - \E_h) + \pd{L}{\l}(u_h;\,\tilde\l - \l_h)\right),
\end{aligned}
\end{equation}
where $D_1\rvert_{\eps - \eps_h}$ denotes differentiation with respect to the first component in $u_h$ and action on $\eps - \eps_h$.

We split $\tilde\E - \E_h=(\tilde\E - \ih\tilde\E) + (\ih\tilde\E - \E_h) = \rh\tilde\E + (\ih\tilde\E - \E_h)$ and use the fact that $\l_h$ solves Problem~\ref{weakadjoint} with coefficient $\eps_h$, so that $\pd{L}{\E}(u_h;\,\ih\tilde\E-\E_h) = 0$ as $\ih\tilde\E-\E_h\in V_h^\mathrm{dir}$. This gives
\begin{align} \label{pfthm2_2}
\pd{L}{\E}(u_h;\,\tilde\E - \E_h) = \pd{L}{\E}(u_h;\,\rh\tilde\E) + \pd{L}{\E}(u_h;\,\ih\tilde\E - \E_h) = \pd{L}{\E}(u_h;\,\rh\tilde\E).
\end{align}
Similarly, we have
\begin{align} \label{pfthm2_3}
\pd{L}{\l}(u_h;\,\tilde\l - \l_h) = \pd{L}{\l}(u_h;\,\rh\tilde\l) + \pd{L}{\l}(u_h;\,\ih\tilde\l - \l_h) = \pd{L}{\l}(u_h;\,\rh\tilde\l).
\end{align}
as $\E_h$ solves Problem~\ref{weakdirect} with coefficient $\eps_h$.

Combining \eqref{pfthm2_1}, \eqref{pfthm2_2}, and \eqref{pfthm2_3}, and recalling \eqref{dl} gives
\begin{align*}
\Theta_1 &\approx D_1\rvert_{\eps - \eps_h}\left(\pd{L}{\E}(u_h;\,\rh\tilde\E) + \pd{L}{\l}(u_h;\,\rh\tilde\l)\right)\\
 &= -\sca{(\eps - \eps_h)\pd{\rh\tilde\E}{t}}{\pd{\l_h}{t}}_\ot
 + s\sca{\div((\eps - \eps_h)\rh\tilde\E)}{\div\l_h}_\ot \\
 & \quad-\sca{(\eps - \eps_h)\pd{\E_h}{t}}{\pd{\rh\tilde\l}{t}}_\ot
 + s\sca{\div((\eps - \eps_h)\E_h}{\div\rh\tilde\l}_\ot.
\end{align*}

We now aim to lift time derivatives from the interpolation residuals $\rh\tilde\E$ and $\rh\tilde\l$ by splitting the integral over $[0,\,T]$ to the sum of integrals over the subintervals in $\mathcal{I}_\tau$, and integrating by parts on each subinterval. Thus we obtain
\begin{align*}
\sca{(\eps - \eps_h)\pd{\rh\tilde\E}{t}}{\pd{\l_h}{t}}_\ot
 &= \sum_{k=1}^{N_\tau} \int_{t_{k-1}}^{t_k}\sca{(\eps - \eps_h)\pd{\rh\tilde\E}{t}}{\pd{\l_h}{t}}_\Omega\rmd t\\
 &= \sum_{k=1}^{N_\tau} \left(- \int_{t_{k-1}}^{t_k}\sca{(\eps - \eps_h)\rh\tilde\E}{\pd{^2\l_h}{t^2}}_\Omega\rmd t\right.\\
 &\qquad +\left. \left.\sca{(\eps - \eps_h)\rh\tilde\E}{\pd{\l_h}{t}}_\Omega\right\rvert_{t = t_k}
  - \left.\sca{(\eps - \eps_h)\rh\tilde\E}{\pd{\l_h}{t}}_\Omega\right\rvert_{t = t_{k-1}}\right).
\end{align*}
We note that $\pd{^2\l_h}{t^2}\equiv0$ on each subinterval, since $\l_h$ is piecewise linear, and identify the jumps form \eqref{time_jump}, obtaining
\begin{align*}
\sca{(\eps - \eps_h)\pd{\rh\tilde\E}{t}}{\pd{\l_h}{t}}_\ot
 &= \sum_{k=1}^{N_\tau}\left.\sca{(\eps - \eps_h)\rh\tilde\E}{\jm{\pd{\l_h}{t}}{t}}_\Omega\right\rvert_{t=t_k}.
\end{align*} 

Next, we use the approximation
\begin{align*}
f(t_k) \approx \frac{1}{\tau}\int_{t_{k-1}}^{t_k}f(t)\rmd t
\end{align*}
to obtain terms defined on the whole time-interval. That is
\begin{align*}
\sca{(\eps - \eps_h)\pd{\rh\tilde\E}{t}}{\pd{\l_h}{t}}_\ot
 &\approx \sum_{k=1}^{N_\tau}\frac{1}{\tau}\int_{t_{k-1}}^{t_k}\sca{(\eps - \eps_h)\rh\tilde\E}{\aj{\pd{\l_h}{t}}{t}}_\Omega\rmd t\\
 & = \sca{(\eps - \eps_h)\rh\tilde\E}{\frac{1}{\tau}\aj{\pd{\l_h}{t}}{t}}_\ot.
\end{align*} 

Similarly, we have
\begin{align*}
\sca{(\eps - \eps_h)\pd{\E_h}{t}}{\pd{\rh\tilde\l}{t}}_\ot
& = \sca{(\eps - \eps_h)\frac{1}{\tau}\aj{\pd{\E_h}{t}}{t}}{\rh\tilde\l}_\ot.
\end{align*} 

Thus
\begin{equation}\label{pfthm2_3.5}
\begin{aligned}
  \abs{\Theta_1} &\lesssim \sca{\abs{\eps - \eps_h}\frac{1}{\tau}\abs{\aj{\pd{\l_h}{t}}{t}}}{ \abs{\rh\tilde\E}}_\ot
 + s\sca{\abs{\div((\eps - \eps_h)\rh\tilde\E)}}{\abs{\div\l_h}}_\ot\\
 & \quad+\sca{\abs{\eps - \eps_h}\frac{1}{\tau}\abs{\aj{\pd{\E_h}{t}}{t}}}{\abs{\rh\tilde\l}}_\ot
 + s\sca{\abs{\div((\eps - \eps_h)\E_h)}}{\abs{\div\rh\tilde\l}}_\ot \\
&\leq \norm{\eps - \eps_h}_{L_\infty(\Omega)}\left(
   \sca{\frac{1}{\tau}\abs{\aj{\pd{\l_h}{t}}{t}}}{ \abs{\rh\tilde\E}}_\ot
 + \sca{\frac{1}{\tau}\abs{\aj{\pd{\E_h}{t}}{t}}}{\abs{\rh\tilde\l}}_\ot
\right) \\&\quad
+ s\norm{\eps - \eps_h}_{L_\infty(\Omega)}\left(
   \sca{\abs{\div\rh\tilde\E}}{\abs{\div\l_h}}_\ot
 + \sca{\abs{\div\E_h}}{\abs{\div\rh\tilde\l}}_\ot
\right) \\&\quad
+ s\norm{\grad(\eps - \eps_h)}_{L_\infty(\Omega)}\left(
   \sca{\abs{\rh\tilde\E}}{\abs{\div\l_h}}_\ot
 + \sca{\abs{\E_h}}{\abs{\div\rh\tilde\l}}_\ot
\right).
\end{aligned}
\end{equation}

Note that, by the equivalence of norms on $V^\eps$, we have
\begin{align*}
\norm{\eps-\eps_h}_{L_\infty(\Omega)} + \norm{\grad(\eps-\eps_h)}_{L_\infty(\Omega)} &\leq C\norm{\eps-\eps_h}_\Omega.
\end{align*}
Finally, we use standard interpolation estimates (see for instance \cite{js95}) for $\rh\tilde\E$
\begin{align*}
\abs{\rh\tilde\E} &\leq C\left(h^2\abs{D^2\tilde\E} + \tau^2\abs{\pd{^2\E}{t^2}}\right)
 \approx C\left(h^2\abs{h^{-1}\aj{\pd{\E_h}{\bnu}}{s}} + \tau^2\abs{\tau^{-1}\aj{\pd{\E_h}{t}}{t}} \right) \\
& \lesssim C\left(h\abs{\aj{\pd{\E_h}{\bnu}}{s}} + \tau\abs{\aj{\pd{\E_h}{t}}{t}} \right),
\end{align*}
and as well as for $\rh\tilde\l$, $\div\rh\tilde\E$, and $\div\rh\tilde\l$
\begin{align*}
\abs{\rh\tilde\l} &\lesssim C\left(h\abs{ \aj{ \pd{\l_h}{\bnu} }{ s } } + \tau\abs{\aj{\pd{\l_h}{t}}{t}}\right),\\
\abs{\div\rh\tilde\E} &\lesssim C\left(\abs{ \aj{ \pd{\E_h}{\bnu} }{ s } } + \tau\abs{\aj{\pd{\div\E_h}{t}}{t}}\right),\\
\abs{\div\rh\tilde\l} &\lesssim C\left(\abs{ \aj{ \pd{\l_h}{\bnu} }{ s } } + \tau\abs{\aj{\pd{\div\l_h}{t}}{t}}\right),
\end{align*}
where the jumps $\aj{\cdot}{s}$ and $\aj{\cdot}{t}$ were defined in equations \eqref{spatial_jump}, \eqref{abs_spatial_jump}, \eqref{time_jump}, and \eqref{abs_time_jump}. Applying these estimates in \eqref{pfthm2_3.5}, we get
\begin{equation}\label{pfthm2_4}
\begin{aligned}
 \Theta_ 1 &\lesssim C\eta\norm{\eps - \eps_h}_\Omega,
\end{aligned}
\end{equation}
with $\eta$ as in \eqref{eta}.

Turning to $\Theta_2$ of \eqref{pfthm2_0}, we recall from \eqref{dl} that
\begin{equation} \label{pfthm2_4.5}
\begin{aligned}
\Theta_2 &= \pd{L}{\eps}(u_h;\,\eps - \eps_h) \\
 & =
 \alpha\sca{\eps_h - \eps_0}{\eps - \eps_h}_\Omega
 - \sca{(\eps - \eps_h)\pd{\E_h}{t}}{\pd{\l_h}{t}}_\ot
 + s\sca{\div((\eps - \eps_h)\E_h)}{\div\l_h}_\ot.
\end{aligned}
\end{equation}
Starting with the last term in the above expression, we split the integral over $\Omega$ into the sum of integrals over $K\in \th$, and integrate by parts to get
\begin{align*}
\sca{\div((\eps - \eps_h)\E_h)}{\div\l_h}_\ot &= \sum_{K\in\th}\sca{\div((\eps - \eps_h)\E_h)}{\div\l_h}_{K_T}\\ &=
\sum_{K\in\th}\left(
 - \sca{(\eps - \eps_h)\E_h}{\grad(\div\l_h)}_{K_T}\right.\\&\qquad\qquad\left.
 + \sca{(\eps - \eps_h)\bnu\cdot\E_h}{\div \l_h}_{\partial K_T}
\right),
\end{align*}
where $K_T\ce K \times (0,\,T)$, and $\partial K_T \ce (\partial K) \times (0,\,T)$. we observe that $\div(\grad\l_h)\equiv0$ on each $K\in \th$ for the piecewise linear $\l_h$, and identify the jumps from \eqref{spatial_jump}, obtaining
\begin{align*}
\sca{\div((\eps - \eps_h)\E_h)}{\div\l_h}_\ot &= 
 \sum_{K\in\th} \sca{(\eps - \eps_h)\bnu\cdot\E_h}{\div \l_h}_{\partial K_T} \\
& =
 \frac{1}{2}\sum_{K\in\th}\sca{\jm{(\div\l_h)(\bnu\cdot\E_h)}{s}}{\eps - \eps_h}_{\partial K_T},
\end{align*}
where the factor $\frac{1}{2}$ appears since every non-zero jump is encountered exactly twice in the sum over $K\in\th$.

Seeking to obtain expressions defined over the whole of $\Omega$, as opposed to once defined only on boundaries  of elements $K\in\th$, we use the following approximation, similar to the one used above for the jumps in time:
\begin{align*}
\int_{\partial K}f\rmd S \approx \int_K\frac{f}{h_K}\rmd\x.
\end{align*} 
This yields
\begin{align*}
\sca{\div((\eps - \eps_h)\E_h)}{\div\l_h}_\ot &\approx
\frac{1}{2}\sum_{K\in\th}\sca{\frac{1}{h_K}\aj{(\div\l_h)(\bnu\cdot\E_h)}{s}}{\eps - \eps_h} \\&=
\sca{\frac{1}{2h}\aj{(\div\l_h)(\bnu\cdot\E_h)}{s}}{\eps - \eps_h}.
\end{align*}
With the above estimate in \eqref{pfthm2_4.5}, we can now conclude that
\begin{equation} \label{pfthm2_5}
\begin{aligned}
\abs{\Theta_2} &\lesssim \abs{\sca{\alpha(\eps_h - \eps_0) - \int_0^T\pd{\E_h}{t}\cdot\pd{\l_h}{t}\rmd t +  \frac{s}{2h}\int_0^T\aj{(\div\l_h)(\bnu\cdot\E_h)}{s}\rmd t}{\eps - \eps_h}}\\
&\leq \norm{R_\eps}_\Omega\norm{\eps - \eps_h}_\Omega,
\end{aligned}
\end{equation}
with $R_\eps$ as defined \eqref{residual}.

Combining estimates \eqref{pfthm2_4} and \eqref{pfthm2_5} with \eqref{pfthm2_00} and \eqref{pfthm2_0}, we conclude that
\begin{equation*}
\norm{\eps - \eps_h}_\Omega^2 \lesssim \frac{2C}{\alpha}\left(\eta\norm{\eps - \eps_h}_\Omega + \norm{R_\eps}_\Omega\norm{\eps - \eps_h}_\Omega\right),
\end{equation*}
and the result \eqref{the_estimate} follows.
\end{proof}

We see that if the numerical errors for solving the direct and adjoint problems are relatively small, that is, when $\tilde u\approx u_h$ with relatively high accuracy, then $\norm{R_\eps}_\Omega\approx\norm{F'(\eps_h)}$ dominates the error estimate.

\begin{corollary}(A posteriori error estimate for the Tikhonov functional.) \label{ape_tikh}
Under the hypothesis of Theorem~\ref{ape_norm}, we have
\begin{equation*}
\abs{F(\eps) - F(\eps_h)} \lesssim \frac{4C^2}{\alpha^2}\left(\eta^2 + \norm{R_\eps}_\Omega^2\right),
\end{equation*}
with $\eta$ as defined in \eqref{eta}, and $R_\eps$ as in \eqref{residual}.
\end{corollary}

\begin{proof}
Using the definition of the Fr\'echet derivative and \eqref{dj2dl} we get
\begin{align*}
F(\eps) - F(\eps_h) &= F'(\eps_h;\eps - \eps_h) + o(\norm{\eps - \eps_h}_\Omega)\\
                   &= \pd{L}{\eps}(\tilde u;\,\eps - \eps_h) + o(\norm{\eps - \eps_h}_\Omega).
\end{align*}
Neglecting the remainder term as it is of higher order with respect to the error, and estimating $\pd{L}{\eps}(\tilde u;\,\eps - \eps_h) = \Theta_1 + \Theta_2$ as in the proof of Theorem~\ref{ape_norm}, we obtain
\begin{align*}
\abs{F(\eps) - F(\eps_h)} \lesssim \frac{2C}{\alpha}(\eta + \norm{R_\eps}_\Omega)\norm{\eps - \eps_h}_\Omega.
\end{align*}
Applying Theorem~\ref{ape_norm} to estimate $\norm{\eps - \eps_h}_\Omega$, we arrive at
\begin{align*}
\abs{F(\eps) - F(\eps_h)} \lesssim \frac{4C^2}{\alpha^2}(\eta + \norm{R_\eps}_\Omega)^2\leq C\left(\eta^2 + \norm{R_\eps}_\Omega^2\right).
\end{align*}
\end{proof}

We will conclude this section by describing how these theorems translates into concrete recommendations for refining the computational mesh in the adaptive algorithm, outlined in the Section~\ref{algorithm}. 

\subsection{Mesh refinement recommendations} \label{sec:ref}
From Theorem~\ref{ape_norm} and Corollary~\ref{ape_tikh}, it is clear that the error in the reconstruction can be estimated in terms of two quantities, $\eta$, and $R_\eps$. The former essentially represents how well the direct and adjoint problems, problems~\ref{weakdirect}, and \ref{weakadjoint}, are approximated by their finite element counterparts, problems~\ref{femdirect}, and \ref{femadjoint}, respectively. The latter, $R_\eps$, describes the error incurred by the approximation of the coefficient itself.

In our computational experience from \cite{btkm14, btkm14b}, on given meshes $\th$ and $\mathcal{I}_\tau$, the solutions to direct and adjoint problems are in generally approximated better than the coefficient itself. As remarked above, this implies that $\eta \ll \norm{R_\eps}_\Omega$, which tells us that the error is the greatest in regions where $\abs{R_\eps}\approx\abs{F'(\eps_h)}$ is close to its maximum value. Thus we propose the following mesh refinement recommendation:

\begin{recommendation} \label{mrr1}
Using Theorem~\ref{ape_norm} we conclude that we should refine the mesh in neighborhoods of those points in $\Omega$ where the function $\abs{R_\eps}$ attains its maximal values. More precisely, let $\beta \in (0,\,1)$ be a tolerance number which should be chosen in computational experiments. Then, refine the mesh $\th$ in such subdomains of $\Omega$ where
\begin{equation*}
 \abs{R_\eps} \geq \beta \max_{\Omega}\abs{R_\eps}.
\end{equation*}
\end{recommendation}

Since a relatively large value of the reconstructed coefficient $\eps_h$ indicates a region where the permittivity is different from the background value of $1$, we can also propose the following heuristic:

\begin{recommendation} \label{mrr2}
We should refine the mesh in neighborhoods of those points in $\Omega$ where the function $\abs{\eps_h}$ attains its maximal values. More precisely, we refine the mesh in such subdomains of $\Omega$ where
\begin{equation*}
\abs{\eps_h} \geq \widetilde{\beta} \max_\Omega \abs{\eps_h}, 
\end{equation*}
where $ \widetilde{\beta} \in (0,1)$ is a number which should be chosen computationally and $h$ is the mesh function \eqref{meshfunction} of the finite element mesh $\th$.
\end{recommendation}

\section{Adaptive algorithms for the inverse problem} \label{algorithm}
In this section we will present different algorithms which can be used for solution of the inverse problem we consider: usual conjugate gradient algorithm and two different adaptive finite element algorithms. Conjugate gradient algorithm is applied on every finite element mesh $\th$ which we use in computations. We note that in our adaptive algorithms the time mesh $\mathcal{I}_\tau$ is refined globally accordingly to the Courant-Friedrichs-Lewy condition of \cite{cfl67}. 

Taking into account remark of Section \ref{sec:ref} we denote by 
\begin{equation} \label{dertikhonov2}
\begin{aligned}
R_\eps^n(\x)\ce \alpha(\eps_h - \eps_0)
 - \int_0^T\pd{\E_h^n}{t}(\x,\,t)\cdot\pd{\l_h^n}{t}(\x,t)\rmd t
 + s\int_0^T\div\E_h^n(\x,\,t)\div\l_h^n(\x,\,t)\rmd t
\end{aligned}
\end{equation}
where functions $\l_h^n$, and  $\E_h^n$ are finite element solutions of direct and adjoint problems computed with $\eps \ce\eps_h^n$, respectively, and $n$ is the number of iteration in the conjugate gradient algorithm.

\subsection{Conjugate Gradient Algorithm}\label{sec:cg}
Here we outline the conjgate gradient algorithm, which will be used in the two adaptive algorithms presented below.

\begin{algorithm} (Conjugate Gradient Algorithm)
\begin{itemize}
\item[Step 0.]  Discretize the computational space-time domain $\Omega \times [0,T]$ using partitions $\th$ and $\mathcal{I}_\tau$, respectively, see Section \ref{error_analysis}.  Start with the initial approximations $\eps_h^0 = \eps_0$ and compute the sequence of $\eps_h^n$, $n = 1,\,2,\,\ldots$, as:
\item[Step 1.] Compute solutions $\E_h^n$ and $\l_h^n$ of problems~\ref{femdirect} and \ref{femadjoint}, respectively, using the coefficient $\eps_h^n$.

\item[Step 2.]  Update the coefficient on $\th$ and $\mathcal{I}_\tau$ via the conjugate gradient method
\begin{equation*}
\begin{aligned}
\eps_h^{n+1}(\x) &=  \eps_h^n(\x) + \gamma_{\eps}^n d_{\eps}^n(\x),
\end{aligned}
\end{equation*}
where
\begin{equation*}
\begin{aligned}
 d_{\eps}^n(\x)&=  -R_\eps^n(\x) + \beta_\eps^n  d_\eps^{n-1}(\x),
\end{aligned}
\end{equation*}
with
\begin{equation*}
\begin{aligned}
\beta_\eps^n &= \frac{ \norm{R_\eps^n}_\Omega^2 }{ \norm{R_\eps^{n-1}}_\Omega^2 }, 
\end{aligned}
\end{equation*}
$d_\eps^0(\x)= -R_\eps^0(\x)$, and $\gamma_\eps^n$ are step-sizes in the gradient update which can be computed as in \cite{p84}
\begin{equation}
\gamma_\eps^n = -\frac{ \sca{R_\eps^n}{d_\eps^n}_\Omega }{\alpha \norm{d_\eps^n}_\Omega^2}.
\end{equation}

\item[Step 3.] Stop computing updates $\eps_h^n$ at the iteration $M \ce n$ and obtain the function $\eps_h \ce \eps_h^M$ if either $\norm{R_\eps^n}_\Omega\leq \theta$ -- where $\theta$ is the tolerance in $n$ updates of the gradient method -- or norms $\norm{\eps_h^n}_\Omega$ are stabilized. Otherwise update $n$ to $n+1$ and return to Step 1.
\end{itemize}
\end{algorithm}

\subsection{Adaptive algorithms} \label{sec:adaptalg}
In this section we present two different adaptive algorithms for the solution of our coefficient inverse problem (more precisely, Problem~\ref{femminlag}, approximating the solution to Problem~\ref{minlag}), where  in the first adaptive algorithm we apply Mesh Refinement Recommendation~\ref{mrr1} of Section~\ref{sec:ref}, while in the second adaptive algorithm we use Mesh Refinement Recommendation~\ref{mrr2} of Section~\ref{sec:ref}.

We define the minimizer of the Tikhonov functional \eqref{tikhonovfunctional}  and its approximated finite element solution on $k$ times adaptively refined mesh $\mathcal{T}_{h_k}$ by $\eps$ and $\eps_{h_k}$, correspondingly. The latter is obtained at the final step of the conjugate gradient iteration of Section~\ref{sec:cg} on the mesh $\mathcal{T}_{h_k}$.

\begin{algorithm} \label{adalg1}
(The First Adaptive Algorithm)

\begin{itemize}
\item[Step 0.] Choose an initial space-time mesh $\mathcal{T}_{h_0} \times \mathcal{I}_{\tau_0}$ in $\Omega \times [0,\,T]$. Compute $\eps_{h_k}$, $k > 0$, via following steps:

\item[Step 1.] Obtain numerical solution $\eps_{h_k}$ on $\mathcal{T}_{h_k}$ using the  Conjugate Gradient Method of  Section~\ref{sec:cg}.

\item[Step 2.] In accordance with the first mesh refinement recommendation, refine such elements in the mesh $\mathcal{T}_{h_k}$ where the expression 
\begin{equation} \label{alg2_2}
\abs{R_\eps^{M_k}} \geq \beta_k \max_\Omega\abs{R_\eps^{M_k}},
\end{equation}
is satisfied. Here, the tolerance numbers $ \widetilde{\beta}_k  \in \left(0,\,1\right) $ are chosen by the user.

\item[Step 3.] Define a new refined mesh as $\mathcal{T}_{h_{k+1}}$ and construct a new time partition $\mathcal{I}_{\tau_{k+1}}$ such that the Courant-Friedrichs-Lewy condition of \cite{cfl67} is satisfied. Interpolate $\eps_{h_k}$ on the new mesh $\mathcal{T}_{h_{k+1}}$ and perform Steps 1--3 on the space-time mesh $\mathcal{T}_{h_{k+1}} \times \mathcal{I}_{\tau_{k+1}}$. 

 \item[Step 4.] Stop mesh refinements when either $\norm{\eps_{h_k} - \eps_{h_{k-1}}}_\Omega < \theta_1$, or for some $n$, $\norm{R_\eps^n}_\Omega < \theta_2$, where $\theta_i$, $i=1,\,2$ are tolerances chosen by the user, and $R_\eps^n$ is the gradient on the $n$:th iteration of the conjugate gradient method on the new mesh. We then set the final number of refinements $k_\mathrm{rec} \ce k$, and the reconstructed coefficient $\eps_\mathrm{rec}\ce\eps_{h_k}$.
\end{itemize}

\end{algorithm}

\begin{algorithm} \label{adalg2}
(The Second Adaptive Algorithm)

This algorithm follows the same procedure as the First Adaptive Algorithm, except that the refinement criterion \eqref{alg2_2} is replaced, in accordance with the second mesh refinement recommendation, by
\begin{equation}\label{62}
\abs{\eps_{h_k}} \geq \widetilde{\beta}_k \max_\Omega \abs{\eps_{h_k}}
\end{equation}
for some tolerance numbers $\beta_k \in (0,\,1)$ , possibly different from $\widetilde{\beta}_k$, chosen by the user. Here, $R_\eps^{M_k}$ is the gradient on the last iteration of the conjugate gradient method on the $k$ times adaptively refined mesh.

\end{algorithm}

%
%
%

Before continuiong with the details of our numerical examples, some remarks are in order:

\begin{itemize}
\item Firstly, we comment on how to choose the tolerance numbers $\beta_k$, and $\widetilde{\beta_k }$ in \eqref{alg2_2}, \eqref{62}. Their values depend on the concrete values of $ \max_\Omega\abs{R_\eps^{M_k}}$ and $\max_\Omega \abs{h_k\eps_{h_k}}$, correspondingly.  If we will take values of $\beta_k$, and $\widetilde{\beta_k }$ which are very close to $1$ then we will refine the mesh in very narrow region of the domain $\Omega$, and if we will choose $\beta_k$, and $\widetilde{\beta_k} \approx 0$ then almost all elements in the finite element mesh  will be refined, and thus, we will get global and not local mesh refinement.  Our numerical tests of Section~\ref{numex} show that the choice of $\beta_k$, and $\widetilde{\beta}_k = 0.7$ is almost optimal one since with these values of the parameters $\beta_k$, and $\widetilde{\beta_k }$ the finite element mesh $\mathcal{T}_{h_k}$ is refined exactly at the places where we have computed non-trivial parts of the functions $\eps_{h_k}$.

\item Secondly, to compute norms $\norm{\eps_{h_k} - \eps_{h_{k-1}}}_\Omega$ in Step~3 of the adaptive algorithms, the solution $\eps_{h_{k-1}}$ is interpolated from the mesh $\mathcal{T}_{h_{k-1}}$ to the mesh $\mathcal{T}_{h_k}$.
\end{itemize}

\section{Numerical examples} \label{numex}

\begin{figure}[tbp]
\begin{center}
\begin{tabular}{cc}
{\includegraphics[width = 0.3\textwidth, clip = true, trim = 3.0cm 3.0cm 3.0cm 3.0cm,, angle = -90]{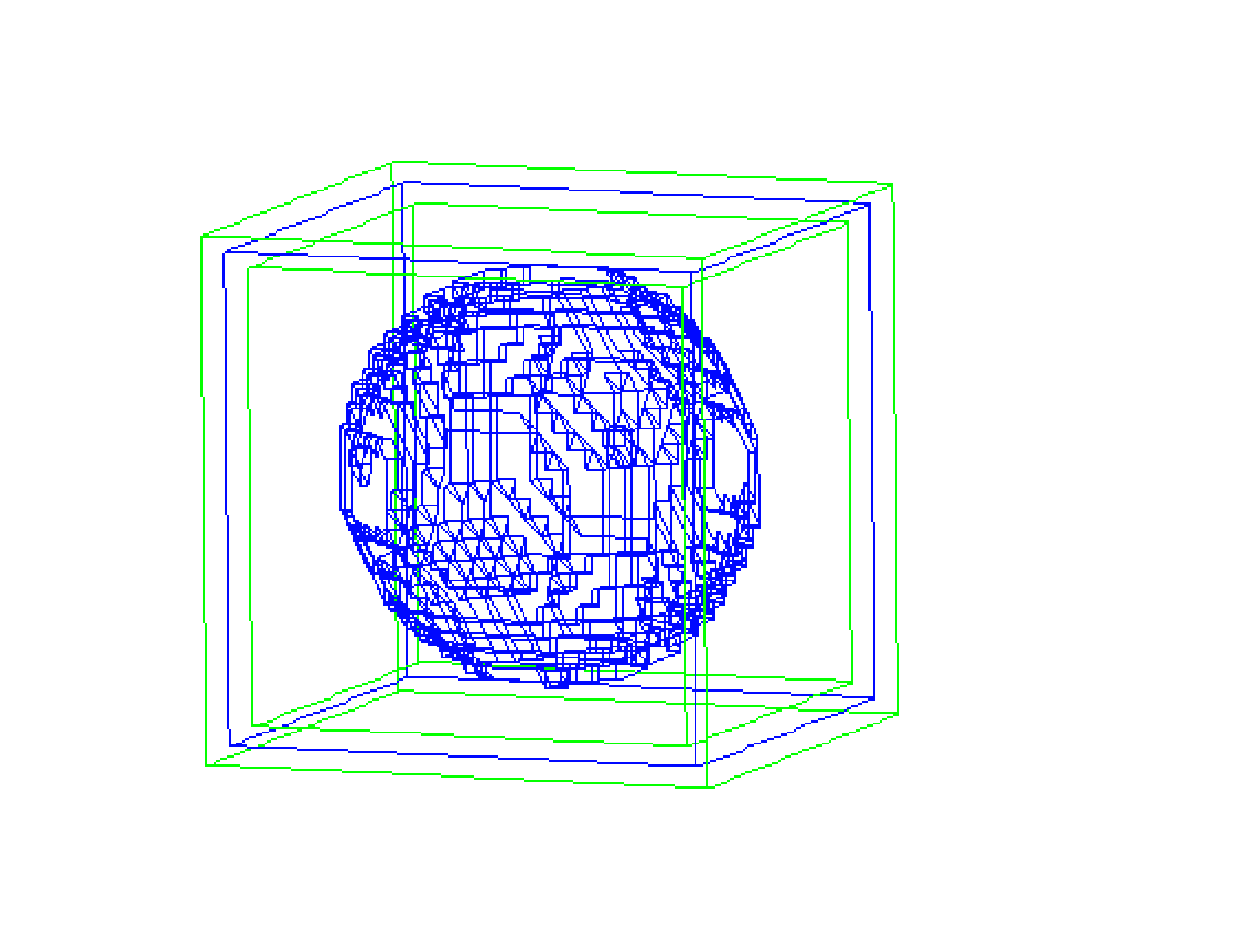}} &
{\includegraphics[width = 0.3\textwidth, clip = true, trim = 3.0cm 3.0cm 3.0cm 3.0cm,, angle = -90]{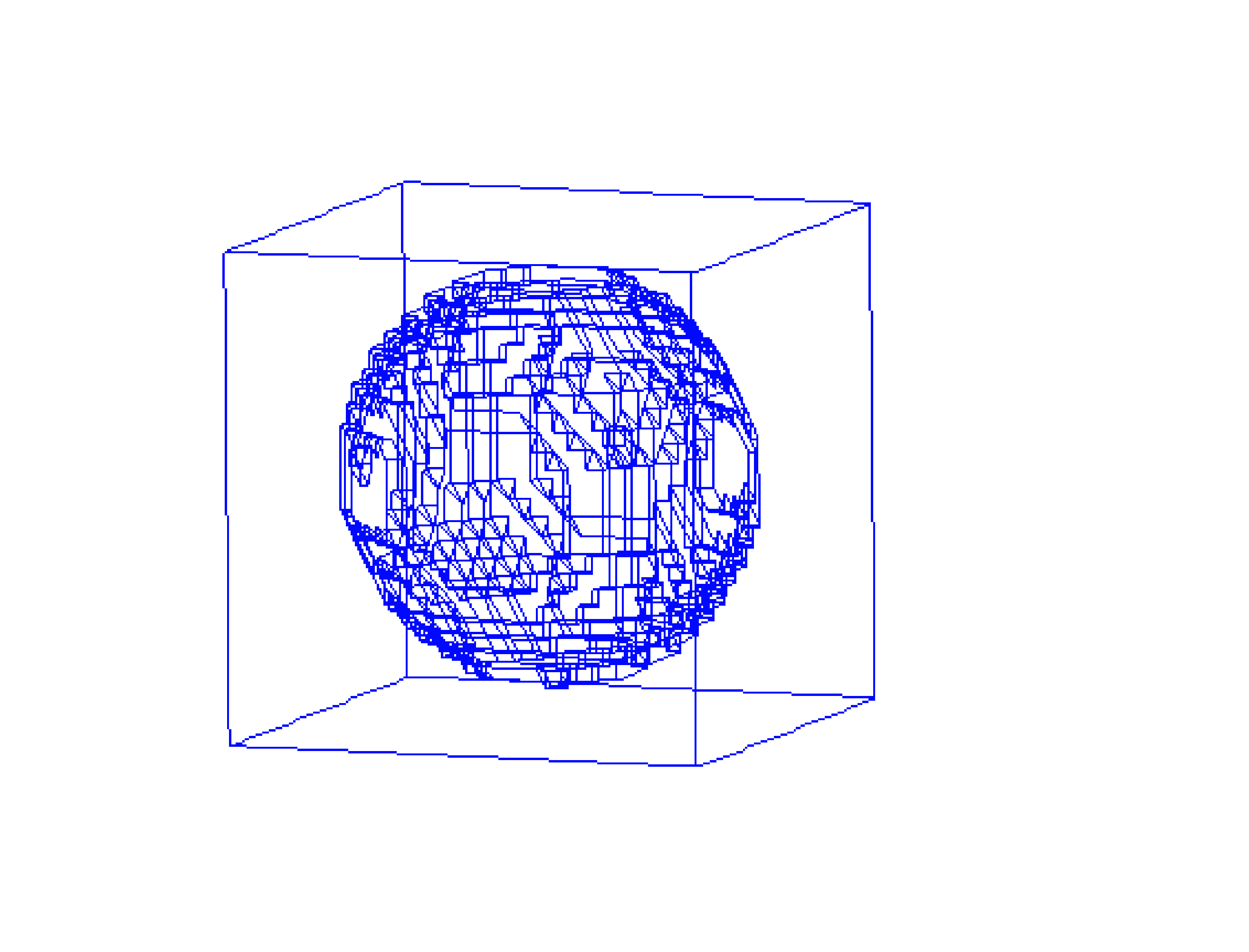}} \\
a) Test1: $\widetilde{\Omega} = \Omega_\mathrm{FEM} \cup \Omega_\mathrm{FDM}$ &  b) Test 1: $\Omega_\mathrm{FEM}$ \\
\end{tabular}
\end{center}
\caption{\protect\small\emph{Domain decomposition in numerical tests of Section~\ref{numex}. a) The  decomposed domain  $\widetilde{\Omega} = \Omega_\mathrm{FEM} \cup \Omega_\mathrm{FDM}$. b) The finite element domain $\Omega_\mathrm{FEM}$.}}
\label{fig:fig1}
\end{figure}

In this section we present numerical studies of the solution our inverse problem using the adaptive algorithms of Section~\ref{sec:adaptalg}. The algorithms are efficiently implemented in the software package WavES (\cite{waves}), using the domain decomposition technique of \cite{b13}. 

To do that we enlarge the domain $\Omega$ to a domain $\widetilde{\Omega} \supset \Omega$, and decompose $\widetilde{\Omega}$ into two subregions $\Omega_\mathrm{FEM}$ and $\Omega_\mathrm{FDM}$ such that $\Omega_\mathrm{FEM} = \Omega$, and $\Omega_\mathrm{FEM} \cap \Omega_\mathrm{FDM} = \emptyset$.  In $\Omega_\mathrm{FEM}$ we will use the finite element method (FEM) and, in $\Omega_\mathrm{FDM}$, the finite difference method (FDM). The boundary $\partial \widetilde{\Omega}$ of the domain $\widetilde{\Omega}$ is such that $\partial \widetilde{\Omega} = \partial_{1} \widetilde{\Omega} \cup \partial _{2} \widetilde{\Omega} \cup \partial_{3}\widetilde{\Omega}$ where $\partial _{1} \widetilde{\Omega}$ and $\partial _{2} \widetilde{\Omega}$ are, respectively, front and back sides of $\widetilde{\Omega}$, and $\partial _{3} \widetilde{\Omega}$ is the union of left, right, top and bottom sides of this domain. We will collect time-dependent observations over $S_T \ce \partial_1 \widetilde{\Omega} \times (0,\,T)$ at the backscattering side $\partial_1 \widetilde{\Omega}$ of $\widetilde{\Omega}$.  We also define $S_{1,\,1} \ce \partial_1 \widetilde{\Omega} \times (0,\,t_1]$, $S_{1,\,2}\ce \partial_1 \widetilde{\Omega} \times (t_1,\,T)$, $S_2 \ce \partial_2 \widetilde{\Omega} \times(0,\, T)$ and $S_3 \ce \partial_3 \widetilde{\Omega} \times (0,\, T)$.

As in \cite{btkm14, btkm14b} we initialize  only one component $E_2$ of the electrical field $\E=(E_1,\,E_2,\,E_3)$  on $S_T$ as a plane wave $f(t)$ such that 
 \begin{equation}\label{f}
 f(t) =\begin{cases}
 \sin \left(\omega t \right)  &\text{if } 0 < t < 2\pi/\omega, \\ 
 0 &\text{if } t > 2\pi /\omega .
 \end{cases} 
 \end{equation}

We assume that the function $\eps \equiv 1$ inside $\Omega_\mathrm{FDM}$. The numerical tests of our previous studies \cite{bh15} show that the best reconstruction results are obtained for $\omega = 40$ in \eqref{f}. Thus, we perform our tests with $\omega = 40$ in \eqref{f} in all our tests.

In our computations, we consider computational domains $\widetilde{\Omega} \ce (-0.8,\,0.8)^3$, and $\Omega_\mathrm{FEM} \ce (-0.7,\,0.7)^3$, where the length scales are in decimeters. We choose the mesh size $h_0 = 0.05$ in the overlapping layers between $\Omega_\mathrm{FEM}$ and $\Omega_\mathrm{FDM}$ as well as in the coarse mesh. We note that we have generated our transmitted data using a four times locally refined initial mesh $\Omega_\mathrm{FEM}$, and in a such way we avoid variational crimes. To generate transmitted data we solve the model problem in time $T = [0,\,3.0]$ with the time step $\tau = 0.006$ which satisfies to the Courant-Friedrichs-Lewy condition \cite{cfl67}. We then pollute our data additive noise of levels $\sigma=3\%$, and $10\%$, respectively.

Similarly with \cite{bh15, b13, b11, btkm14, btkm14b} in all our computations we choose constant regularization parameter $\alpha = 0.01$ because it gives smallest relative error in the reconstruction of the function $\eps$. This parameter was chosen via trial and error because of our computational experience: such choice for the regularization parameter gave the smallest relative error $e_{\eps} = \norm{\eps - \eps_{h_k}}_{\Omega}/\norm{\eps_{h_k}}_{\Omega}$. An iteratively regularized adaptive finite element method when both functions $\eps$ and $\mu$ are reconstructed, has recently been presented in \cite{h15}. Here, iterative regularization is performed via the algorithms of \cite{bks11}. We also refer to \cite{ehn96, tgsy95} for different techniques for the choice of regularization parameters.

We perform four different tests:
\begin{enumerate}

\item[Test 1:] The goal of this numerical test is to reconstruct a smooth function $\eps$ only inside $\Omega_\mathrm{FEM}$. We define this function as
\begin{equation}\label{2gaussians}
\begin{aligned}
\eps(\x)  &\ce1.0 + 1.0\rme^{-\abs{\x - \x_1}^2/0.2} + 1.0\rme^{-\abs{\x - \x_2}^2/0.2}, && \x\in\Omega_\mathrm{FEM},\\
\x_1 &\ce (0.3,\, 0.0,\, 0.0) \in\Omega_\mathrm{FEM},\\
\x_2 &\ce (-0.4,\, 0.2,\, 0.0) \in\Omega_\mathrm{FEM}.
\end{aligned}
\end{equation}

\item[Test 2:] In this test we  reconstruct three small inclusions of diameter $d=2$~mm with the centers of the inclusions at $(-0.3,\,0.0,\,-0.25), (0.3,\,0.2,\,-0.25)$ and $(0.3,\,-0.2,\,-0.25)$, respectively.

\item[Test 3:] In this test we reconstruct four small inclusions of diameter $d=2$~mm with the centers of the inclusions at $(-0.3,\, 0.0,\, 0.25), (0.0,\, 0.2,\, 0.25)$, $(0.0,\, -0.2,\, -0.25)$, and $(0.3,\, -0.2,\, -0.25)$, respectively.

\item[Test 4:]  The inclusions of this test are the same as in Test~3, but here the data consists measurements of two backscattered wave fields: one backscattered field initiated at the front boundary $\partial_1 \widetilde{\Omega} $, and another one at the back boundary $\partial_2 \widetilde{\Omega}$.
\end{enumerate}

We start to run the adaptive algorithm with a homogeneous initial approximation $\eps_0 \equiv 1.0$ in $\Omega_\mathrm{FEM}$. In our recent work \cite{bh15}, it was shown that such choice of the initial approximation gives a good reconstruction. Such homogeneous initial approximations were also used in \cite{elhk09}. See also \cite{b13, b11, btkm14, btkm14b} for a similar choice of initial approximations. We also assume the upper bound $\eps_\mathrm{max} = 5$ in \eqref{admissible}. This is a reasonable value, given that our target applications are in medical imaging, and typically involves low contrasts.

To get final images of our reconstructed function $\eps_{h_k}$ we use a post-processing procedure which has been described before in \cite{bh15, b13, b11, btkm14}.

\begin{table}[tbp] 
{\footnotesize Table 1. \emph{Results obtained on the coarse mesh. We present reconstructions of the maximal contrast $\tilde{\eps} = \max_{\Omega_\mathrm{FEM}} \eps_{h_0}^{M_0}$ together with computational errors in percents. Here, $M_0$ is the final number of iteration in the conjugate gradient method on the coarse mesh.}}  \par
\vspace{2mm}
\centerline{
\begin{tabular}{|c|c|}
\hline
$\sigma=3\%$ &  $\sigma = 10\%$ \\
\begin{tabular}{l|l|l|l} \hline
 & $\tilde{\eps}$ &  error, \% & $M_0$  \\ \hline
Test 1 & 1.93 & 3.5 & 2   \\
Test 2 & 2.94  & 47   &  2   \\
Test 3 &  1.77  & 11.5    &  2   \\
Test 4 & 1.9 & 5 &  2  \\
\end{tabular}
 & 
\begin{tabular}{l|l|l|l} \hline
&  $\tilde{\eps}$ &    error, \% & $M_0$  \\ \hline
Test 1 & 1.94 & 3  & 2  \\
Test 2 & 2.81  & 40.5 &  2  \\
Test 3 & 2.04 & 2 &  2  \\
Test 4 & 2.03 & 1.5 &  2  \\
\end{tabular} 
\\
\hline
\end{tabular}}
\end{table}

\begin{table}[tbp] 
{\footnotesize Table 2. \emph{Results obtained on $k_\mathrm{rec}$ times adaptively refined mesh. We present reconstructions of the maximal contrast $\tilde{\eps} = \max_{\Omega_\mathrm{FEM}} \eps_\mathrm{rec}$ together with computational errors in percents.}}  \par
\vspace{2mm}
\centerline{
\begin{tabular}{|c|c|}
\hline
$\sigma=3\%$ &  $\sigma = 10\%$ 
\\
\hline
\begin{tabular}{l|l|l|l|l} \hline
Case & $\tilde{\eps}$ &  error, \% & $M_{k_\mathrm{rec}}$  & $k_\mathrm{rec}$ \\ \hline
Test 1 & 2.04 & 2 & 5 & 4 \\
Test 2 & 1.99  & 0.5  & 1 & 5  \\
Test 3 & 1.55   & 22.5   & 1 & 2   \\
Test 4 & 1.9   &   5 & 1 & 1  \\
\end{tabular}
 & 
\begin{tabular}{l|l|l|l|l} \hline
Case & $\tilde{\eps}$ &    error, \% & $M_{k_\mathrm{rec}}$   & $k_\mathrm{rec}$ \\ \hline
Test 1 & 1.97 & 1.5  & 1 & 5   \\
Test 2 &  1.92 & 4 &  1 & 5 \\
Test 3 & 1.88 & 6 &  1 & 2 \\
Test 4 & 2.15 & 7.5 &  1 & 1 \\
\end{tabular} 
\\
\hline
\end{tabular}}
\end{table}

\begin{figure}
\begin{center}
\begin{tabular}{ccc}
{\includegraphics[scale=0.22, trim = 2.0cm 6.0cm 2.0cm 6.0cm, clip=true,]{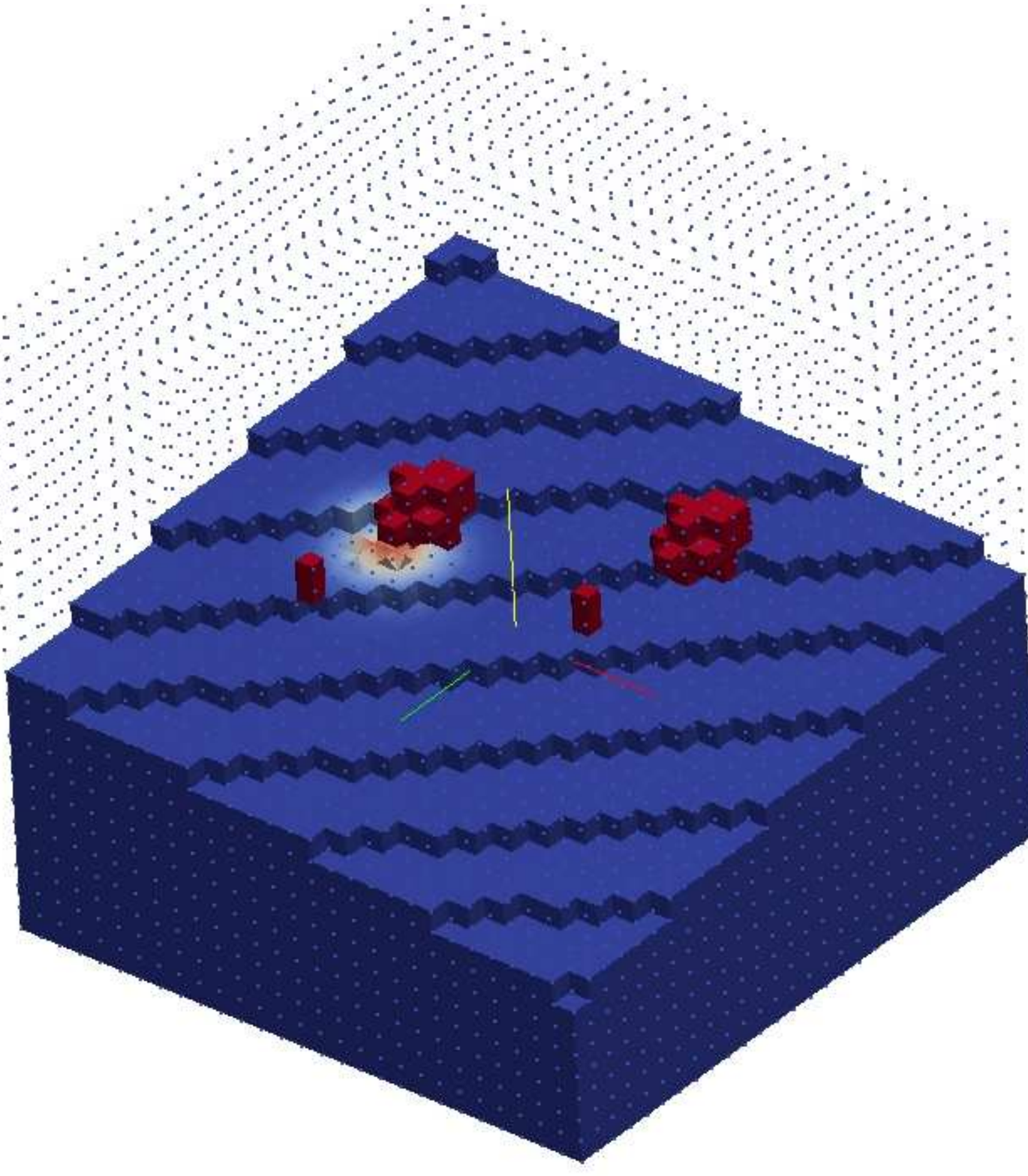}} &
{\includegraphics[scale=0.22, trim = 2.0cm 6.0cm 2.0cm 6.0cm, clip=true,]{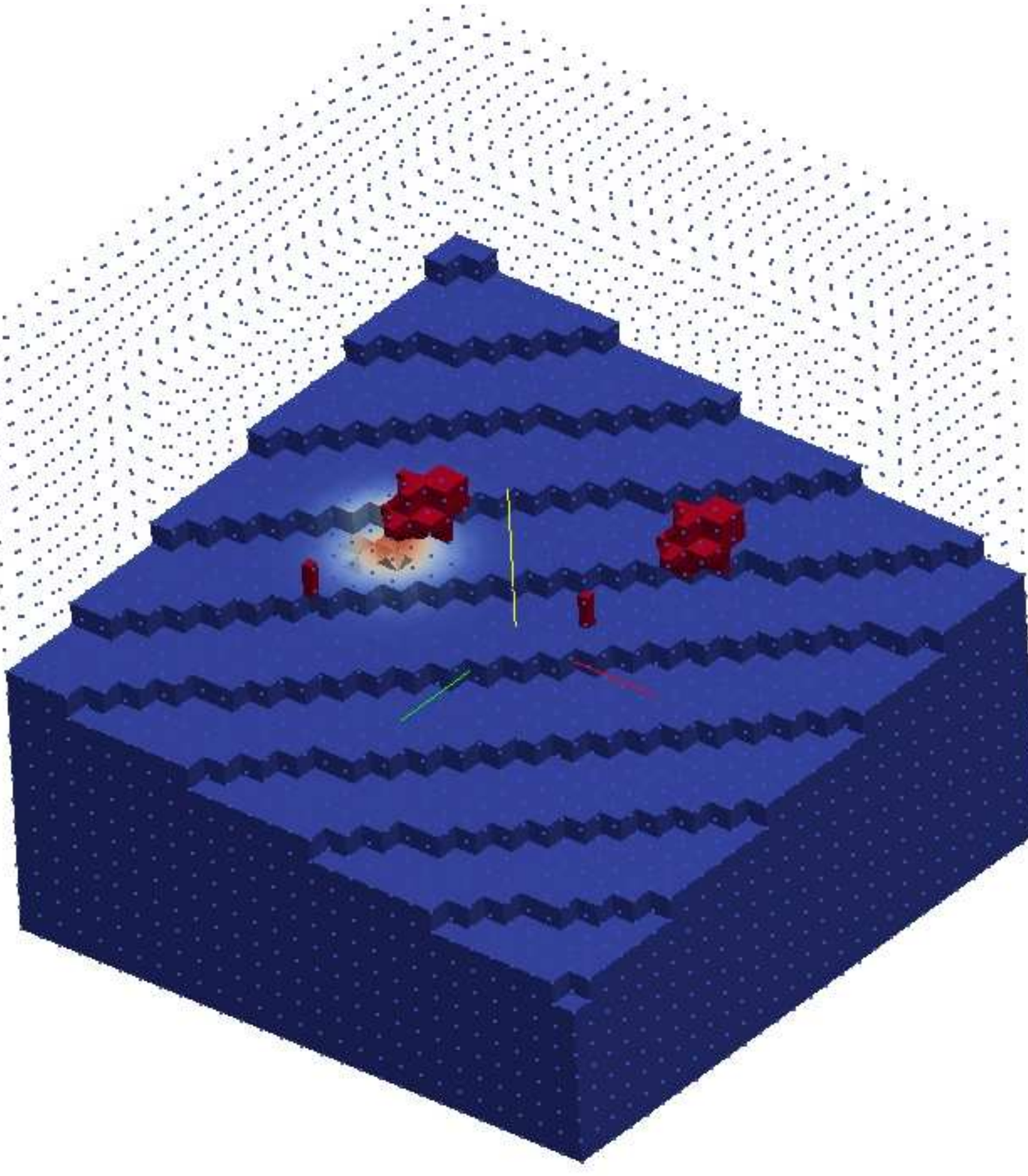}} &
{\includegraphics[scale=0.22, trim = 2.0cm 6.0cm 2.0cm 6.0cm, clip=true,]{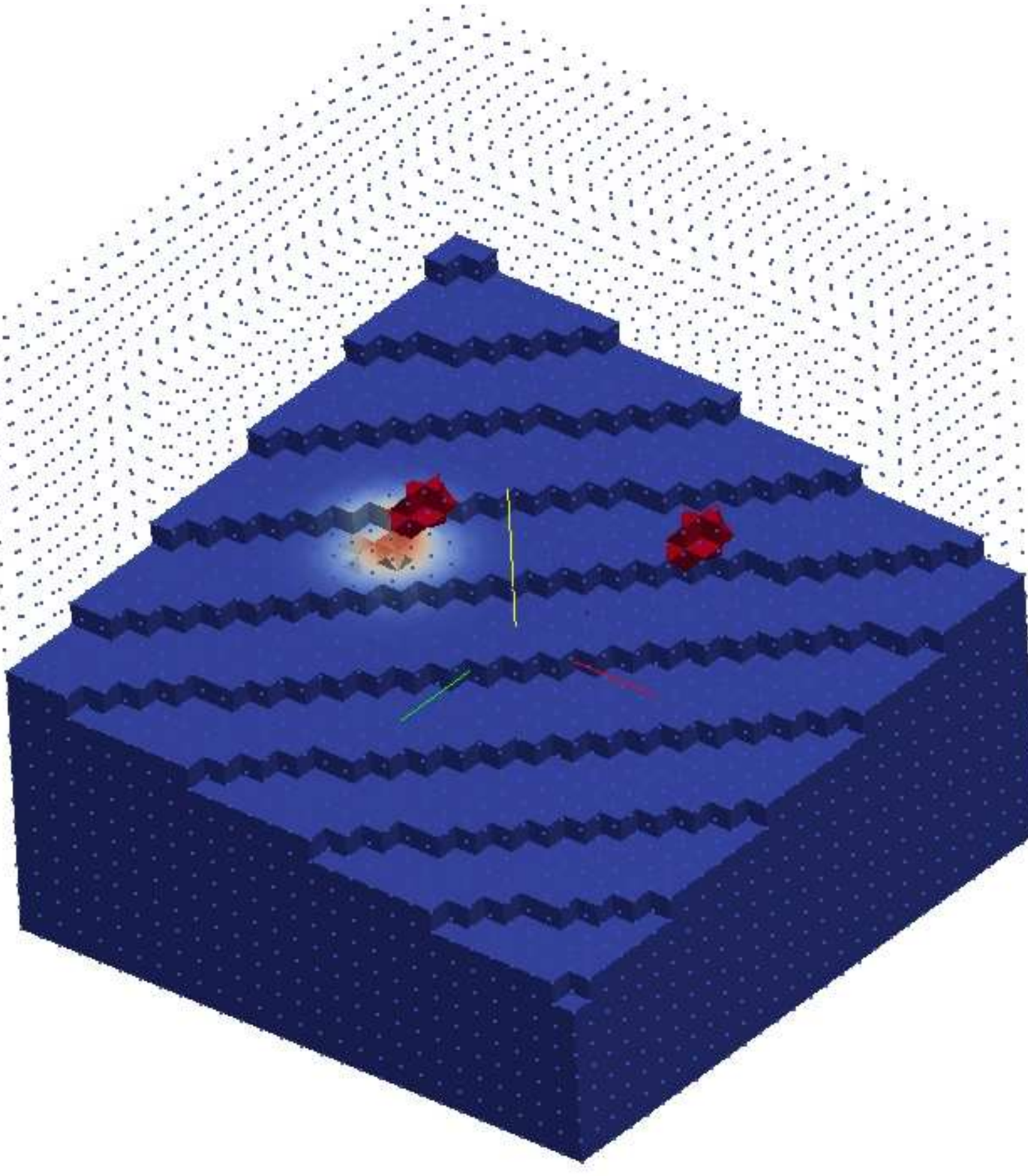}} \\
$\x\in\Omega_\mathrm{FEM}:\eps_{h_0}(\x) = 1.2$ & $\x\in\Omega_\mathrm{FEM}:\eps_{h_0}(\x) = 1.5 $ & $\x\in\Omega_\mathrm{FEM}:\eps_{h_0}(\x) = 1.8 $ 
\end{tabular}
\end{center}
\caption{\small\emph{Test 1. Prospect views of reconstructed isosurfaces of $\eps_{h_0}$ obtained on a coarse mesh. In this test we have reconstructed  $\max_{\Omega_\mathrm{FEM}} \eps_{h_0} = 1.94$, see Table 1. Noise level in data is $\sigma=10\%$.}}
 \label{fig:test1cellscoarse}
 \end{figure}

 \begin{figure}
 \begin{center}
 \begin{tabular}{ccc}
 {\includegraphics[scale=0.19, trim = 1.0cm 4.0cm 1.0cm 4.0cm, clip=true,]{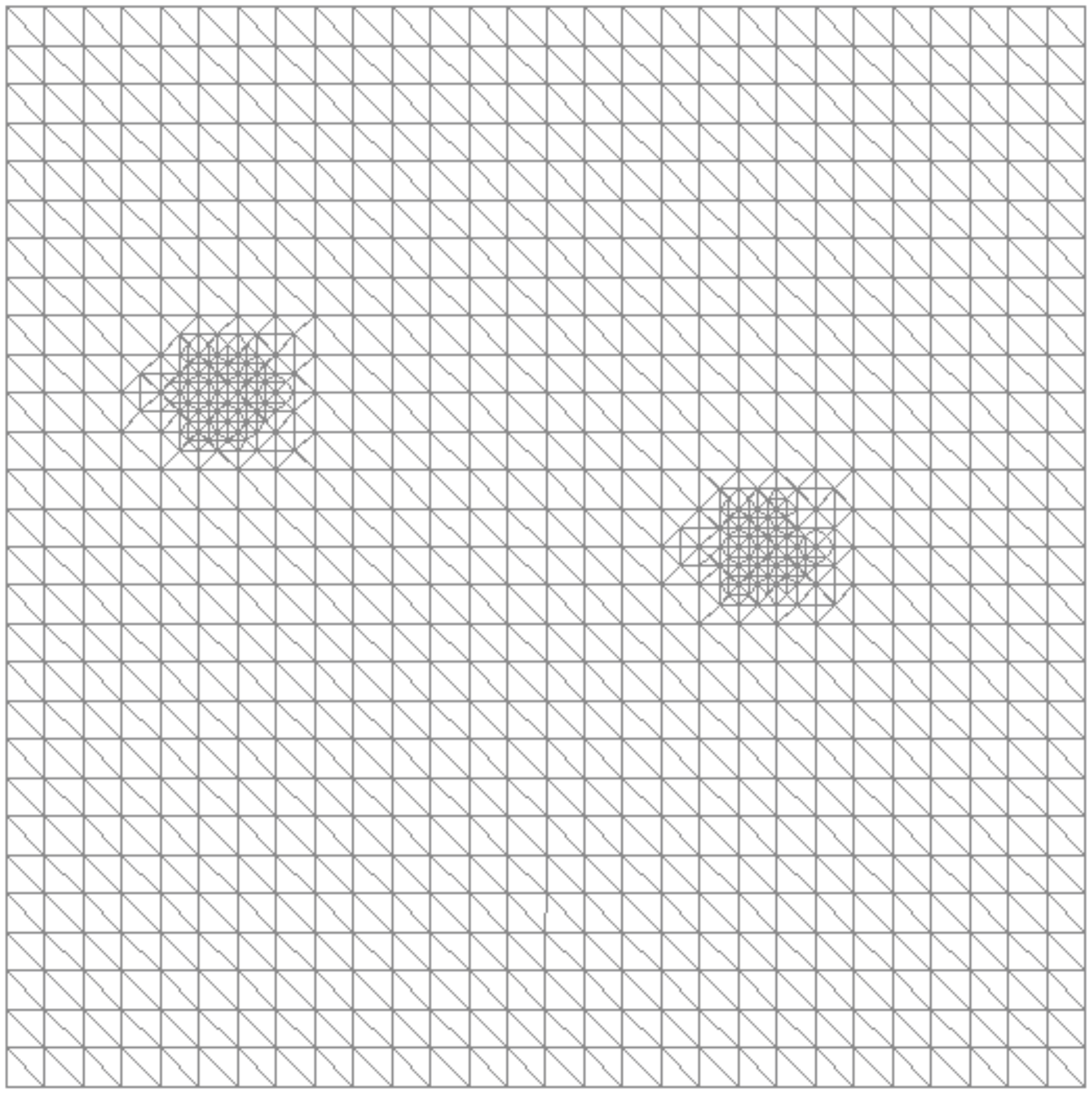}} &
 {\includegraphics[scale=0.19, trim = 1.0cm 4.0cm 1.0cm 4.0cm, clip=true,]{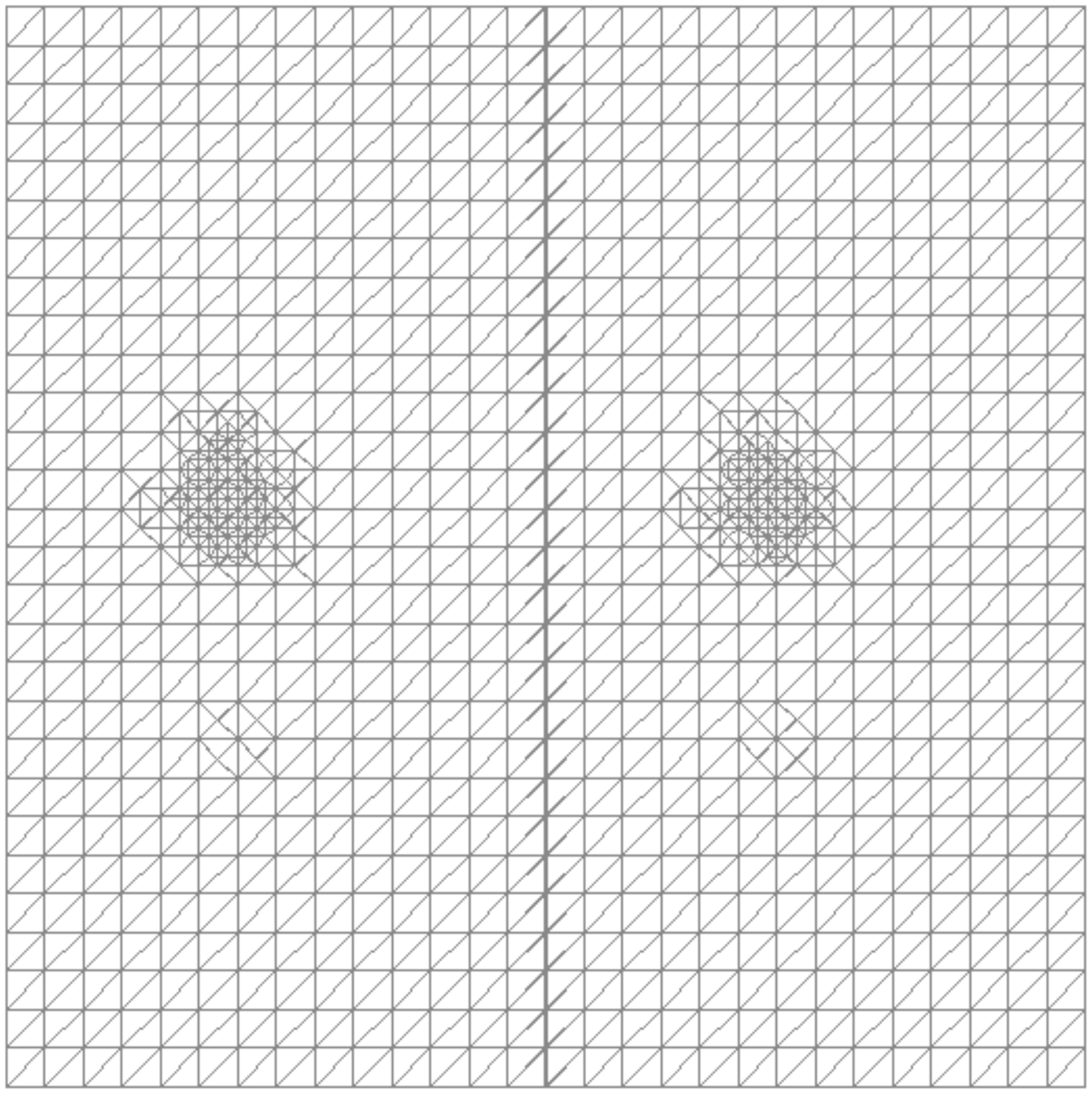}} &
 {\includegraphics[scale=0.19, trim = 1.0cm 4.0cm 1.0cm 4.0cm, clip=true,]{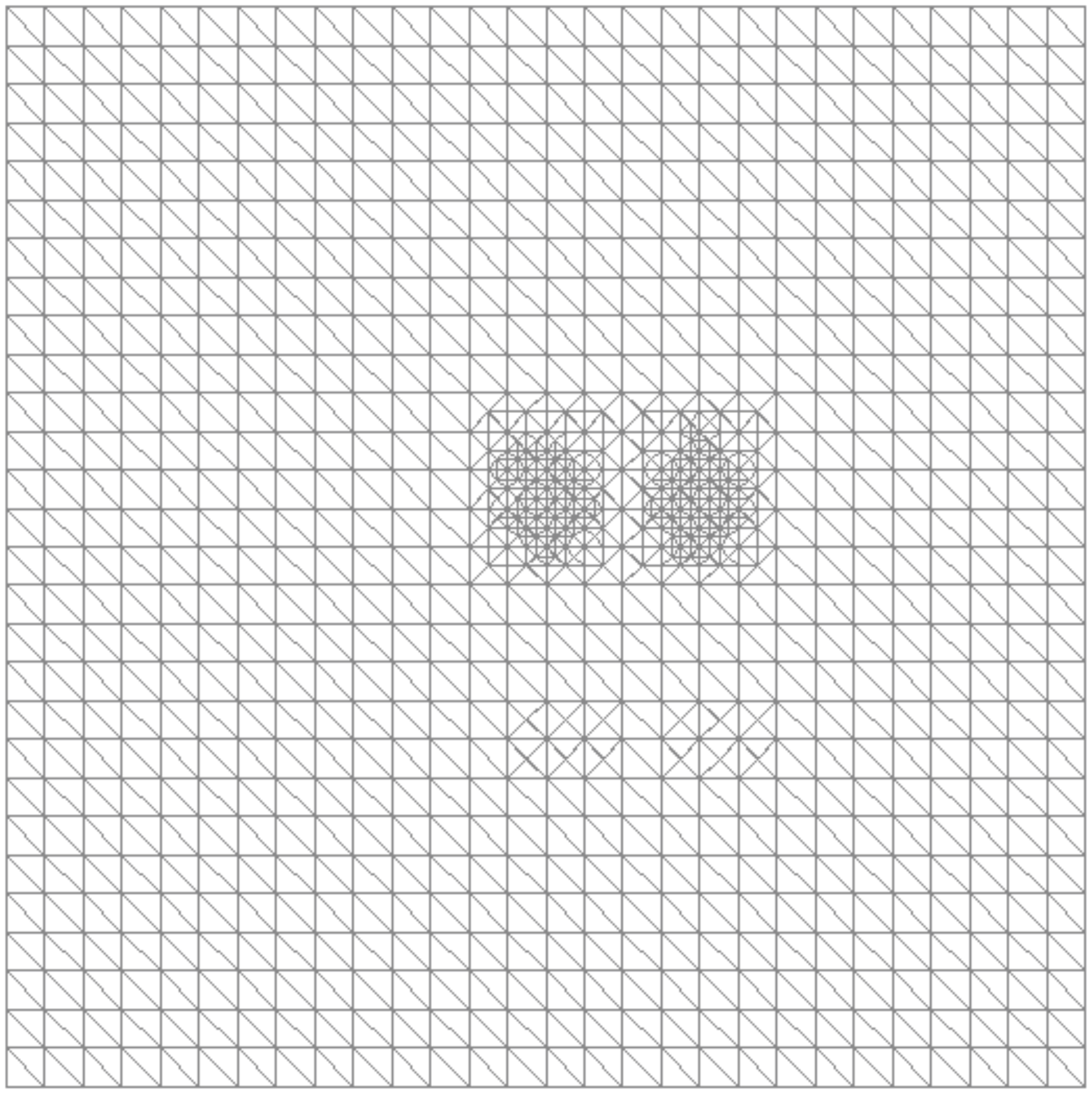}} \\
 $x_1 x_2$-view &  $x_1 x_3$-view & $x_2 x_3$-view
 \end{tabular}
 \end{center}
 \caption{\small\emph{Test 1. Five times adaptively refined mesh when the level of the noise in the data was $\sigma=10\%$.}}
 \label{fig:test1meshes}
 \end{figure}

\begin{figure}
\begin{center}
\begin{tabular}{ccc}
{\includegraphics[scale=0.22, trim = 2.0cm 6.0cm 2.0cm 6.0cm, clip=true,]{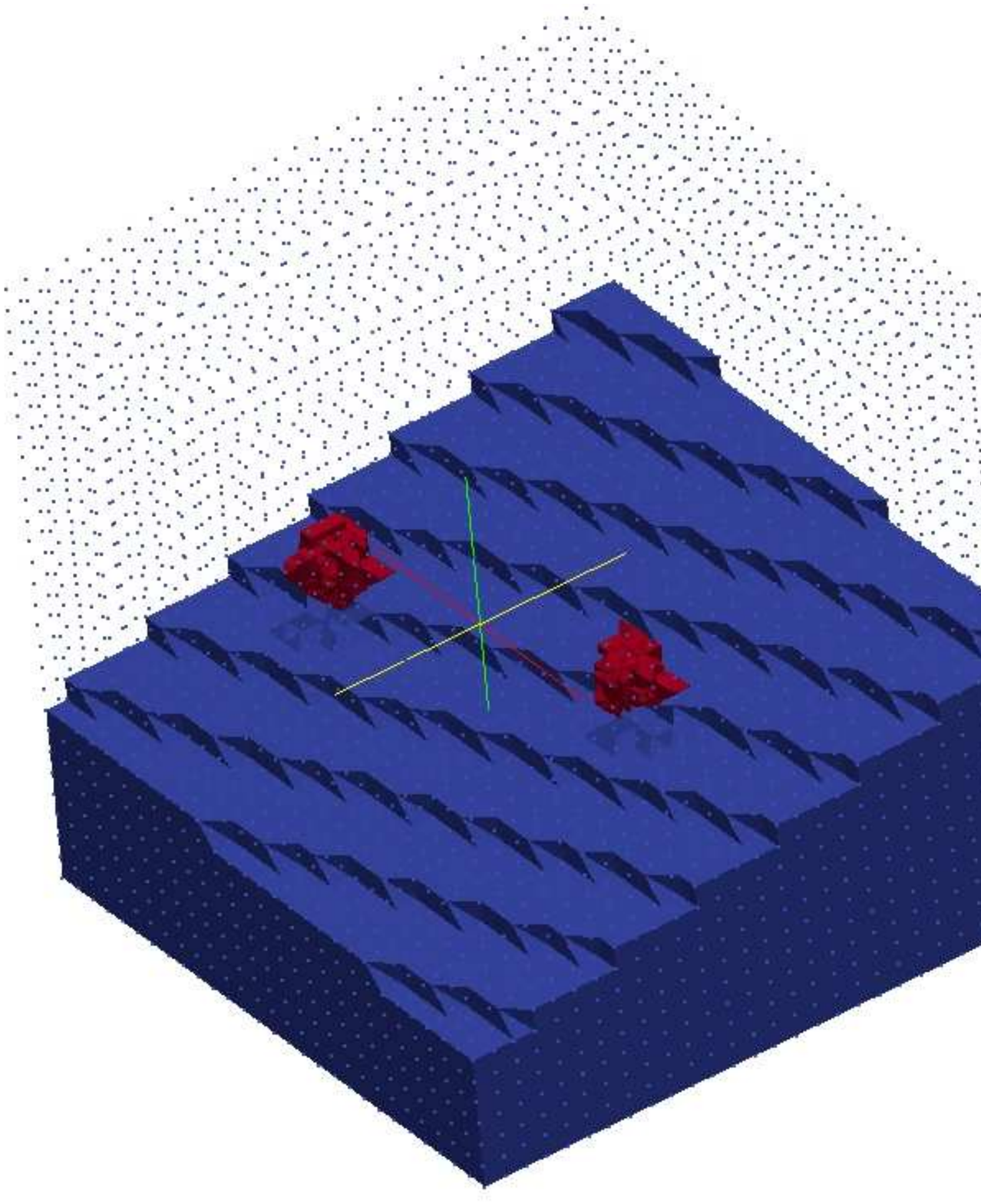}} &
{\includegraphics[scale=0.22, trim = 2.0cm 6.0cm 2.0cm 6.0cm, clip=true,]{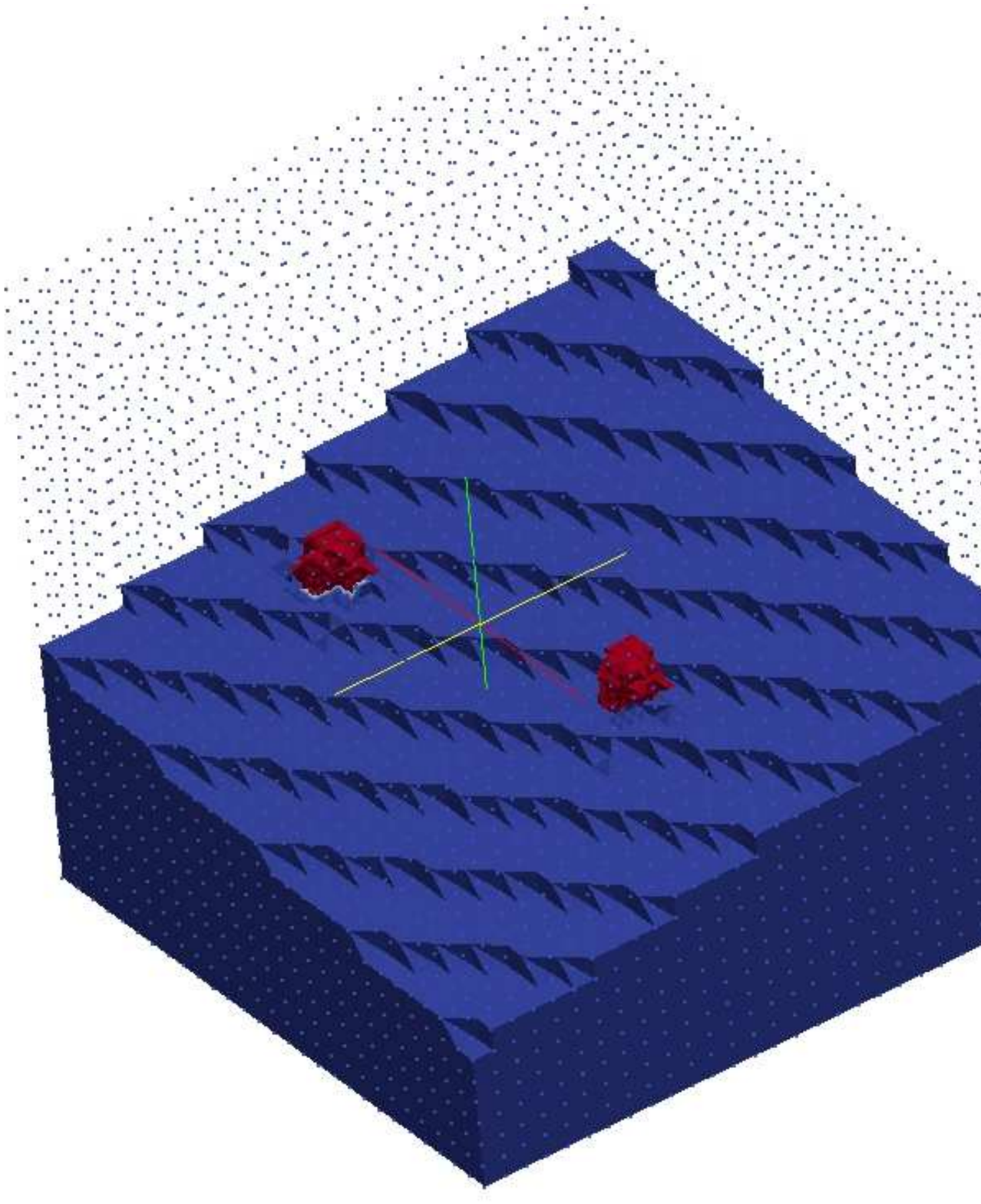}} &
{\includegraphics[scale=0.22, trim = 2.0cm 6.0cm 2.0cm 6.0cm, clip=true,]{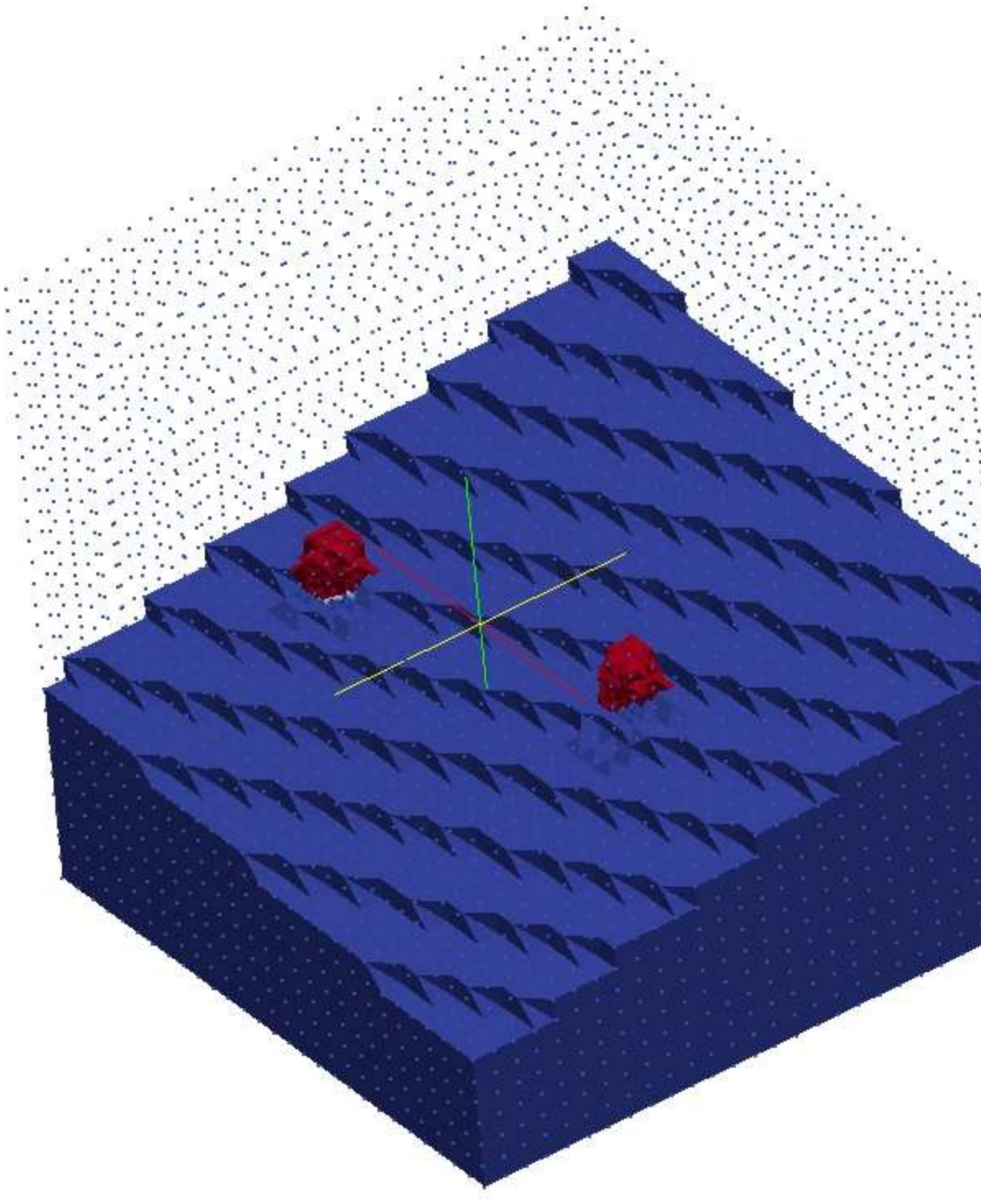}} \\
$\x\in\Omega_\mathrm{FEM}:\eps_\mathrm{rec}(\x) = 1.2$ & $\x\in\Omega_\mathrm{FEM}:\eps_\mathrm{rec}(\x) = 1.5 $ & $\x\in\Omega_\mathrm{FEM}:\eps_\mathrm{rec}(\x) = 1.8 $ 
\end{tabular}
\end{center}
\caption{\small\emph{Test 1. Prospect views of reconstructed isosurfaces of $\eps_\mathrm{rec}$ on a 5 times adaptively refined mesh. In this test we have obtained  $\max_{\Omega_\mathrm{FEM}} \eps_\mathrm{rec} = 1.97$, see Table 2. The noise level in the data is $\sigma=10\%$.}}
\label{fig:test1cellsref5}
\end{figure}

\begin{figure}
\begin{center}
\begin{tabular}{cc}
{\includegraphics[scale=0.4, clip=true,]{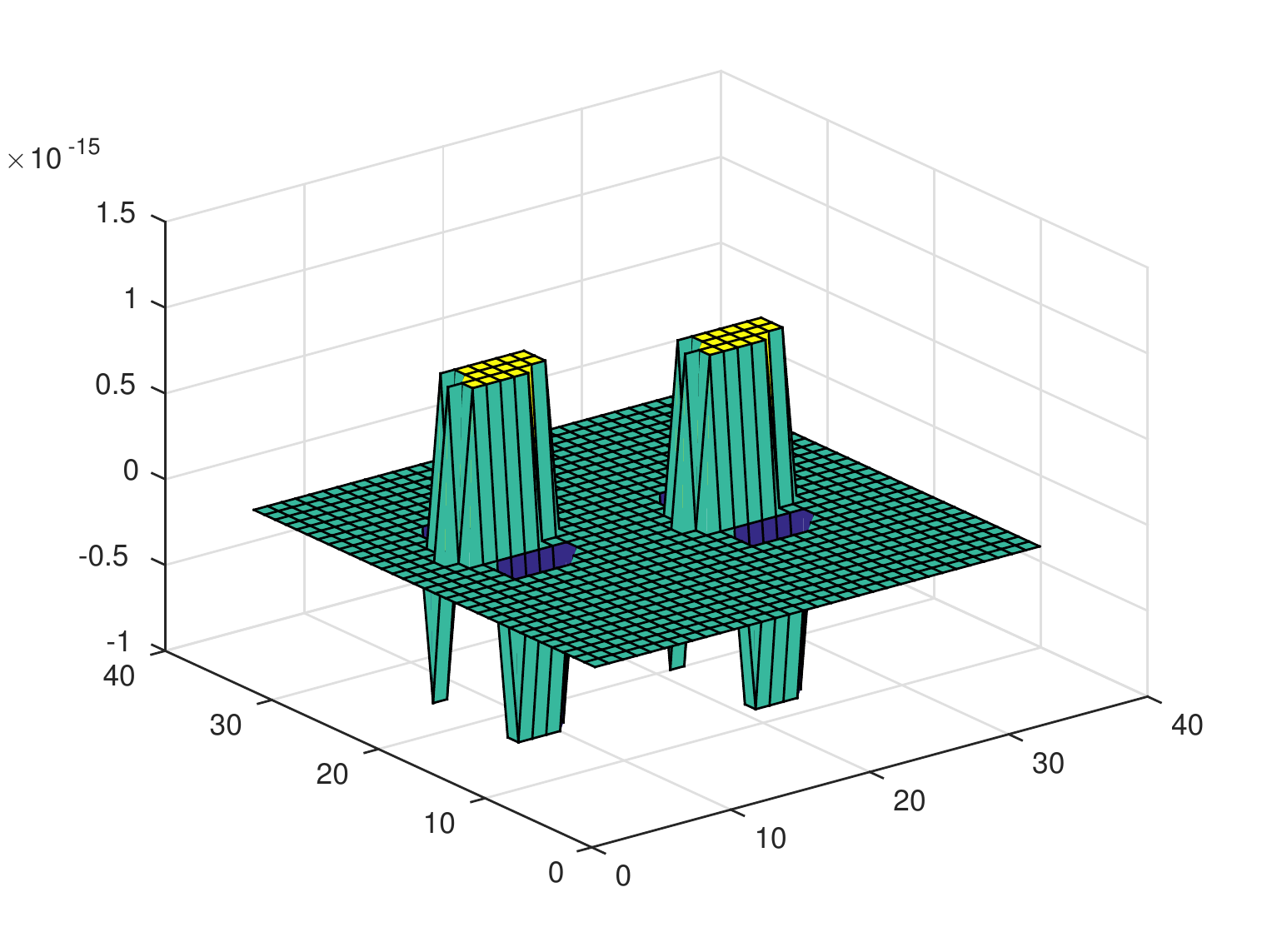}} &
{\includegraphics[scale=0.4, clip=true,]{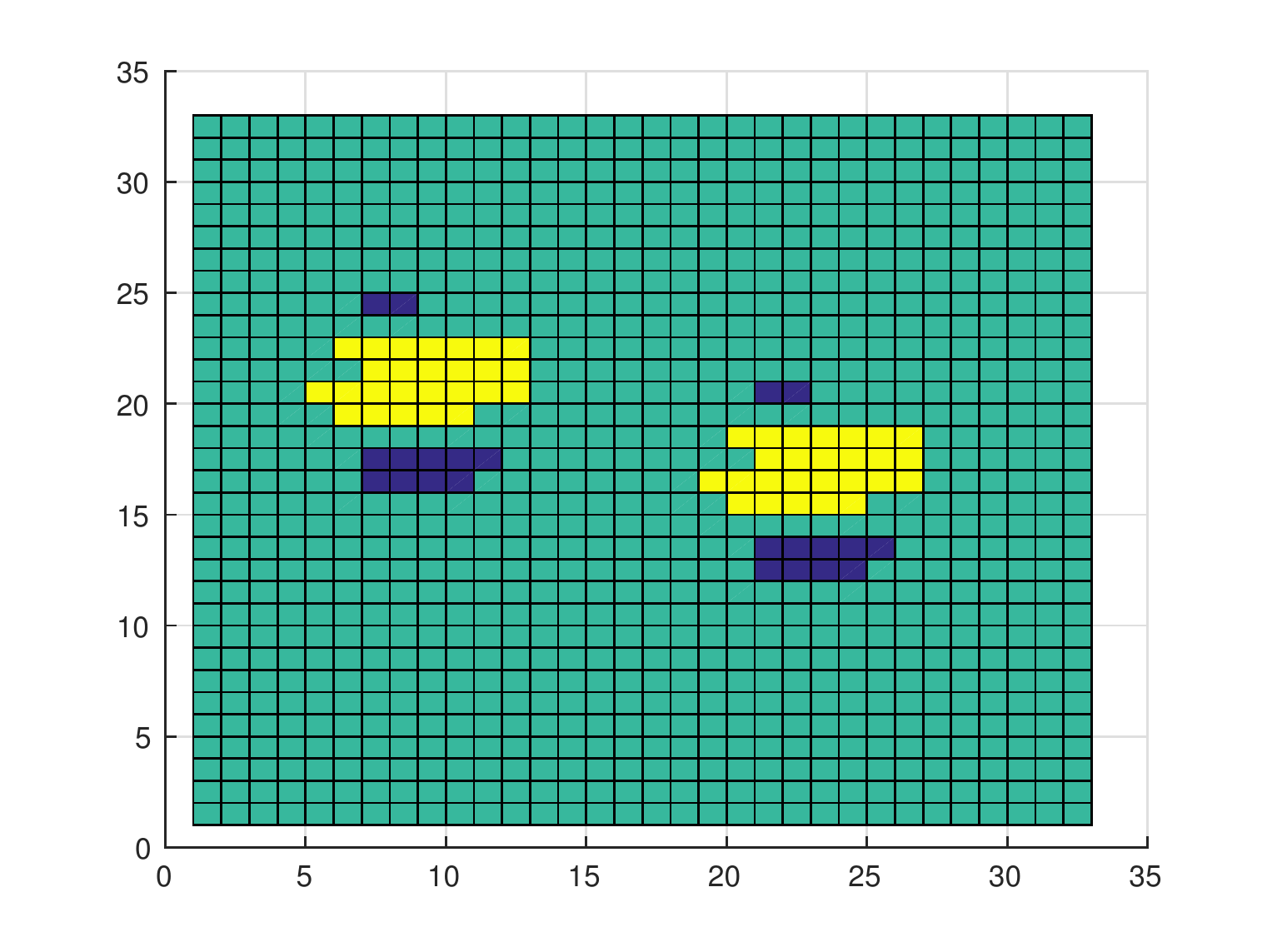}} \\
$t=0.6$  & $t=0.6, x_1 x_2 $ view \\
{\includegraphics[scale=0.4, clip=true,]{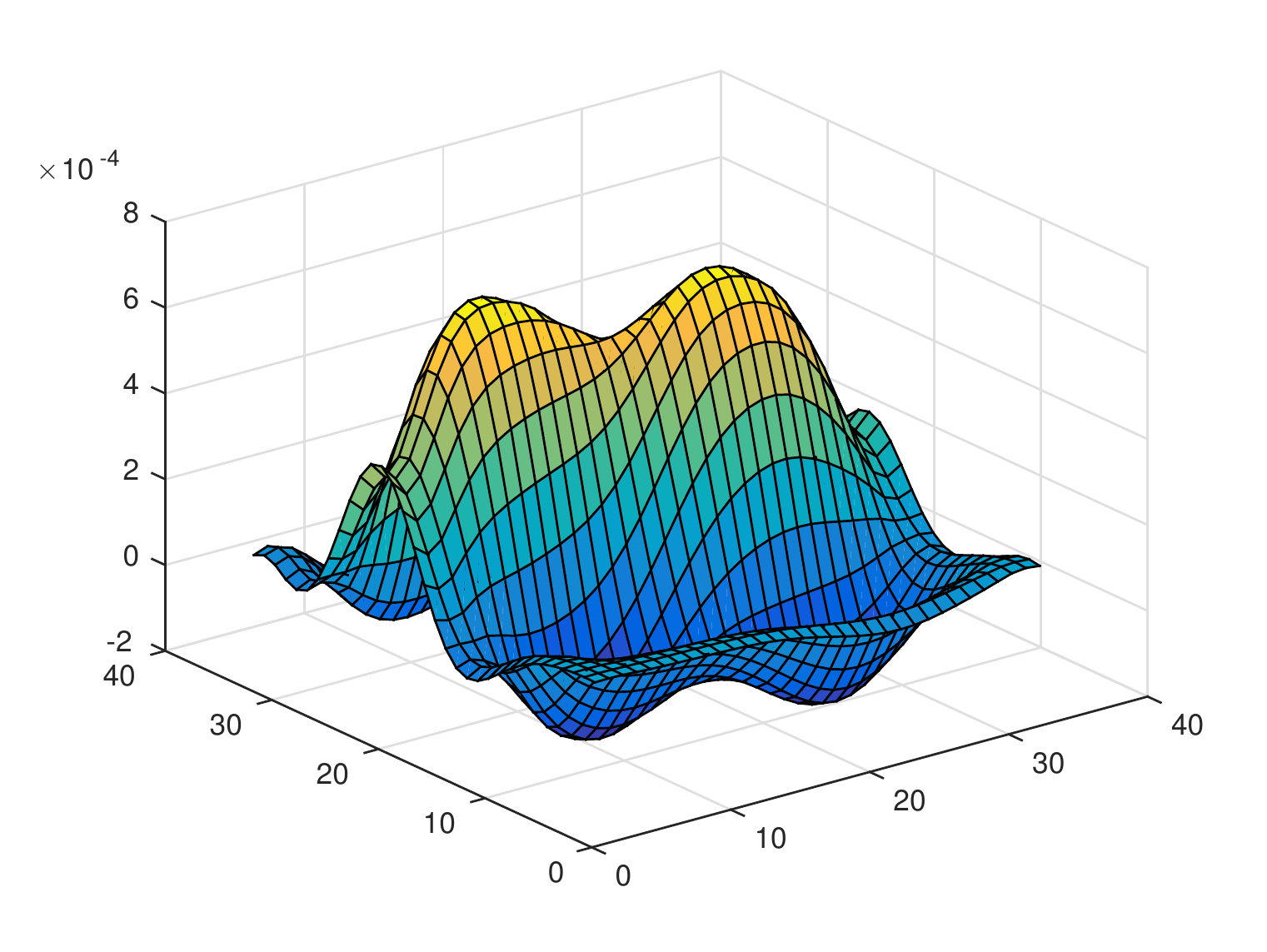}} &
{\includegraphics[scale=0.4, clip=true,]{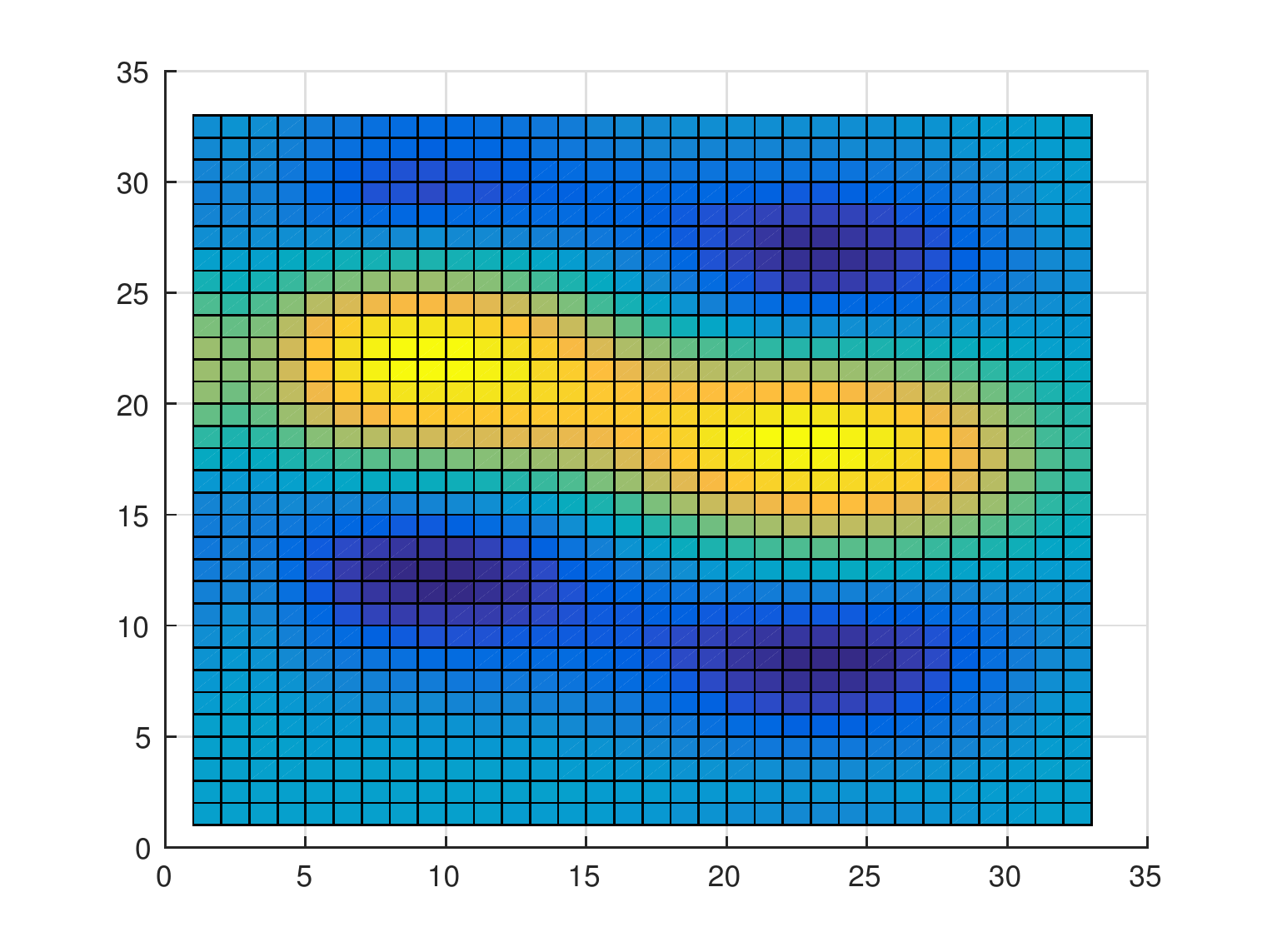}} \\
$t=1.2$  & $t=1.2, x_1 x_2$ view \\
{\includegraphics[scale=0.4, clip=true,]{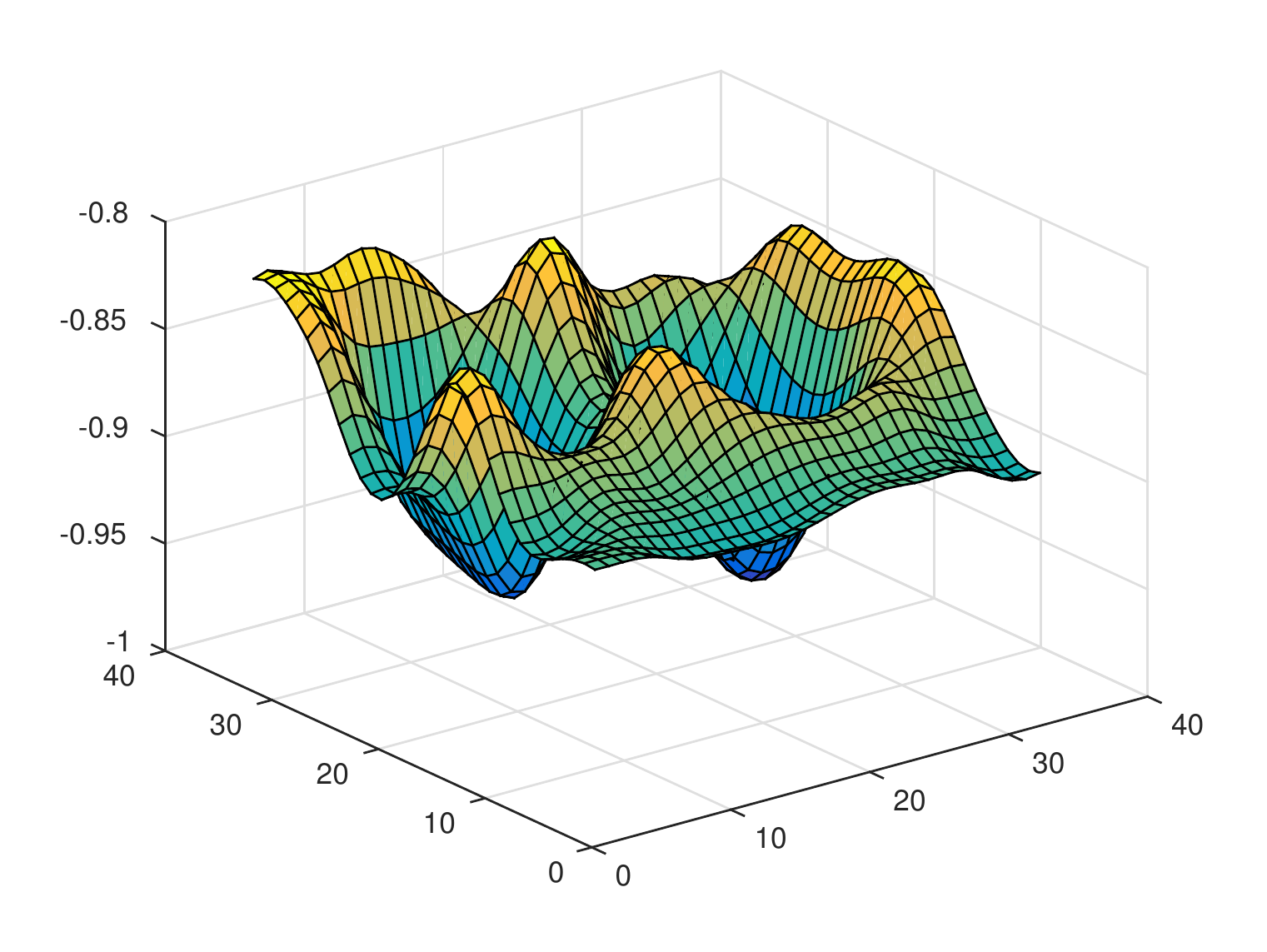}} &
{\includegraphics[scale=0.4, clip=true,]{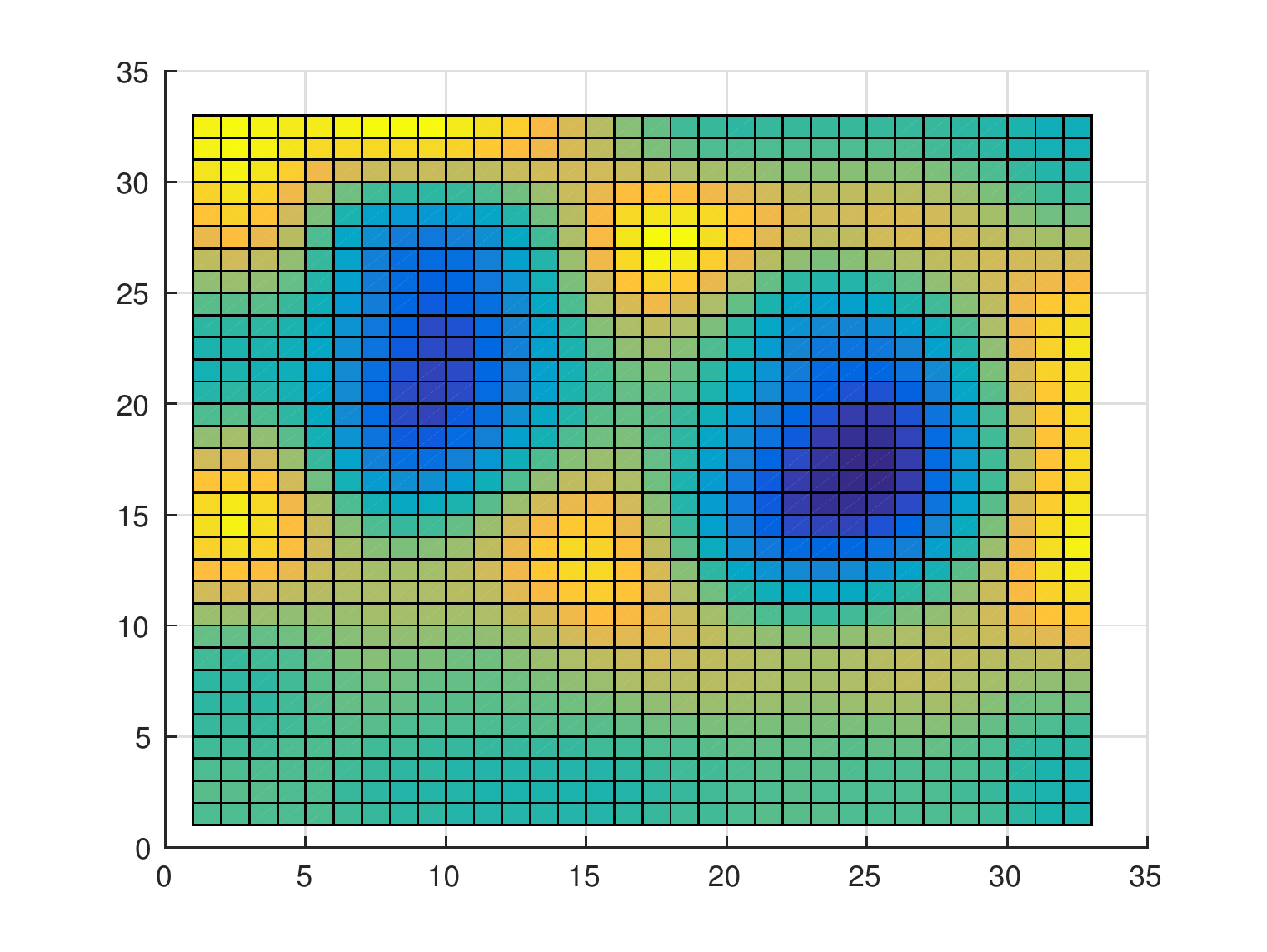}} \\
$t=1.8$  & $t=1.8, x_1 x_2$ view \\
\end{tabular}
\end{center}
\caption{\small\emph{Test 1. Transmitted data of component $E_2$ at different times. The noise in the data is $\sigma=10\%$.}}
\label{fig:test1data}
\end{figure}

 \begin{figure}
 \begin{center}
 \begin{tabular}{ccc}
 {\includegraphics[scale=0.18, trim = 2.0cm 6.0cm 2.0cm 6.0cm, clip=true,]{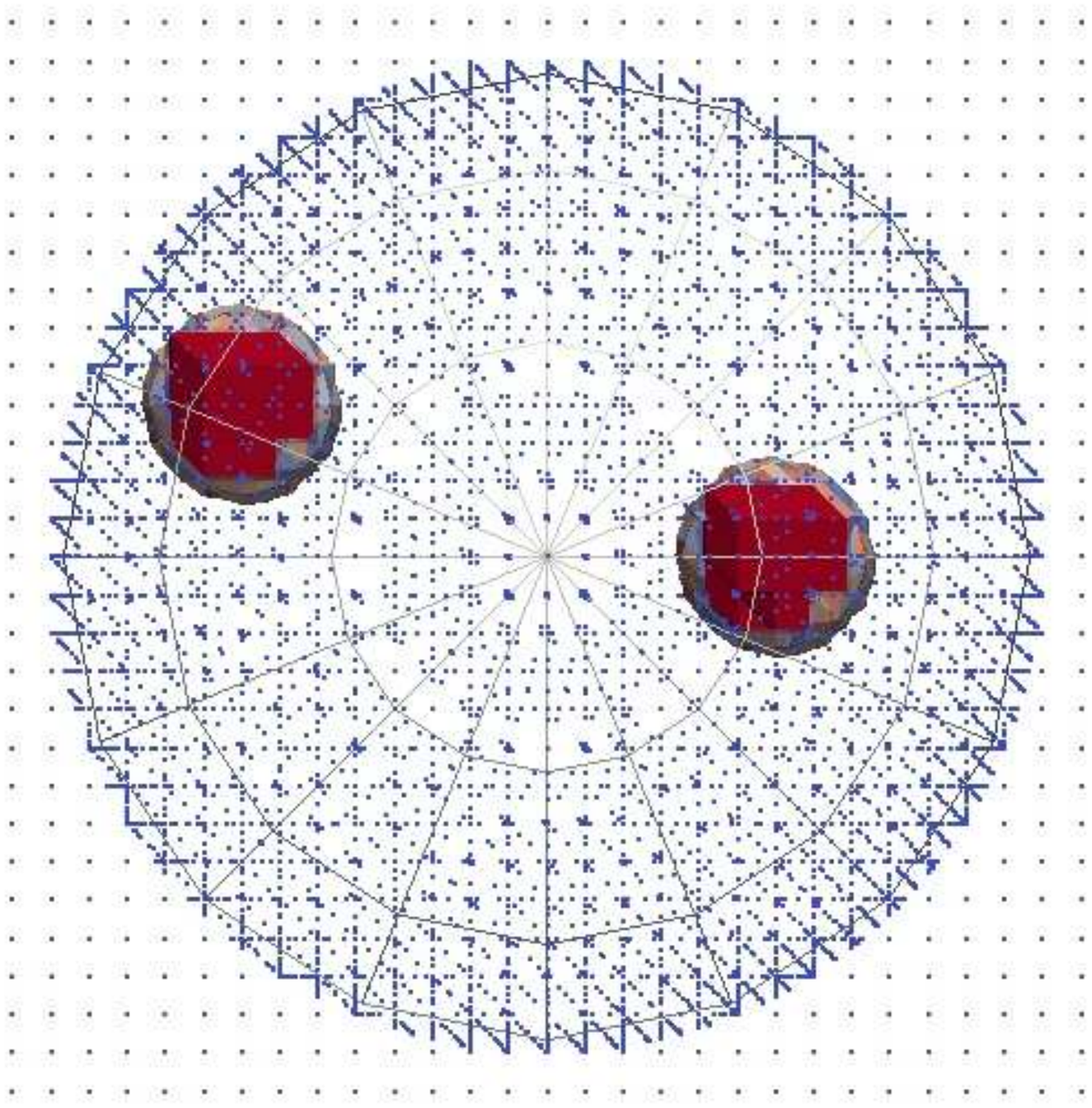}} &
 {\includegraphics[scale=0.18, trim = 2.0cm 6.0cm 2.0cm 6.0cm, clip=true,]{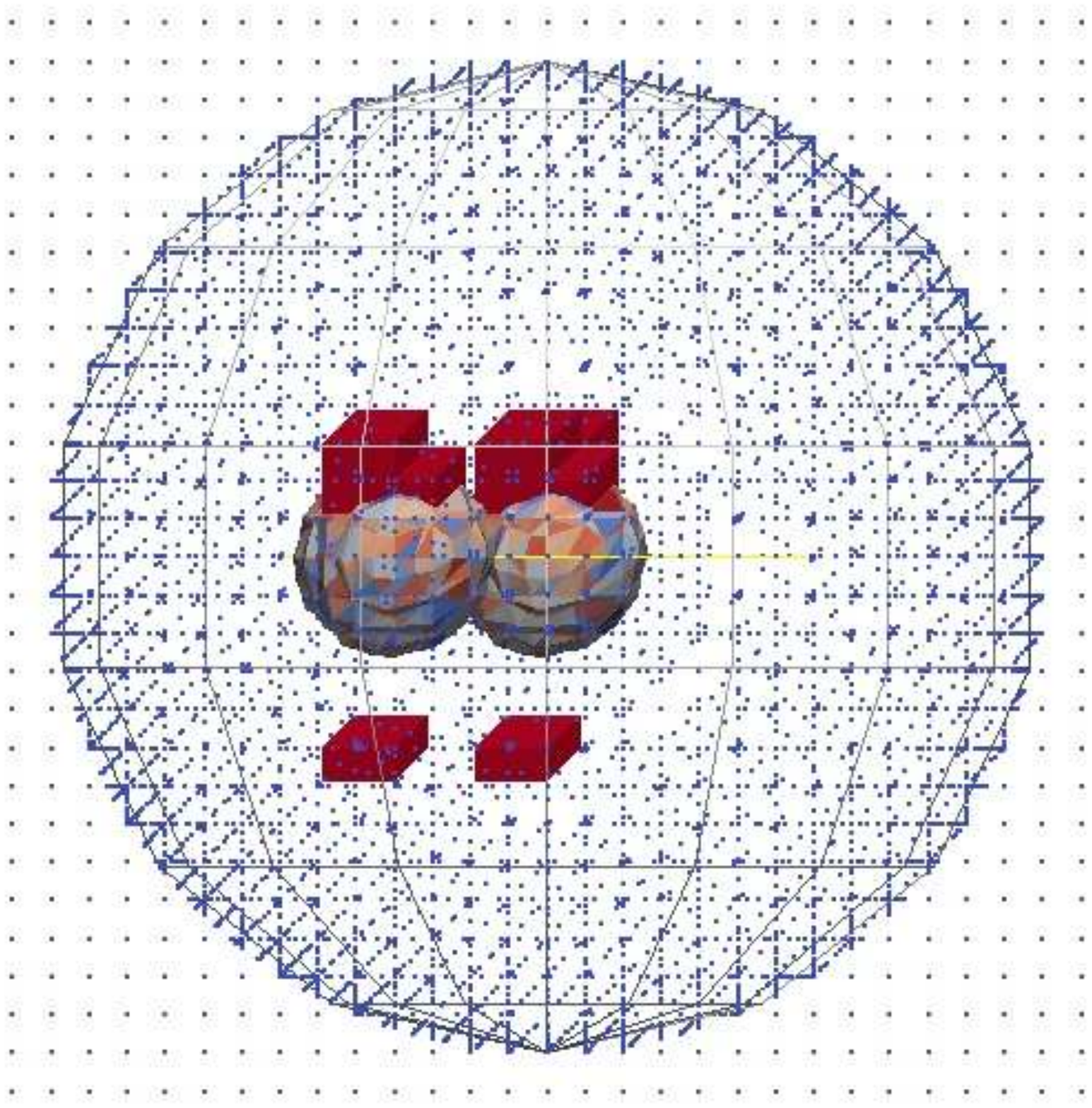}} &
 {\includegraphics[scale=0.18, trim = 2.0cm 6.0cm 2.0cm 6.0cm, clip=true,]{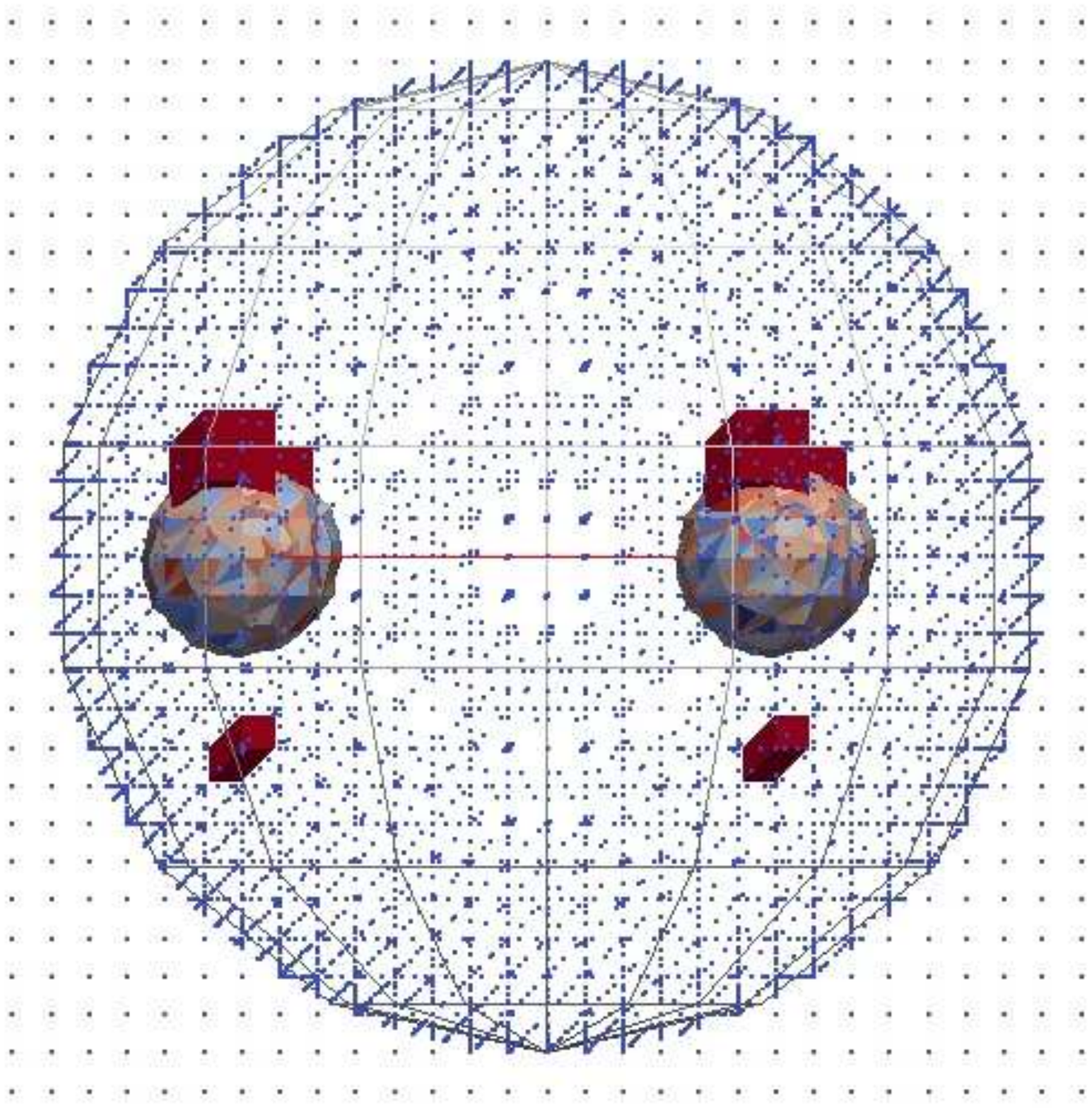}} \\
 $x_1 x_2$ view & $x_2 x_3$ view &   $x_1 x_3$ view \\
 {\includegraphics[scale=0.18, trim = 2.0cm 6.0cm 2.0cm 6.0cm, clip=true,]{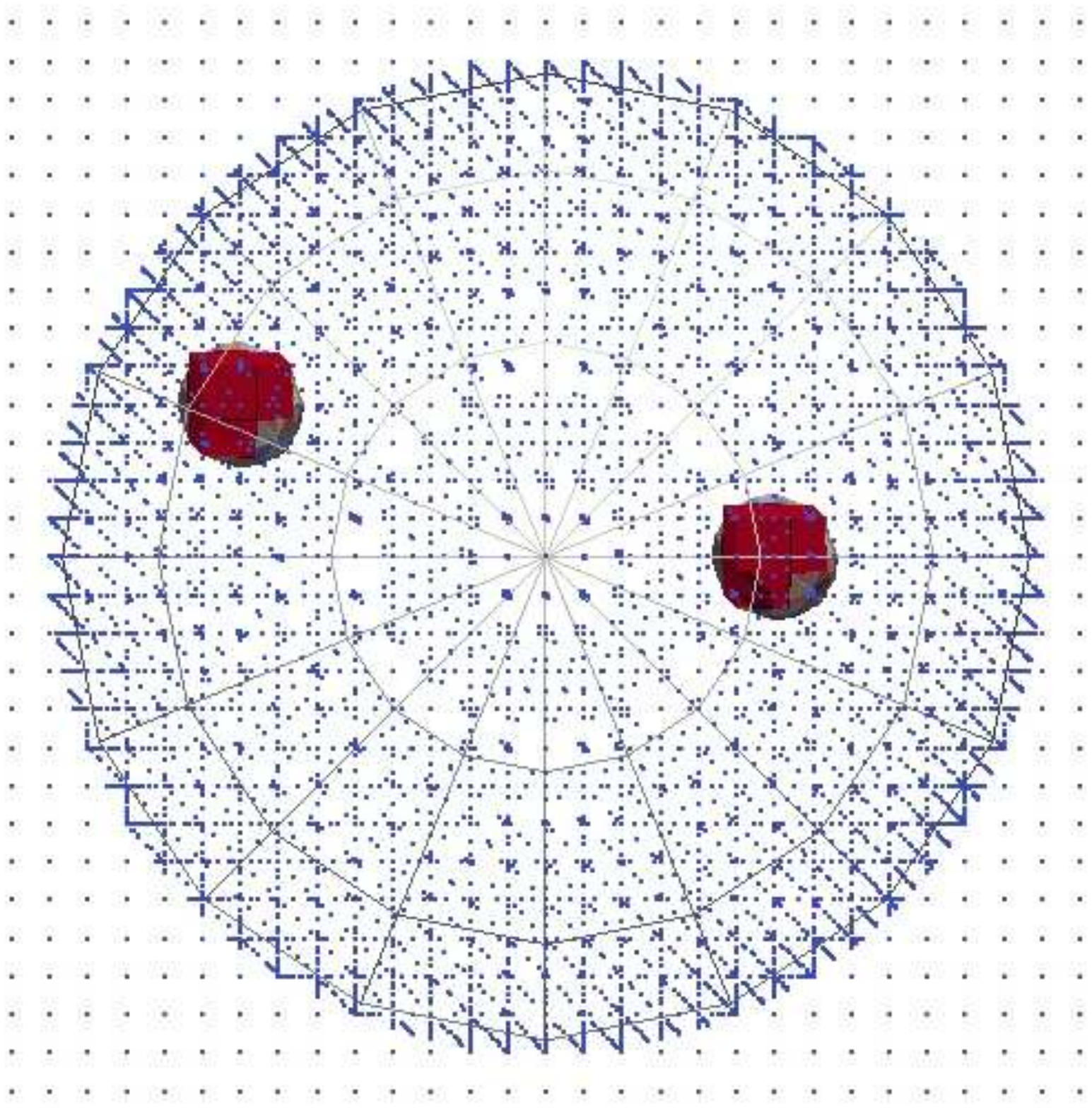}} &
 {\includegraphics[scale=0.18, trim = 2.0cm 6.0cm 2.0cm 6.0cm, clip=true,]{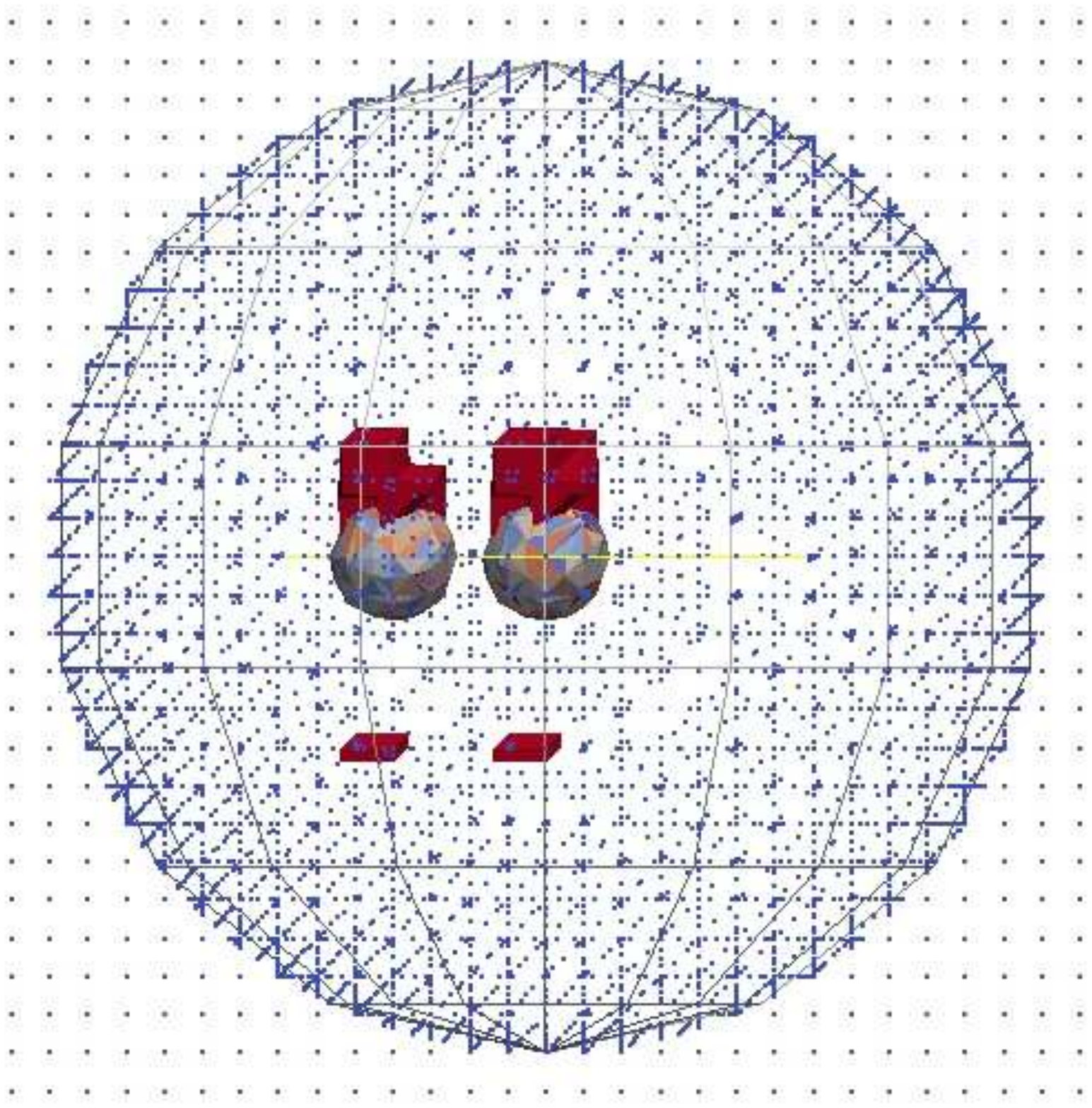}} &
 {\includegraphics[scale=0.18, trim = 2.0cm 6.0cm 2.0cm 6.0cm, clip=true,]{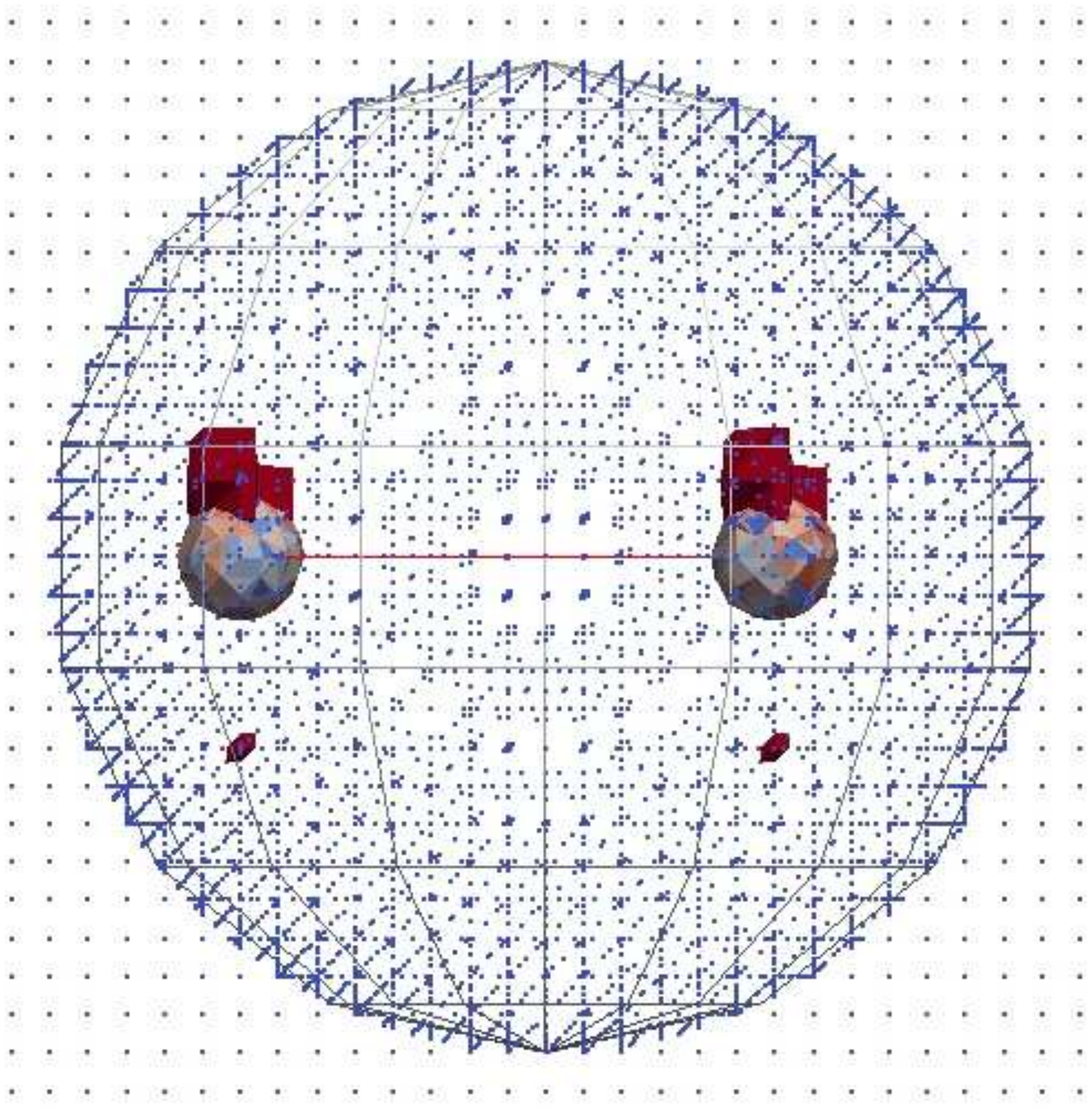}} \\
 $x_1 x_2$ view & $x_2 x_3$ view &   $x_1 x_3$ view \\
 {\includegraphics[scale=0.18, trim = 2.0cm 6.0cm 2.0cm 6.0cm, clip=true,]{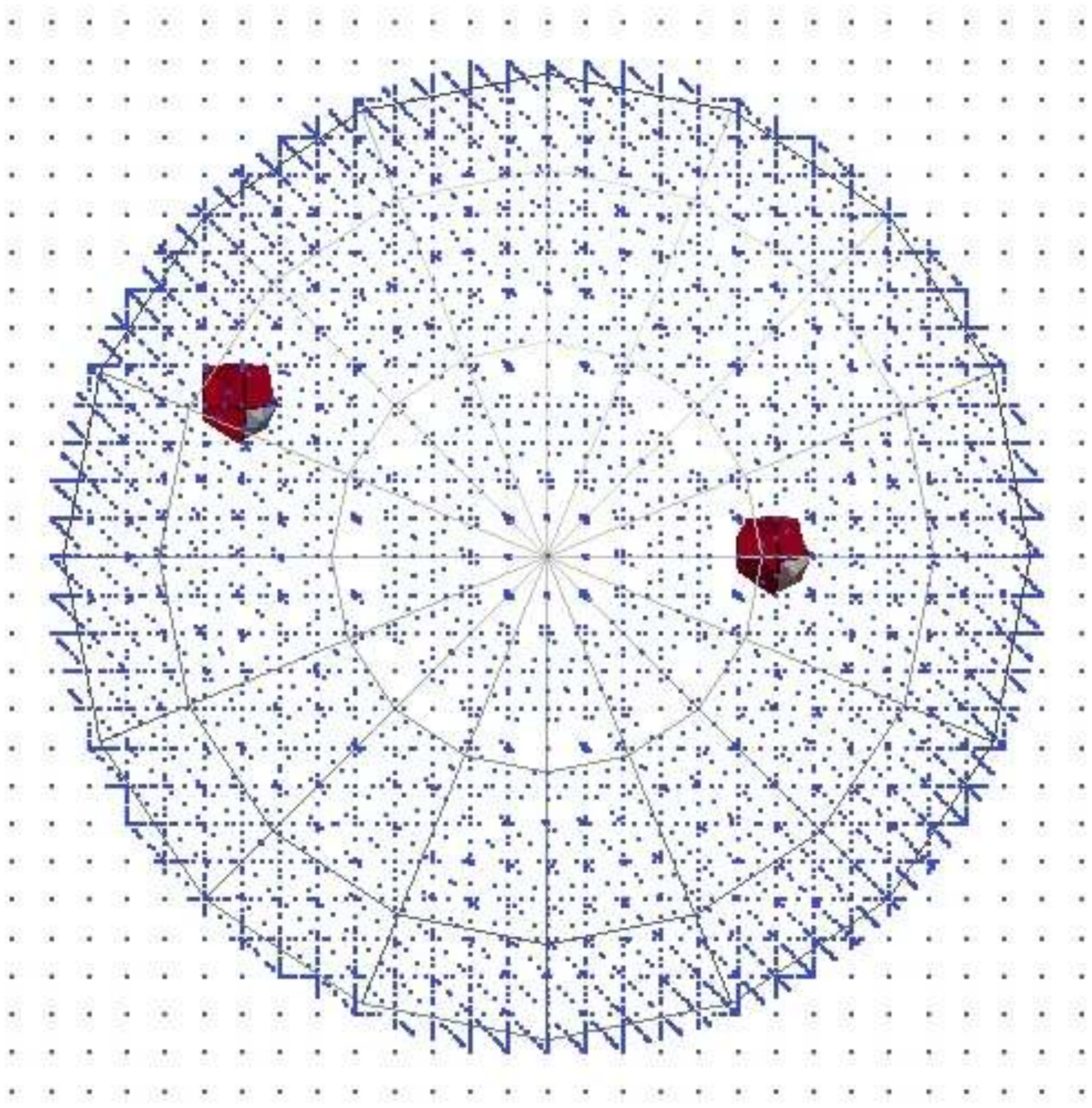}} &
 {\includegraphics[scale=0.18, trim = 2.0cm 6.0cm 2.0cm 6.0cm, clip=true,]{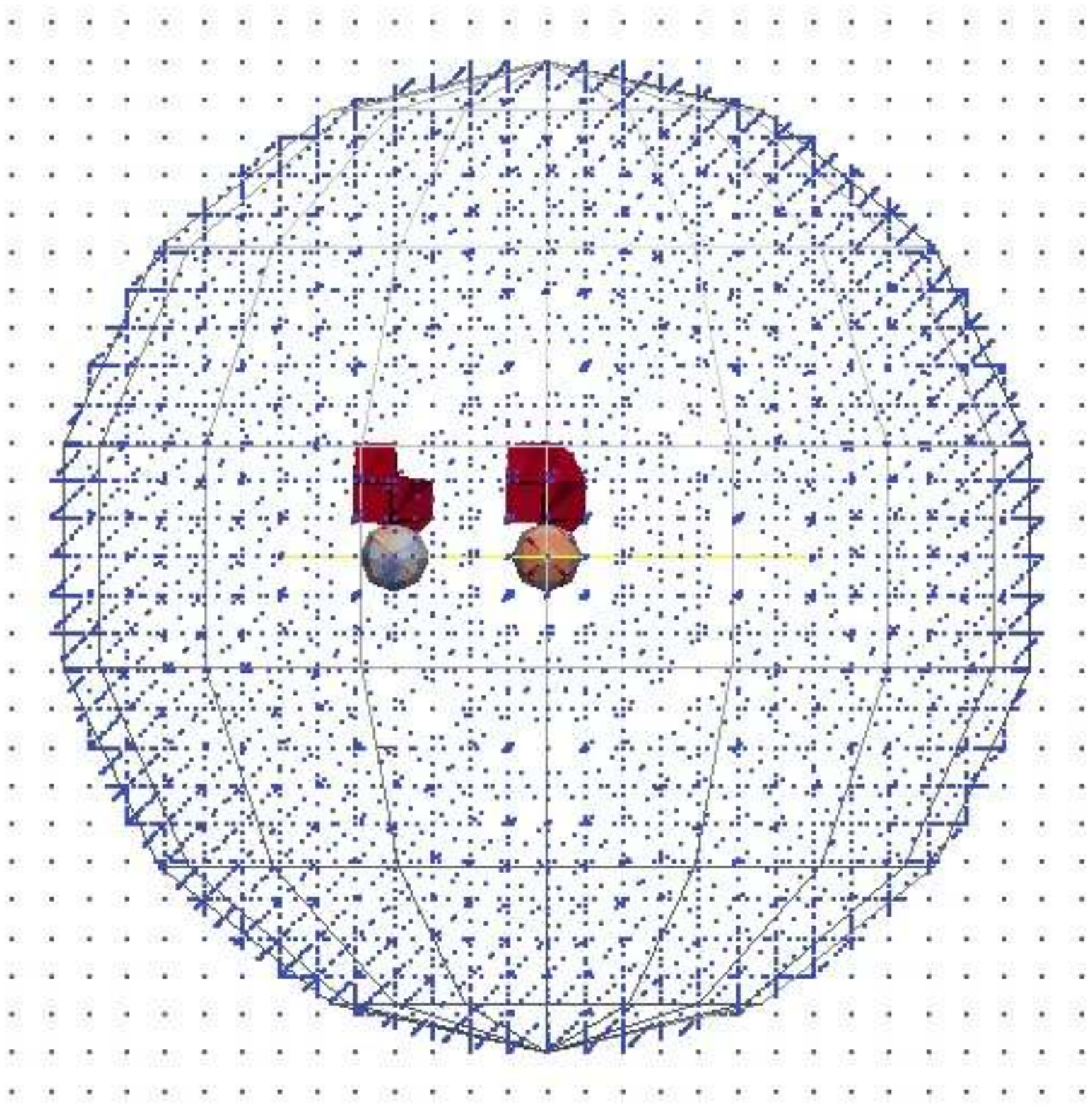}} &
 {\includegraphics[scale=0.18, trim = 2.0cm 6.0cm 2.0cm 6.0cm, clip=true,]{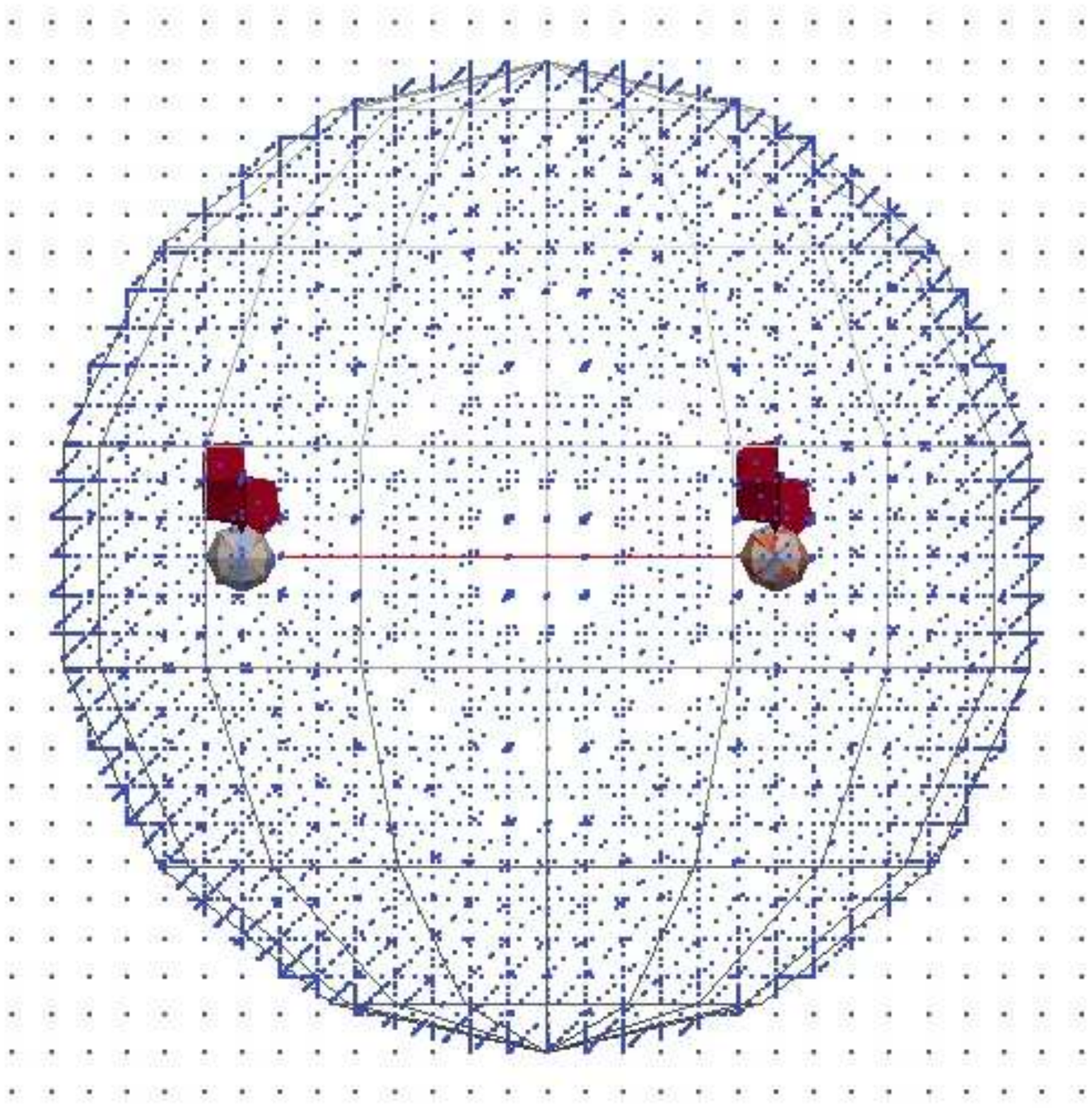}} \\
 $x_1 x_2$ view & $x_2 x_3$ view &   $x_1 x_3$ view  \\
\end{tabular}
\end{center}
\caption{\small\emph{Test 1. In red: isosurfaces $\{\x\in\Omega_\mathrm{FEM}:\eps_{h_0}(\x) = 0.8 \max_{\Omega_\mathrm{FEM}} \eps_{h_0}\}$ on a coarse mesh. Here, $\max_{\Omega_\mathrm{FEM}} \eps_\mathrm{h_0} = 1.94 $, and the noise level in the data is $\sigma=10\%$. For comparison we also present as wireframes the corresponding isosurface of the function \eqref{2gaussians} in every figure.}}
\label{fig:test1noise10coarse}
\end{figure}

\begin{figure}
\begin{center}
\begin{tabular}{ccc}
$\x\in \Omega_\mathrm{FEM}:\eps_\mathrm{rec}(\x) = 1.2$ & $\x\in \Omega_\mathrm{FEM}:\eps_\mathrm{rec}(\x) = 1.5 $ & $\x\in \Omega_\mathrm{FEM}:\eps_\mathrm{rec}(\x) = 1.8 $ \\
{\includegraphics[scale=0.18, trim = 2.0cm 6.0cm 2.0cm 6.0cm, clip=true,]{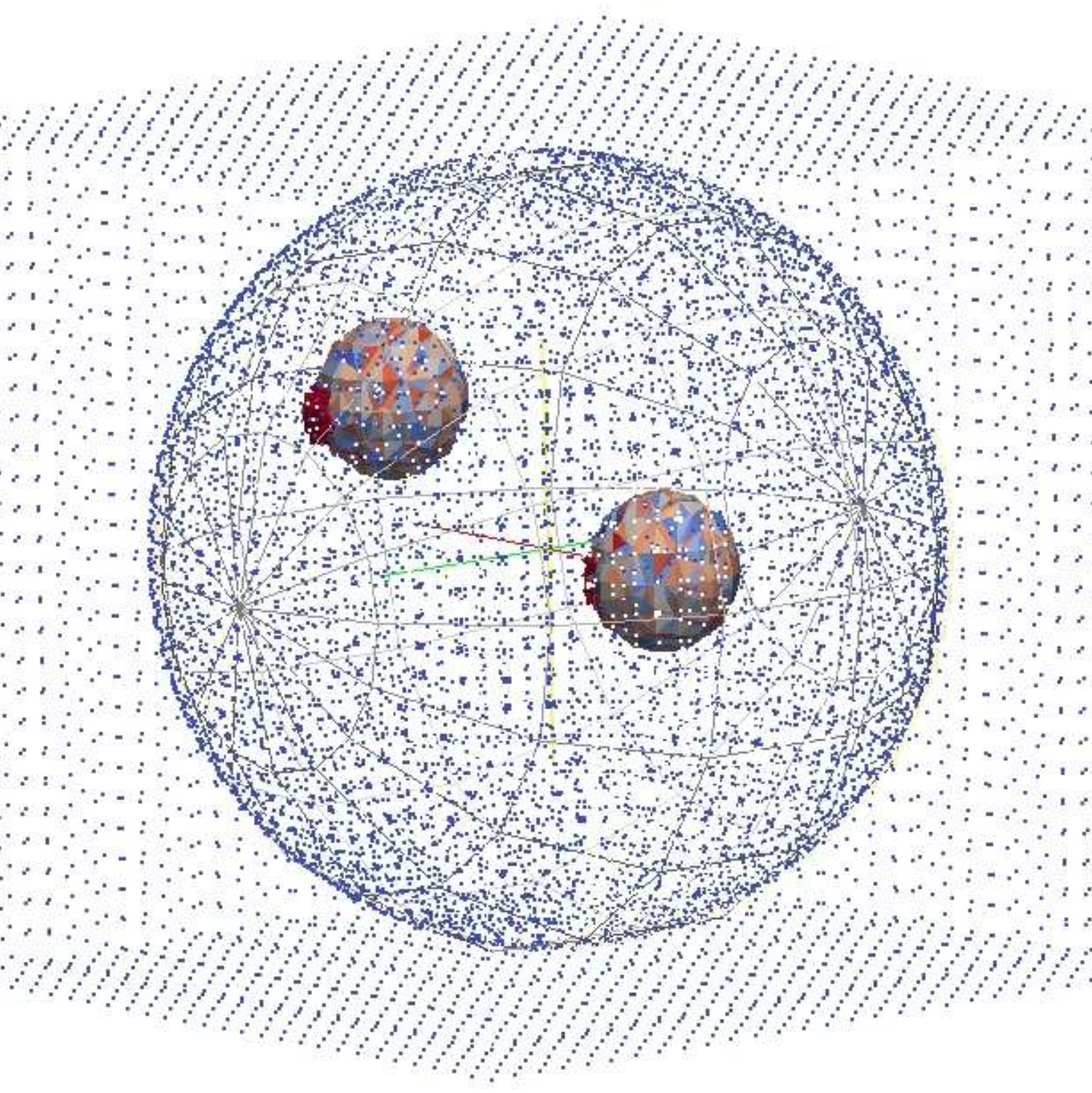}} &
{\includegraphics[scale=0.18, trim = 2.0cm 6.0cm 2.0cm 6.0cm, clip=true,]{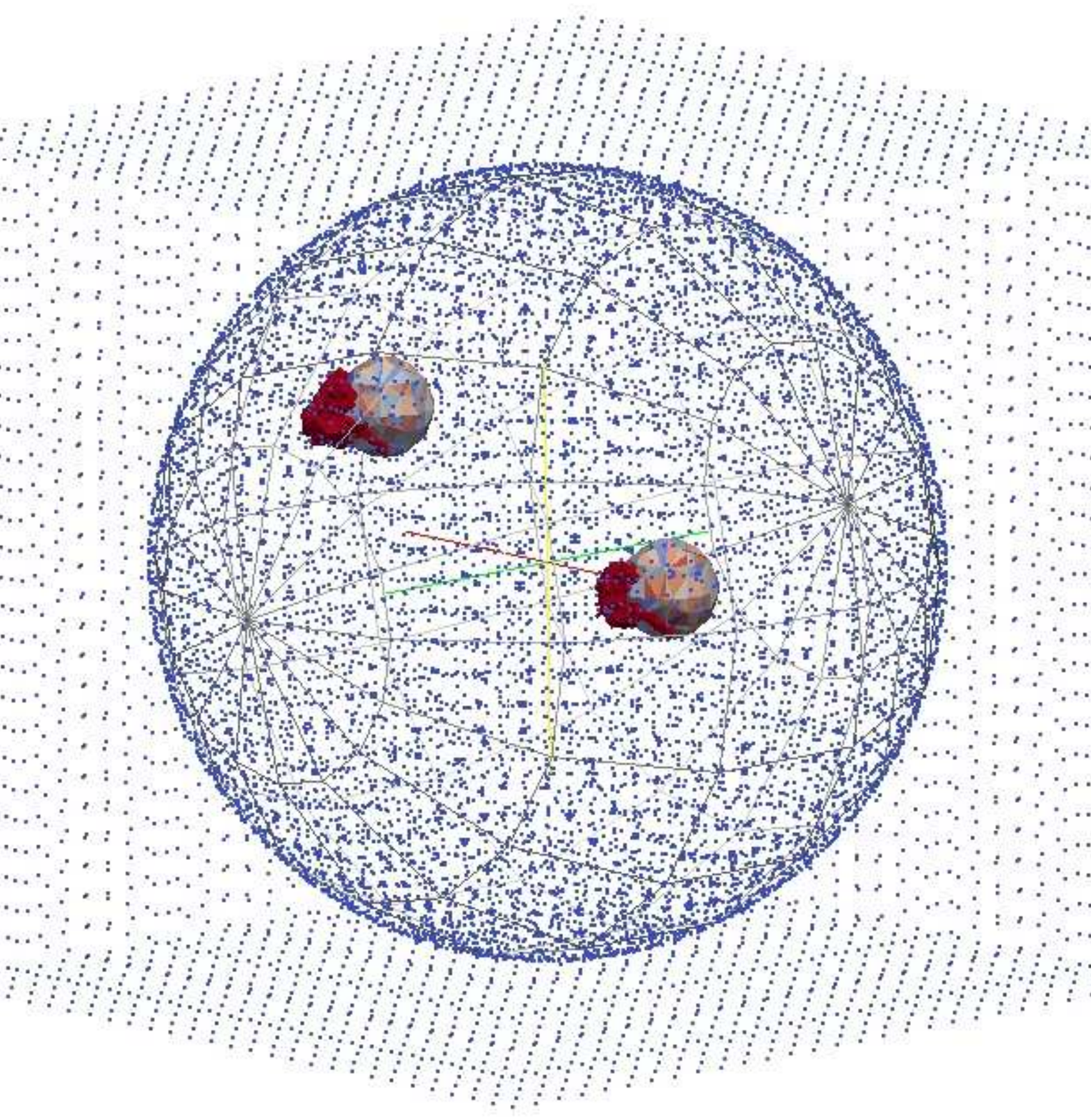}} &
{\includegraphics[scale=0.18, trim = 2.0cm 6.0cm 2.0cm 6.0cm, clip=true,]{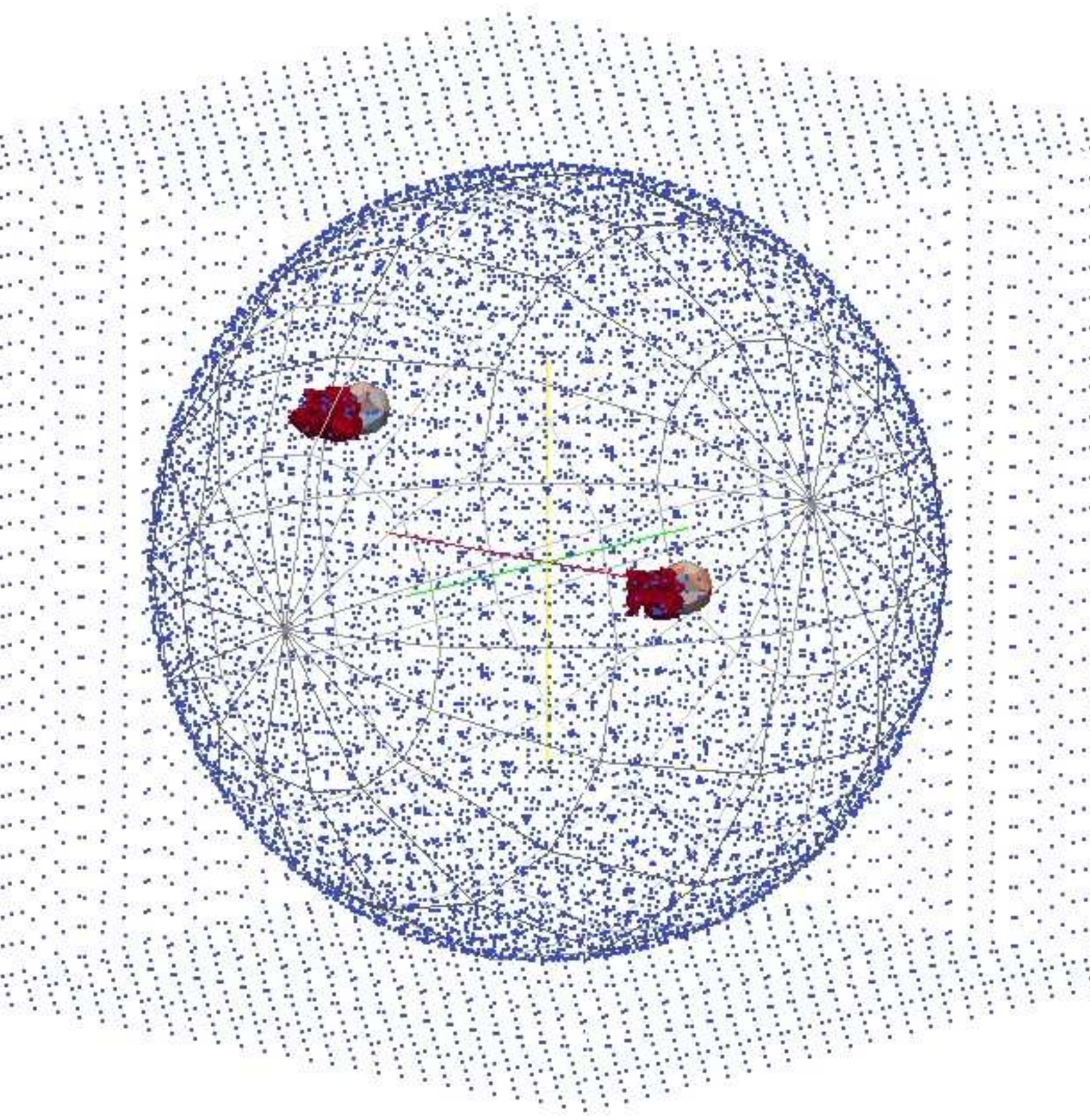}} \\
& prospect view & \\
{\includegraphics[scale=0.18, trim = 2.0cm 6.0cm 2.0cm 6.0cm, clip=true,]{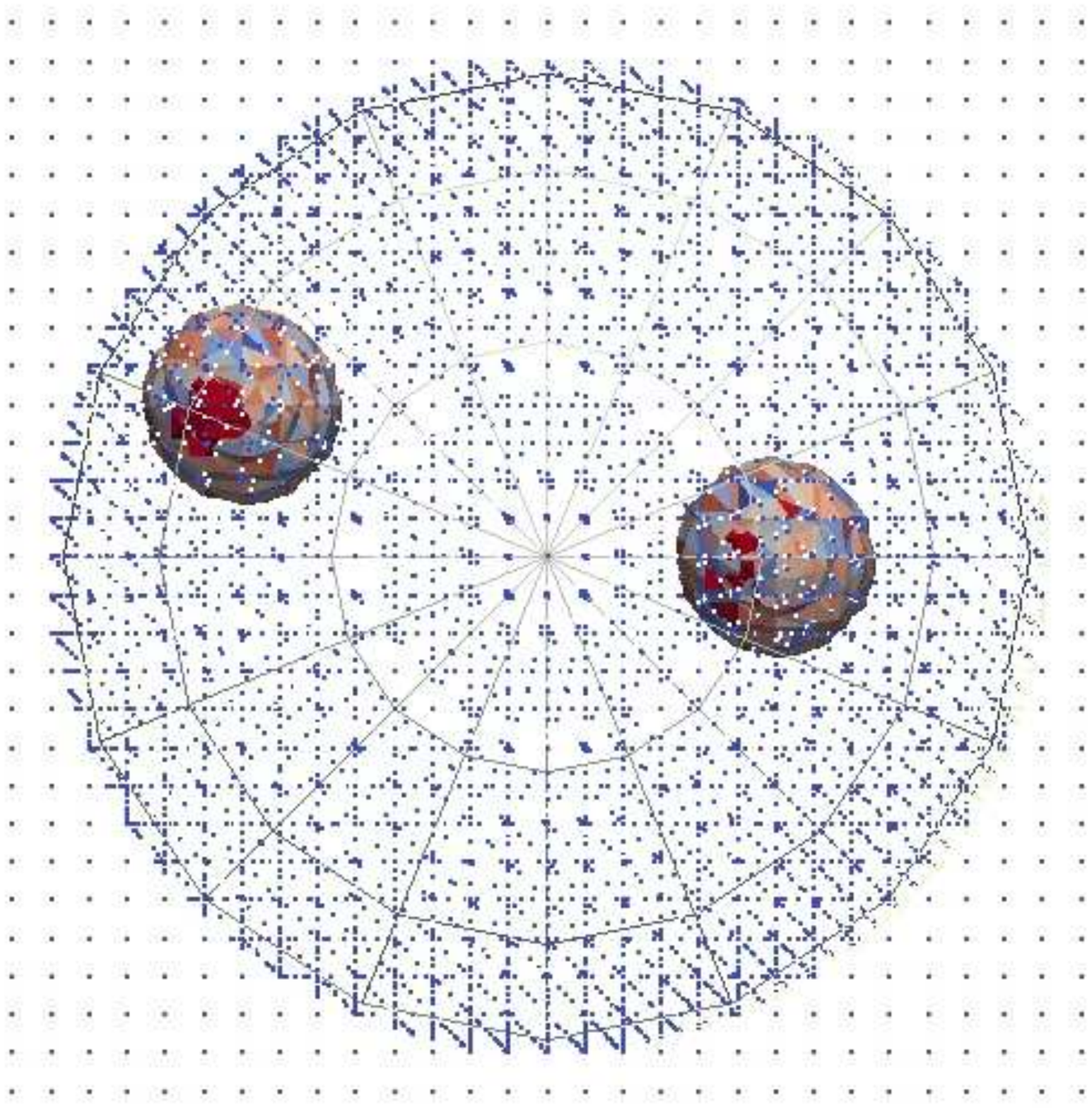}} &
{\includegraphics[scale=0.18, trim = 2.0cm 6.0cm 2.0cm 6.0cm, clip=true,]{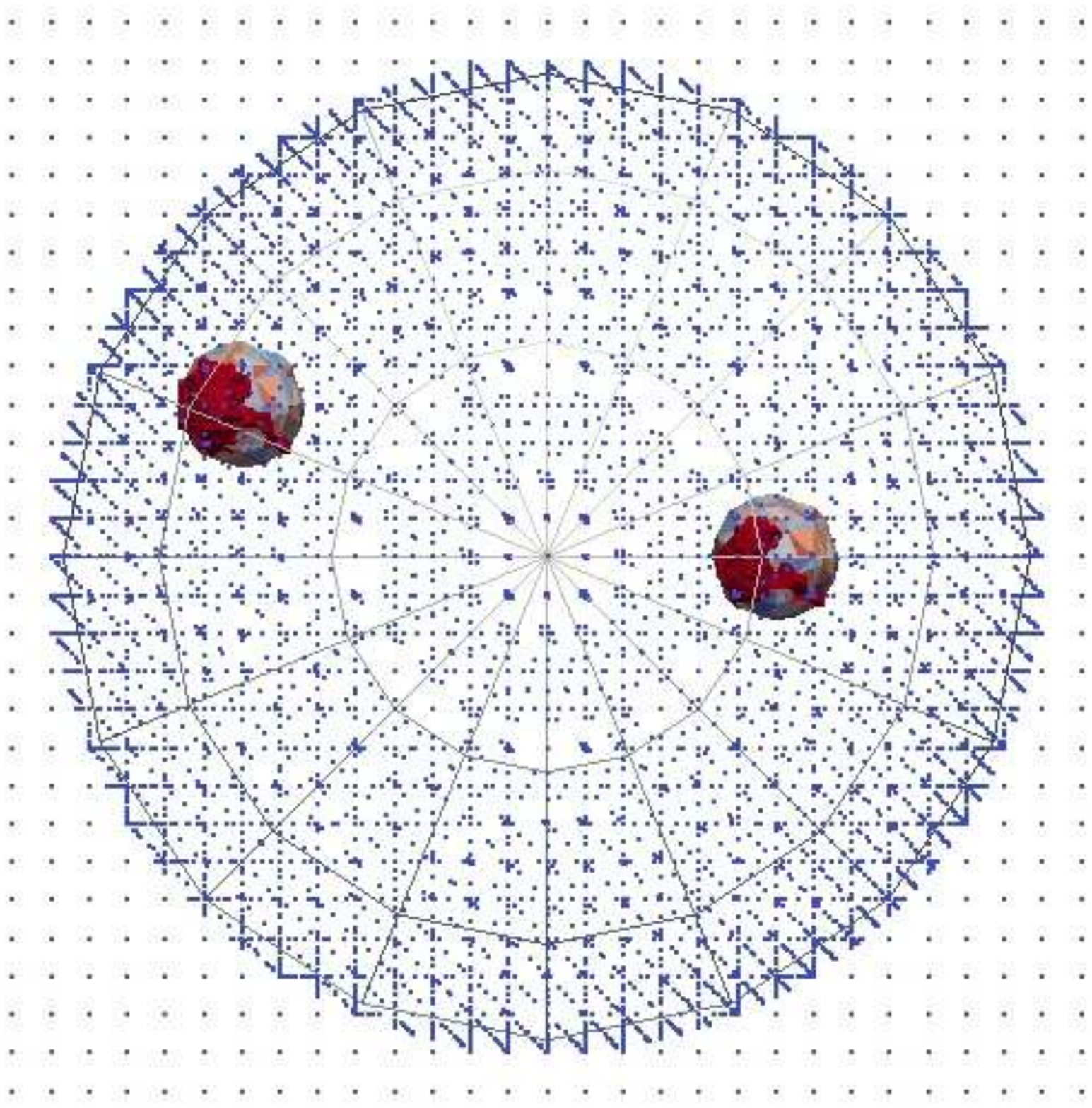}} &
{\includegraphics[scale=0.18, trim = 2.0cm 6.0cm 2.0cm 6.0cm, clip=true,]{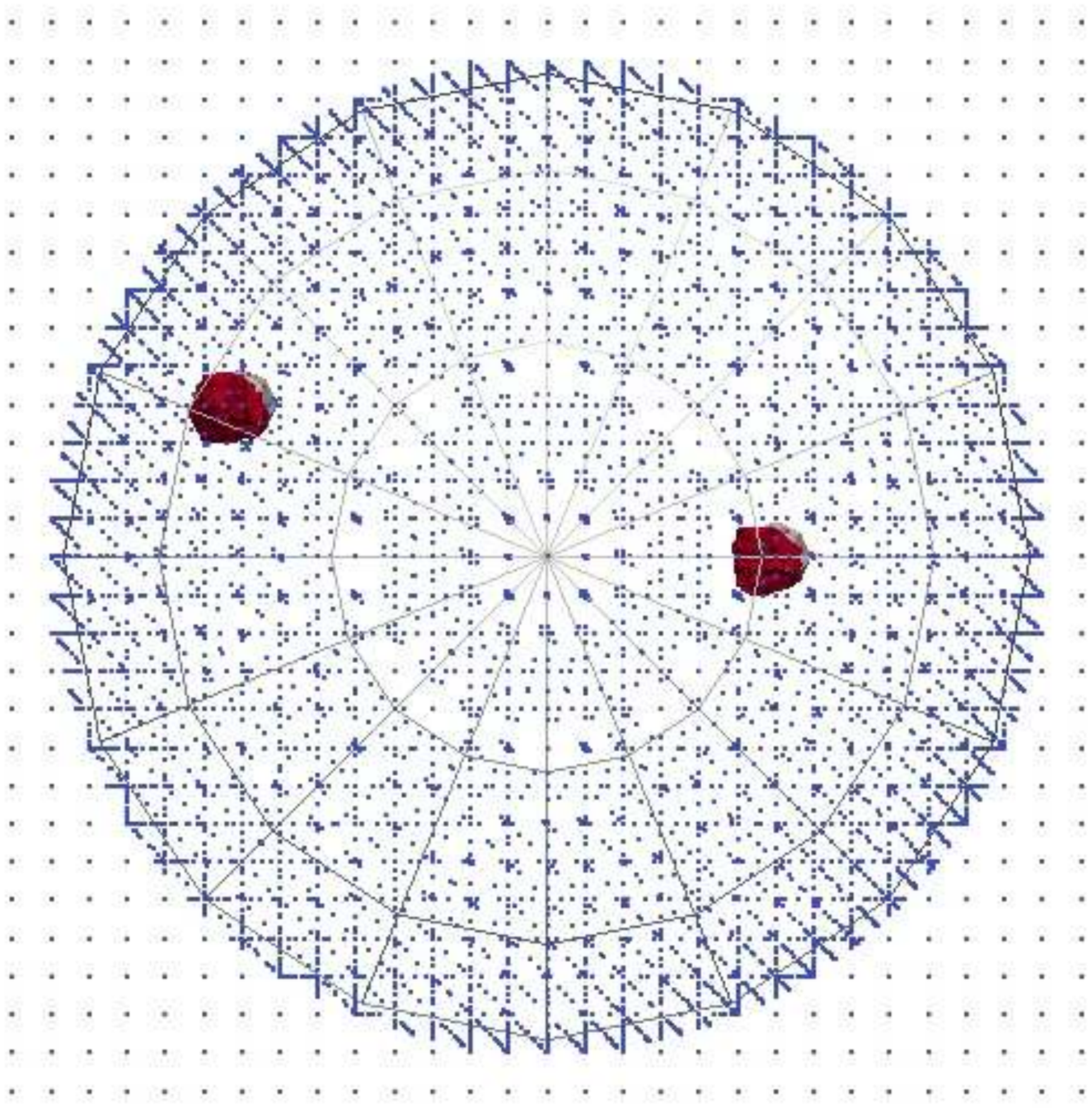}} \\
&  $x_1 x_2$ view & \\
{\includegraphics[scale=0.18, trim = 2.0cm 6.0cm 2.0cm 6.0cm, clip=true,]{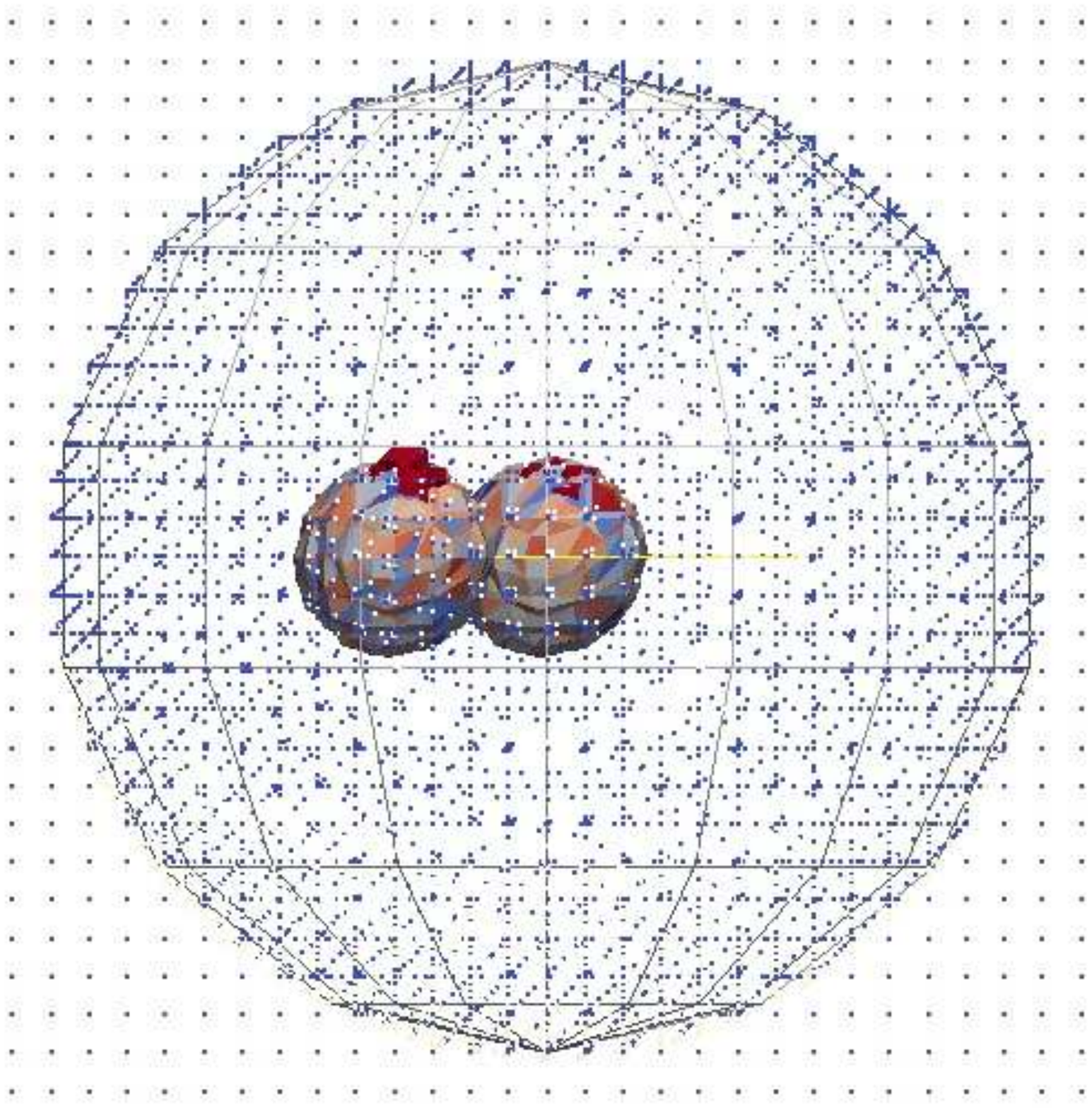}} &
{\includegraphics[scale=0.18, trim = 2.0cm 6.0cm 2.0cm 6.0cm, clip=true,]{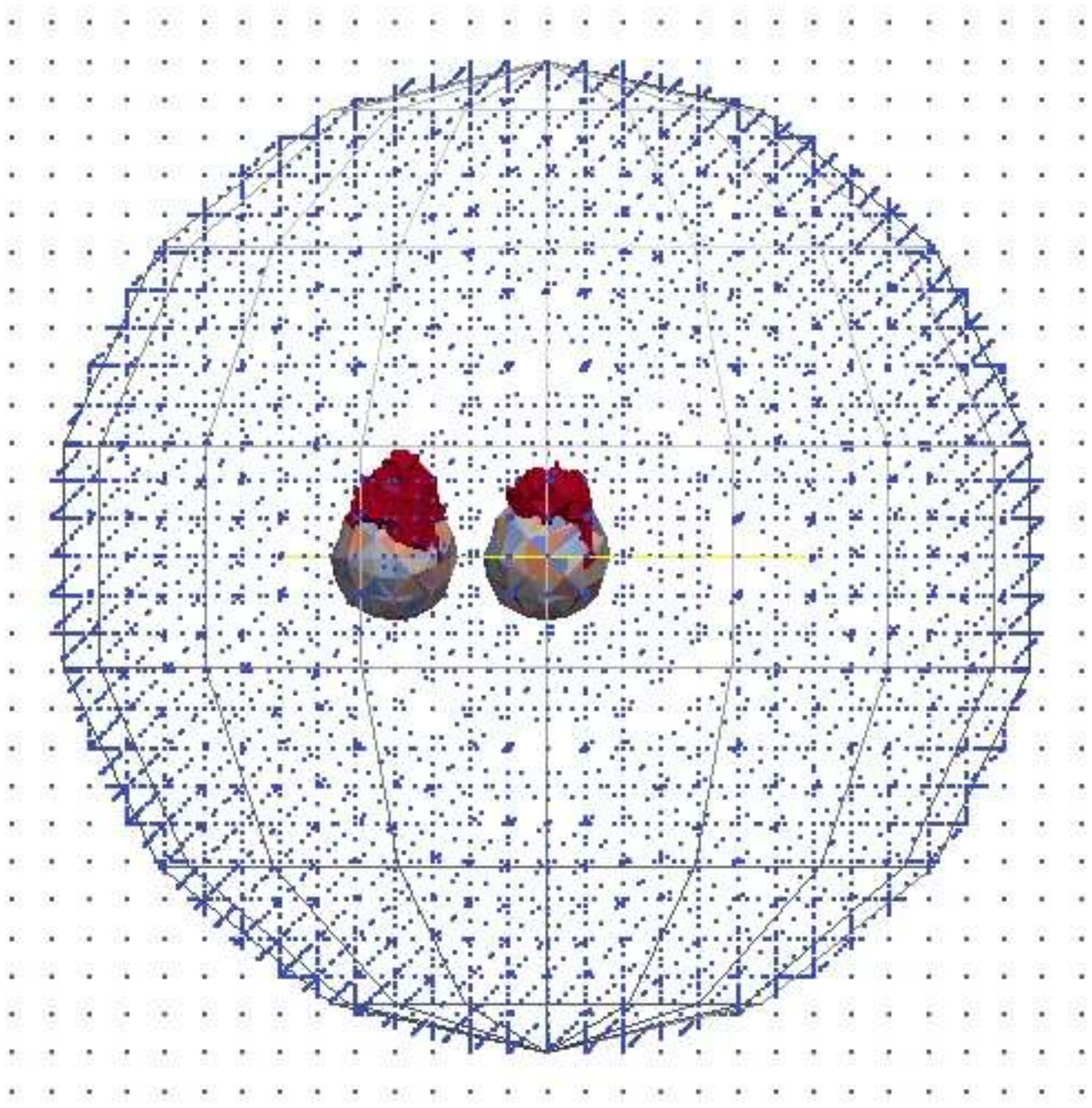}} &
{\includegraphics[scale=0.18, trim = 2.0cm 6.0cm 2.0cm 6.0cm, clip=true,]{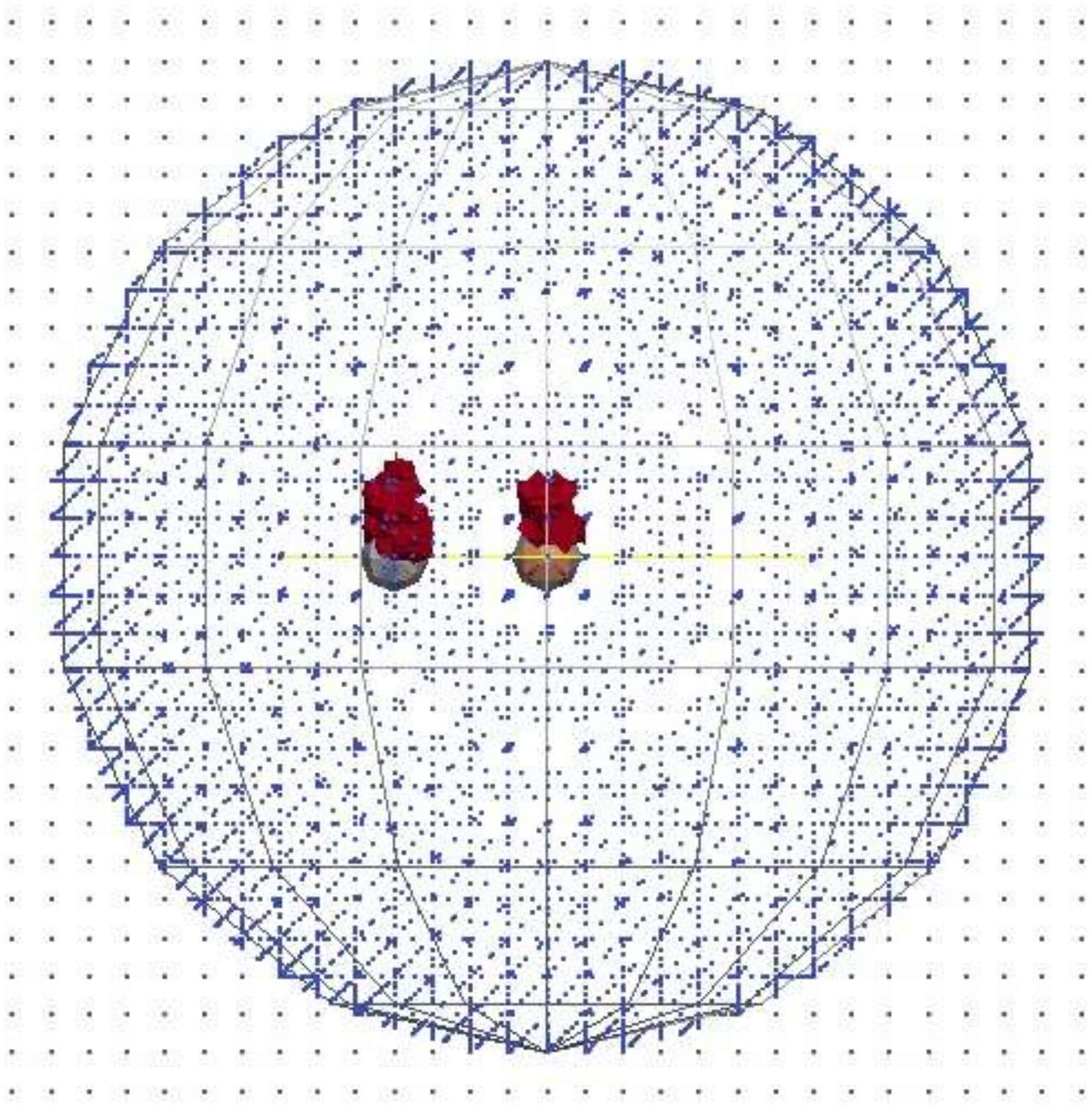}} \\
&  $x_2 x_3$ view & \\
{\includegraphics[scale=0.18, trim = 2.0cm 6.0cm 2.0cm 6.0cm, clip=true,]{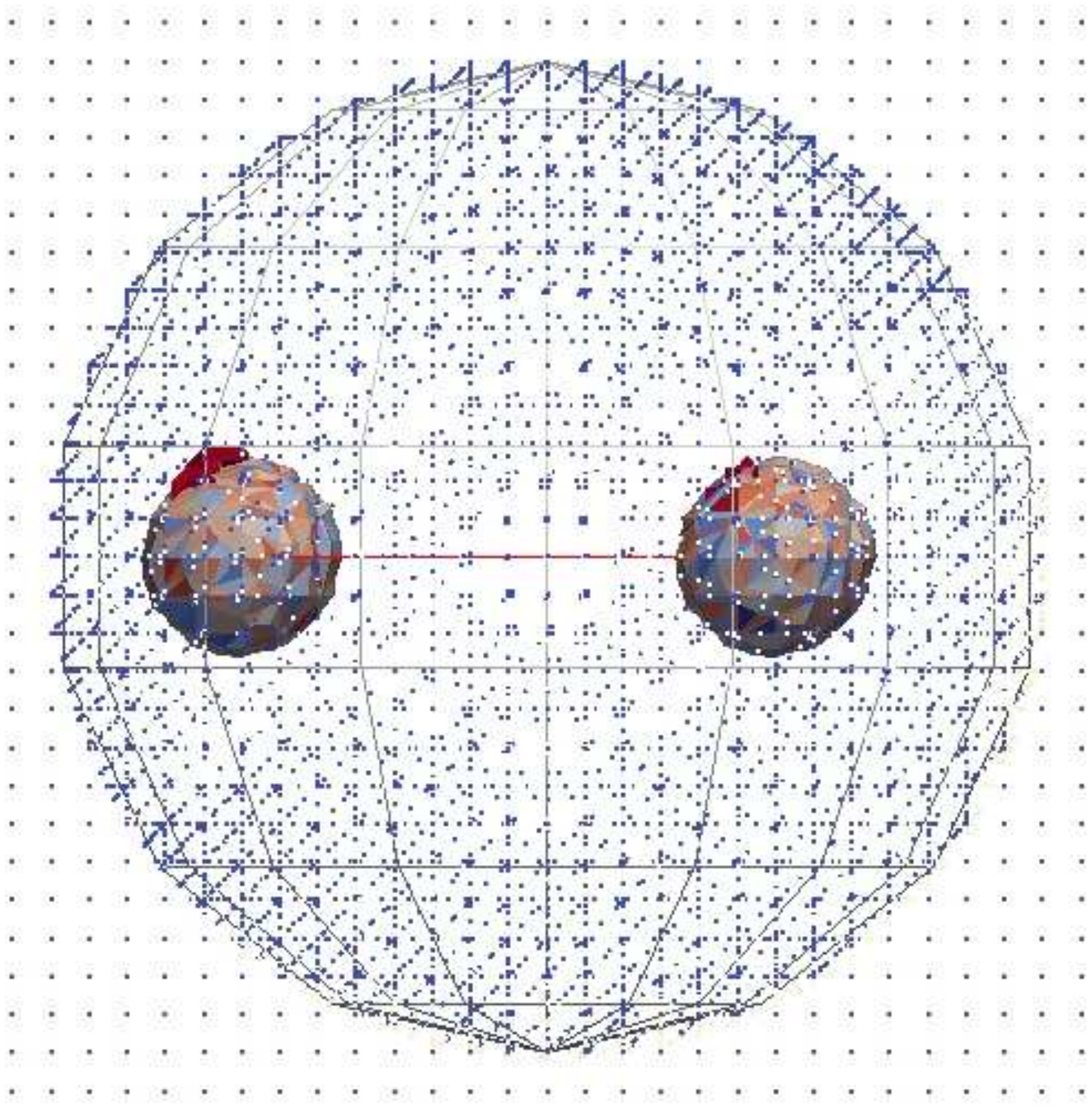}} &
{\includegraphics[scale=0.18, trim = 2.0cm 6.0cm 2.0cm 6.0cm, clip=true,]{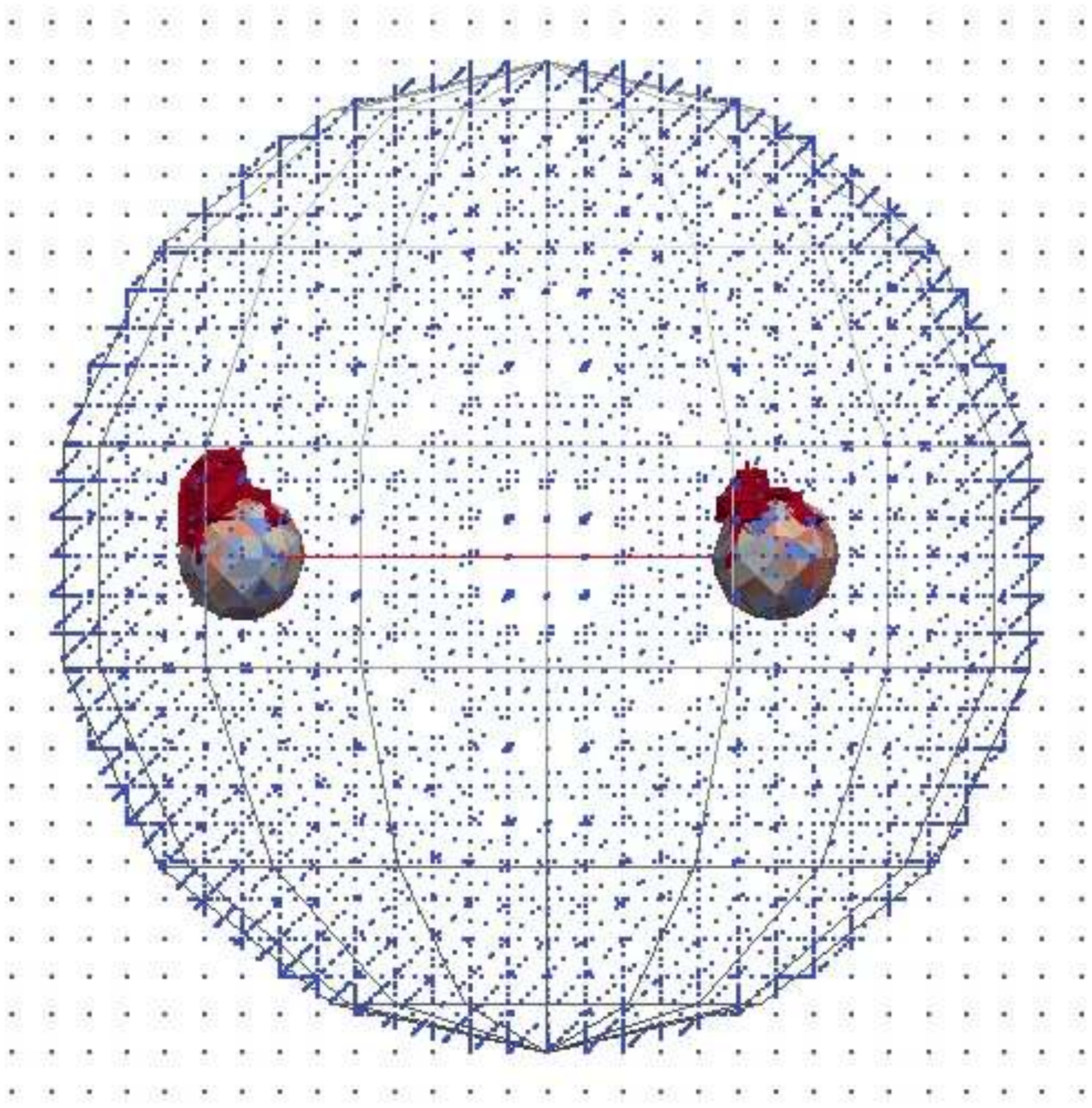}} &
{\includegraphics[scale=0.18, trim = 2.0cm 6.0cm 2.0cm 6.0cm, clip=true,]{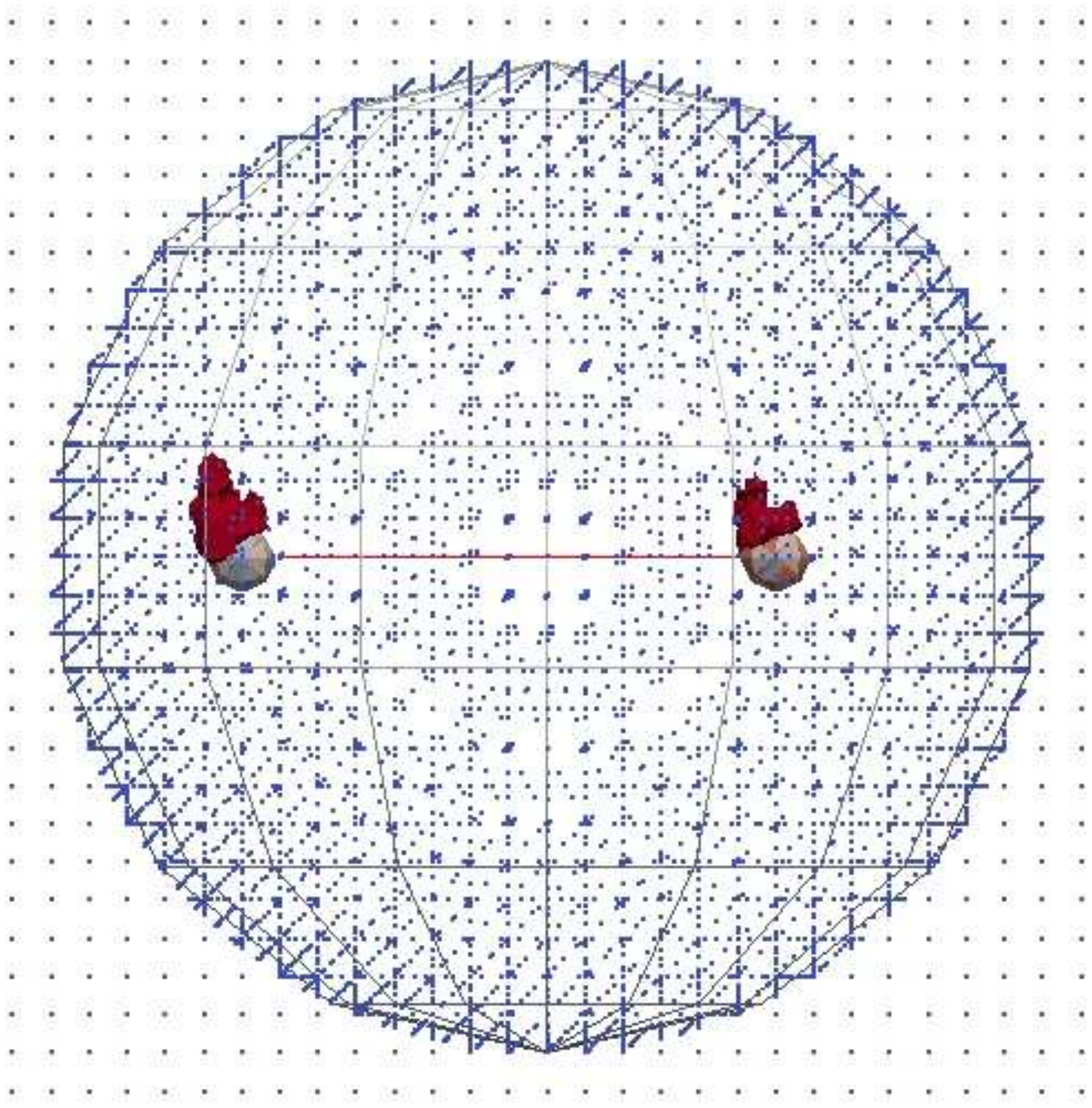}} \\
&  $x_1 x_3$ view & 
\end{tabular}
\end{center}
\caption{\small\emph{Test 1. In red:  isosurfaces $\{\x\in\Omega_\mathrm{FEM}:\eps_\mathrm{rec} = 0.8\max_{\Omega_\mathrm{FEM}}\eps_\mathrm{rec}\}$. Here,  $\eps_\mathrm{rec}$ was obtained on a five times refined mesh, $\max_{\Omega_\mathrm{FEM}}\eps_\mathrm{rec} = 1.97$, and the noise level in the data is $\sigma=10\%$. For comparison we also present as wireframes the corresponding isosurface of the function \eqref{2gaussians} in every figure.}}
\label{fig:test1noise10ref5}
\end{figure}

\subsection{Reconstructions} \label{sec:test1}

\subsubsection{Test 1}
In this section we present numerical results of the reconstruction of the function $\eps$ given by \eqref{2gaussians}. Tables 1--2 present computed results of the reconstructions on adaptively refined meshes after applying the First Adaptive Algorithm. Figures~\ref{fig:test1cellscoarse}--\ref{fig:test1noise10ref5} display results of the reconstruction of the function given by \eqref{2gaussians} with additive noise of the level $\sigma=10\%$. Quite similar results are obtained for $\sigma=3\%$, see Tables 1, and 2, and thus they are not presented here. In Figures~\ref{fig:test1cellscoarse}--\ref{fig:test1noise10ref5} we observe that the location of the maximal value of the function \eqref{2gaussians} is imaged correctly. It follows from Figure~\ref{fig:test1noise10coarse} and Table~1 that on the coarse mesh we obtain good contrast, with $\max_{\Omega_\mathrm{FEM}} \eps_{h_0} = 1.94 $. However, Figure~\ref{fig:test1noise10coarse} reveals that it is desirable to improve the location of the maxima of the reconstructed function in $x_3$ direction as well as remove some artifacts which appeared in the reconstruction on the coarse mesh.

The reconstruction $\eps_\mathrm{rec}$ of $\eps$ on a final, five times adaptively refined mesh are presented in Figure~\ref{fig:test1noise10ref5}. We observe significant improvement of the reconstruction of the function $\eps$ obtained on the final adaptively refined mesh: the artifacts are removed and the reconstructed function is moved more closer to the exact function in $x_3$ direction, compared with results of Figure~\ref{fig:test1noise10coarse}.

\begin{figure}
\begin{center}
\begin{tabular}{cc}
{\includegraphics[scale=0.4, clip=true,]{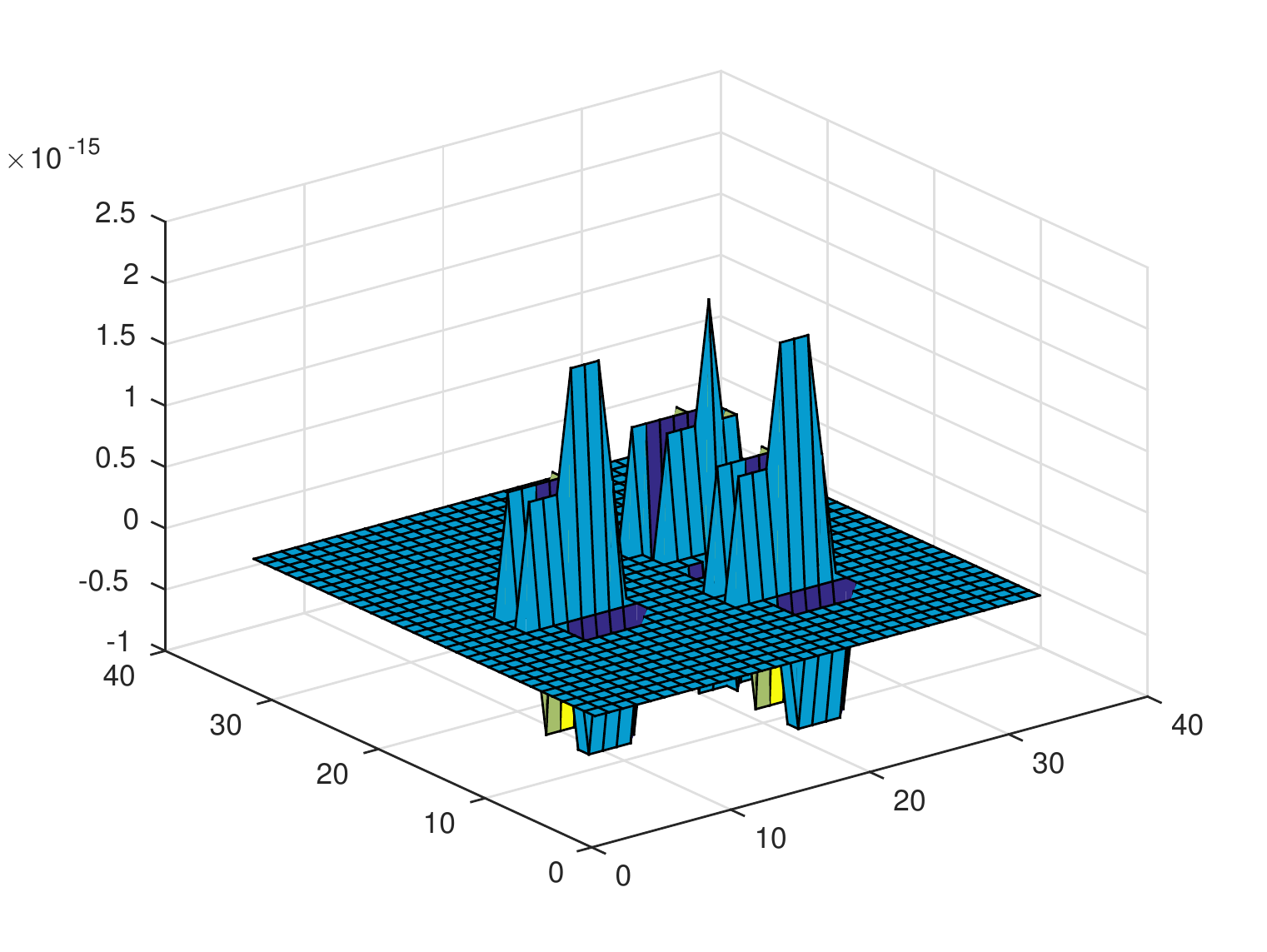}} &
{\includegraphics[scale=0.4, clip=true,]{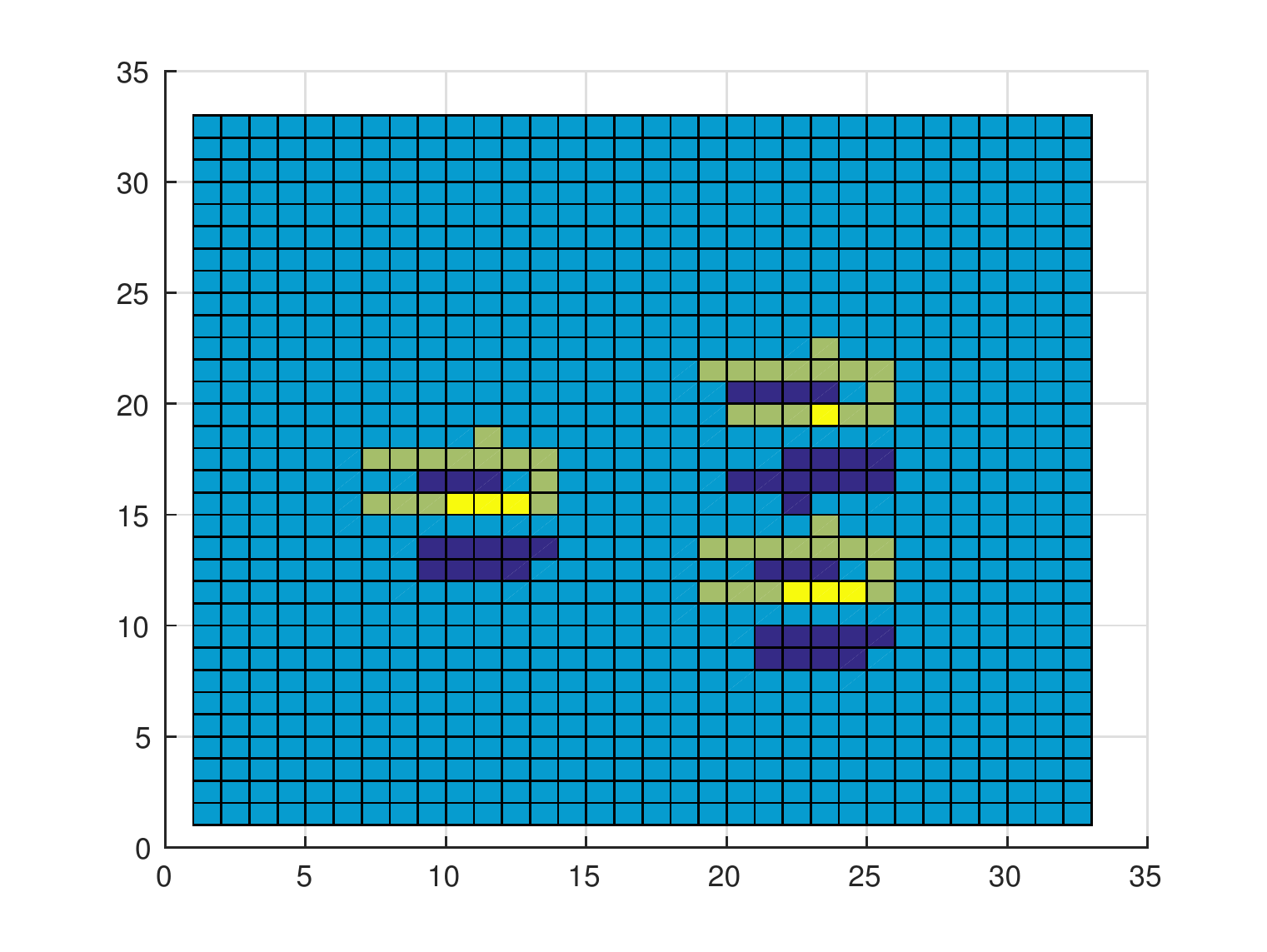}} \\
$t=0.6$  & $t=0.6, x_1 x_2 $ view \\
{\includegraphics[scale=0.4, clip=true,]{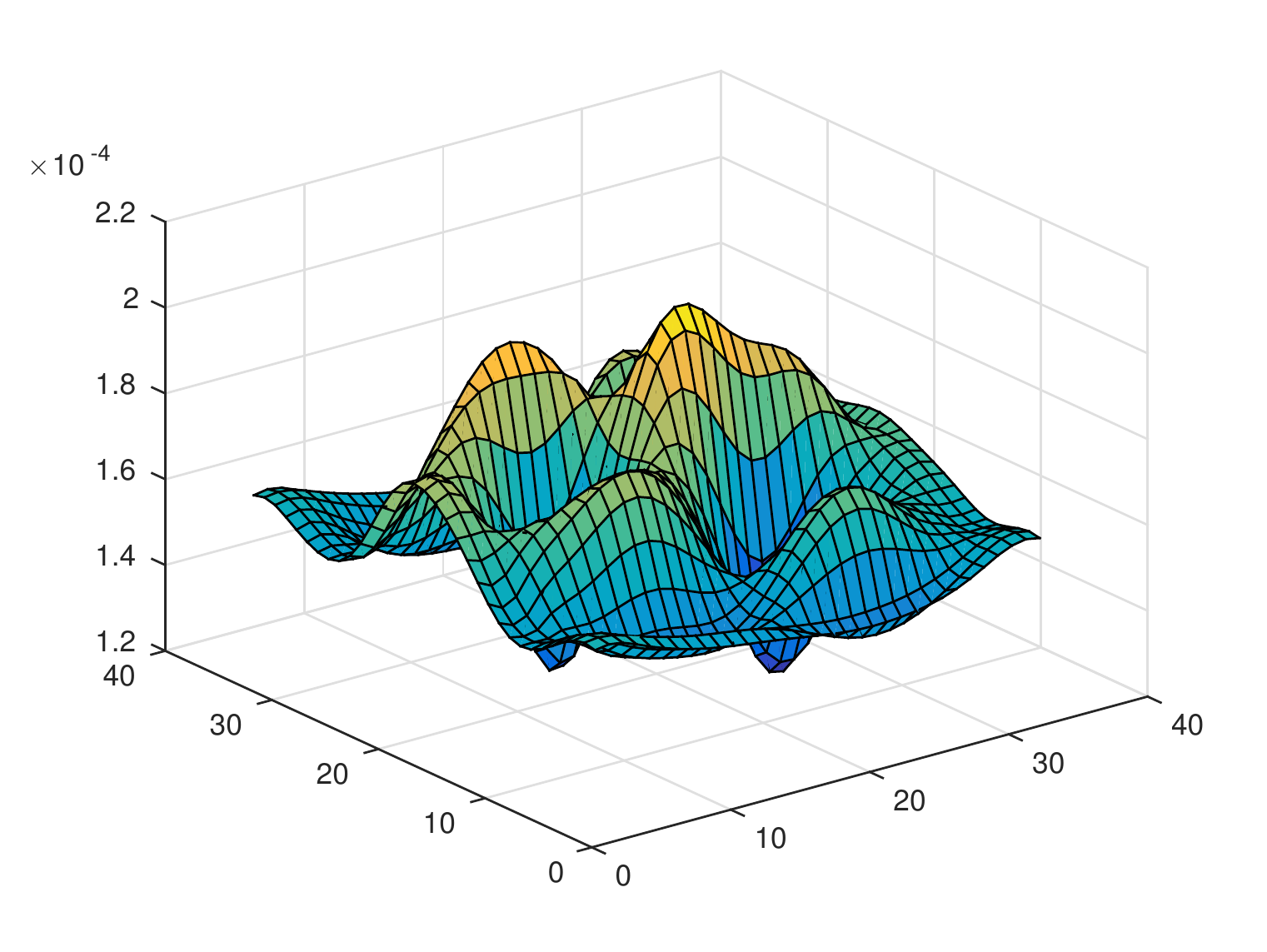}} &
{\includegraphics[scale=0.4, clip=true,]{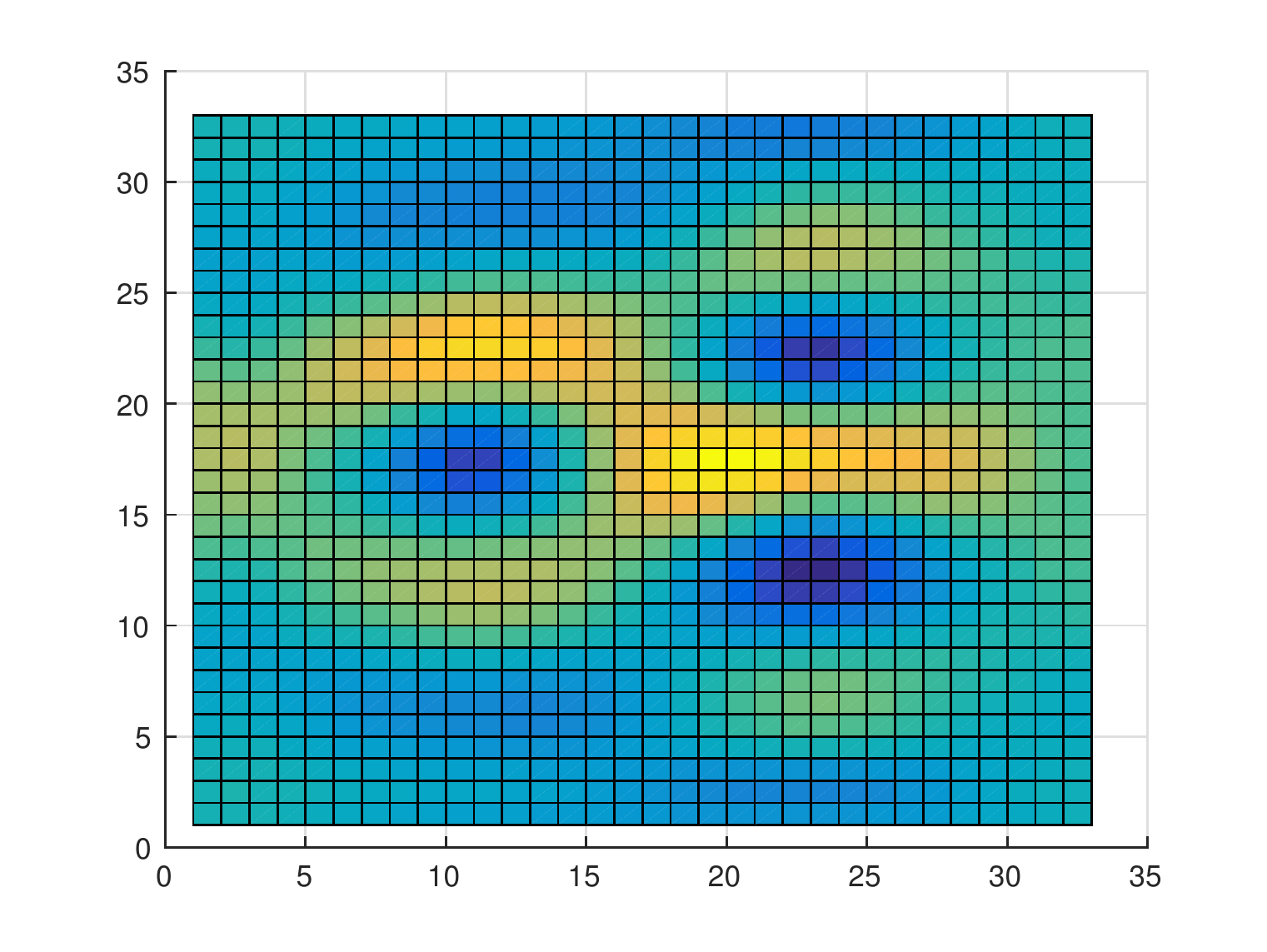}} \\
$t=1.2$  & $t=1.2, x_1 x_2$ view \\
{\includegraphics[scale=0.4, clip=true,]{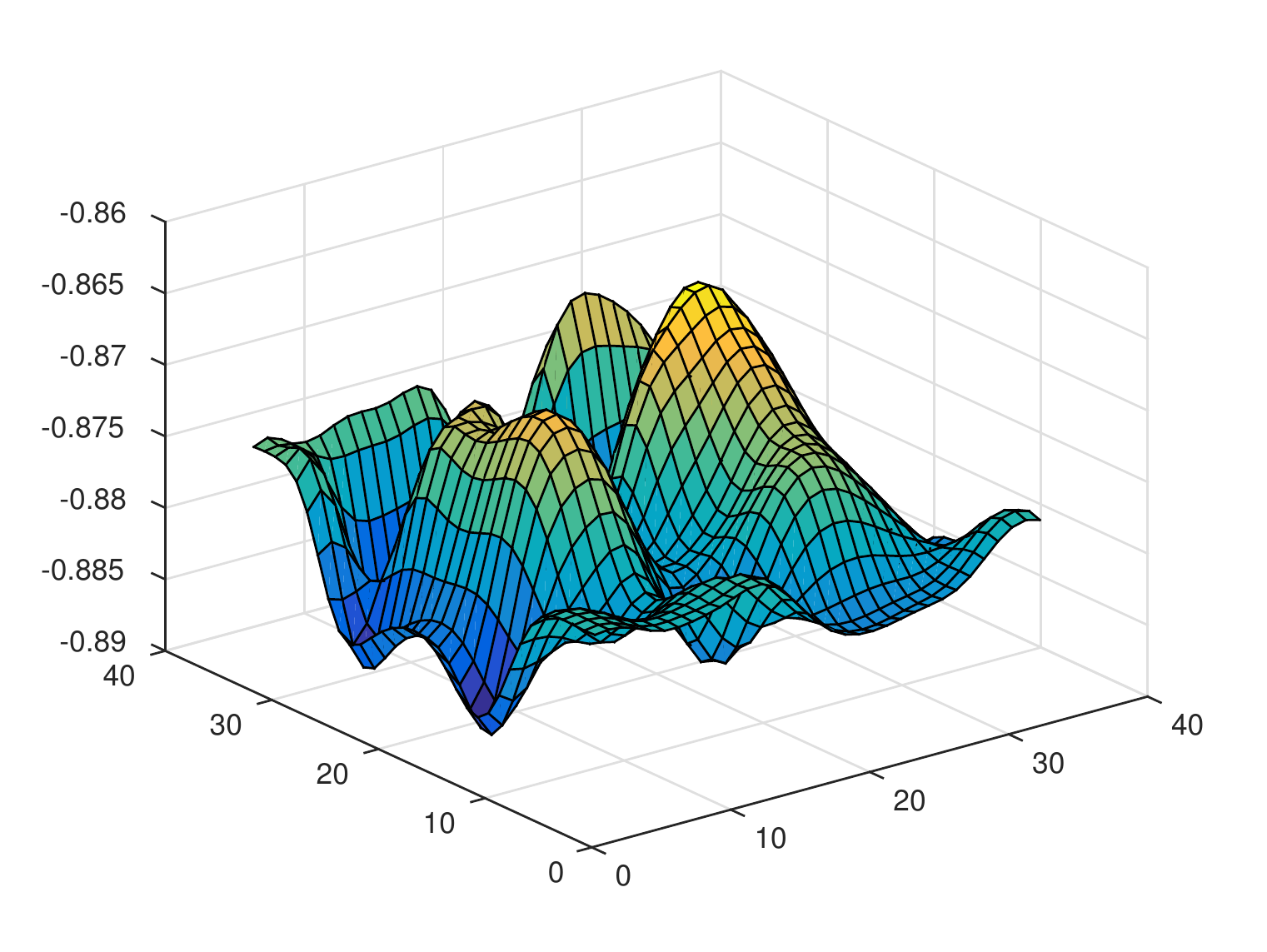}} &
{\includegraphics[scale=0.4, clip=true,]{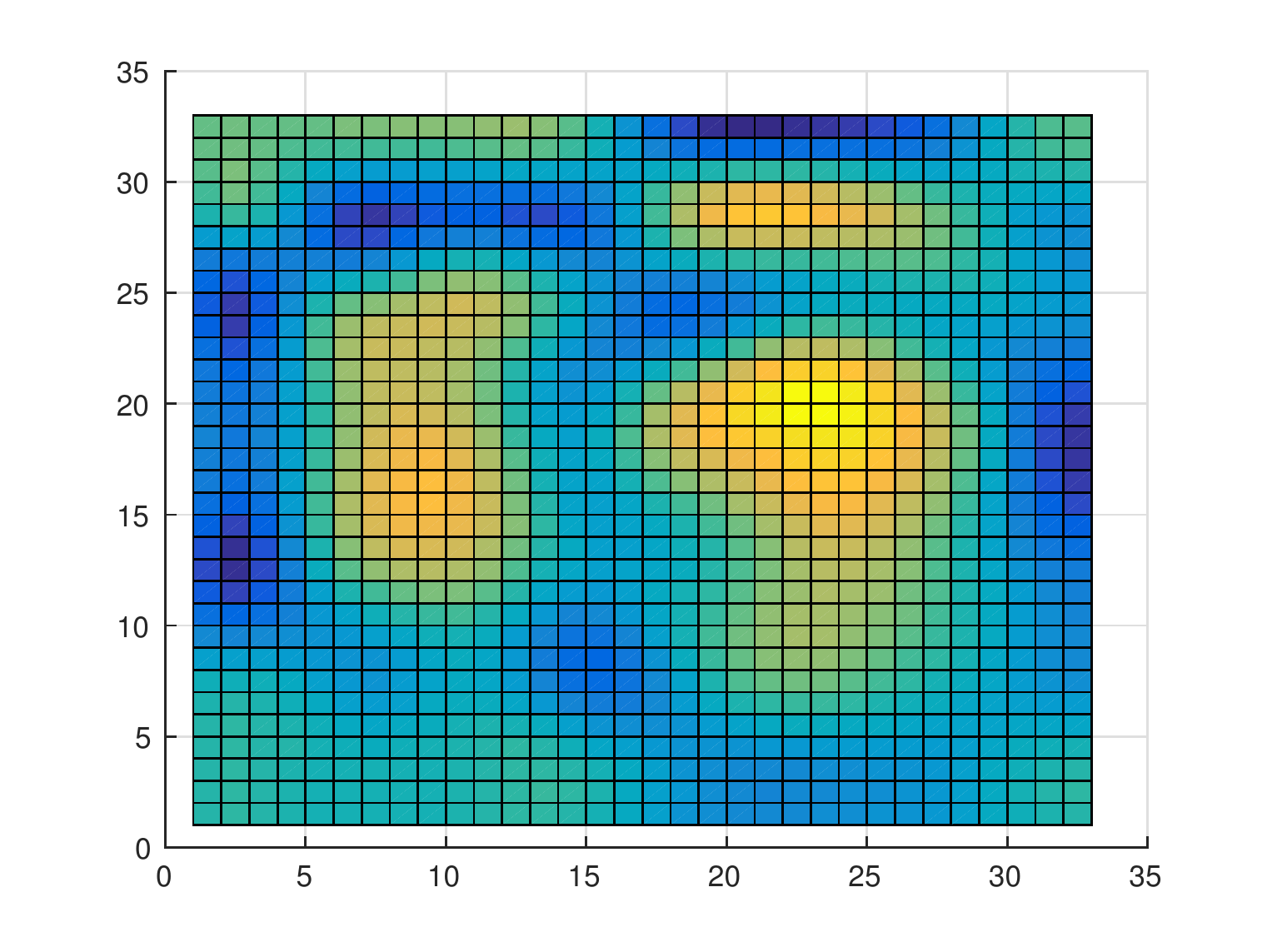}} \\
$t=1.8$  & $t=1.8, x_1 x_2$ view \\
\end{tabular}
\end{center}
\caption{\small\emph{Test 2. Transmitted data of the component $E_2$ at different times.  The noise level in the data is $\sigma=10\%$.}}
\label{fig:test2data}
\end{figure}

\begin{figure}
\begin{center}
\begin{tabular}{cc}
{\includegraphics[scale=0.3, trim = 2.0cm 6.0cm 2.0cm 6.0cm, clip=true,]{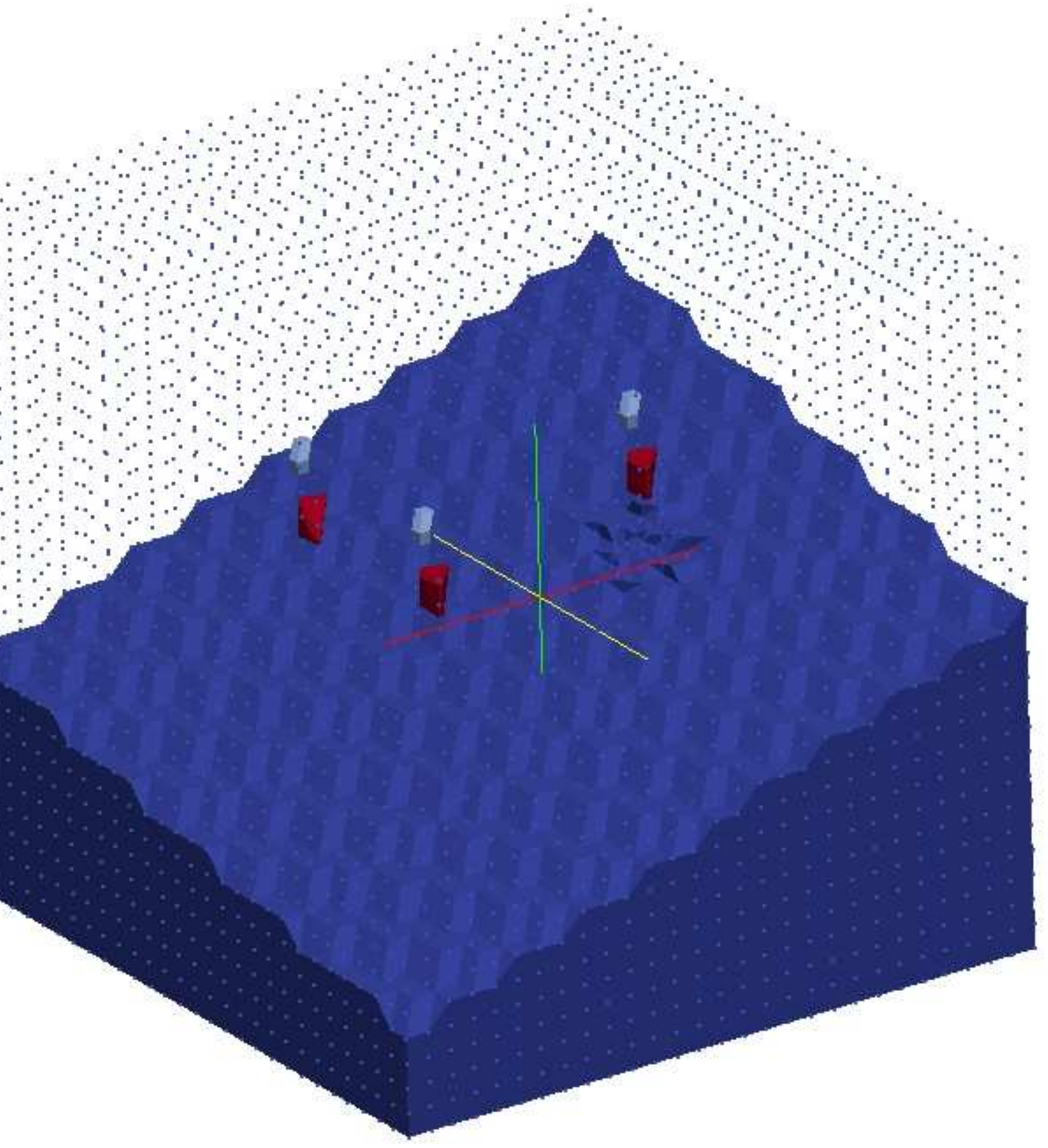}} &
{\includegraphics[scale=0.3, trim = 2.0cm 6.0cm 2.0cm 6.0cm, clip=true,]{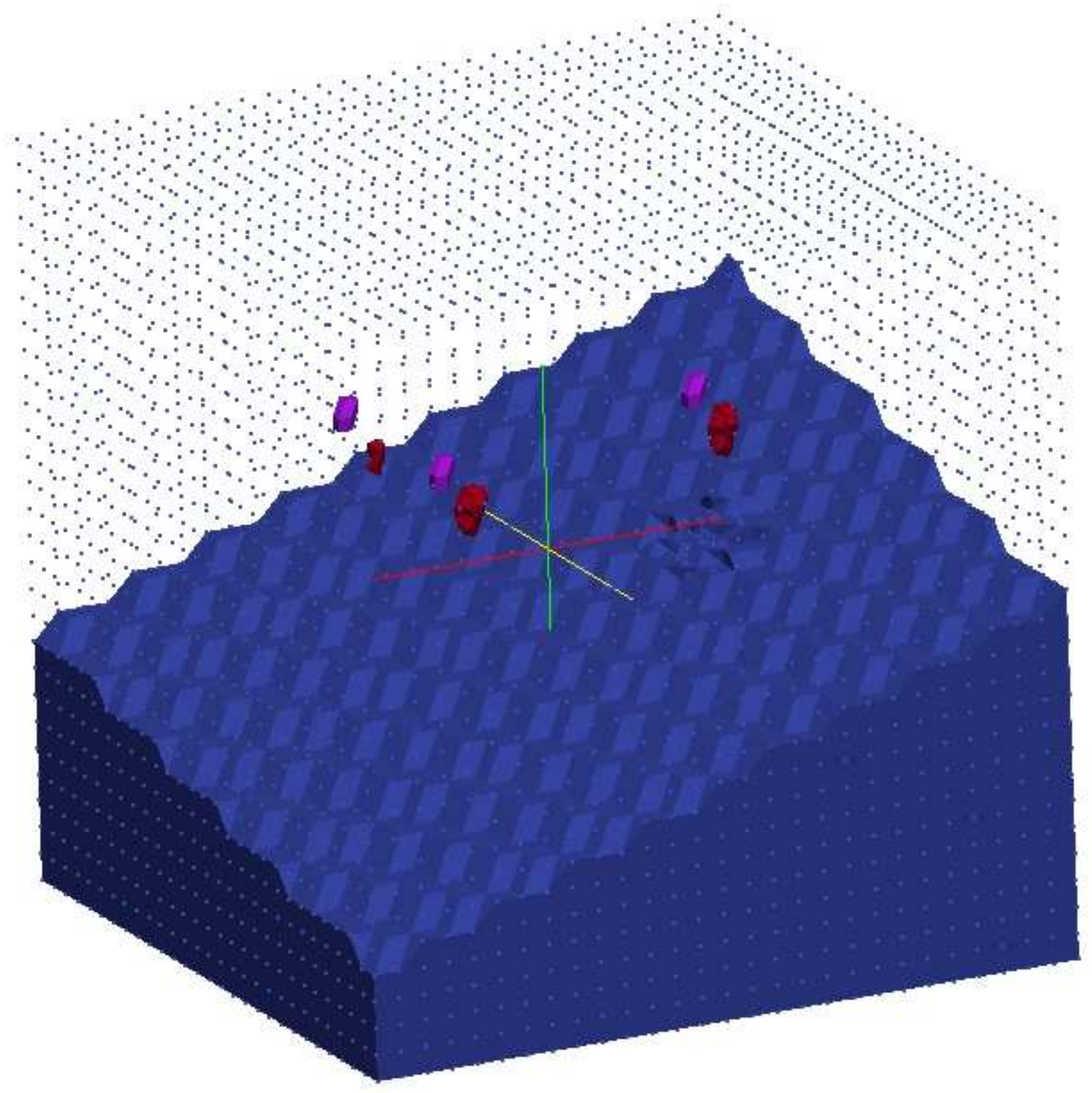}} \\
{\includegraphics[scale=0.3, trim = 2.0cm 6.0cm 2.0cm 6.0cm, clip=true,]{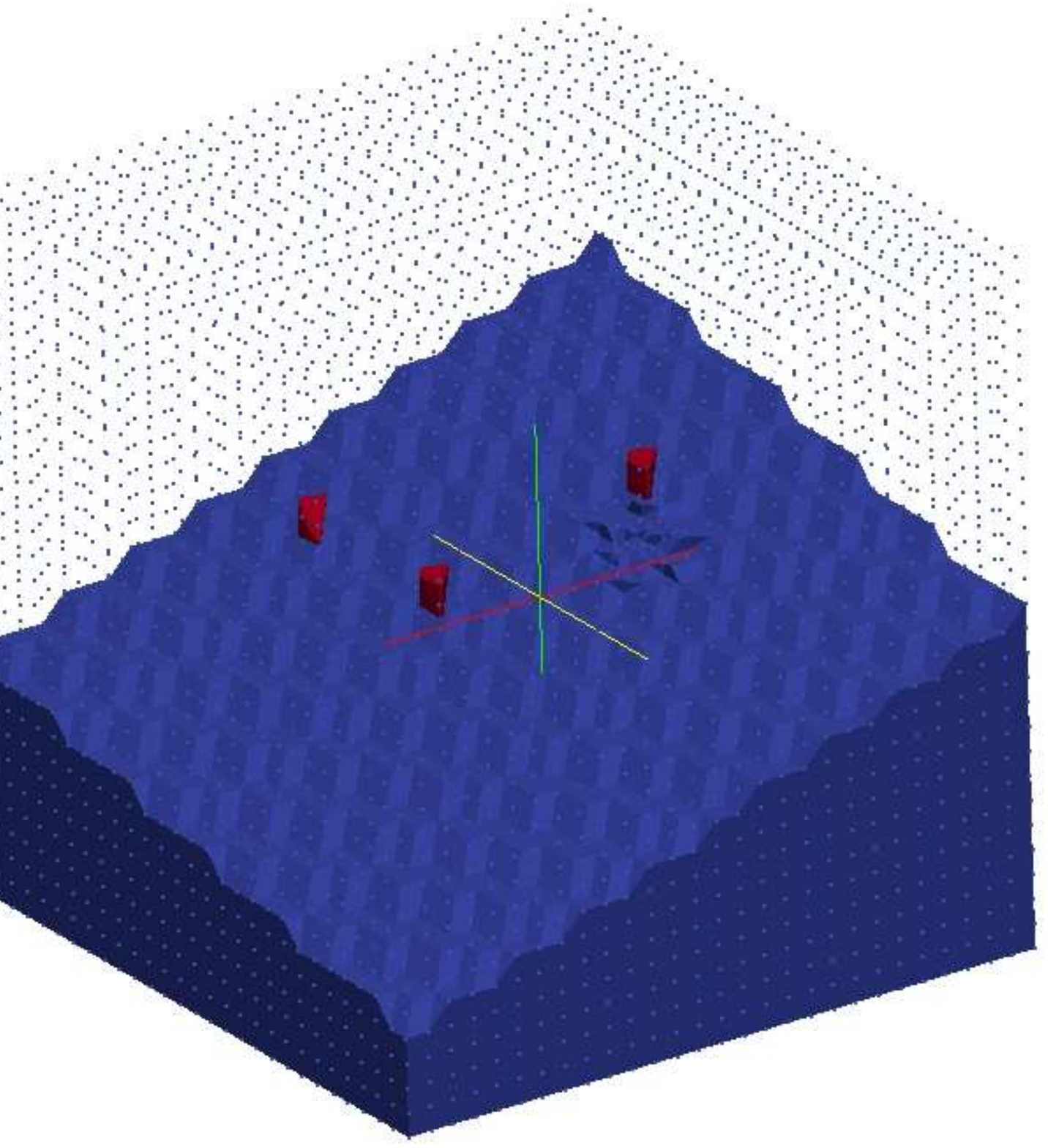}} &
{\includegraphics[scale=0.3, trim = 2.0cm 6.0cm 2.0cm 6.0cm, clip=true,]{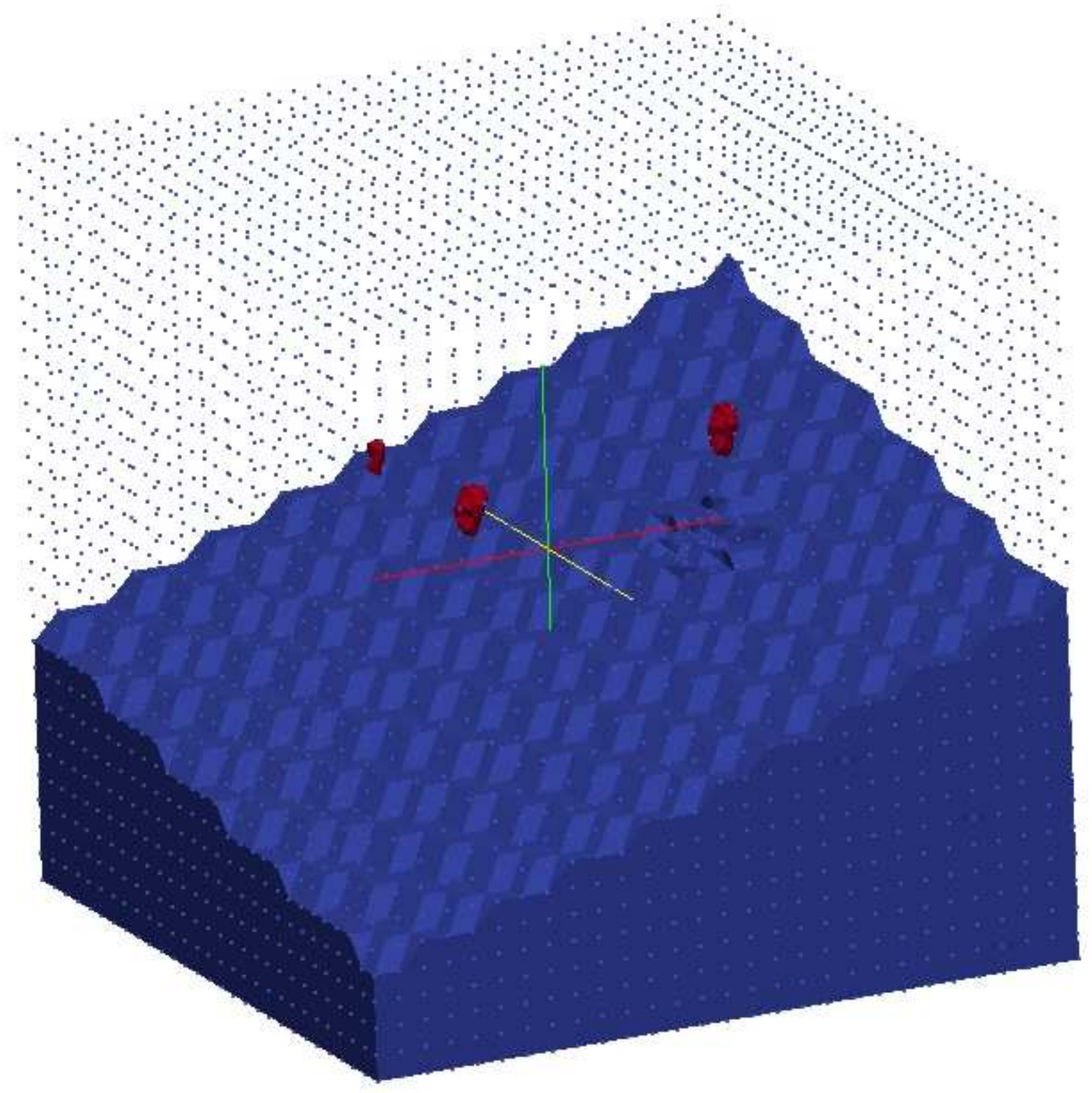}}\\
$\x\in\Omega_\mathrm{FEM}:\eps_{h_0}(\x) = 2.8 $ & $\x\in\Omega_\mathrm{FEM}:\eps_\mathrm{rec}(\x) = 1.9 $ 
\end{tabular}
\end{center}
\caption{\small\emph{Test 2. Prospect views of reconstructions of three inclusions on a coarse mesh (on the left) and on the five times adaptively refined mesh (on the right). The noise level in the data is $\sigma=10\%$. The top figures also present exact inclusions (in light blue and pink colors).}}
\label{fig:test2cells}
\end{figure}

\begin{figure}
\begin{center}
\begin{tabular}{ccc}
{\includegraphics[scale=0.19, trim = 1.0cm 4.0cm 1.0cm 4.0cm, clip=true,]{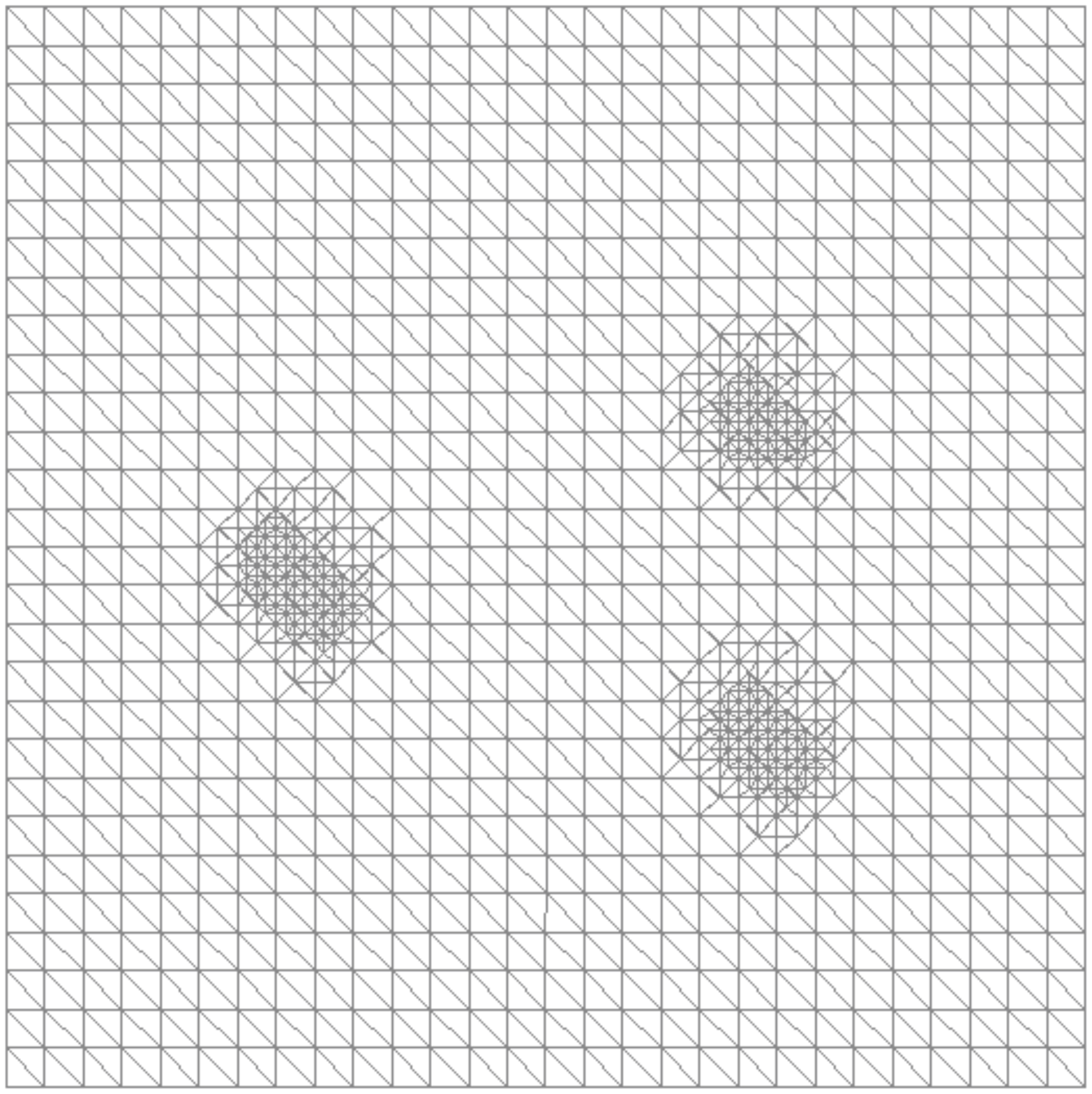}} &
{\includegraphics[scale=0.19, trim = 1.0cm 4.0cm 1.0cm 4.0cm, clip=true,]{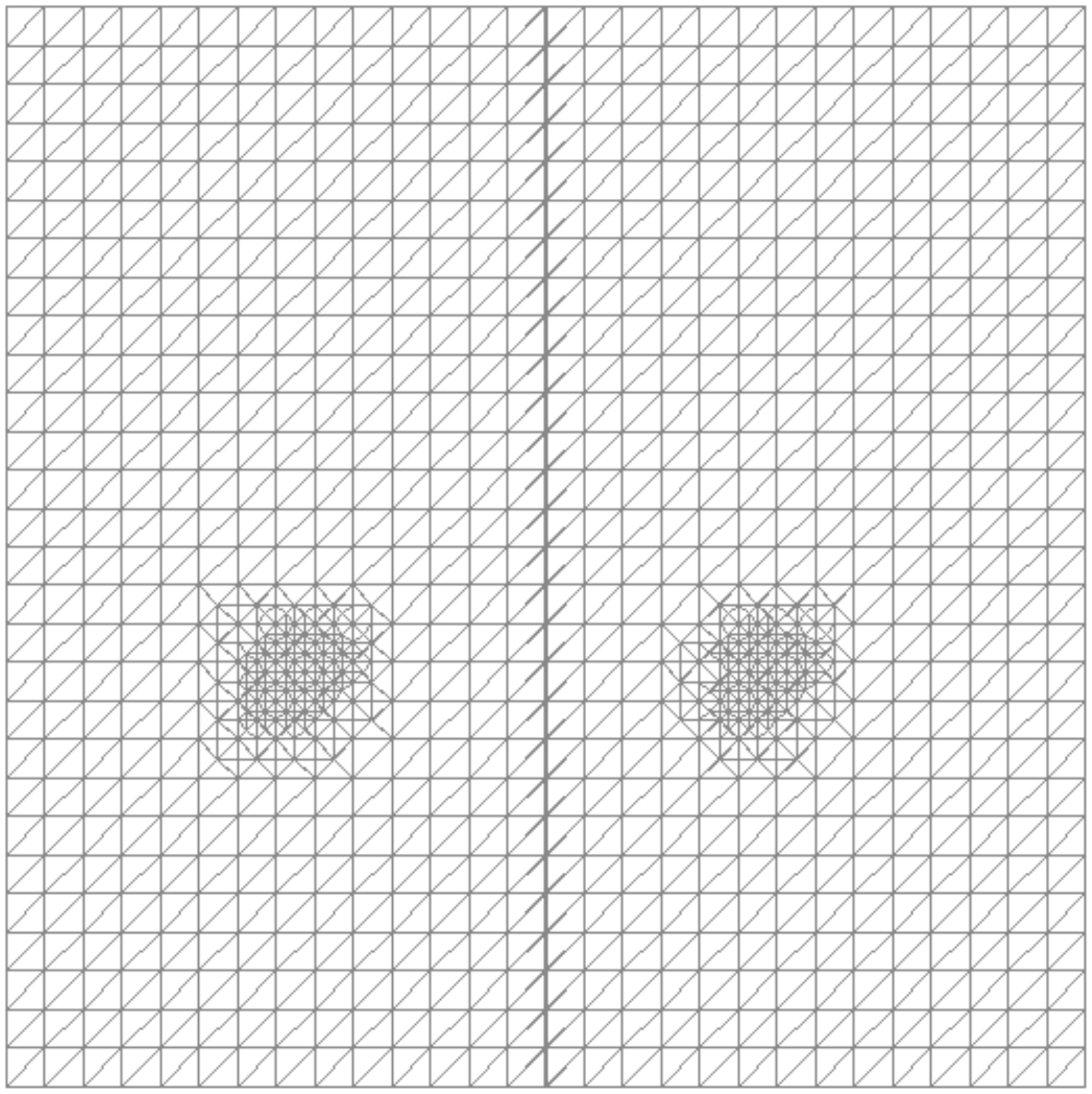}} &
{\includegraphics[scale=0.19, trim = 1.0cm 4.0cm 1.0cm 4.0cm, clip=true,]{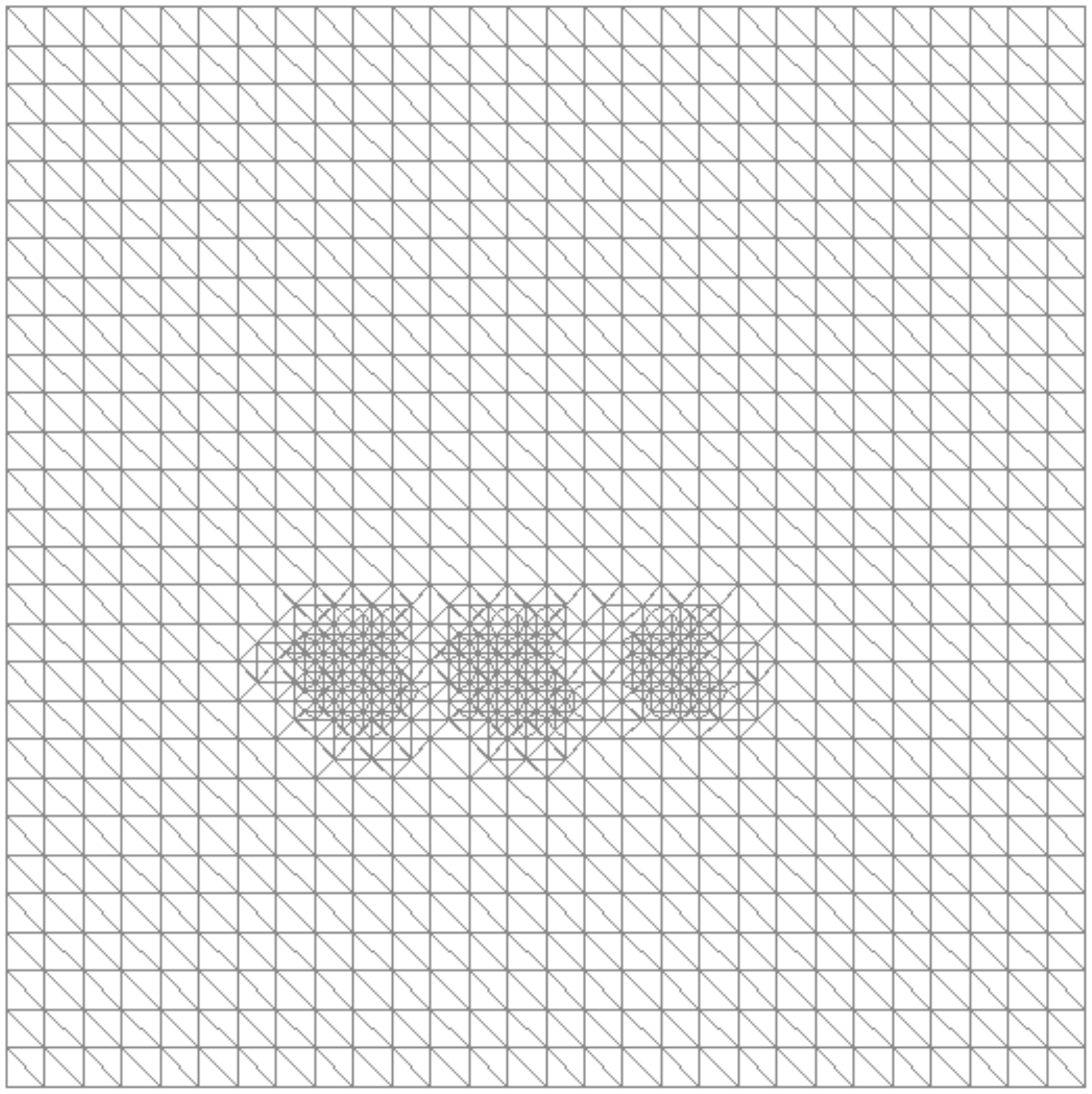}} \\
$x_1 x_2$ view &  $x_1 x_3$ view & $x_2 x_3$ view
\end{tabular}
\end{center}
\caption{\small\emph{Test 2. Five times adaptively refined mesh when the noise level in the data was $\sigma=10\%$.}}
\label{fig:test2meshes}
\end{figure}

\begin{figure}
\begin{center}
\begin{tabular}{cc}
 $\x\in\Omega_\mathrm{FEM}: \eps_{h_0}(\x) = 2.8 $ & $\x\in\Omega_\mathrm{FEM}: \eps_\mathrm{rec}(\x) = 1.9 $ \\
{\includegraphics[scale=0.22, trim = 2.0cm 6.0cm 2.0cm 6.0cm, clip=true,]{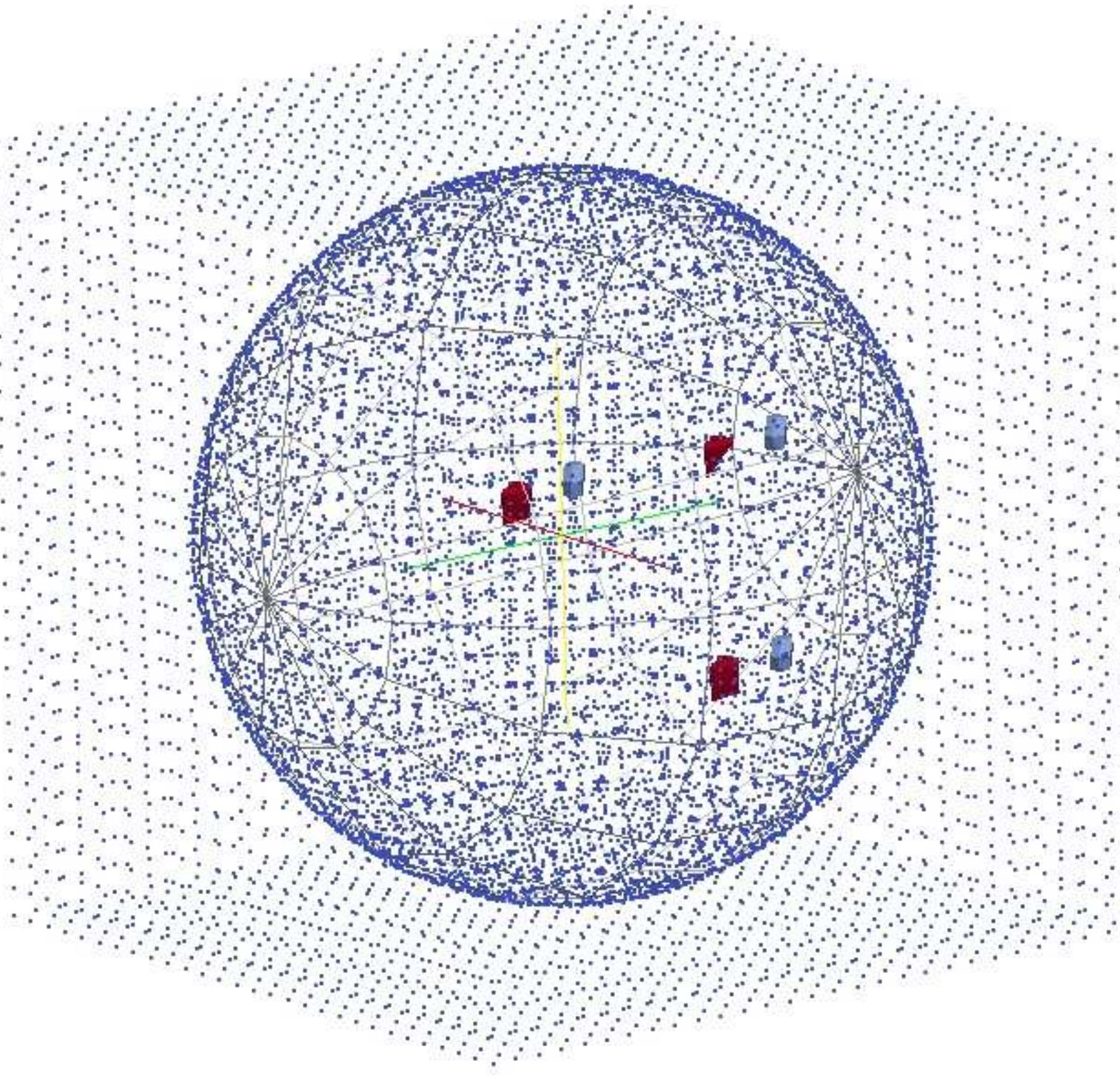}} &
{\includegraphics[scale=0.22, trim = 2.0cm 6.0cm 2.0cm 6.0cm, clip=true,]{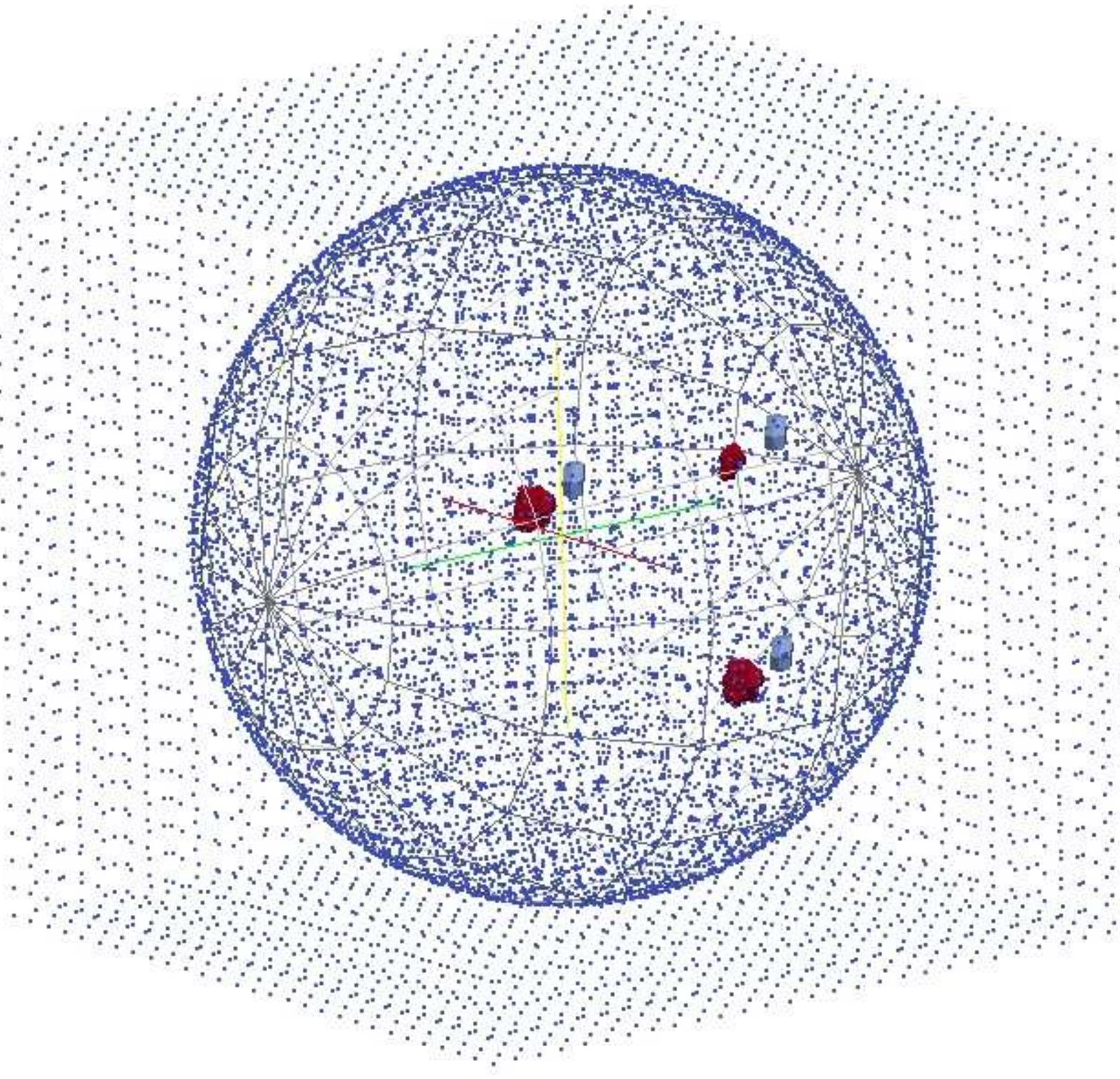}}  \\
a) prospect view &  b) prospect view \\
{\includegraphics[scale=0.22, trim = 2.0cm 6.0cm 2.0cm 6.0cm, clip=true,]{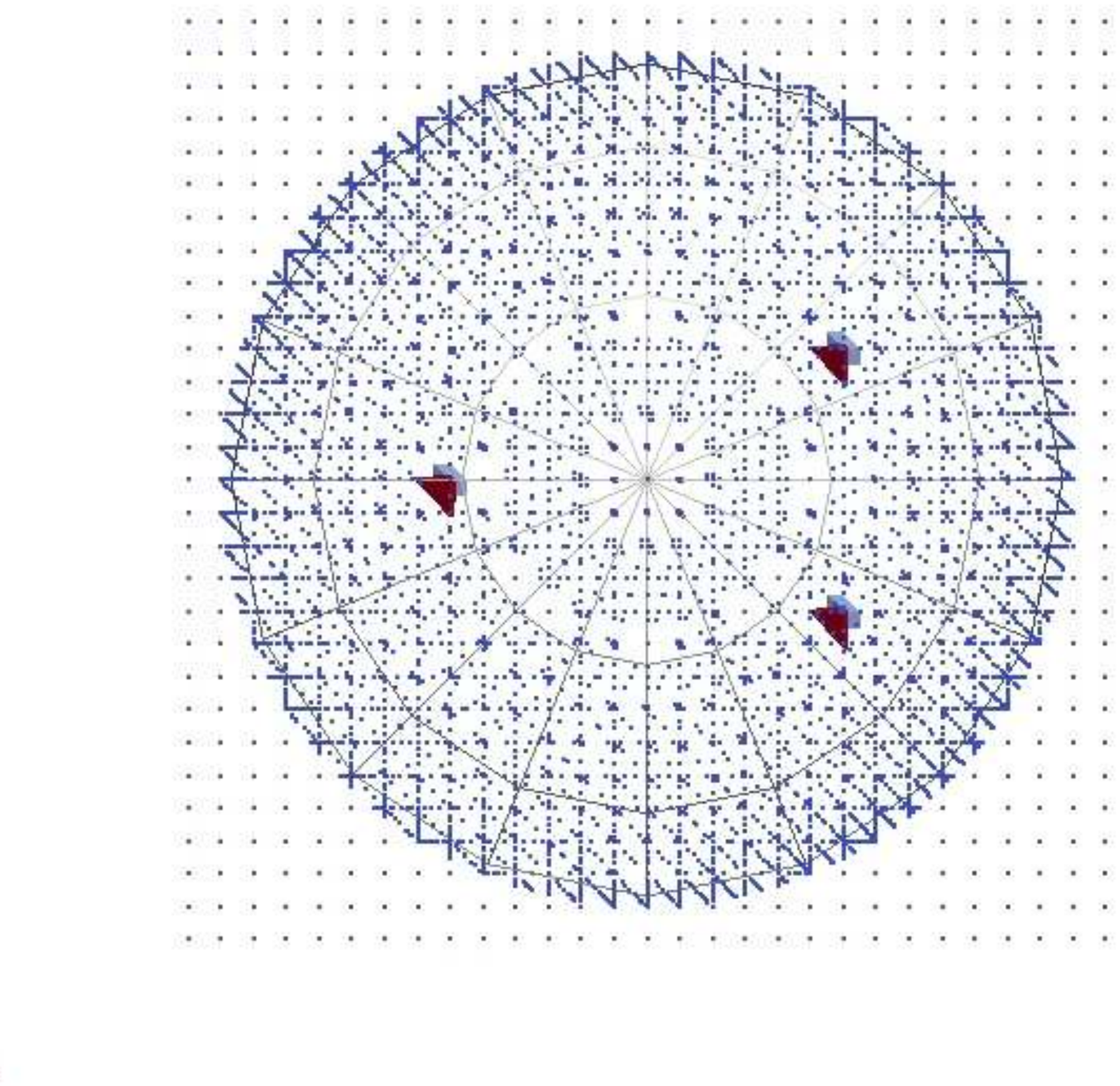}} &
{\includegraphics[scale=0.22, trim = 2.0cm 6.0cm 2.0cm 6.0cm, clip=true,]{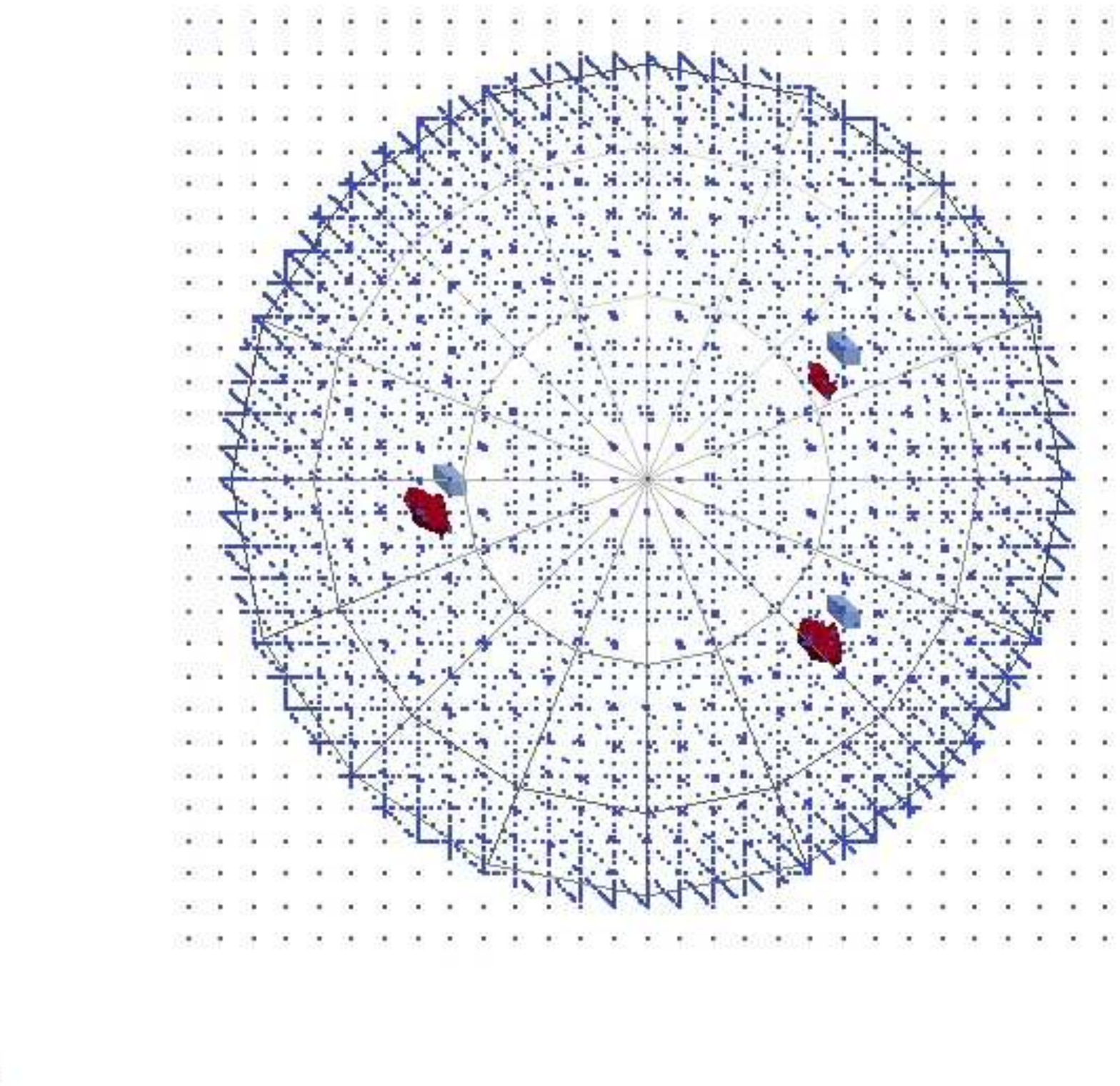}} \\
c)  $x_1 x_2$ view & d)  $x_1 x_2$ view \\ 
{\includegraphics[scale=0.22, trim = 2.0cm 6.0cm 2.0cm 6.0cm, clip=true,]{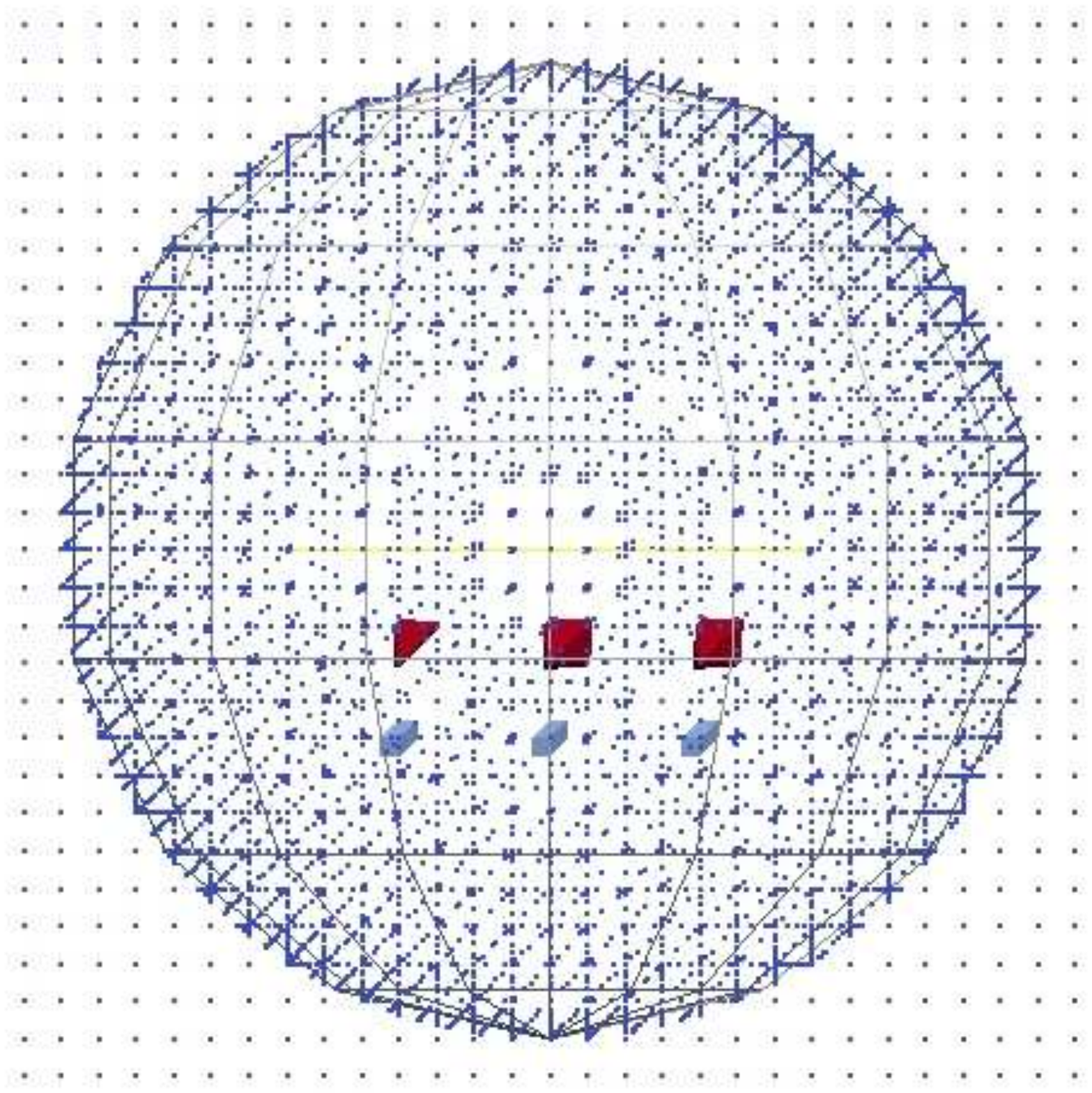}} &
{\includegraphics[scale=0.22, trim = 2.0cm 6.0cm 2.0cm 6.0cm, clip=true,]{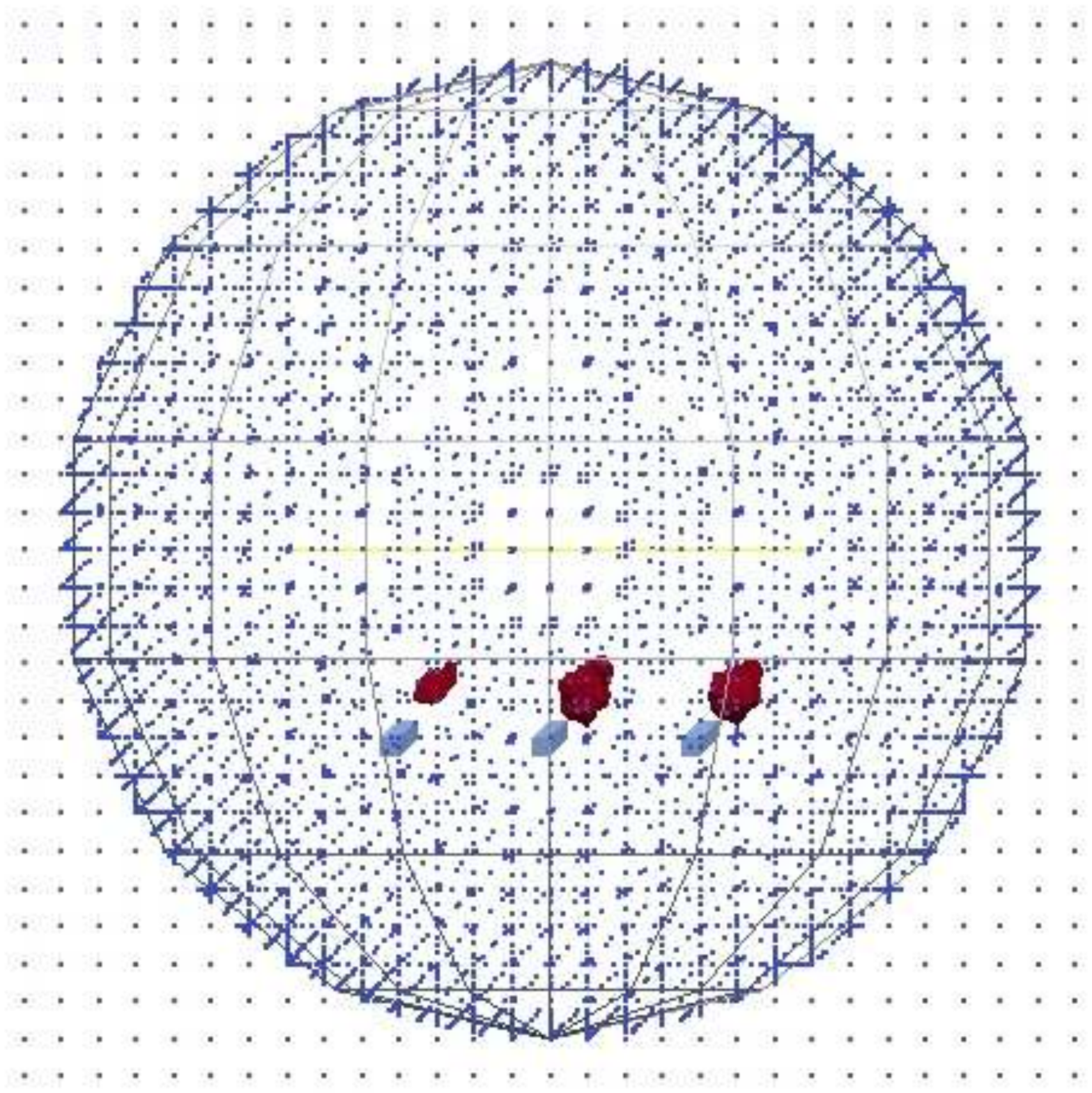}} \\
e) $x_2 x_3$ view & f)   $x_2 x_3$ view \\
{\includegraphics[scale=0.22, trim = 2.0cm 6.0cm 2.0cm 6.0cm, clip=true,]{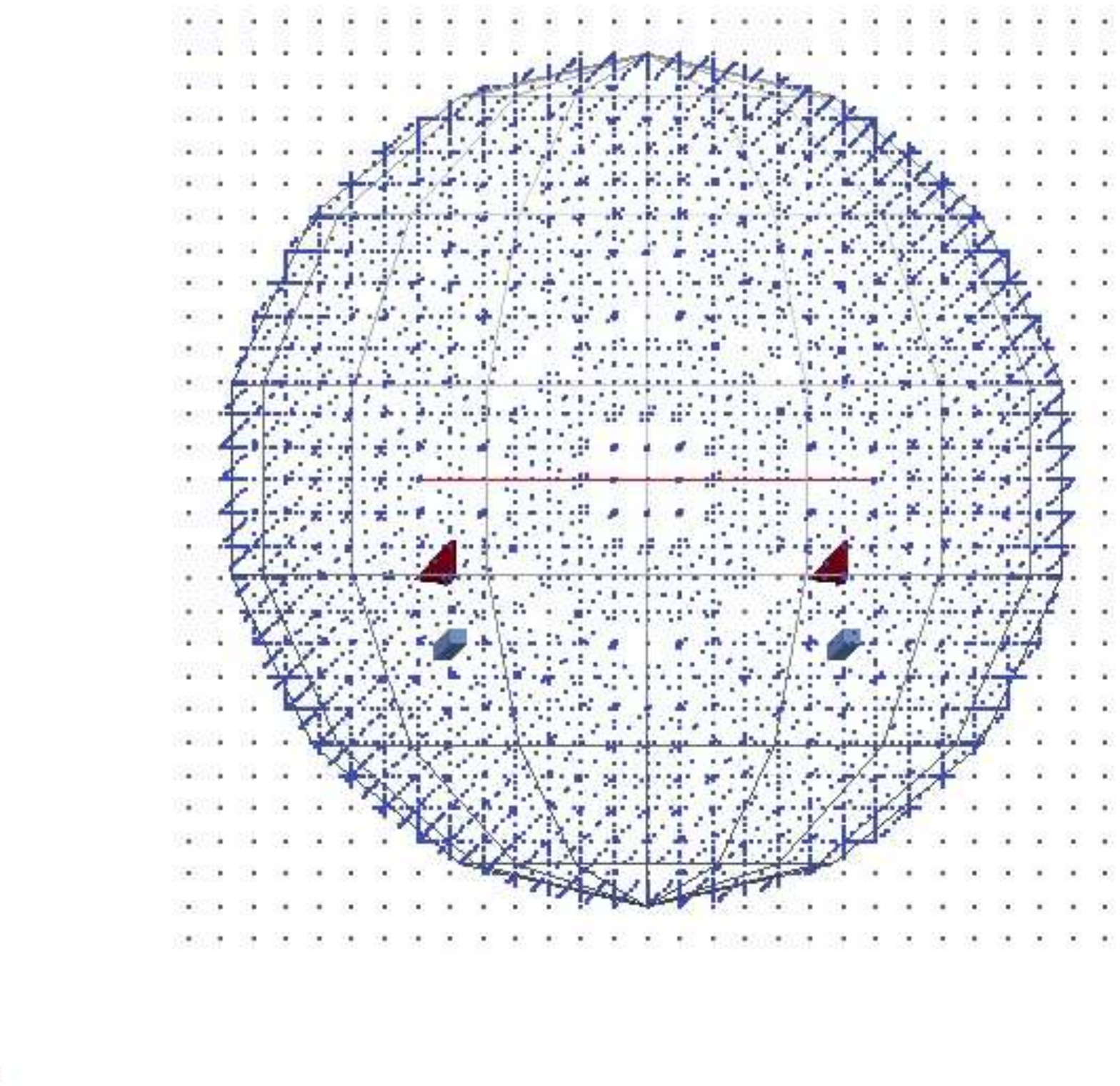}} &
{\includegraphics[scale=0.22, trim = 2.0cm 6.0cm 2.0cm 6.0cm, clip=true,]{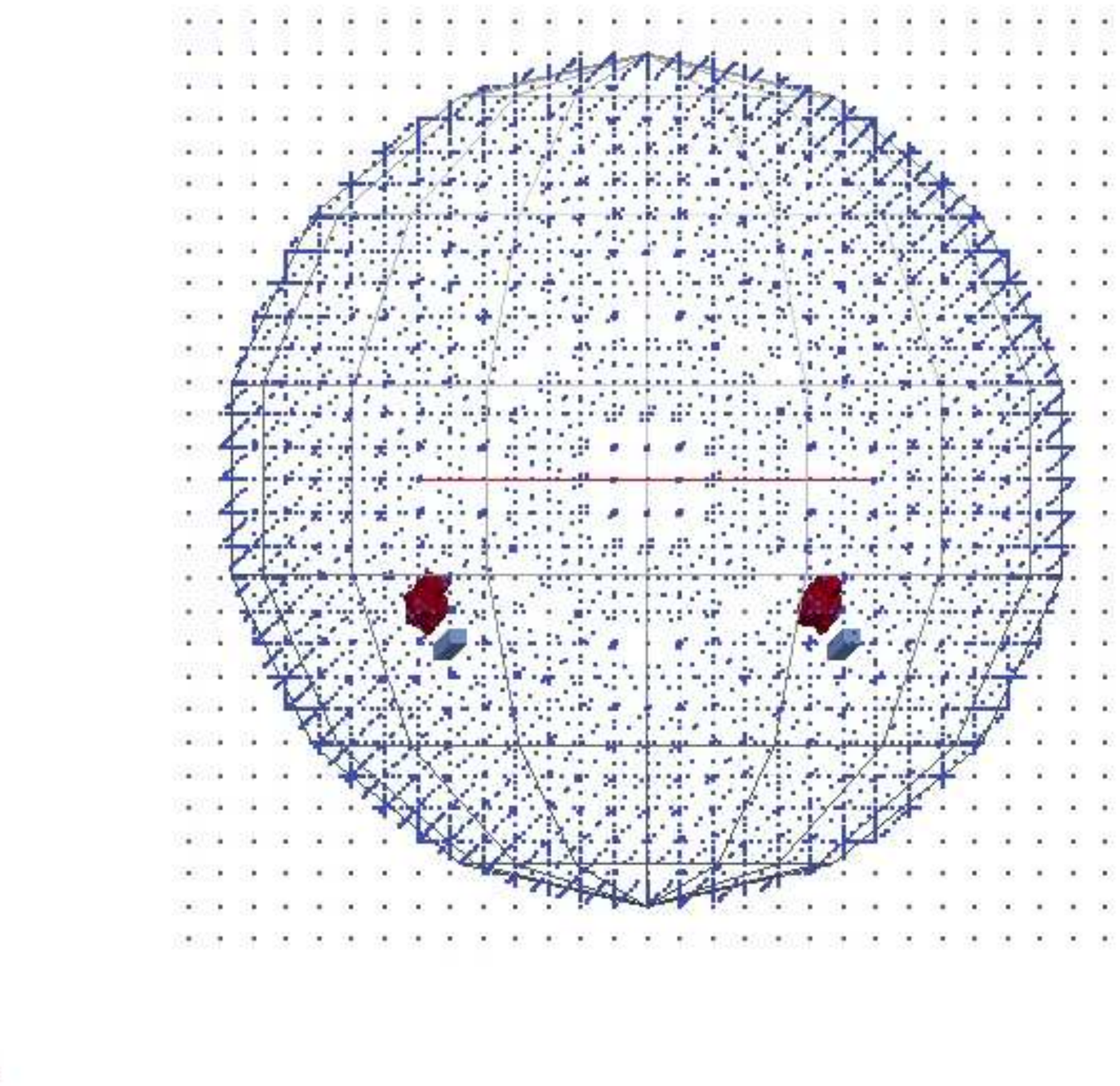}} \\
 g) $x_1 x_3$ view & h) $x_1 x_3$ view 
\end{tabular}
\end{center}
\caption{\small\emph{Test 2. Reconstructions (in red color) $\eps_{h_0}$ of $\eps$ obtained on the coarse mesh (left figures), and $\eps_\mathrm{rec}$ on the five times adaptively refined mesh (right figures).  The noise level in the data is $\sigma=10\%$. For comparison we also present exact isosurfaces of the three small inclusions to be reconstructed (in light blue color).}}
\label{fig:test2}
\end{figure}

\subsubsection{Test 2}
In the test of this section we consider the problem arisen in microwave imaging: reconstruction of high contrast in malign tumors and detection of small tumor sizes (less than 1 cm) in breast cancer screening. Most of existing numerical methods for solution of these problems uses minimization of conventional least-squares functionals and Gauss-Newton methods. See, for example, \cite{bssps04, jzlj94, cw90, jph91, fhp06}, \cite{egl15, egl15b}. We propose to use an adaptive finite element method which will allow achievement of high contrast in the malign tumor and efficiently detect very small sizes of inclusions during adaptive mesh refinement.  Although only reconstruction of a real dielectric permittivity function is considered in our test, the obtained results can be extended for the reconstruction of the complex permittivity function which is one of the goals of our current research.

We tested the Second Adaptive Algorithm on the reconstruction of three small inclusions with centers at $(-0.3,\,0.0,\,-0.25), (0.3,\,0.2,\,-0.25)$, and $(0.3,\,-0.2,\,-0.25)$ located inside spherical geometry of Figure~\ref{fig:fig1}, see also Figure \ref{fig:test2cells} where they are visualized. These inclusions model three small malign tumors of a size 2 mm.  We performed simulations with two additive noise levels in the data: $\sigma= 3 \%$ and $\sigma= 10\%$, see Tables~1--2 for the results.

The reconstruction of the three inclusions on the initial coarse mesh with $\sigma= 10\%$ is presented on the left figures of Figure~\ref{fig:test2}. In this figure and Table~1 we observe that, on the coarse mesh, we obtain quite correct locations of all inclusions and achieve maximal contrast of $\max_{\Omega_\mathrm{FEM}} \eps_{h_0}= 2.8$ in the inclusions. However, Figures~\ref{fig:test2}-e), and g) show us that the locations of all inclusions in $x_3$-direction can still be improved. The figures on the right of Figure~\ref{fig:test2} present the reconstruction $\eps_\mathrm{rec}$ of $\eps$ on the five times locally adaptively refined mesh.  In Figure~\ref{fig:test2}-f), and h) we observe an improvement of the reconstructions of the three inclusions in the $x_3$-direction on the final adaptively refined mesh, compared with reconstructions obtained on the coarse mesh. Comparing Tables~1 and 2 we also can conclude that adaptive mesh refinement allowed us to obtain more correct contrast for all three inclusions.

\begin{figure}
\begin{center}
\begin{tabular}{cc}
{\includegraphics[scale=0.4, clip=true,]{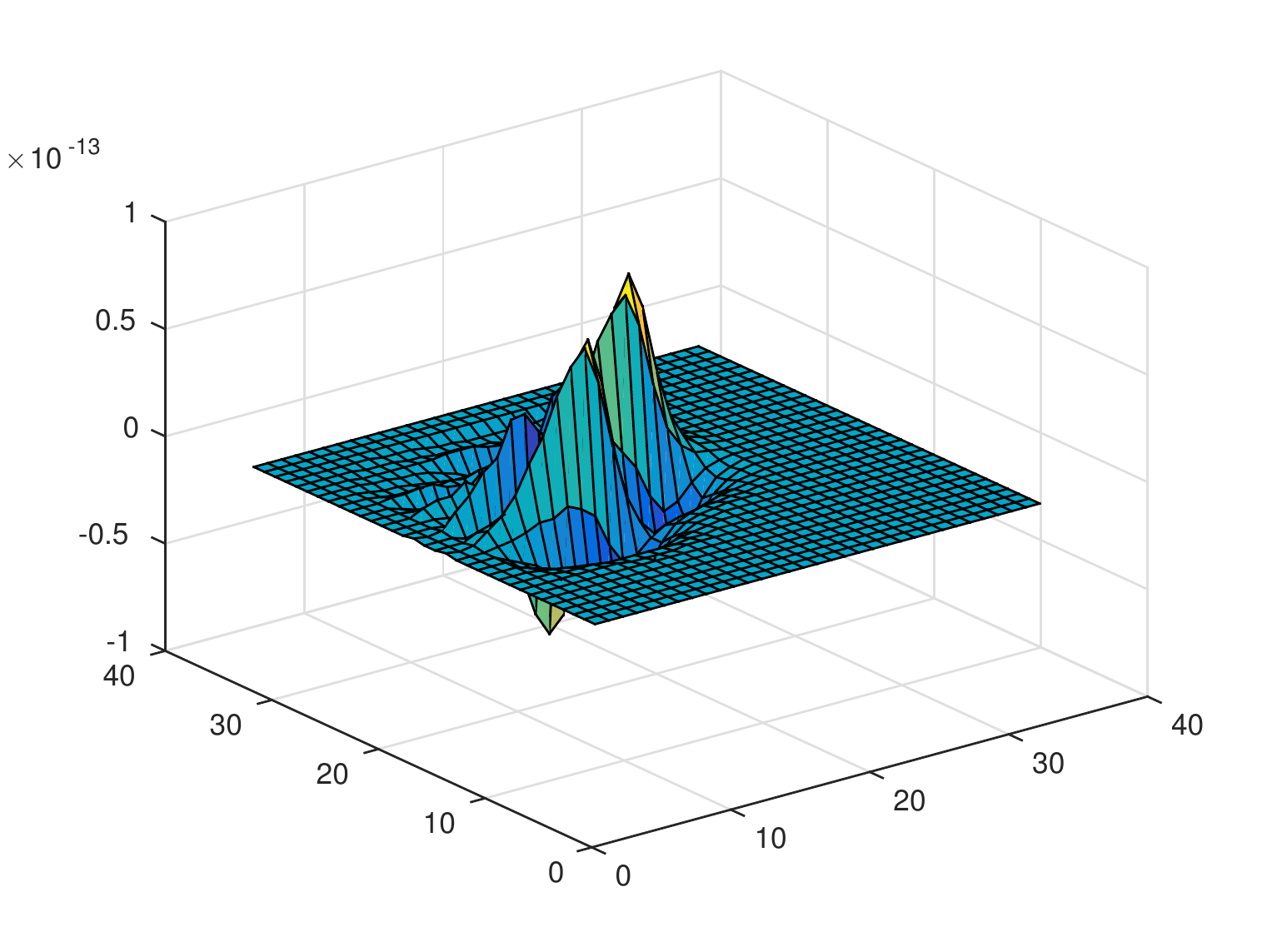}} &
{\includegraphics[scale=0.4, clip=true,]{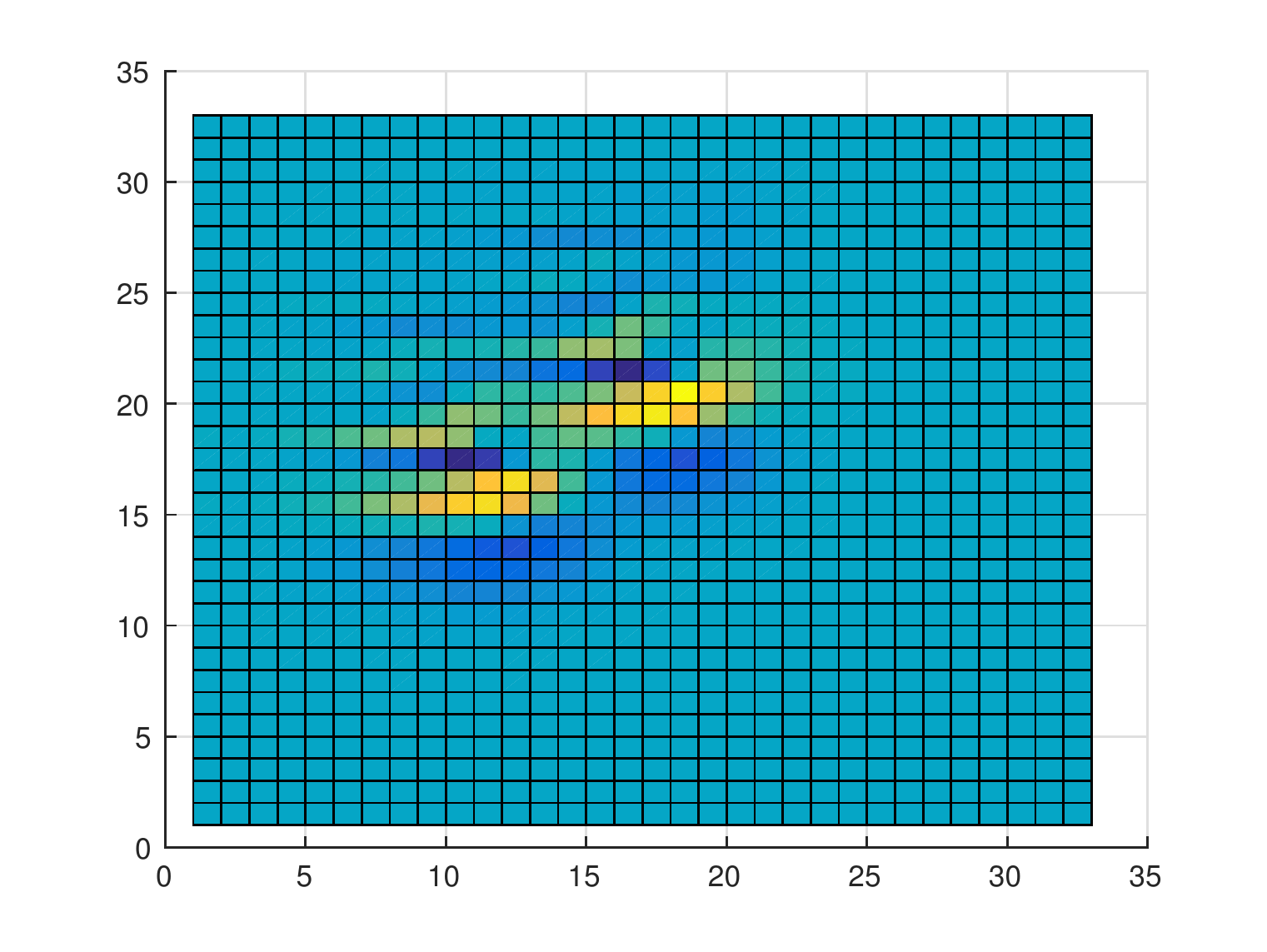}} \\
$t=0.6$  & $t=0.6$, $x_1 x_2$ view \\
{\includegraphics[scale=0.4, clip=true,]{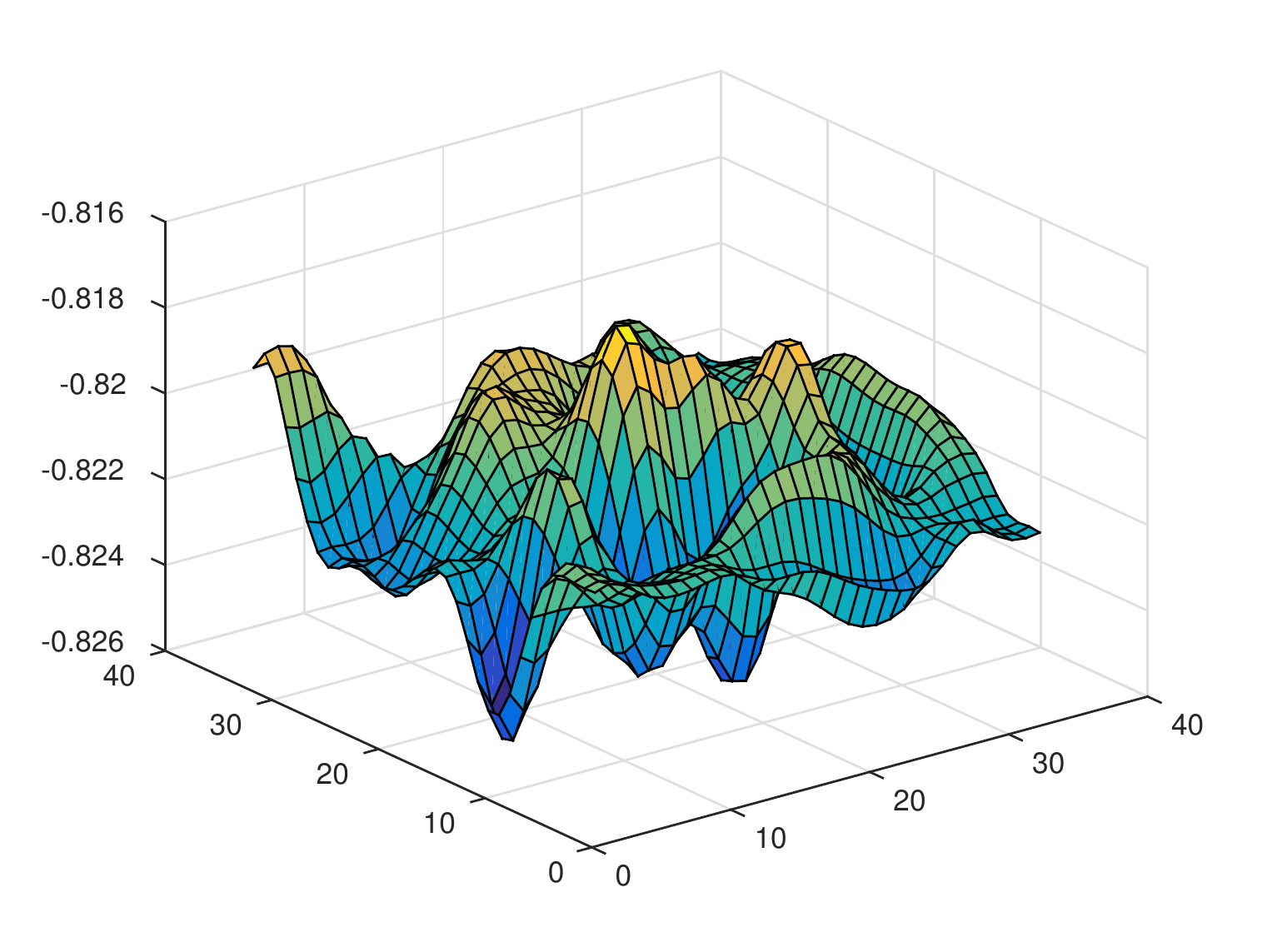}} &
{\includegraphics[scale=0.4, clip=true,]{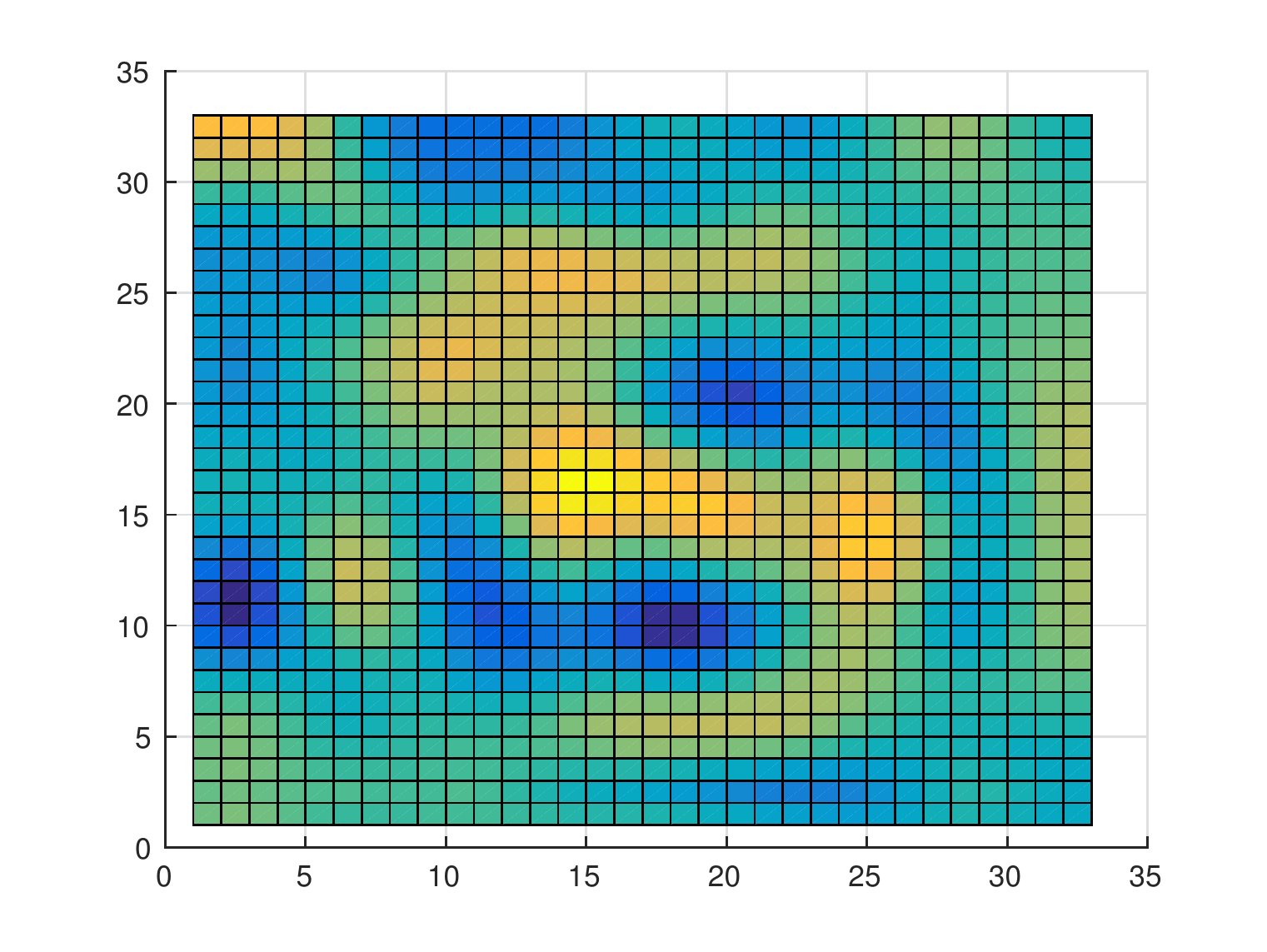}} \\
$t=1.8$  & $t=1.8$, $x_1 x_2$ view \\
{\includegraphics[scale=0.4, clip=true,]{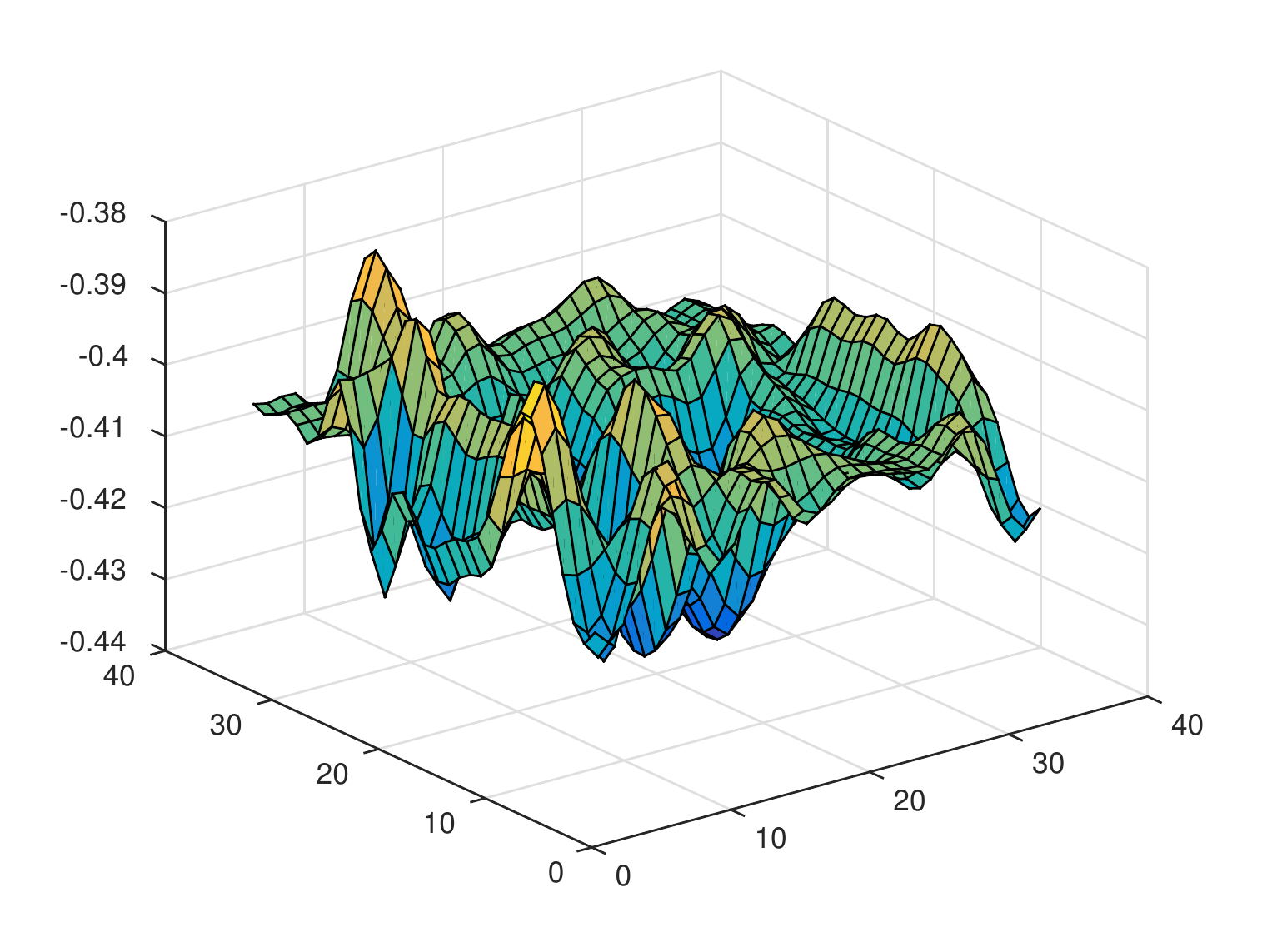}} &
{\includegraphics[scale=0.4, clip=true,]{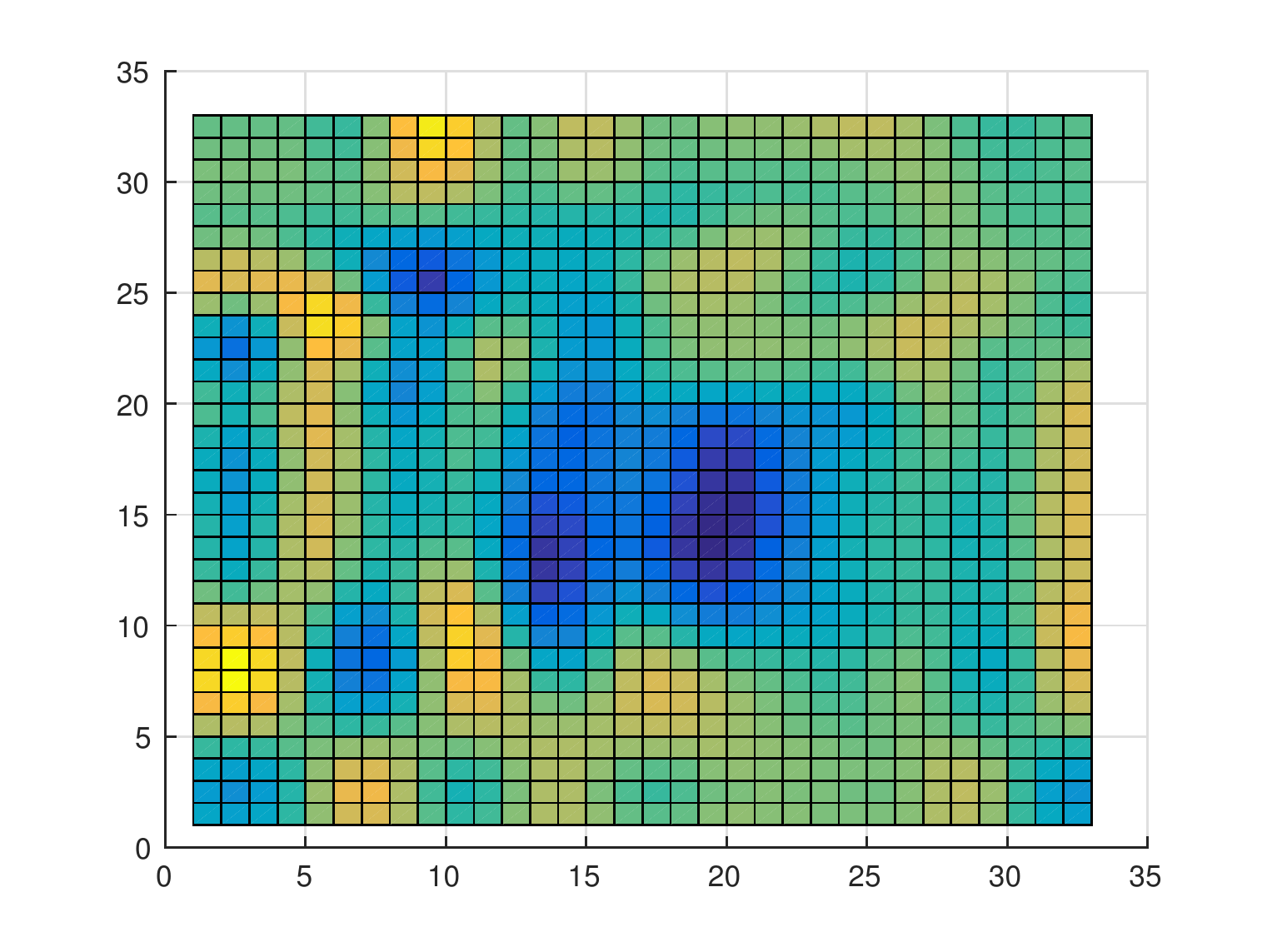}} \\
$t=2.4$  & $t=2.4$, $x_1 x_2$ view \\
\end{tabular}
\end{center}
\caption{\small\emph{Test 3. Transmitted data of the component $E_2$ at different times.  The noise level in the data is $\sigma=3\%$.}}
\label{fig:test3data}
\end{figure}

\begin{figure}
\begin{center}
\begin{tabular}{cc}
$\x\in\Omega_\mathrm{FEM}:\eps_{h_0}(\x) = 2.04 $ & $\x\in\Omega_\mathrm{FEM}:\eps_\mathrm{rec}(\x) = 1.88 $ \\
{\includegraphics[scale=0.22, trim = 2.0cm 6.0cm 2.0cm 6.0cm, clip=true,]{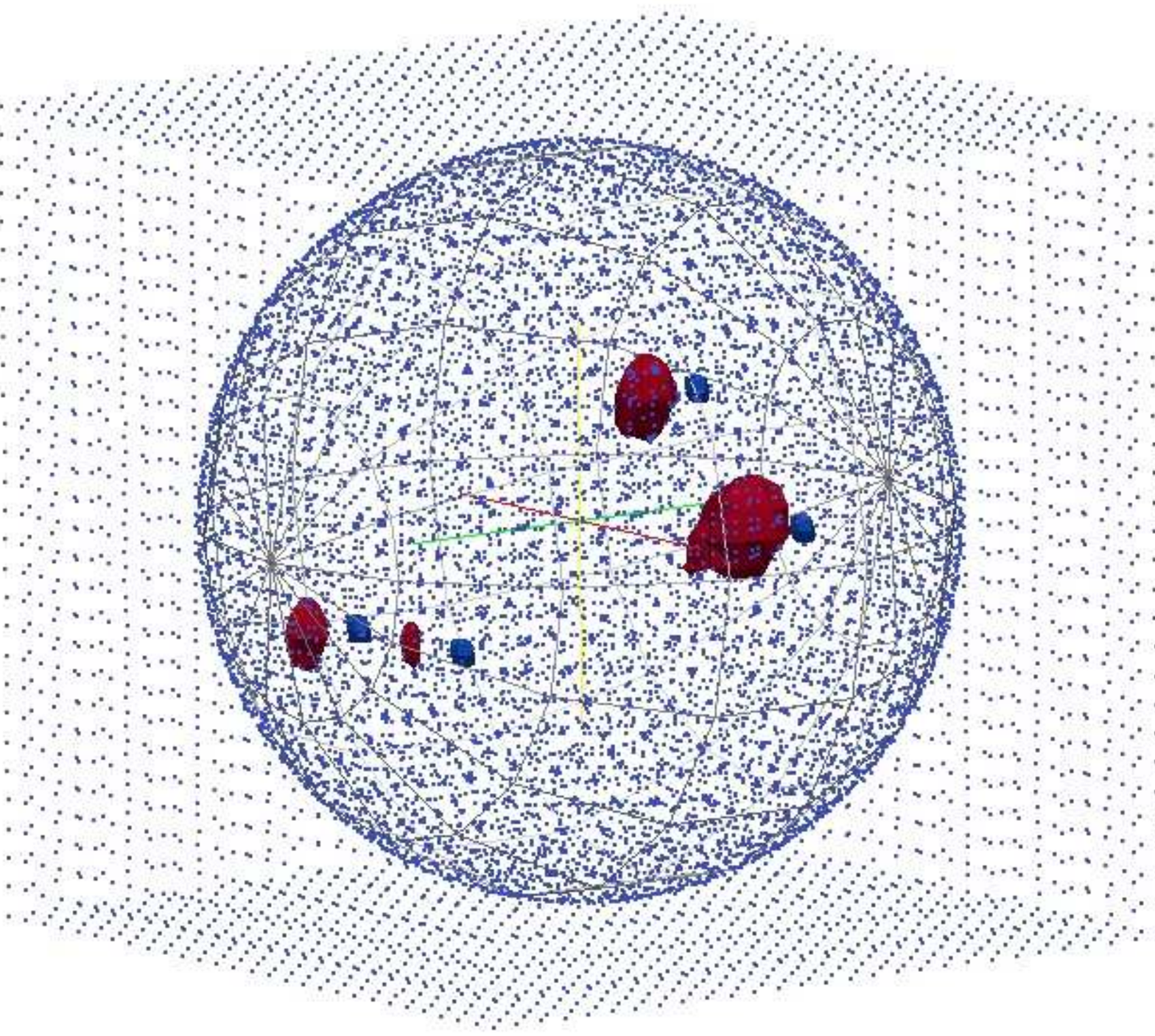}} &
{\includegraphics[scale=0.22, trim = 2.0cm 6.0cm 2.0cm 6.0cm, clip=true,]{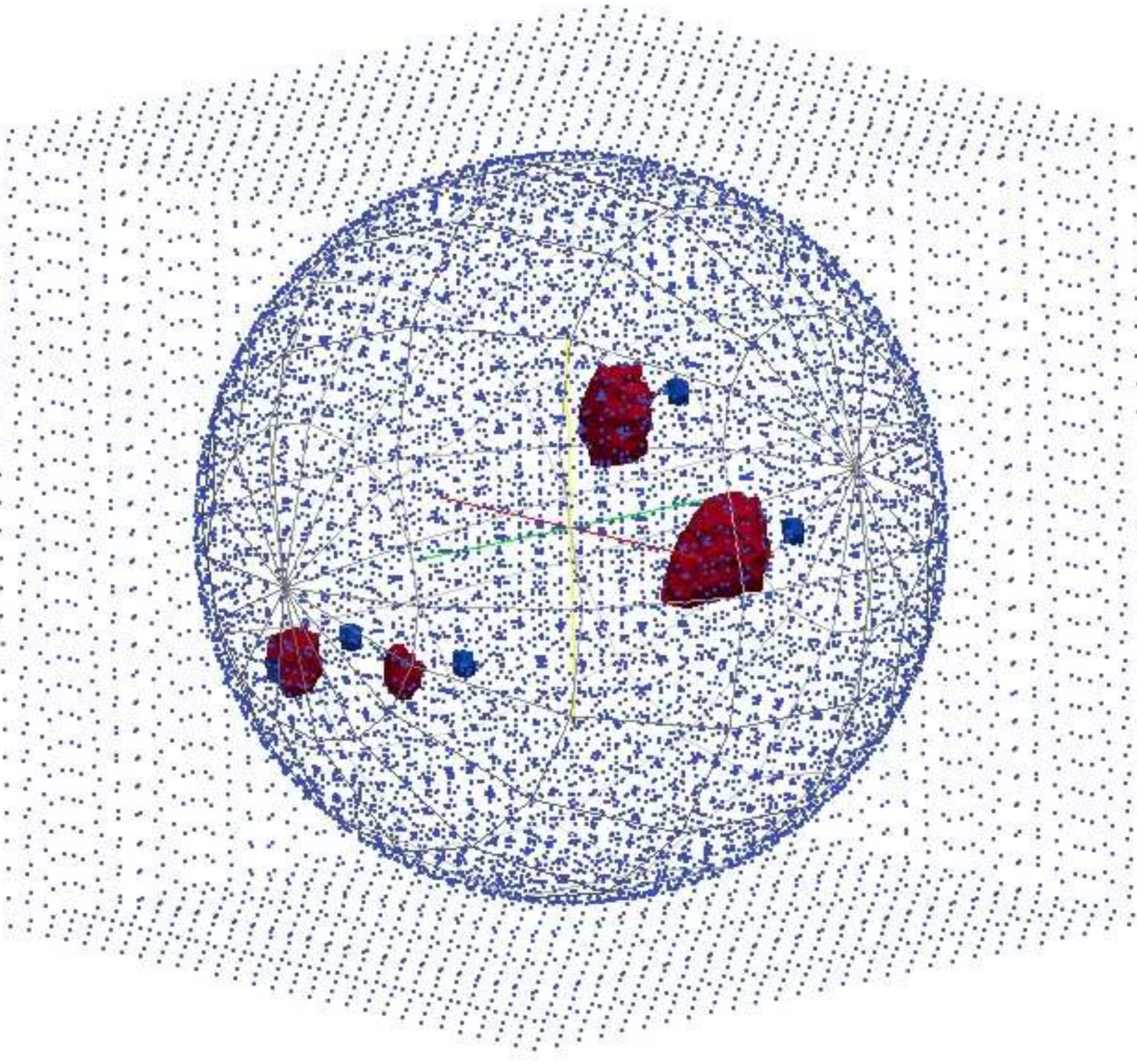}}  \\
a) prospect view &  b) prospect view \\
{\includegraphics[scale=0.22, trim = 2.0cm 6.0cm 2.0cm 6.0cm, clip=true,]{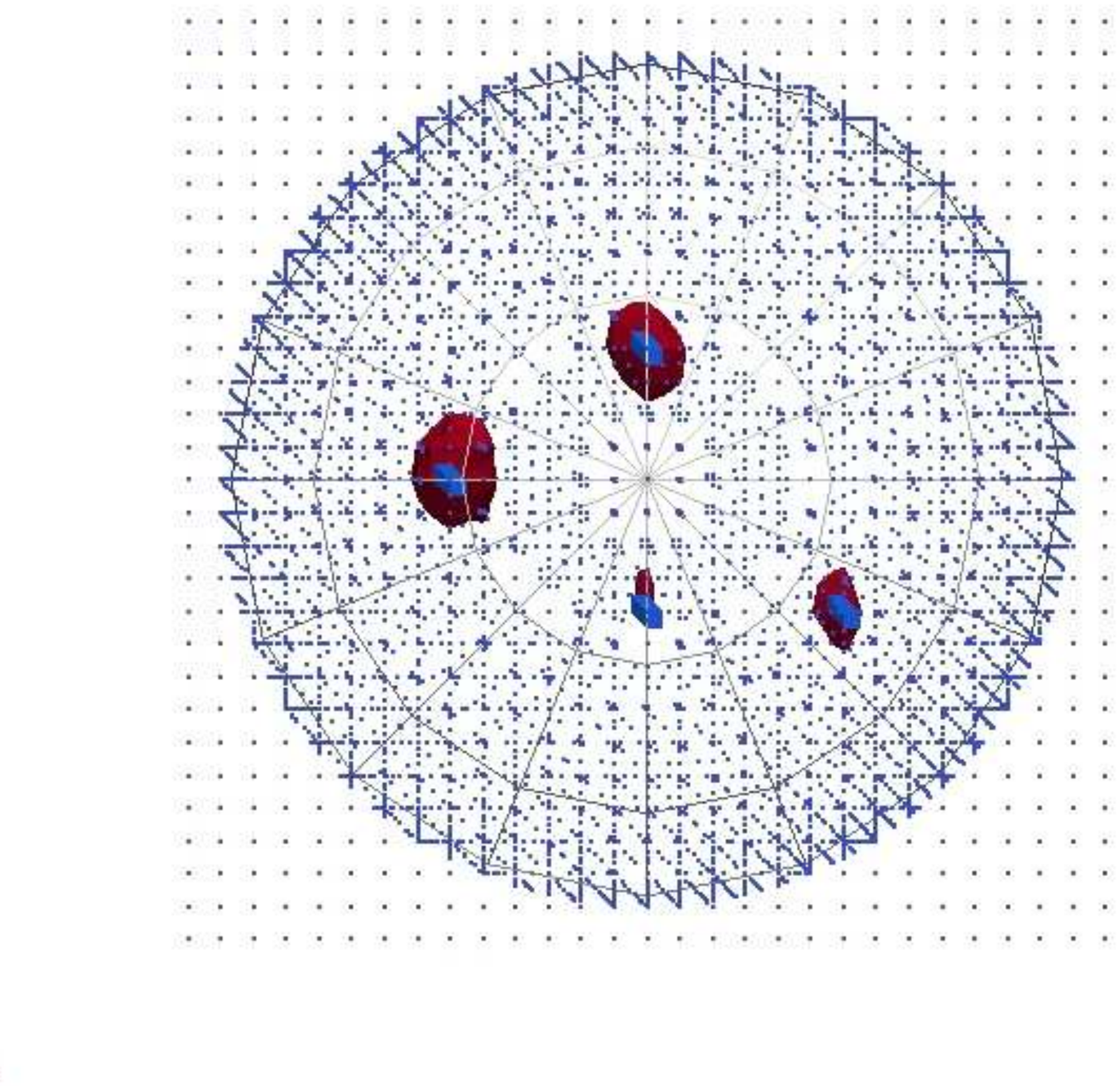}} &
{\includegraphics[scale=0.22, trim = 2.0cm 6.0cm 2.0cm 6.0cm, clip=true,]{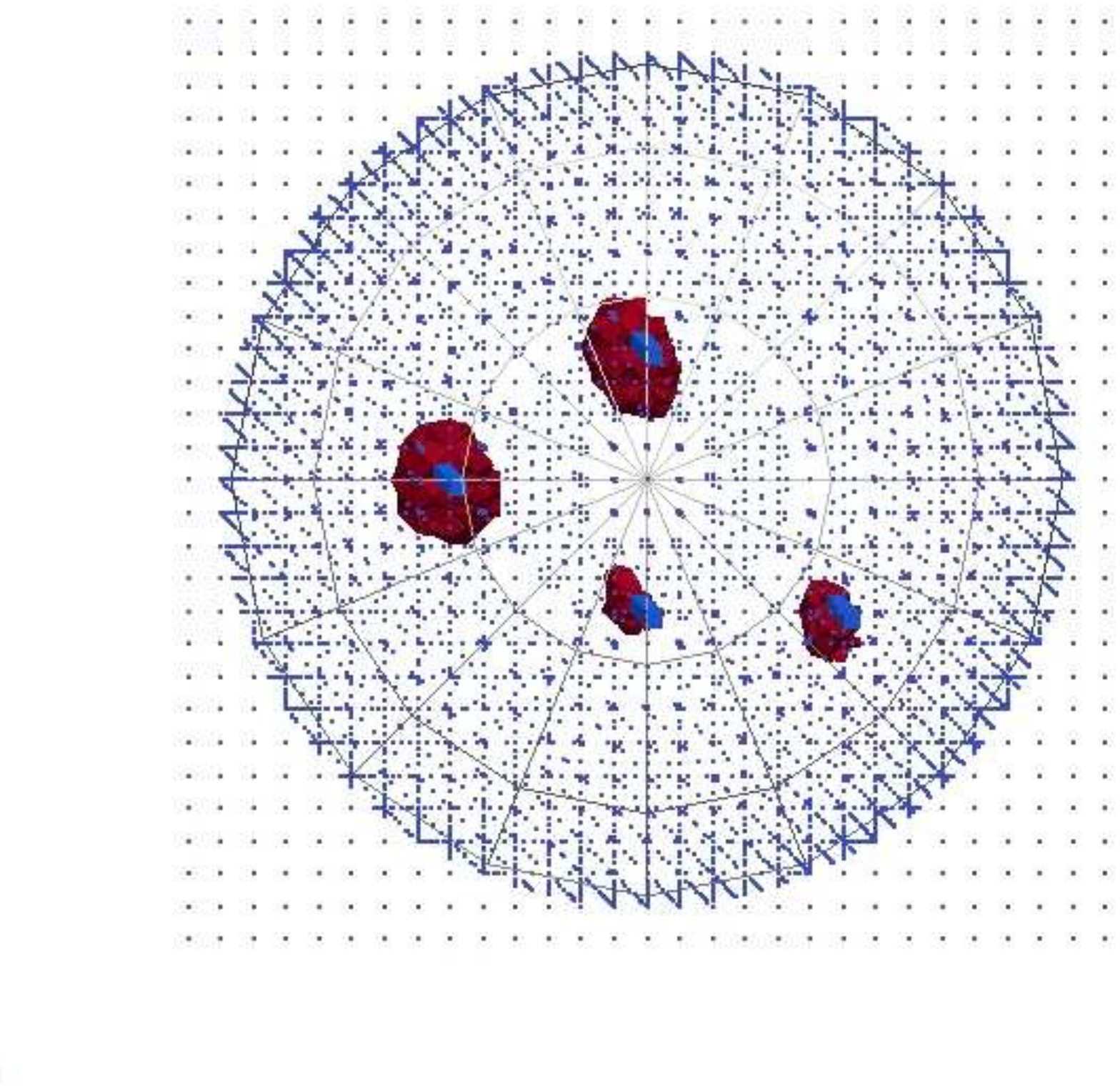}} \\
c) $x_1 x_2$ view & d) $x_1 x_2$ view \\ 
{\includegraphics[scale=0.22, trim = 2.0cm 6.0cm 2.0cm 6.0cm, clip=true,]{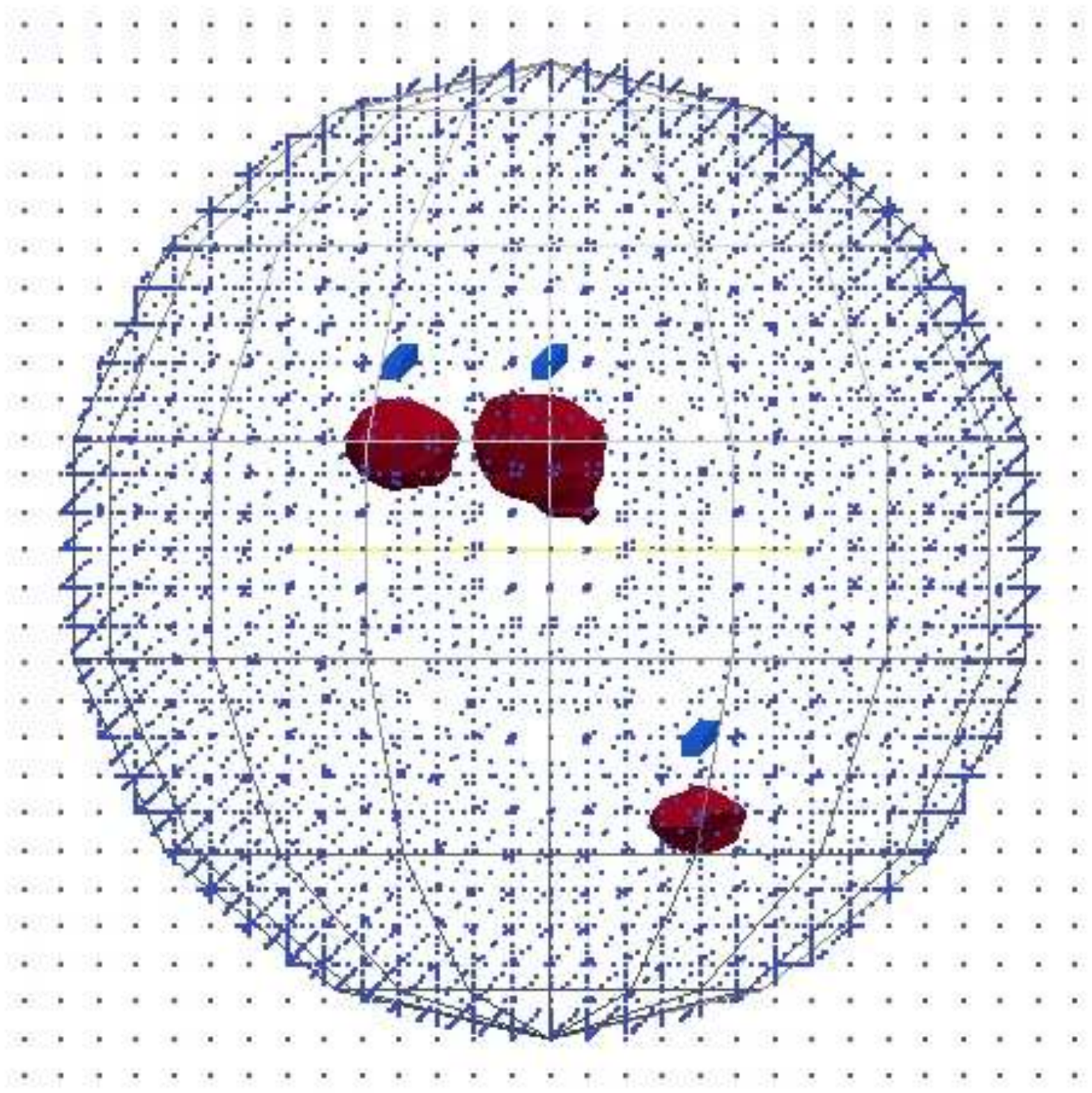}} &
{\includegraphics[scale=0.22, trim = 2.0cm 6.0cm 2.0cm 6.0cm, clip=true,]{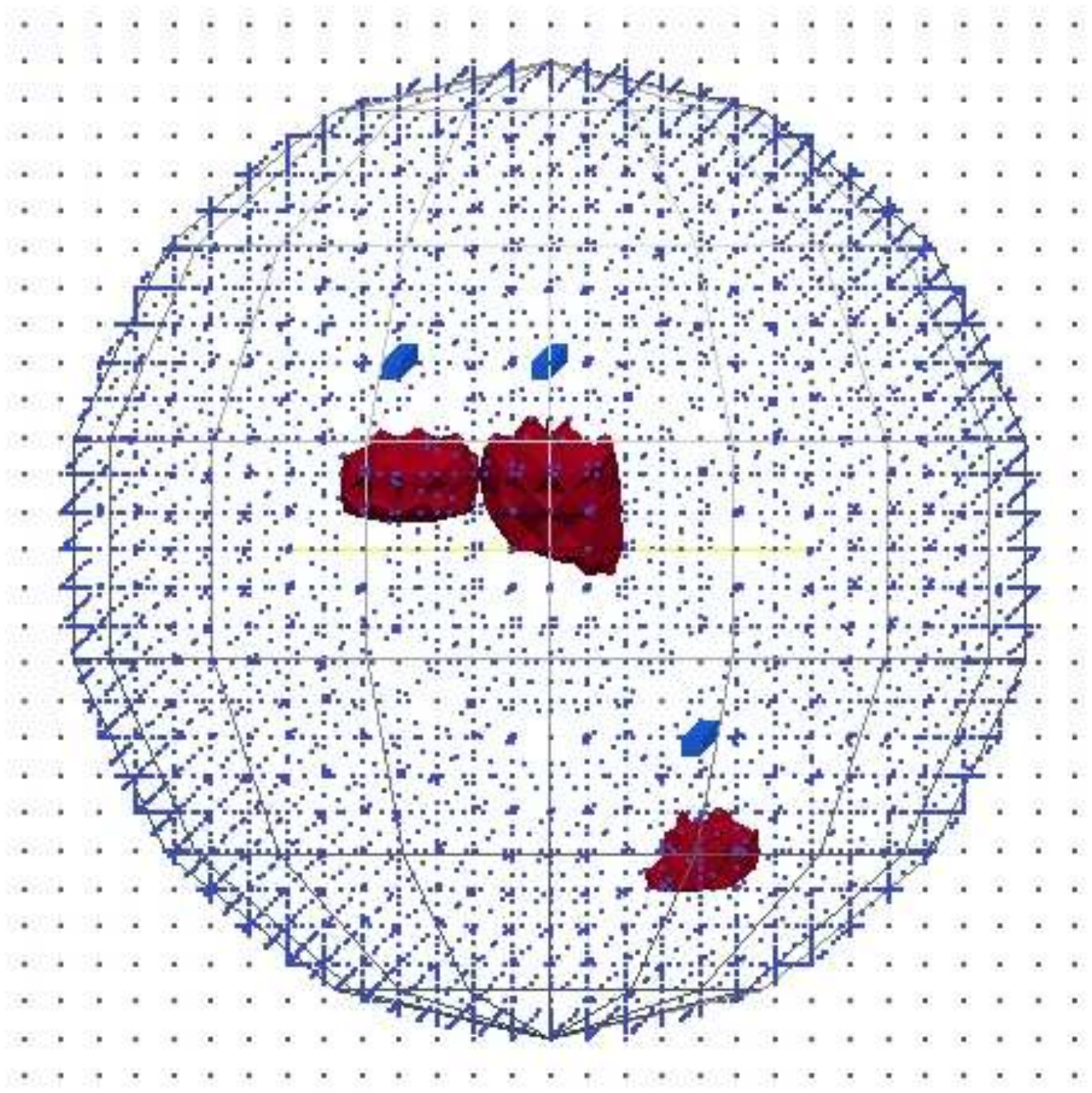}} \\
e) $x_2 x_3$ view & f) $x_2 x_3$ view \\
{\includegraphics[scale=0.22, trim = 2.0cm 6.0cm 2.0cm 6.0cm, clip=true,]{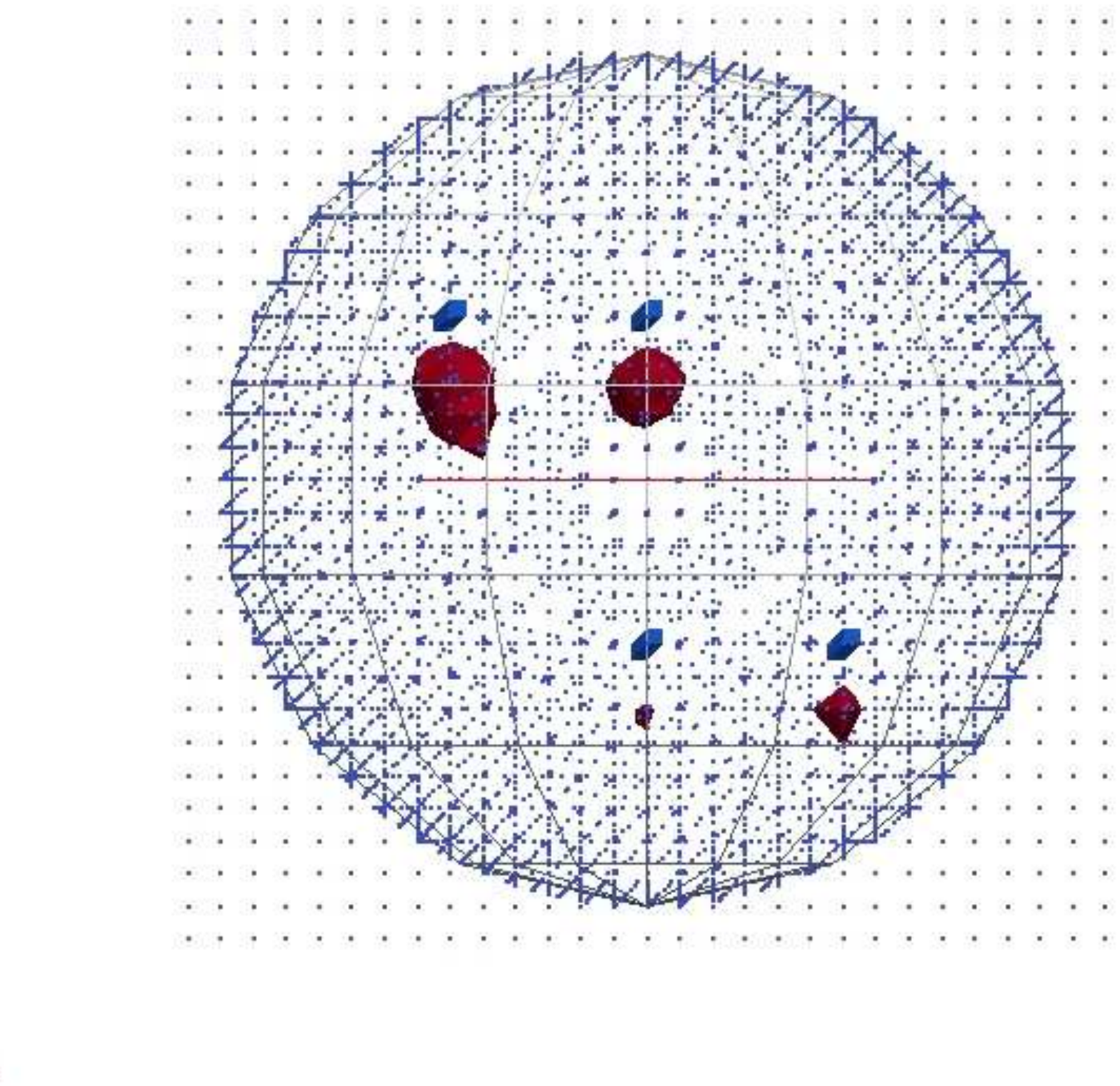}} &
{\includegraphics[scale=0.22, trim = 2.0cm 6.0cm 2.0cm 6.0cm, clip=true,]{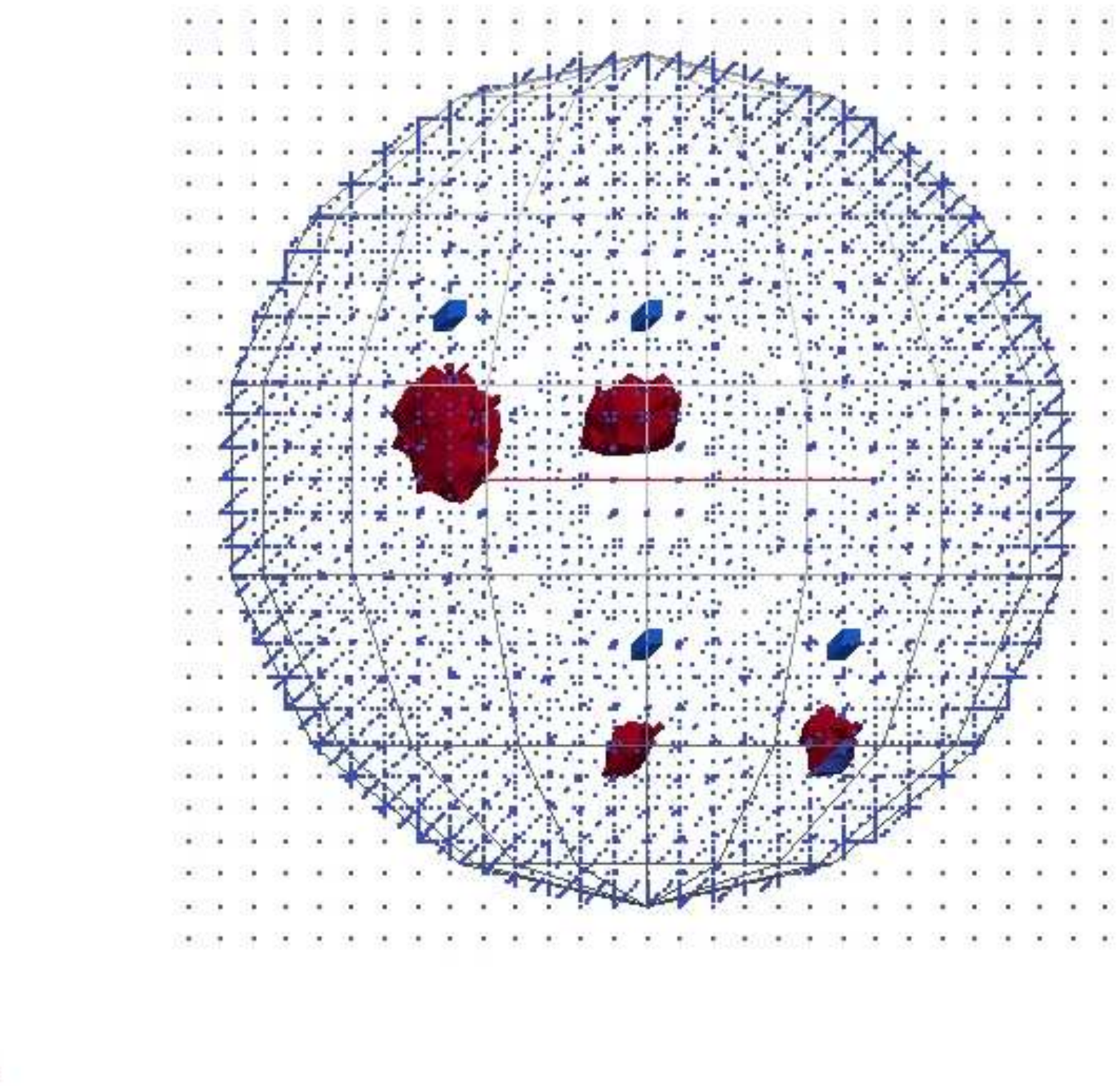}} \\
g) $x_1 x_3$ view & h) $x_1 x_3$ view 
\end{tabular}
\end{center}
\caption{ \protect\small \emph{Test 3. Reconstruction of four inclusions (in red color) obtained on the coarse mesh (left figures) and on the two times adaptively refined mesh (right figures). The noise level in the data is $\sigma=10\%$. For comparison we also present exact isosurfaces of the four small inclusions to be reconstructed (in light blue color).}}
\label{fig:test3noise10}
\end{figure}

\subsubsection{Test 3}
This test is similar to Test~3, only here the goal was to reconstruct four small inclusions located at different parts of $\Omega_\mathrm{FEM}$: two inclusions were located closer to the backscattering boundary, and another two inclusions were placed closer to the transmission boundary of $\Omega_\mathrm{FEM}$, see  Figure~\ref{fig:test3noise10} where they are presented. These inclusions model four small malign tumors of a size 2 mm.  We performed simulations with two additive noise levels in the data: $\sigma= 3 \%$ and $\sigma= 10\%$, see Tables 1--2 for the results.

The reconstruction on the initial coarse mesh with a noise level $\sigma= 10\%$ in the data is presented on the left figures of Figure~\ref{fig:test3noise10}. From this figure and Table~1 we observe that we get quite correct locations of all inclusions and achieve a maximal contrast of $\max_{\Omega_\mathrm{FEM}} \eps_{h_0} = 2.04 $ on the coarse mesh. However, Figure~\ref{fig:test2}-e), and g) show that the locations of all inclusions in the $x_3$-direction can still be improved. These figures also show that one of the four small inclusions is almost not present in the initial reconstruction.

The figures on the right of Figure~\ref{fig:test3noise10} present the reconstruction $\eps_\mathrm{rec}$ on the two times locally adaptively refined mesh.  In Figure~\ref{fig:test3noise10}-f), and h) we observe that the two lower inclusions are well reconstructed. However, since in this quite challenging test we have used only transmitted data resulted from a single measurement of a plane wave, we do not observe significant improvement of the reconstruction of the four inclusions in the $x_3$-direction.

\begin{figure}
\begin{center}
\begin{tabular}{cc}
{\includegraphics[scale=0.4, clip=true,]{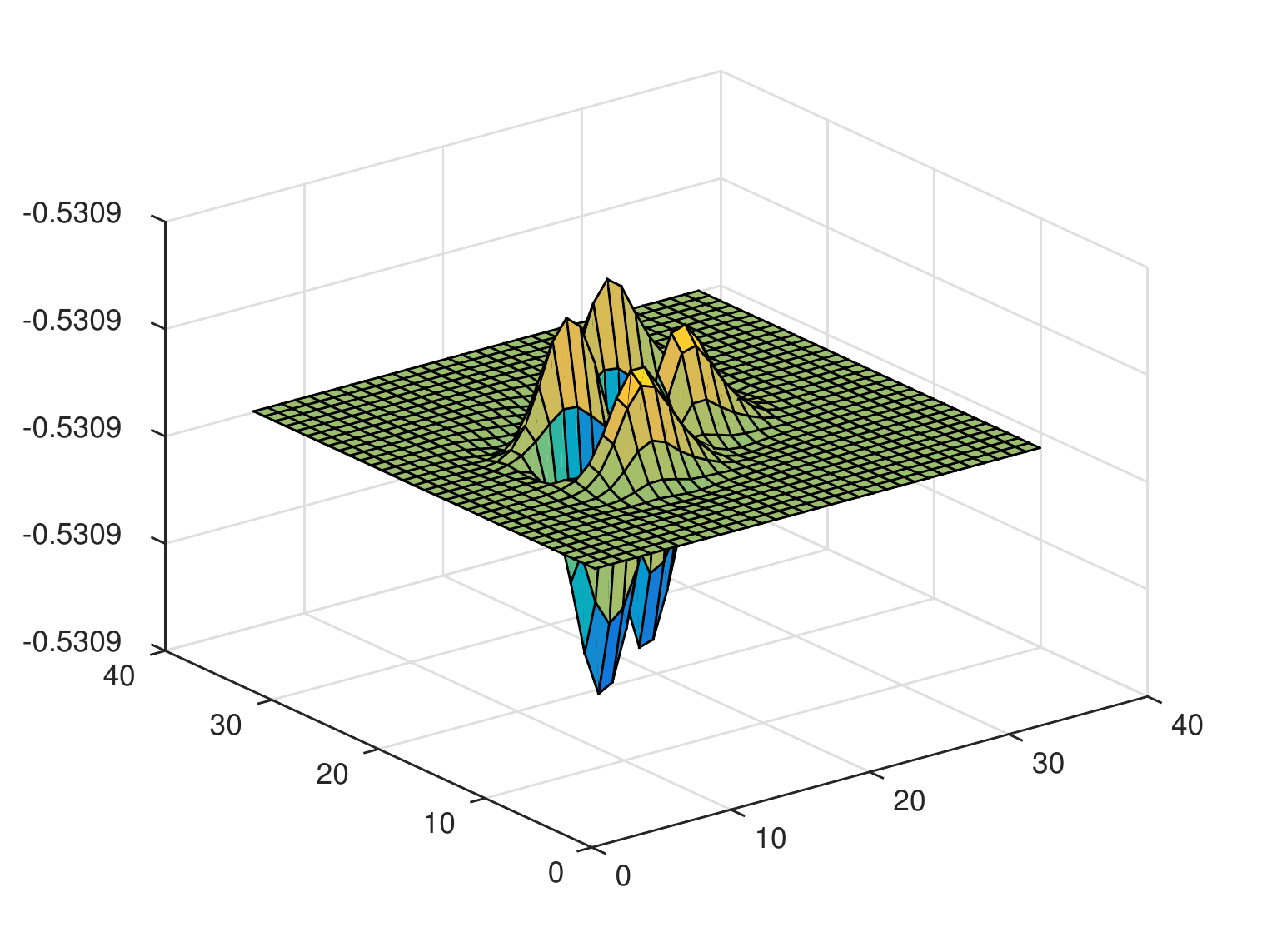}} &
{\includegraphics[scale=0.4, clip=true,]{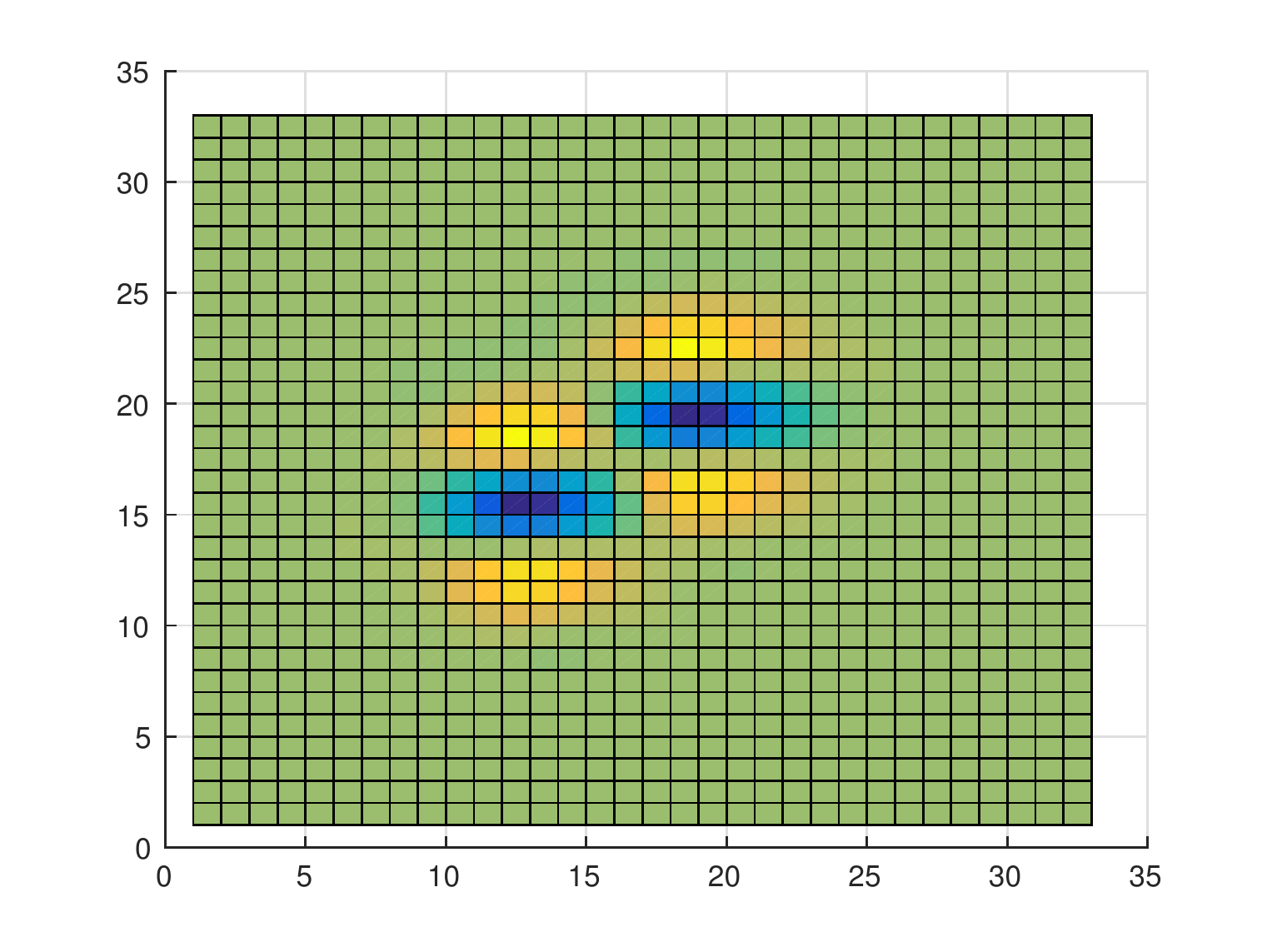}} \\
$t=0.6$  & $t=0.6$, $x_1 x_2 $ view \\
{\includegraphics[scale=0.4, clip=true,]{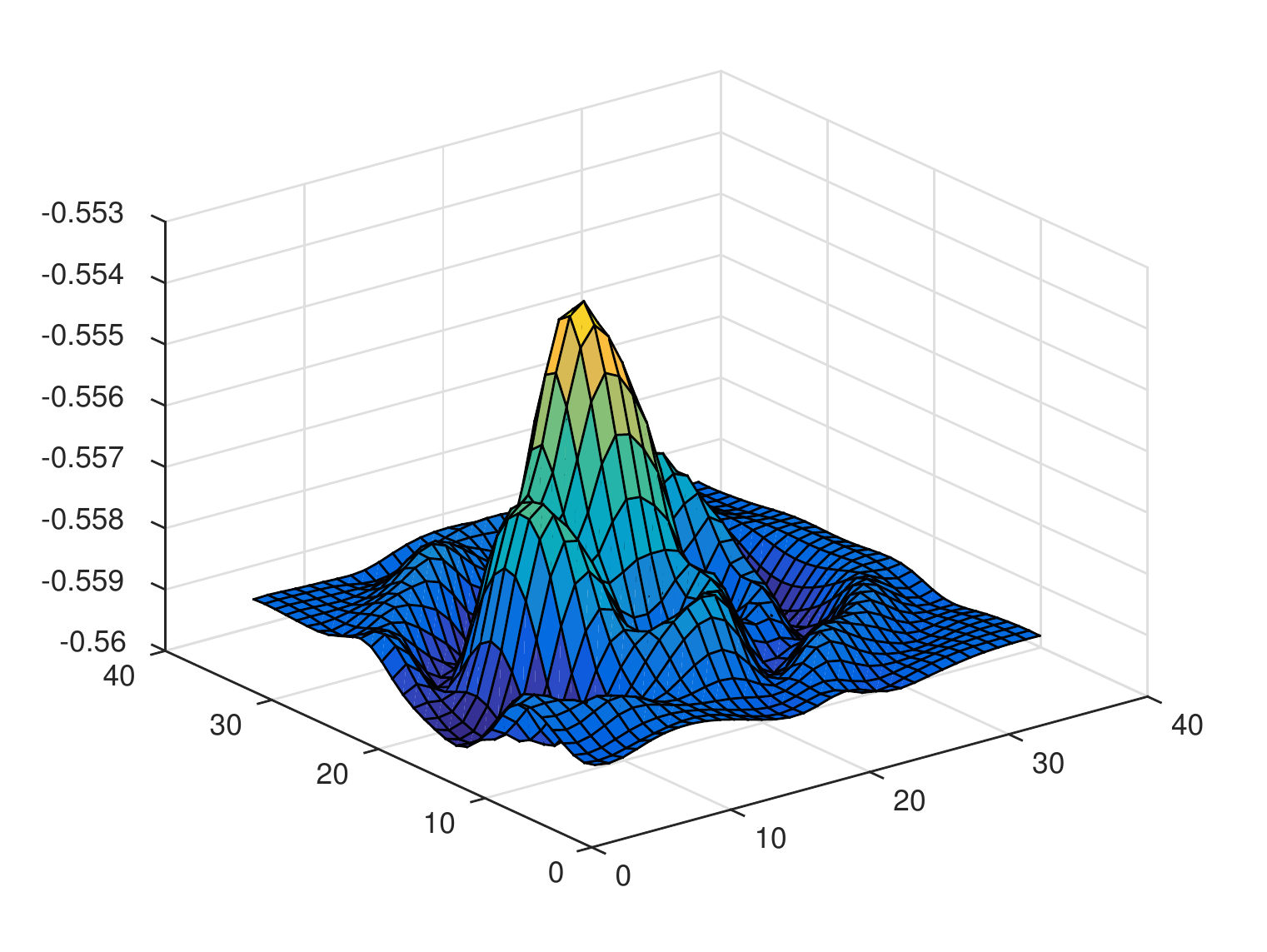}} &
{\includegraphics[scale=0.4, clip=true,]{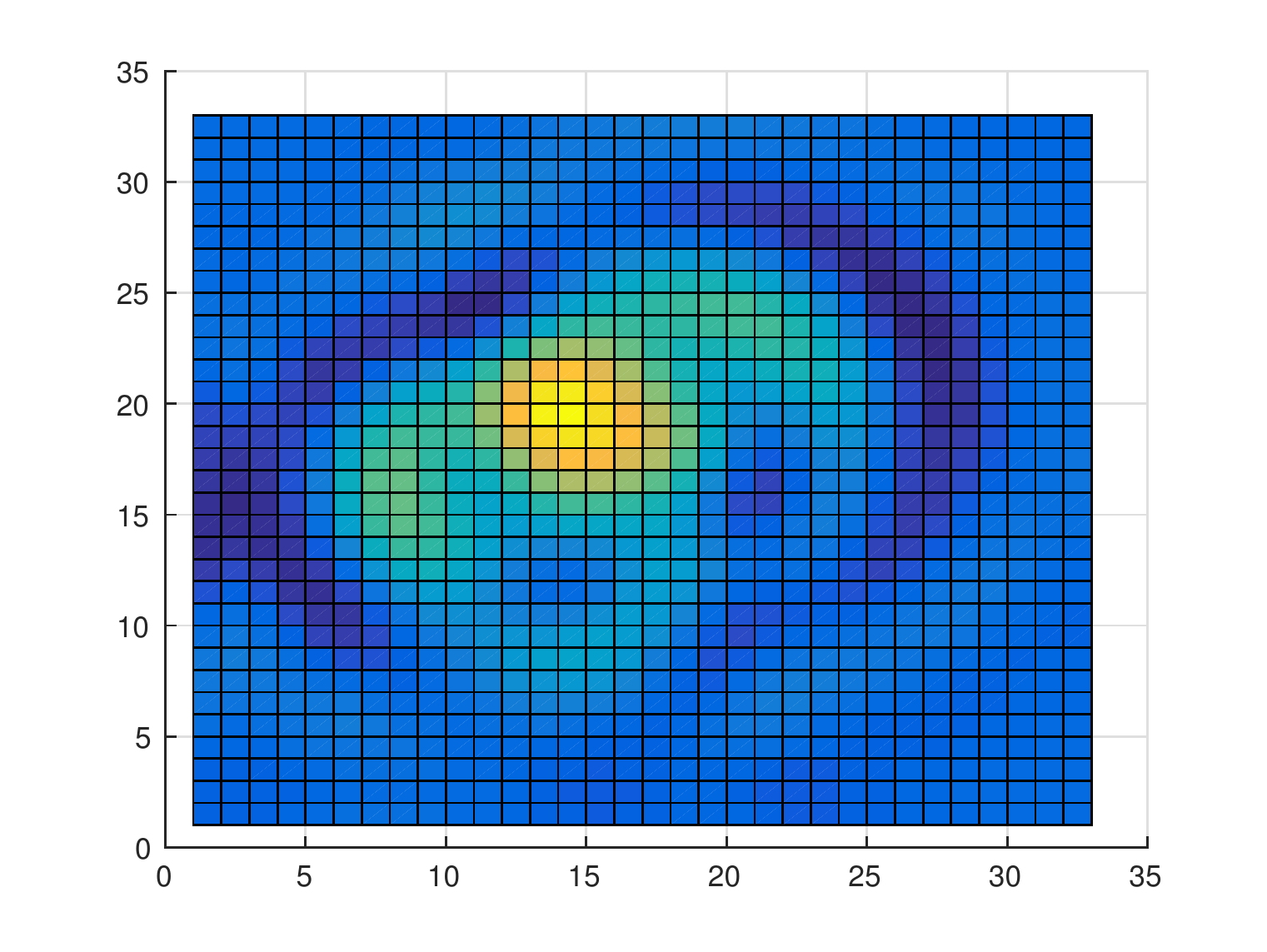}} \\
$t=1.2$  & $t=1.2$, $x_1 x_2$ view \\
{\includegraphics[scale=0.4, clip=true,]{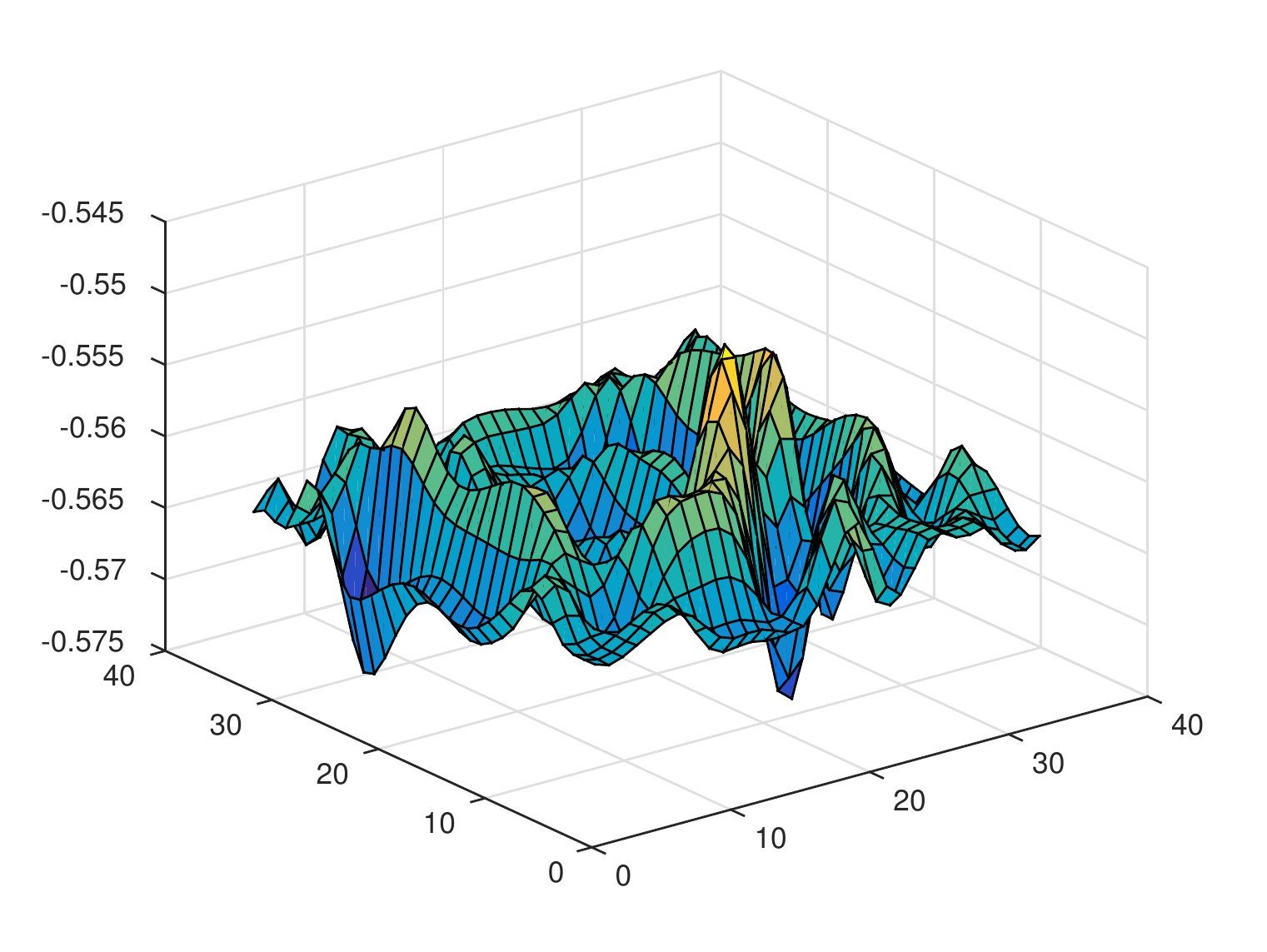}} &
{\includegraphics[scale=0.4, clip=true,]{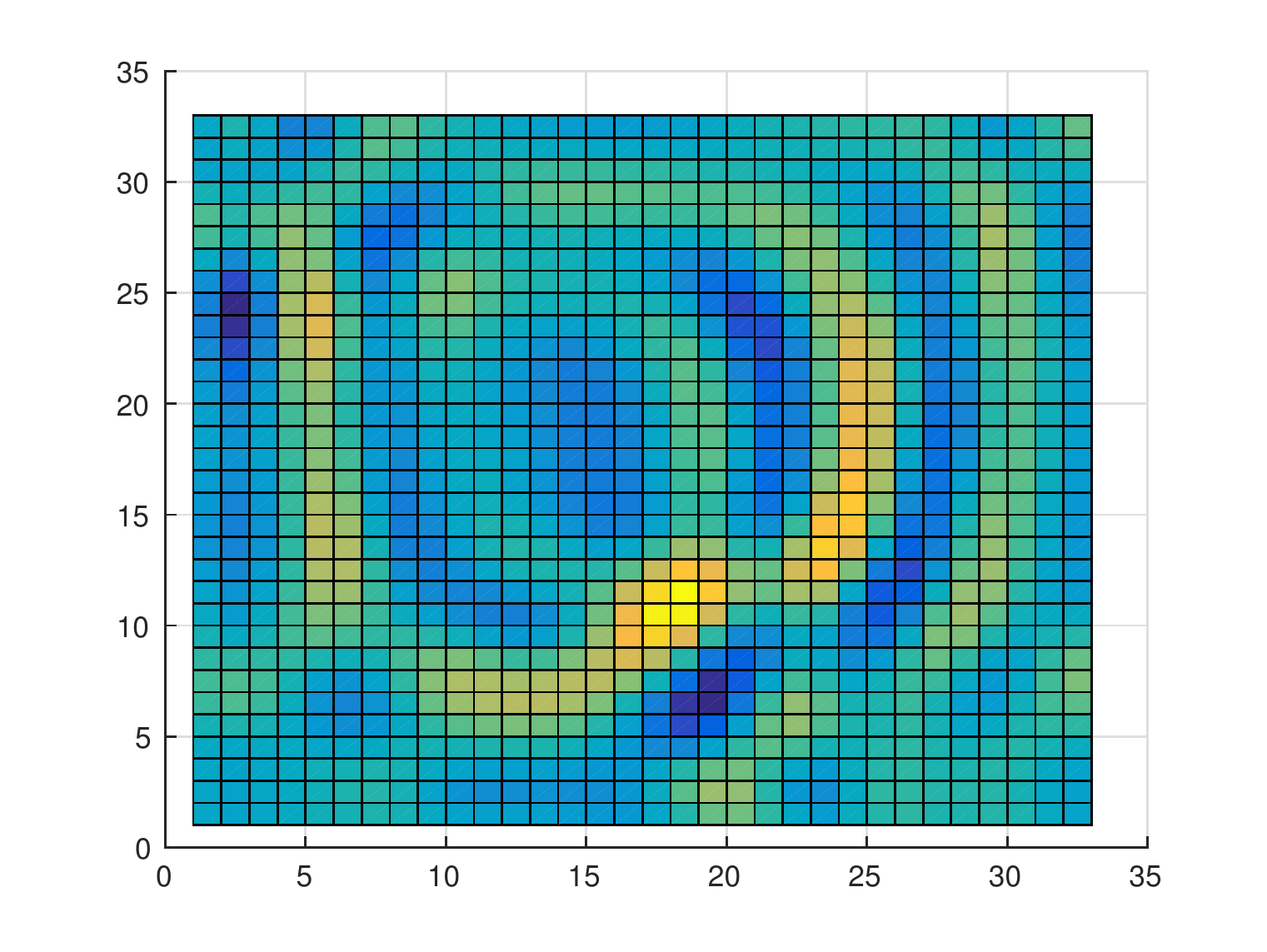}} \\
$t=1.8$  & $t=1.8$, $x_1 x_2$ view \\
\end{tabular}
\end{center}
\caption{\small\emph{Test 4. Backscattered data of the component $E_2$ at $\partial_1\widetilde{\Omega}$ at different times. The level of noise in the data is $\sigma=10\%$.}}
\label{fig:test4data}
\end{figure}

\begin{figure}
\begin{center}
\begin{tabular}{cc}
{\includegraphics[scale=0.4, clip=true,]{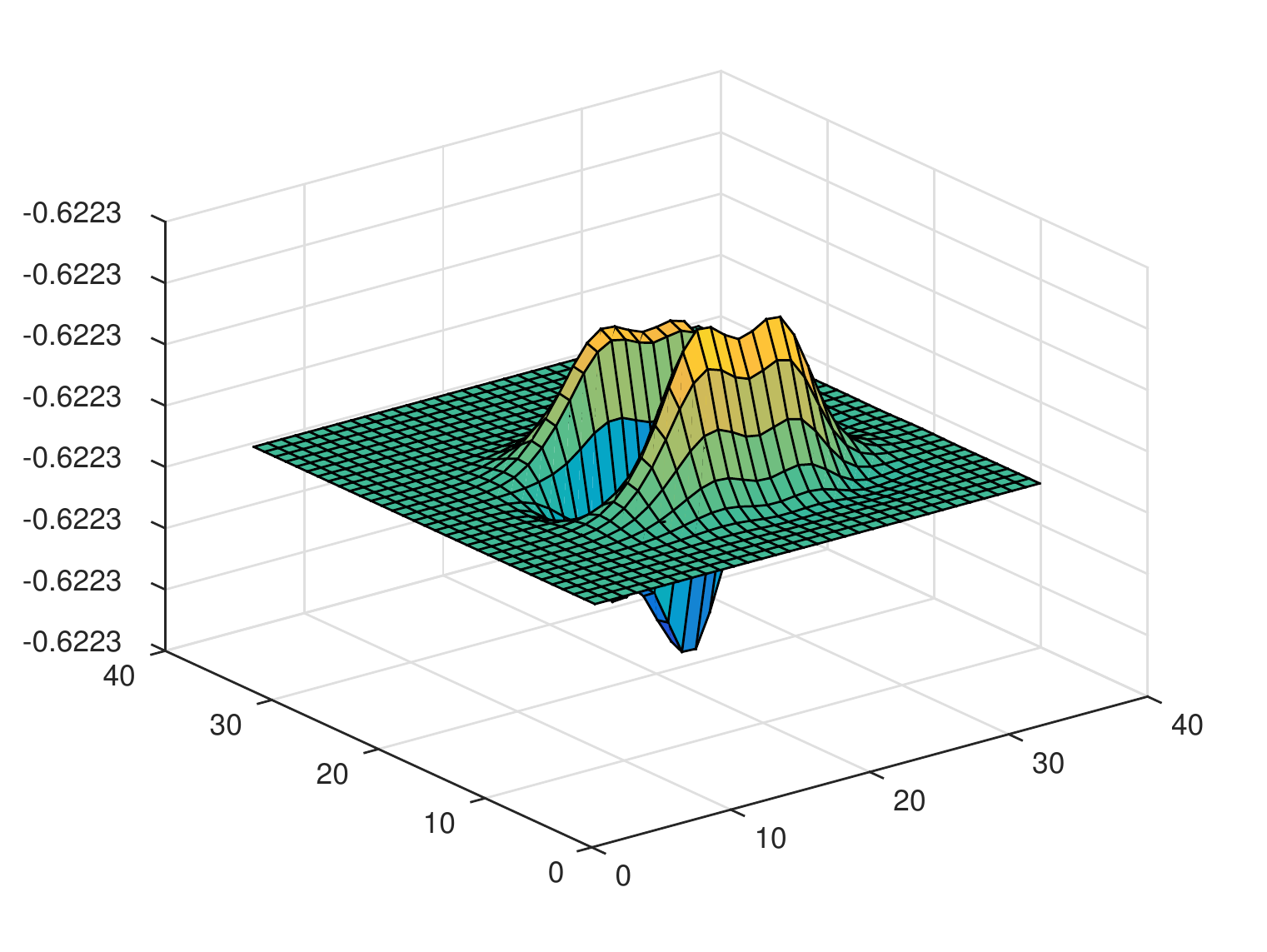}} &
{\includegraphics[scale=0.4, clip=true,]{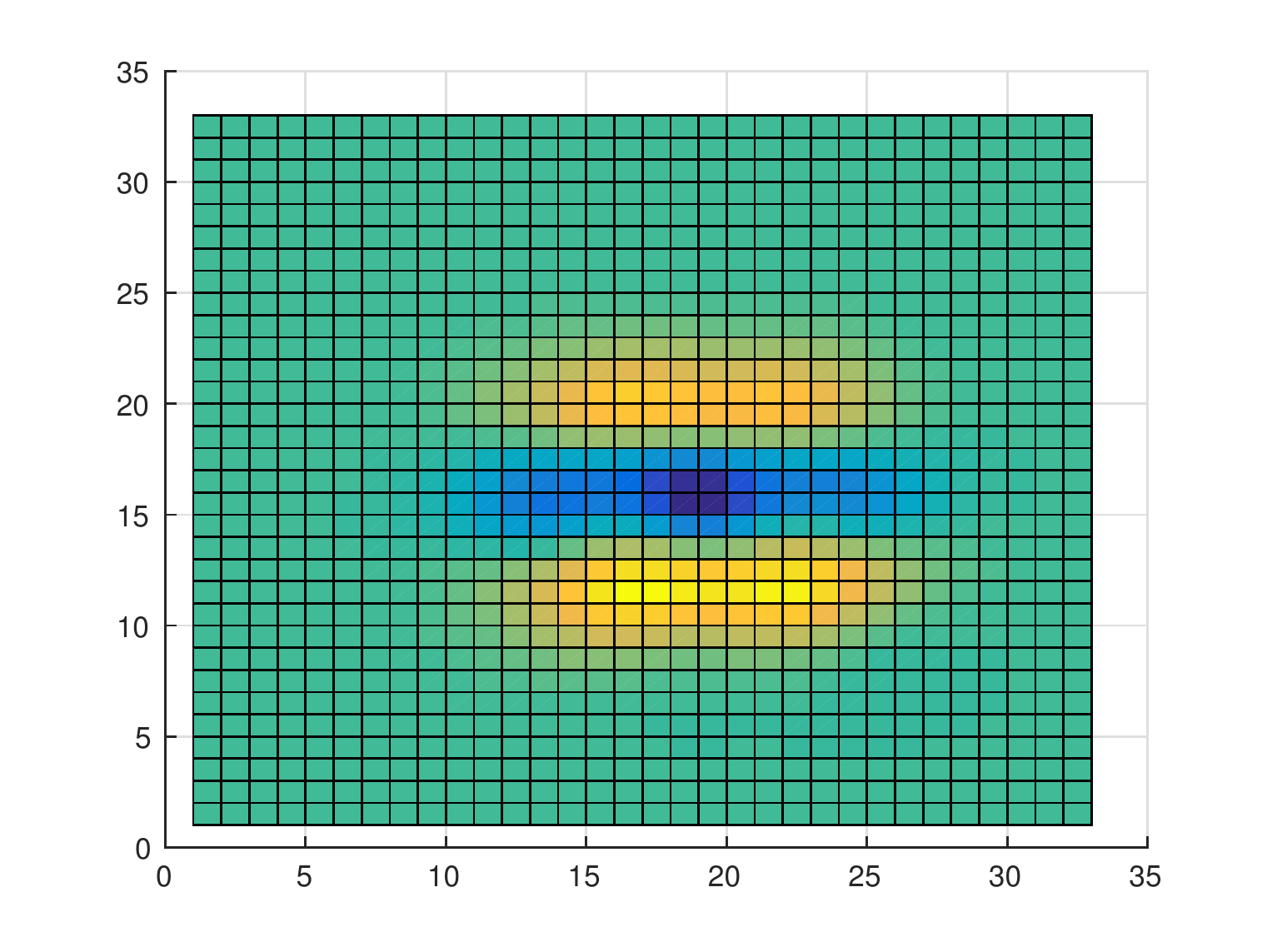}} \\
$t=0.6$  & $t=0.6$, $x_1 x_2 $ view \\
{\includegraphics[scale=0.4, clip=true,]{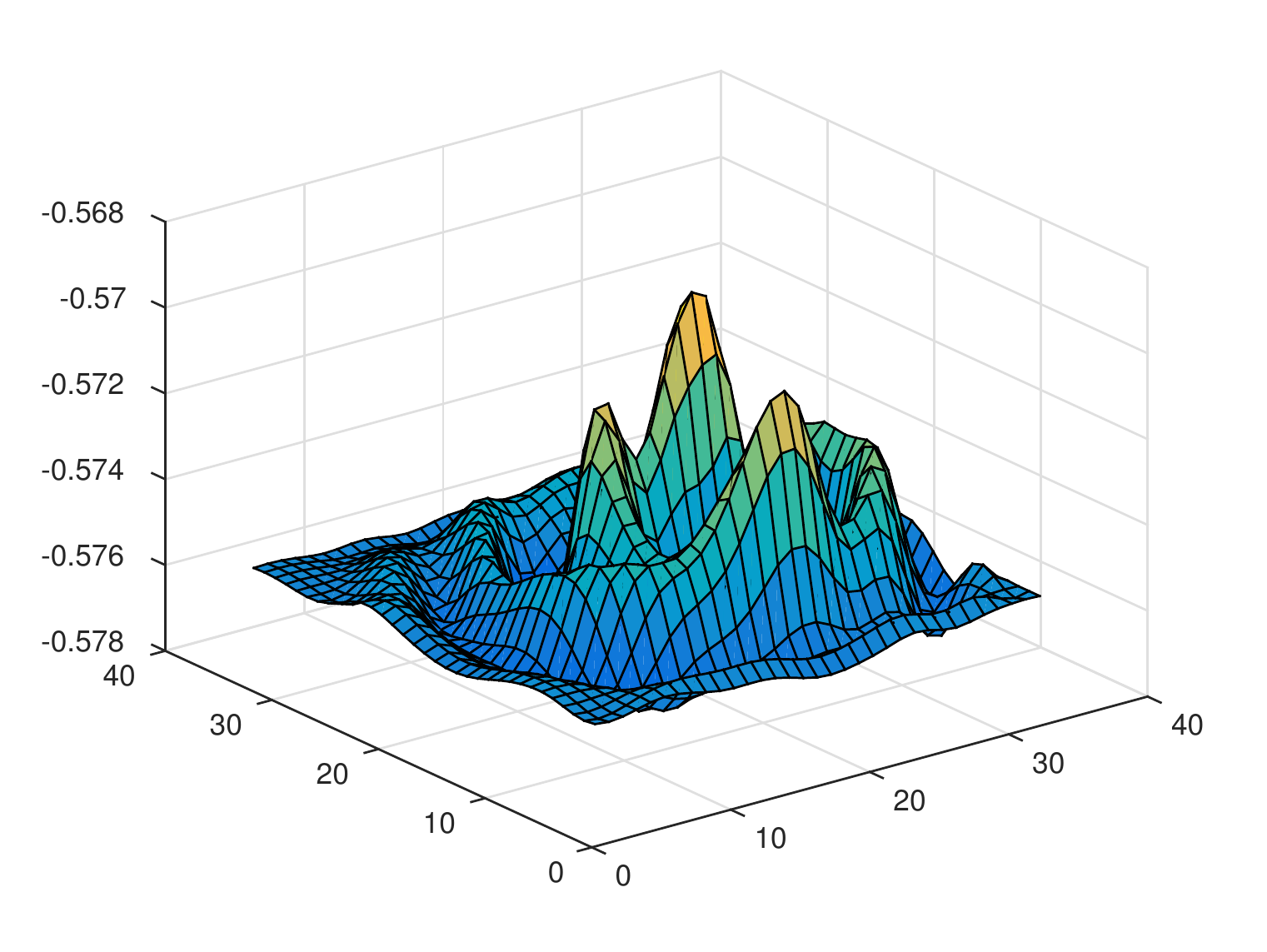}} &
{\includegraphics[scale=0.4, clip=true,]{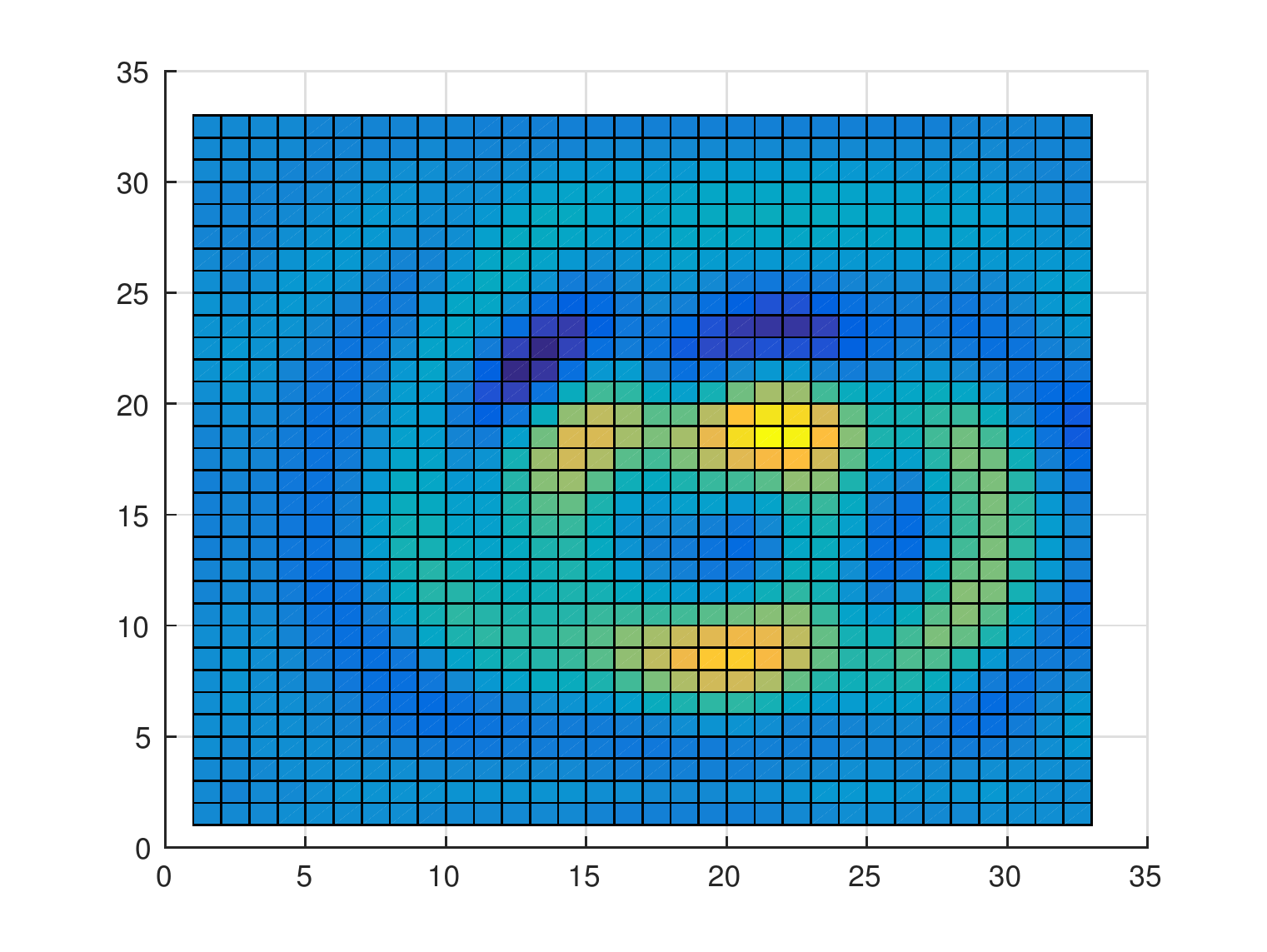}} \\
$t=1.2$  & $t=1.2$, $x_1 x_2$ view \\
{\includegraphics[scale=0.4, clip=true,]{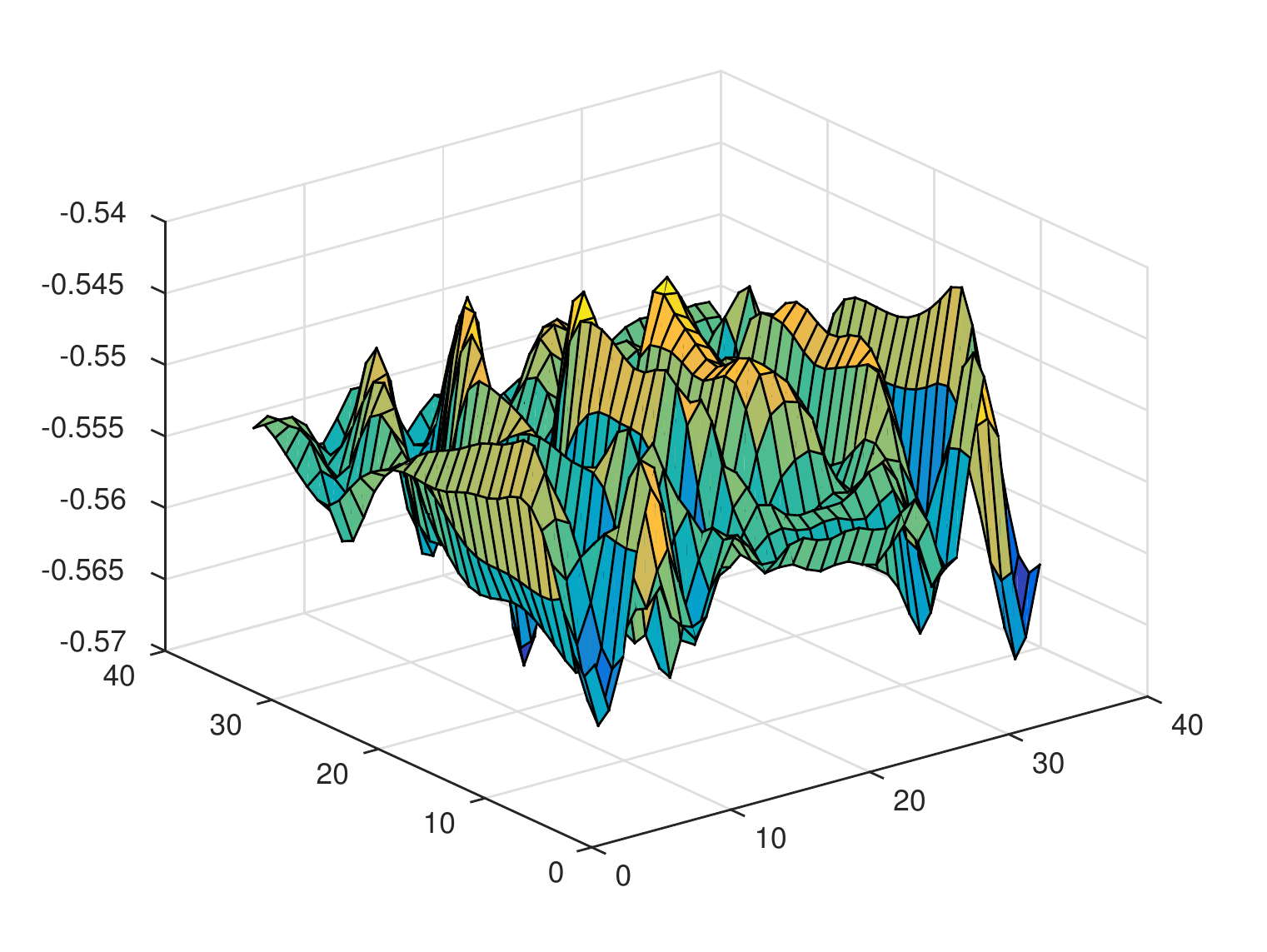}} &
{\includegraphics[scale=0.4, clip=true,]{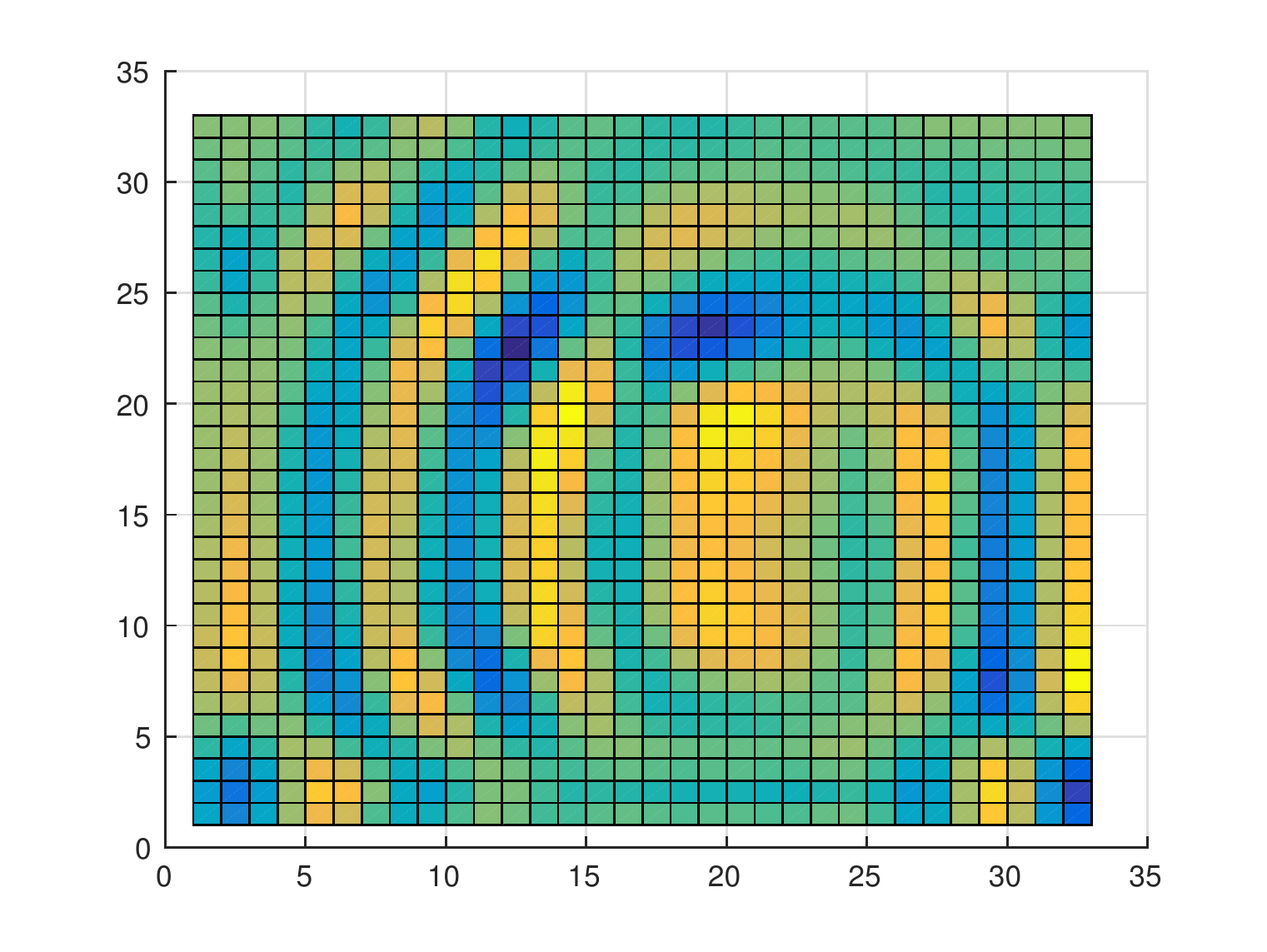}} \\
$t=1.8$  & $t=1.8$, $x_1 x_2$ view \\
\end{tabular}
\end{center}
\caption{\small\emph{Test 4. Backscattered data of the component $E_2$ at $\partial_2 \widetilde{\Omega}$ at different times.  The level of noise in the data is $\sigma=10\%$.}}
\label{fig:test4datab}
\end{figure}

\begin{figure}
\begin{center}
\begin{tabular}{cc}
$\x\in\Omega_\mathrm{FEM}:\eps_{h_0}(\x) = 2.03 $ & $\x\in\Omega_\mathrm{FEM}:\eps_\mathrm{rec}(\x) = 2.15 $ \\
{\includegraphics[scale=0.25, trim = 2.0cm 6.0cm 2.0cm 6.0cm, clip=true,]{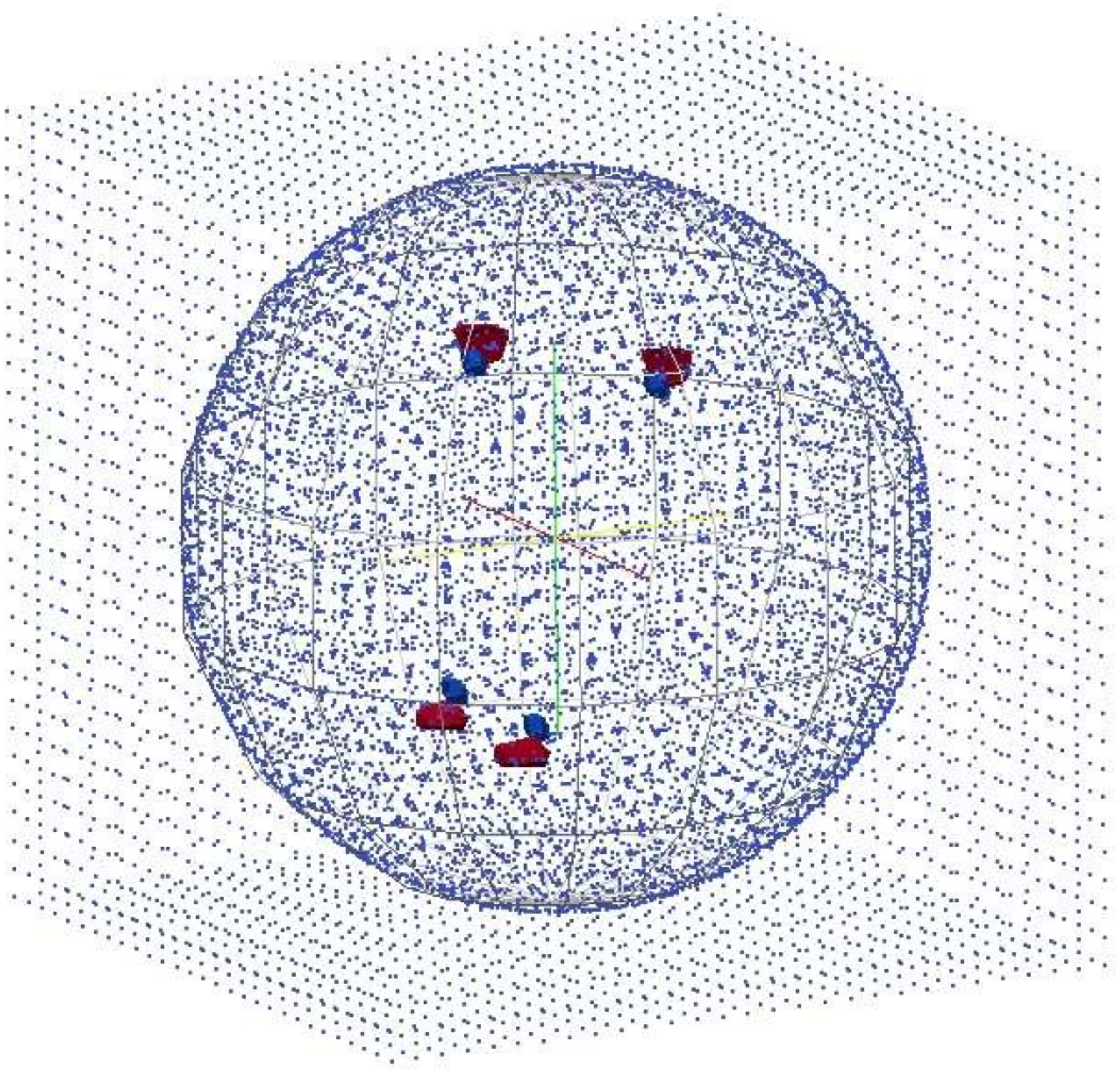}} &
{\includegraphics[scale=0.25, trim = 2.0cm 6.0cm 2.0cm 6.0cm, clip=true,]{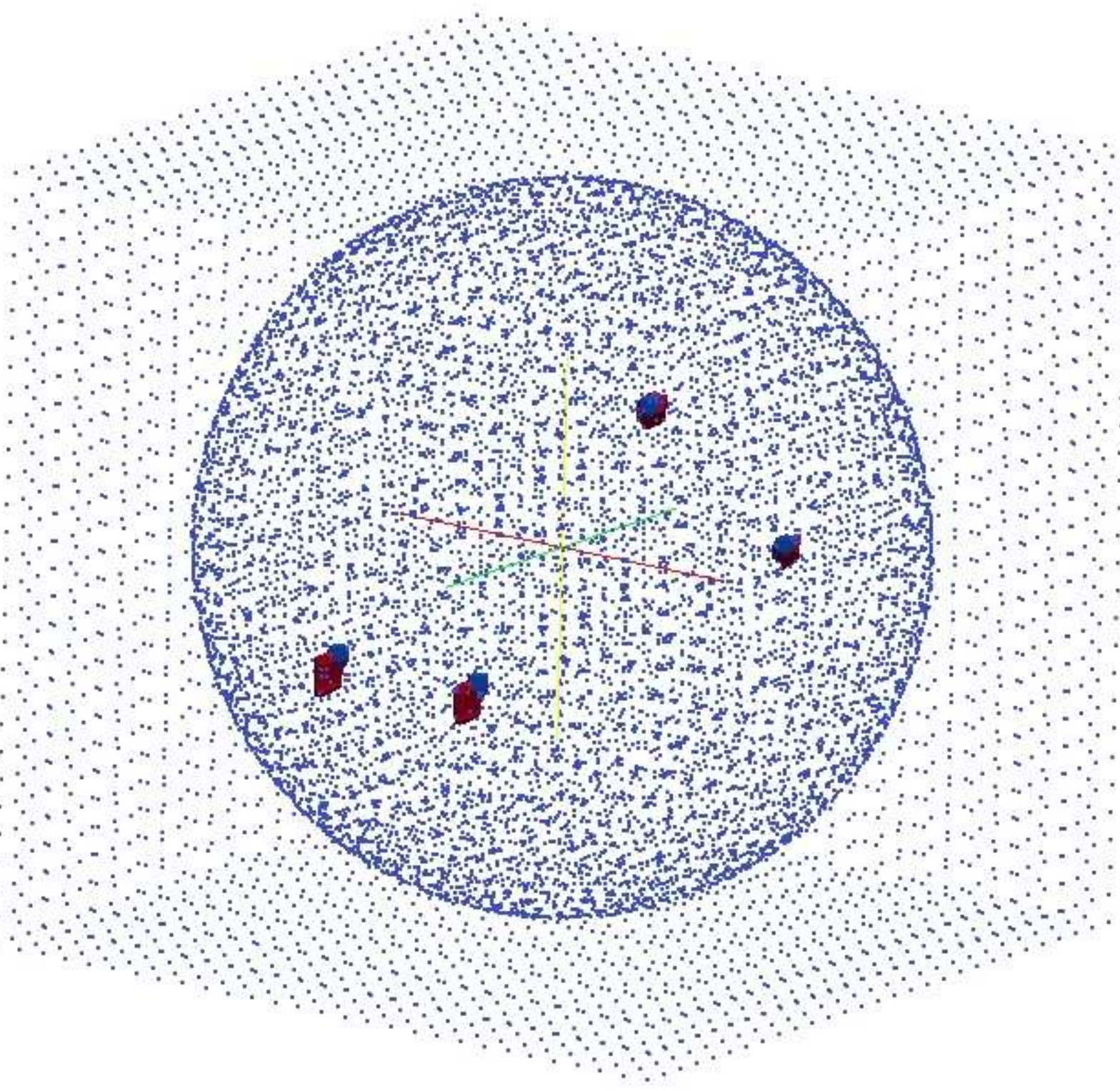}}  \\
a) prospect view &  b) prospect view \\
{\includegraphics[scale=0.25, trim = 2.0cm 6.0cm 2.0cm 6.0cm, clip=true,]{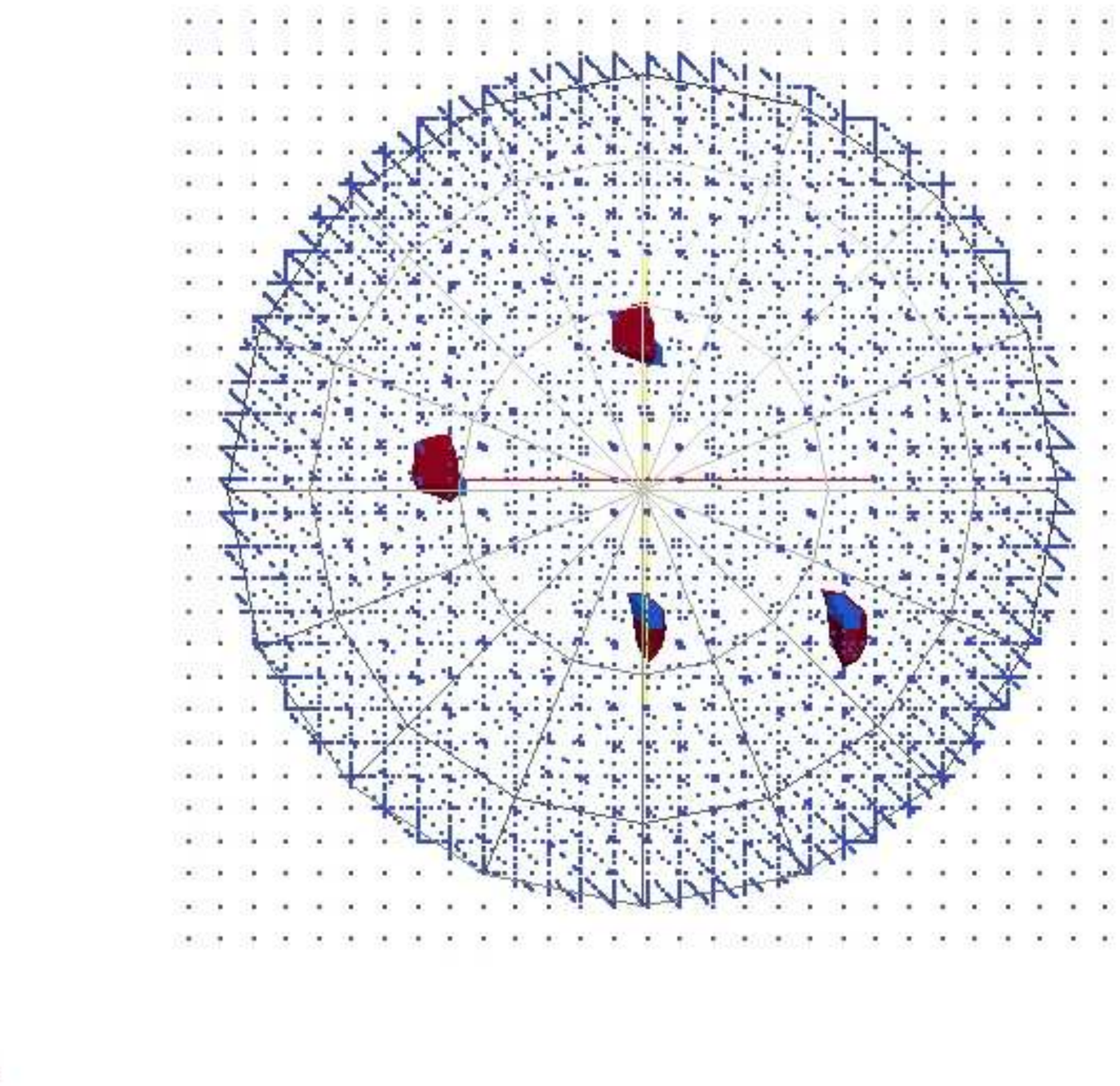}} &
{\includegraphics[scale=0.25, trim = 2.0cm 6.0cm 2.0cm 6.0cm, clip=true,]{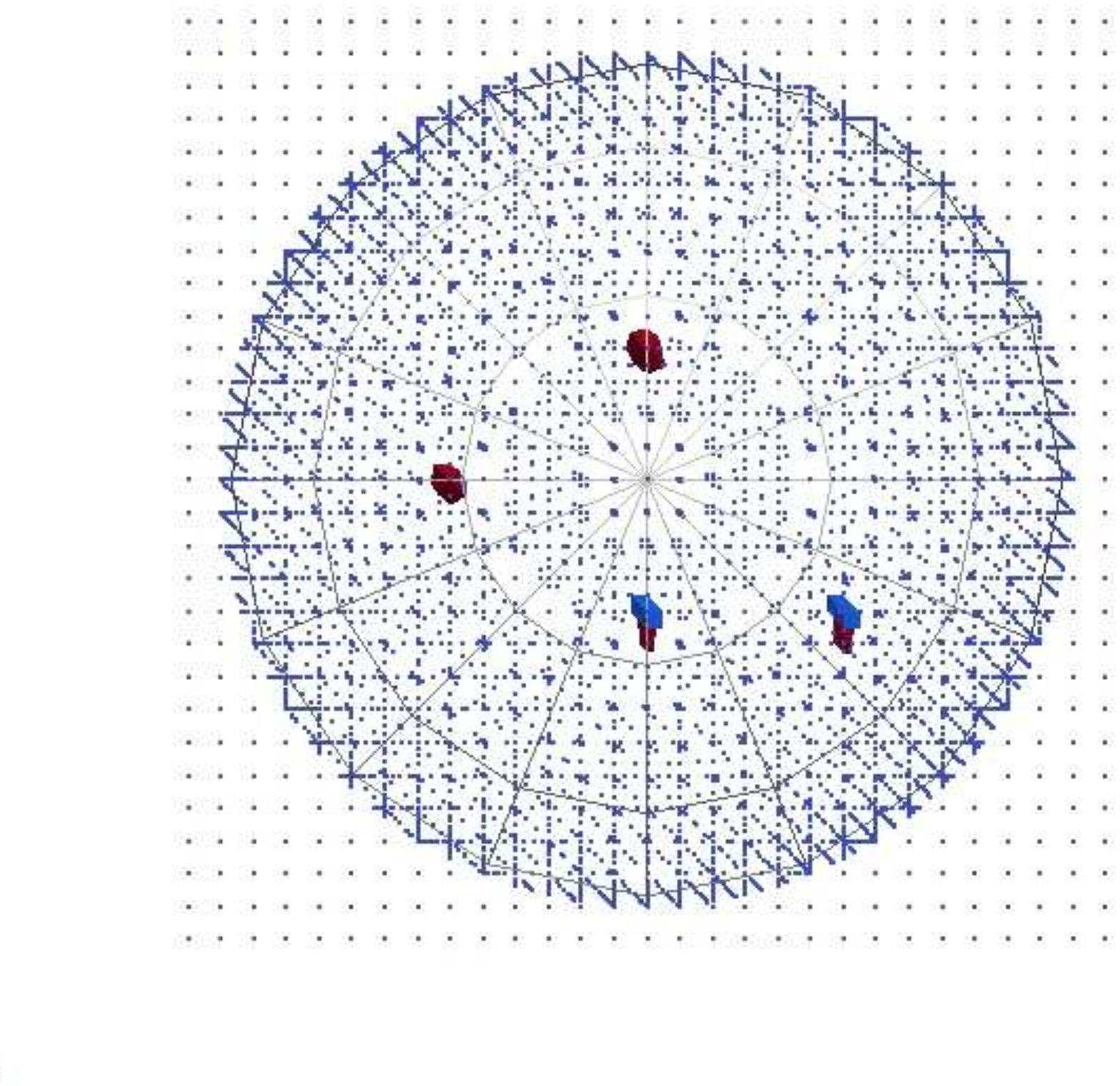}} \\
c)  $x_1 x_2$ view & d)  $x_1 x_2$ view \\ 
{\includegraphics[scale=0.25, trim = 2.0cm 6.0cm 2.0cm 6.0cm, clip=true,]{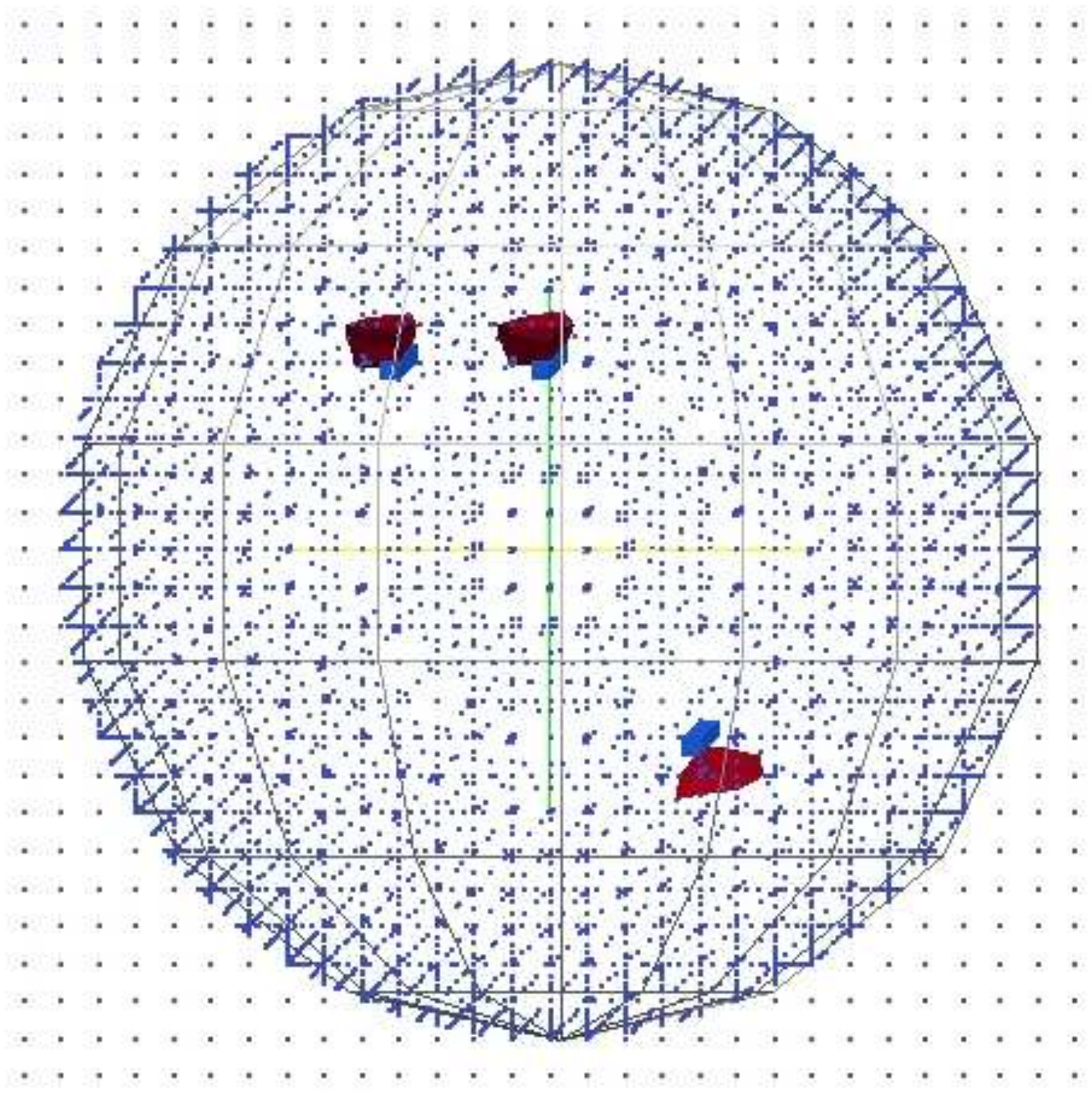}} &
{\includegraphics[scale=0.25, trim = 2.0cm 6.0cm 2.0cm 6.0cm, clip=true,]{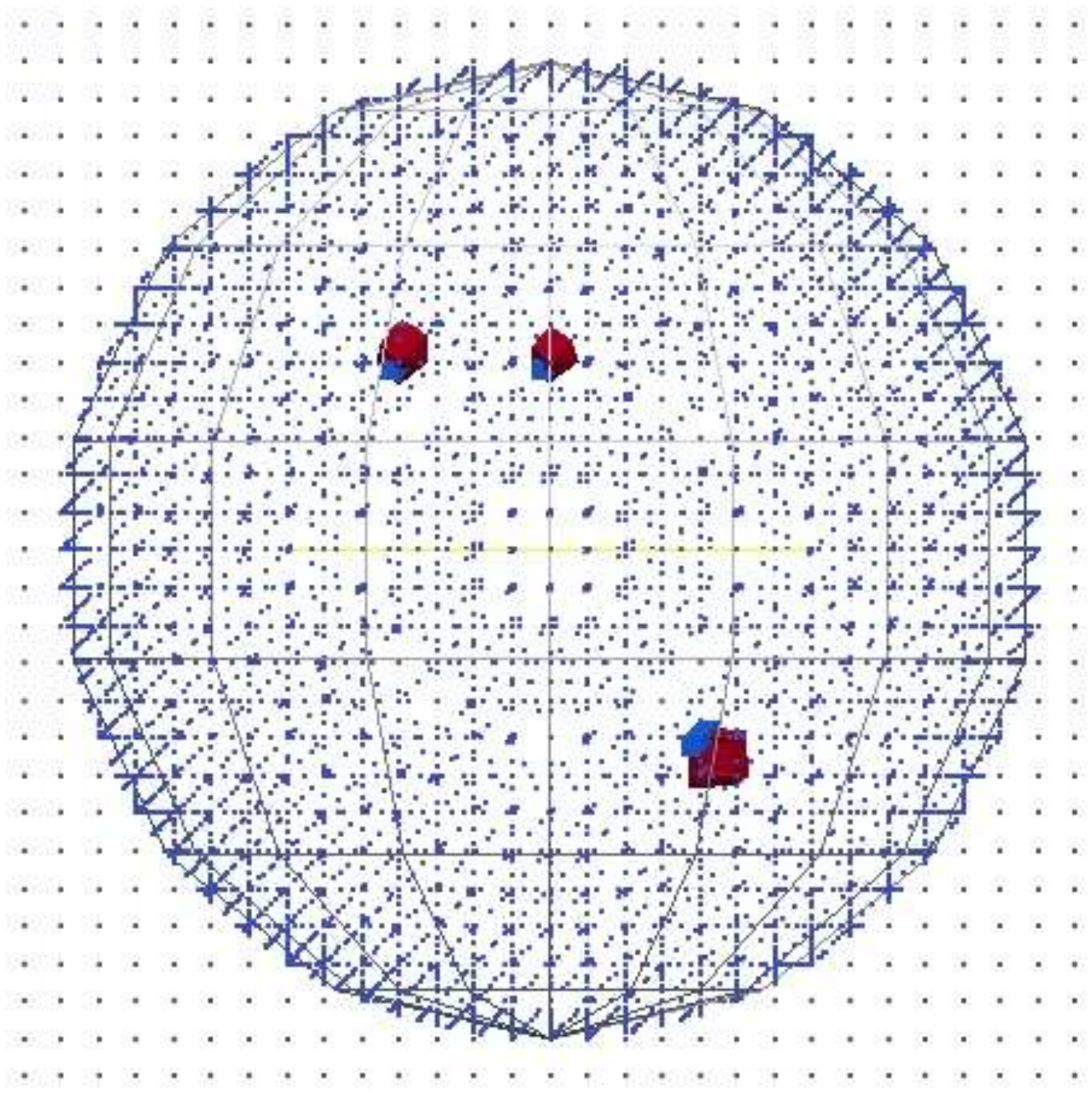}} \\
e) $x_2 x_3$ view & f)   $x_2 x_3$ view \\
{\includegraphics[scale=0.25, trim = 2.0cm 6.0cm 2.0cm 6.0cm, clip=true,]{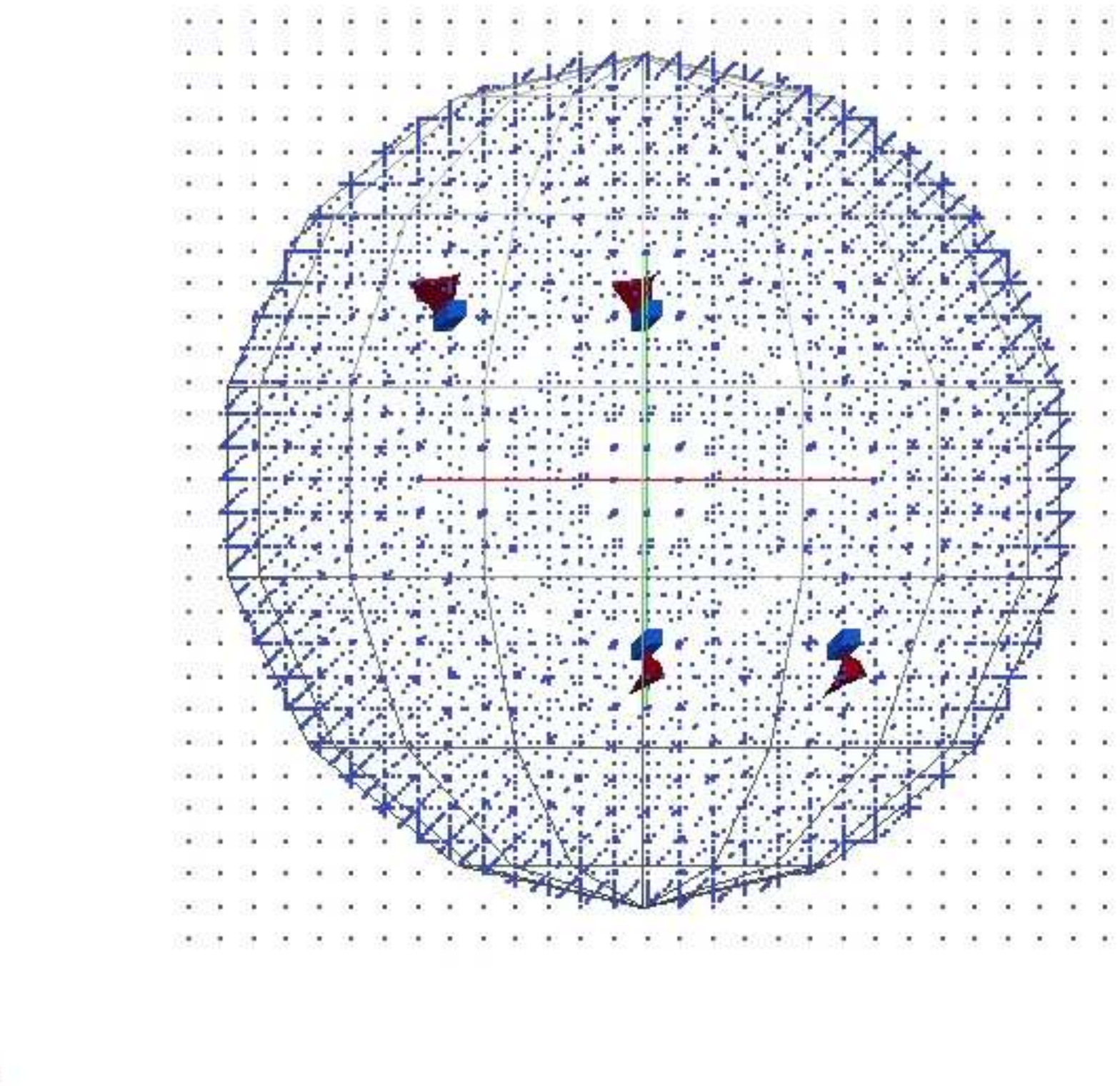}} &
{\includegraphics[scale=0.25, trim = 2.0cm 6.0cm 2.0cm 6.0cm, clip=true,]{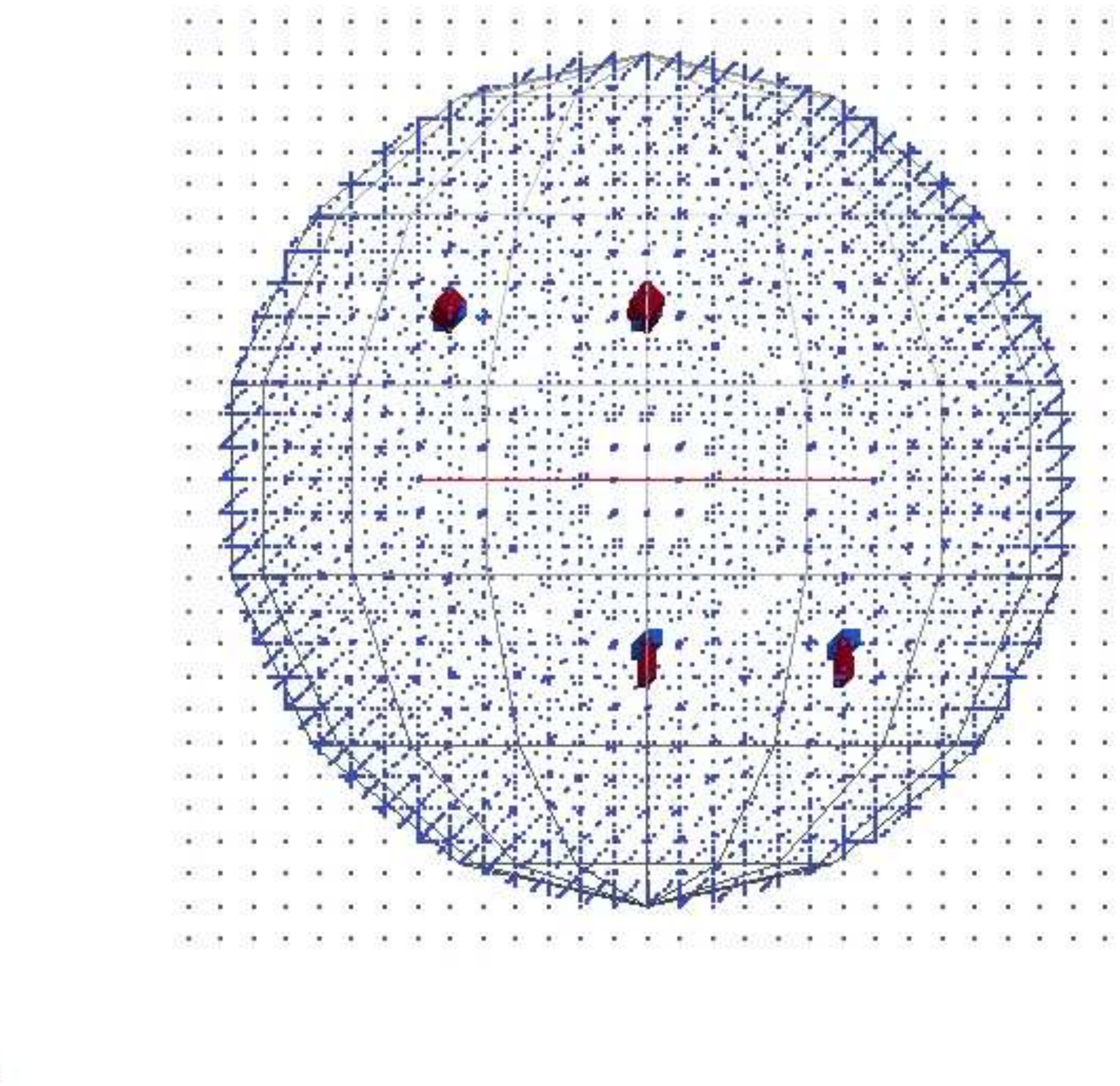}} \\
g) $x_1 x_3$ view & h) $x_1 x_3$ view
\end{tabular}
\end{center}
\caption{\small\emph{Test 4. Reconstruction obtained with two plane waves. We present reconstruction of the four inclusions (in red color) obtained on the coarse mesh (left figures) and on the two times adaptively refined mesh (right figures). The level of noise in the data is $\sigma=10\%$. For comparison we also present exact isosurfaces of the four small inclusions to be reconstructed (in light blue color).}}
\label{fig:test4noise10}
\end{figure}

\subsubsection{Test 4}
In this test we decided improve the results of Test~3 and we tried to reconstruct the four inclusions using measurements of two backscattered wave fields. First, we initialized a plane wave at the front boundary $\partial_1\widetilde{\Omega} $ in time $[0,\,T]$ and collected backscattered data here. Next, we initialized a plane wave at the back boundary $\partial_2\widetilde{\Omega}$ in time $[T,\,2T]$ and collected backscattered data for this wave field. We choose $T = 3$ and $\tau = 0.006$ as in all previous tests. Figure~\ref{fig:test4noise10} and Tables 1--2 show the results of the reconstruction. Now we see that all inclusions are of the same size and they are reconstructed in correct positions and with correct contrasts on the coarse mesh as well as on the refined mesh. From Tables 1--2 we see that adaptive algorithm converged already after first mesh refinement. The drawback of this test is that computations of one optimization iteration took twice as much time as the corresponding iterations of the previous tests, because of two measurements of the backscattered wave fields.

\section{Conclusion} \label{conclusion}
In this work we have derived two a posteriori error estimates, for the direct error in the approximated permittivity as well as in the Tikhonov functional, in the finite element approximation of the Lagrangian approach to our coefficient inverse problem. Both estimates are consisting of two parts which can be interpreted as representing the error incurred by the approximation of the solution to the direct and adjoint problems, and the error incurred by the approximation of the coefficient itself, respectively. These estimates are important in the adaptive algorithms we have studied. Moreover, they further justify the use of similar error indicators in our previous works \cite{btkm14, btkm14b}.

Numerically we have tested the adaptive algorithms with two different additive noise levels, $\sigma = 3 \%$, and $10\%$, in the data. Our numerical tests show that with mesh refinements, as was expected, the quality of the reconstruction is improved a lot. Compare, for example, the results of Figure~\ref{fig:test1noise10coarse} with those of Figure~\ref{fig:test1noise10ref5}. Using these figures and tables~1 and 2 we observe that, with mesh refinements, the artifacts obtained on a coarse mesh are removed and the reconstructed function $\eps_{h_k}$ has more correct location in $x_3$ direction.

We can conclude that we have supported the tests of our previous works \cite{b02, bh15, bj05, bc06, b13, btkm14, btkm14b, bkk10}, and have shown that the adaptive finite element method is powerful tool for the reconstruction of coefficients in Maxwell's equations from limited observations.

Our adaptive algorithms can also be applied for the case when edge elements are used for the numerical simulation of the solutions of forward and adjoint problems, see \cite{cz99, cwz14} for finite element analysis in this case. This, as well as development of different techniques for iterative choice of regularization the parameter in the Tikhonov functional, can be considered as a challenge for future research.

For the applications to medical imaging, an additional important challenge, bot from a computational and theoretical point of view, is the extension of the adaptive Lagrangian methods to the case of a complex coefficient. This corresponds to considering a conductive medium, which is more realistic for organic tissues.

\section*{Acknowledgment}
This research is supported by the Swedish Research Council (VR). The computations were performed on resources at Chalmers Centre for Computational Science and Engineering (C3SE) provided by the Swedish National Infrastructure for Computing (SNIC).

\bibliographystyle{unsrt}
\bibliography{../../mybibliography}

\end{document}